\documentclass[12pt]{amsart}
\usepackage{amssymb}
\usepackage{longtable}
\usepackage{hyperref}
\usepackage{tikz-cd}
\usepackage[total={6.25truein,9truein},centering]{geometry}

\allowdisplaybreaks

\numberwithin{equation}{subsection}

\newtheorem{theorem}[equation]{Theorem}
\newtheorem{lemma}[equation]{Lemma}
\newtheorem{proposition}[equation]{Proposition}
\newtheorem{corollary}[equation]{Corollary}

\newtheorem{theoremintro}{Theorem}
\newtheorem{propositionintro}[theoremintro]{Proposition}

\newtheorem{corollaryintro}[theoremintro]{Corollary}

\theoremstyle{definition}
\newtheorem{remark}[equation]{Remark}
\newtheorem{definition}[equation]{Definition}
\newtheorem{example}[equation]{Example}
\newtheorem{question}{Question}
\newtheorem*{acknowledgments}{Acknowledgments}

\theoremstyle{remark}

\newcommand{\FF}{\mathbb{F}}
\newcommand{\ZZ}{\mathbb{Z}}

\newcommand{\NN}{\mathbb{N}}
\newcommand{\TT}{\mathbb{T}}
\newcommand{\GG}{\mathbb{G}}
\newcommand{\EE}{\mathbb{E}}
\newcommand{\LL}{\mathbb{L}}
\newcommand{\MM}{\mathbb{M}}
\newcommand{\CC}{\mathbb{C}}

\newcommand{\KK}{\mathbb{K}}

\newcommand{\bA}{\mathbf{A}}

\newcommand{\bFF}{\mathbf{F}}

\newcommand{\cD}{\mathcal{D}}
\newcommand{\cE}{\mathcal{E}}
\newcommand{\cF}{\mathcal{F}}
\newcommand{\cG}{\mathcal{G}}

\newcommand{\cM}{\mathcal{M}}
\newcommand{\cN}{\mathcal{N}}

\newcommand{\cW}{\mathcal{W}}

\newcommand{\rd}{\mathrm{d}}

\newcommand{\rF}{\mathrm{F}}
\newcommand{\rG}{\mathrm{G}}
\newcommand{\rH}{\mathrm{H}}
\newcommand{\rJ}{\mathrm{J}}
\newcommand{\rN}{\mathrm{N}}
\newcommand{\rP}{\mathrm{P}}
\newcommand{\rQ}{\mathrm{Q}}

\newcommand{\fg}{\mathfrak{g}}
\newcommand{\fG}{\mathfrak{G}}
\newcommand{\fh}{\mathfrak{h}}

\newcommand{\bsalpha}{\boldsymbol{\alpha}}
\newcommand{\bsa}{\boldsymbol{a}}
\newcommand{\bsA}{\mathbf{A}}
\newcommand{\bsbeta}{\boldsymbol{\beta}}
\newcommand{\bsb}{\boldsymbol{b}}
\newcommand{\bsc}{\boldsymbol{c}}
\newcommand{\bsdelta}{\boldsymbol{\delta}}
\newcommand{\bsd}{\boldsymbol{d}}
\newcommand{\bsD}{\boldsymbol{D}}
\newcommand{\bse}{\boldsymbol{e}}
\newcommand{\bsepsilon}{\boldsymbol{\epsilon}}
\newcommand{\bsf}{\boldsymbol{f}}
\newcommand{\bsg}{\boldsymbol{g}}

\newcommand{\bsgamma}{\boldsymbol{\gamma}}
\newcommand{\bsh}{\boldsymbol{h}}
\newcommand{\bsi}{\boldsymbol{i}}
\newcommand{\bskappa}{\boldsymbol{\kappa}}

\newcommand{\bslambda}{\boldsymbol{\lambda}}
\newcommand{\bsm}{\boldsymbol{m}}
\newcommand{\bsmu}{\boldsymbol{\mu}}
\newcommand{\bsn}{\boldsymbol{n}}
\newcommand{\bsnu}{\boldsymbol{\nu}}
\newcommand{\bss}{\boldsymbol{s}}

\newcommand{\bsvartheta}{\boldsymbol{\vartheta}}
\newcommand{\bsu}{\boldsymbol{u}}
\newcommand{\bsU}{\boldsymbol{U}}
\newcommand{\bsv}{\boldsymbol{v}}
\newcommand{\bsw}{\boldsymbol{w}}
\newcommand{\bsx}{\boldsymbol{x}}
\newcommand{\bsxi}{\boldsymbol{\xi}}

\newcommand{\bsy}{\boldsymbol{y}}

\newcommand{\bsz}{\boldsymbol{z}}
\newcommand{\bspr}{\boldsymbol{pr}}

\DeclareMathOperator{\Der}{Der}
\DeclareMathOperator{\diag}{diag}
\DeclareMathOperator{\DR}{DR}
\DeclareMathOperator{\Exp}{Exp}

\DeclareMathOperator{\Log}{Log}

\DeclareMathOperator{\GL}{GL}
\DeclareMathOperator{\Lie}{Lie}
\DeclareMathOperator{\Mat}{Mat}

\DeclareMathOperator{\End}{End}

\DeclareMathOperator{\Hom}{Hom}

\DeclareMathOperator{\rank}{rank}
\DeclareMathOperator{\Res}{Res}

\newcommand{\oK}{\mkern2.5mu\overline{\mkern-2.5mu K}}

\newcommand{\bzero}{\mathbf{0}}
\newcommand{\Id}{\mathrm{I}}
\newcommand{\inn}{\mathrm{in}}
\newcommand{\si}{\mathrm{si}}
\newcommand{\sr}{\mathrm{sr}}
\newcommand{\sep}{\mathrm{sep}}

\newcommand{\tbsD}{\widetilde{\boldsymbol{D}}}
\newcommand{\tC}{\widetilde{C}}
\newcommand{\tbslambda}{\widetilde{\boldsymbol{\lambda}}}
\newcommand{\tpi}{\widetilde{\pi}}

\newcommand{\tPhi}{\widetilde{\Phi}}

\newcommand{\trho}{\widetilde{\rho}}
\newcommand{\tU}{\widetilde{U}}
\newcommand{\tV}{\widetilde{V}}

\newcommand{\Ga}{\GG_{\mathrm{a}}}

\newcommand{\iso}{\stackrel{\sim}{\longrightarrow}}
\newcommand{\power}[2]{{#1 [\![ #2 ]\!]}}
\newcommand{\laurent}[2]{{#1 (\!( #2 )\!)}}
\newcommand{\Span}[1]{\operatorname{Span}_{#1}}

\newcommand{\dnorm}[1]{\lVert #1 \rVert}
\newcommand{\inorm}[1]{{\lvert #1 \rvert}_{\infty}}

\newcommand{\pd}{\partial}

\newcommand{\tr}{{\mathsf{T}}}
\newcommand{\twistop}[2]{\langle #1 \mathbin{|} #2 \rangle}
\newcommand{\mayeq}{\stackrel{?}{=}}
\newcommand{\assign}{\mathrel{\vcenter{\baselineskip0.5ex \lineskiplimit0pt
                     \hbox{\scriptsize.}\hbox{\scriptsize.}}}%
                     =}
\newcommand{\rassign}{=%
                     \mathrel{\vcenter{\baselineskip0.5ex \lineskiplimit0pt
                     \hbox{\scriptsize.}\hbox{\scriptsize.}}}%
                     }

\begin{document}

\title[Hyperderivatives of periods and quasi-periods]{Hyperderivatives of periods and quasi-periods \\ for Anderson $t$-modules}

\author[C.~Namoijam]{Changningphaabi Namoijam}
\address{Department of Mathematics, National Tsing Hua University, Hsinchu City 30042, Taiwan
R.O.C.}
\email{namoijam@math.nthu.edu.tw}

\author[M.~A.~Papanikolas]{Matthew A. Papanikolas}
\address{Department of Mathematics, Texas A{\&}M University, College Station,
TX 77843, U.S.A.}
\email{papanikolas@tamu.edu}

\thanks{This project was partially supported by NSF Grant DMS-1501362. The first author was supported by MoST Grant 109-2811-M-007-554.}


\date{December 27, 2021}

\begin{abstract}
We investigate periods, quasi-periods, logarithms, and quasi-logarithms of Anderson $t$-modules, as well as their hyperderivatives.  We develop a comprehensive account of how these values can be obtained through rigid analytic trivializations of abelian and $\bA$-finite $t$-modules.  To do this we build on the exponentiation theorem of Anderson and investigate quasi-periodic extensions of $t$-modules through Anderson generating functions.  By applying these results to prolongation $t$-modules of Maurischat, we integrate hyperderivatives of these values together with previous work of Brownawell and Denis in this framework.
\end{abstract}


\maketitle

\tableofcontents


\section{Introduction} \label{S:Intro}

\subsection{Motivating problems from transcendence}
Over the past decades, Drinfeld modules and Anderson $t$-modules have provided an abundant supply of quantities of interest in transcendental number theory over the rational function field $K=\FF_q(\theta)$ in one variable over a finite field, where $q$ is the positive power of a prime number $p$.  Periods, quasi-periods, logarithms, and quasi-logarithms associated to $t$-modules defined over the algebraic closure~$\oK$ of~$K$ give rise to questions about transcendence and algebraic independence that were first systematically investigated by Yu~\cite{Yu85}--\cite{Yu97}.

For example, taking $K_{\infty} = \laurent{\FF_q}{1/\theta}$ to be the completion of $K$ with respect to its $\infty$-adic absolute value $\inorm{\,\cdot\,}$ and $\KK$ to be the completion of an algebraic closure of $K_{\infty}$, the \emph{Carlitz period}
\begin{equation} \label{E:Carlitzpidef}
  \tpi \assign -(-\theta)^{q/(q-1)} \prod_{i=1}^{\infty} \Bigl( 1 - \theta^{1-q^i} \Bigr)^{-1} \in K_{\infty}\bigl( (-\theta)^{1/(q-1)} \bigr)
\end{equation}
is the fundamental period of the Carlitz module and lies within the separable closure of $K_{\infty}$ within~$\KK$.  It was established by Wade~\cite{Wade41} that $\tpi$ is transcendental over~$\oK$.  The Anderson-Thakur power series
\begin{equation} \label{E:Omegadef}
  \Omega \assign (-\theta)^{-q/(q-1)} \prod_{i=1}^{\infty} \biggl( 1 - \frac{t}{\theta^{q^i}} \biggr) \in \TT_{\theta}^{\times}
\end{equation}
converges in the Tate algebra $\TT_{\theta} \subseteq \power{\KK}{t}$ of power series that are regular on the closed disk of radius $|\theta|_{\infty}$ in $\KK$ with respect to the Gauss norm $\dnorm{\,\cdot\,}_{\theta}$.  It discernibly evaluates as
\begin{equation} \label{E:Omegaeval}
  \Omega|_{t=\theta} = -\frac{1}{\tpi}.
\end{equation}
Furthermore, if we define \emph{Frobenius twisting} for $n \in \ZZ$ as the assignment $f \mapsto f^{(n)} : \laurent{\KK}{t} \to \laurent{\KK}{t}$ given by $\sum a_i t^i \mapsto \sum a_i^{q^n} t^i$, then $\Omega$ satisfies the difference equation
\begin{equation} \label{E:Omegarat}
  \Omega^{(-1)} = (t-\theta)\Omega.
\end{equation}
This difference equation furnishes $\Omega$ as the \emph{rigid analytic trivialization} of the Carlitz module (see \S\ref{subS:AndersonThm}), and the evaluation \eqref{E:Omegaeval} manifests an early glimpse at the connections between rigid analytic trivializations and periods of $t$-modules.

Now one can consider the family of hyperdifferential operators $\pd_{\theta}^j : \FF_q[\theta] \to \FF_q[\theta]$ defined $\FF_q$-linearly by $\pd_{\theta}^j(\theta^n) \assign \binom{n}{j} \theta^{n-j}$, where $\binom{n}{j}$ is the usual binomial coefficient but modulo $p$, and which extend uniquely to operators
\[
  \pd_{\theta}^j : K_{\infty}^{\sep} \to K_{\infty}^{\sep}, \quad j \geqslant 0.
\]
In contrast to characteristic~$0$, we can apply these hyperdifferential operators to our transcendental quantities themselves and ask in the case of the Carlitz module if
\[
  \tpi,\, \pd_{\theta}^1(\tpi),\, \pd_{\theta}^2(\tpi), \ldots
\]
are algebraically related over~$\oK$.  Denis~\cite{Denis93}--\cite{Denis97} showed that certain subsets of these quantities are $\oK$-linearly independent, and later that the first $p$ are algebraically independent~\cite{Denis00}.  More recently Maurischat~\cite{Maurischat21b} proved that these values are in fact all algebraically independent over~$\oK$, using his construction of \emph{prolongations} of tensor powers of the Carlitz module~\cite{Maurischat18} and applying techniques of rigid analytic trivializations from~\cite{ABP04}, \cite{P08}.

Other investigations by Anderson, Brownawell, Chang, Denis, Thakur, Yu, and many others have produced transcendence results on function field $\Gamma$-values~\cite{ABP04}, \cite{BP02}, \cite{CPTY10}, \cite{CPY10}, \cite{Sinha97}, \cite{Thakur91}, \cite{Thakur96};  Drinfeld logarithms and quasi-logarithms~\cite{Brownawell93}--\cite{Brownawell98}, \cite{CP11}, \cite{CP12}, \cite{DavidDenis02}, \cite{Denis97}, \cite{Denis06}, \cite{P08}, \cite{Pellarin20}, \cite{Yu85}, \cite{Yu86}, \cite{Yu90}, \cite{Yu97}; zeta values and multiple zeta values~\cite{Chang14}--\cite{ChangMishiba21}, \cite{CPY11}--\cite{ChangYu07}, \cite{Green19a}, \cite{GreenNgoDac20a}, \cite{LP13}, \cite{Yu91}, \cite{Yu97}; and of particular interest to the present paper, hyperderivatives of Drinfeld logarithms and quasi-logarithms \cite{Brownawell98}--\cite{BrownawellDenis00}, \cite{Denis93}--\cite{Denis97}, \cite{Maurischat18}.

\begin{question} \label{Q:one}
For a $t$-module defined over~$K^{\sep}$, to what extent can one determine algebraic relations over~$\oK$ among hyperderivatives of its periods, quasi-periods, logarithms, and quasi-logarithms with respect to~$\theta$?
\end{question}

Brownawell and Denis~\cite{Brownawell98}--\cite{BrownawellDenis00} first investigated this question in the case of logarithms and quasi-logarithms of a Drinfeld module.  If we let $\bA \assign \FF_q[t]$ be the polynomial ring in~$t$ over~$\FF_q$ and take $\KK[\tau]$ to be the ring of twisted polynomials in the $q$-th power Frobenius $\tau : x \mapsto x^q$ on $\KK$, then we define a Drinfeld module to be the $\FF_q$-algebra homomorphism $\phi : \bA \to \KK[\tau]$ determined by the value
\[
  \phi_t = \theta + b_1\tau + \cdots + b_r \tau^r, \quad b_r \neq 0.
\]
Its exponential series $\exp_{\phi}(z) \in \power{\KK}{z}$ is $\FF_q$-linear, entire, and surjective on~$\KK$, and it satisfies the functional equation
\[
  \exp_{\phi}(a(\theta) z) = \phi_a(\exp_{\phi}(z)), \quad \forall\, a \in \bA.
\]
Its kernel is a free and finitely generated $\bA$-submodule $\Lambda_\phi$ of~$\KK$ of rank~$r$, and we choose generators $\lambda_1, \dots, \lambda_r$.  Using the theory of biderivations (see \S\ref{subS:biderivations}), one can construct entire $\FF_q$-linear power series $\rF_1(z), \dots, \rF_{r-1}(z) \in \power{\KK}{z}$ so that
\[
  \rF_i(\theta z) = \theta \rF_i(z) + \exp_{\phi}(z)^{q^i}, \quad 1 \leqslant i \leqslant r-1,
\]
and these form a maximal system of \emph{quasi-periodic functions} for $\phi$.  The quantities
\[
  \rF_i(\lambda_j), \quad 1 \leqslant i \leqslant r-1,\ 1 \leqslant j \leqslant r,
\]
comprise the \emph{quasi-periods} of $\phi$, and moreover, there is an extension of $t$-modules,
\[
  0 \to \Ga^{r-1} \to X \to \phi \to 0,
\]
which forms a quasi-periodic extension of $\phi$ (see \S\ref{subS:biderivations} or~\cite{Brownawell93}, \cite{Gekeler89a}, \cite{Goss94}, \cite{Yu90}, for more details).  Of particular note is that the period lattice of $X$ is generated by vectors whose coordinates consist of the periods and quasi-periods of $\phi$.  More generally, if $y \in \KK$, then $\rF_i(y)$ is a \emph{quasi-logarithm} associated to~$y$ (see \eqref{E:quasilogdef}).

When $\phi$ is defined over $K_{\infty}^{\sep}$, Brownawell and Denis defined for each $n \geqslant 0$ a $t$-module $\trho_n : \bA \to \Mat_{(n+1)r}(K_{\infty}^{\sep}[\tau])$ that sits in a short exact sequence,
\[
  0 \to \Ga^{(n+1)r-1} \to \trho_n \to \phi \to 0,
\]
with the following properties.
\begin{enumerate}
\item[(i)] Periods and logarithms of $K_{\infty}^{\sep}$-valued points on $\trho_n$ are given in terms of the first~$n$ hyperderivatives of periods, quasi-periods, logarithms, and quasi-logarithms of $K_{\infty}^{\sep}$-valued points on $\phi$.
\item[(ii)] The operation of $\bA$ on the power of $\Ga$ in $\trho_n$ is not through scalar multiplication, in contrast to any quasi-periodic extension, and as a result the Brownawell-Denis $t$-module $\trho_n$ is not isomorphic to any quasi-periodic extension of~$\phi$.
\end{enumerate}
As such, if $\phi$ is defined over $K^{\sep}$, the quasi-periodic extensions of $\phi$ and the hyperderivative $t$-modules of Brownawell and Denis produce quantities that should not be inherently algebraically related over $\oK$.  This was confirmed for $\oK$-linear relations by Brownawell~\cite{Brownawell01} using Yu's Theorem of the Sub-$t$-module~\cite{Yu97}.

For questions of algebraic independence, current methods revolve around techniques in~\cite{ABP04}, \cite{P08}, which require $t$-modules that have rigid analytic trivializations and so are abelian in the sense of Anderson~\cite{And86} or $\bA$-finite in the sense of \cite{ABP04}, \cite{HartlJuschka20}.  Unfortunately quasi-periodic extensions and Brownawell-Denis $t$-modules are not well-suited for these purposes because they are neither abelian nor $\bA$-finite.  However, results in~\cite{CP11}, \cite{CP12}, employing work of Anderson (see~\cite{Goss94}) and Pellarin~\cite{Pellarin08}, demonstrate that the rigid analytic trivialization of a Drinfeld module can be obtained through Anderson generating functions and that the specialization of this matrix yields the full complement of its periods and quasi-periods.  Work in~\cite{BP02} on special $\Gamma$-values demonstrates also that quasi-periods are incorporated in evaluations of rigid analytic trivializations of Sinha's soliton $t$-modules~\cite{Sinha97}.

To be more precise, for our Drinfeld module $\phi$ defined over $\KK$ and $y \in \KK$, we define the \emph{Anderson generating function} $\cG_y$ by the infinite series,
\begin{equation}
  \cG_y \assign \sum_{n=0}^\infty \exp_\phi \biggl( \frac{y}{\theta^{n+1}} \biggr) t^n \in \TT.
\end{equation}
As a function of $t$, $\cG_y$ converges on the closed unit disk of $\KK$, and thus is in the Tate algebra $\TT$ of the closed unit disk. As an example, for the Carlitz module the Anderson-Thakur function
\begin{equation} \label{E:omegadef}
  \omega \assign \frac{1}{(t-\theta)\Omega}
\end{equation}
turns out to be the Anderson generating function associated to~$\tpi$. In general $\cG_y$ satisfies several important properties (see \S\ref{subS:AGF}), and in particular,
\begin{equation} \label{E:AGFfneqintro}
  \theta \cG_y + b_1 \cG_y^{(1)} + \cdots + b_r \cG_y^{(r)} = t\cG_y + \exp_\phi(y)
\end{equation}
and
\begin{equation} \label{E:AGFResintro}
  \Res_{t=\theta}(\cG_y) = -y.
\end{equation}
Furthermore, as observed by Pellarin~\cite{Pellarin08} (see also \cite{Gekeler89b}, \cite{Goss94}), for each $1\leqslant i \leqslant r-1$ we have the relation with the quasi-logarithms of~$y$ given by
\begin{equation} \label{E:AGFquasiintro}
  \cG_y^{(i)}(\theta) = \rF_{i}(y).
\end{equation}
For the generators $\lambda_1, \dots, \lambda_r$ of the period lattice of $\phi$, we form the $r \times r$ matrix with entries in $\TT_\theta$,
\begin{equation} \label{E:Upsilonintro}
  \Upsilon \assign \begin{pmatrix}
  \cG_{\lambda_1}^{(1)} & \cdots & \cG_{\lambda_r}^{(1)} \\
  \vdots & & \vdots \\
  \cG_{\lambda_1}^{(r)} & \cdots & \cG_{\lambda_r}^{(r)}
  \end{pmatrix}.
\end{equation}
Using \eqref{E:AGFfneqintro}--\eqref{E:AGFquasiintro}, the entries of the specialization $\Upsilon|_{t=\theta}$ consist of linear combinations of the periods $\lambda_j$ and quasi-periods $\rF_{i}(\lambda_j)$ over the field of definition of~$\phi$ for $1 \leqslant i \leqslant r-1$ and $1\leqslant j \leqslant r$.  Thus if $\phi$ is defined over $\oK$, then the $\oK$-span of the entries of $\Upsilon|_{t=\theta}$ is the same as the $\oK$-span of the periods and quasi-periods of~$\phi$.  Likewise, for $y \in \KK$, the specialization
\[
  \bigl( \cG_y^{(1)}, \dots, \cG_y^{(r)} \bigr)\big|_{t=\theta}
\]
consists of $\oK$-linear combinations of $y$ and the quasi-logarithms associated to $y$.

For a $t$-module $\phi : \bA \to \Mat_d(\KK[\tau])$ of dimension~$d$ (see \S\ref{subS:modules} for precise definitions), one can associate two algebraic structures to it, called its \emph{$t$-motive} $\cM_{\phi}$ and \emph{dual-$t$-motive} $\cN_{\phi}$.  The former is a left $\KK[t,\tau]$-module, where $\KK[t,\tau]$ is the polynomial ring in $t$ and~$\tau$, subject to the relation that $\tau f = f^{(1)}\tau$ for any $f \in \KK[t]$, and the latter is a left $\KK[t,\sigma]$-module, where $\KK[t,\sigma]$ is likewise the polynomial ring in $t$ and $\sigma = \tau^{-1}$ so that $\sigma f = f^{(-1)}\sigma$.  (See \S\ref{subS:modules}--\S\ref{subS:motives} for precise definitions of $t$-modules and $t$-motives.)  The $t$-motive $\cM_{\phi}$ is free and finitely generated over $\KK[\tau]$ of rank~$d$, and likewise the dual $t$-motive $\cN_{\phi}$ is free and finitely generated over $\KK[\sigma]$ of rank~$d$.  If $\cM_{\phi}$ is free and finitely generated over $\KK[t]$, then $\cM_{\phi}$ and $\phi$ are said to be \emph{abelian}; if $\cN_{\phi}$ is free and finitely generated over $\KK[t]$, then $\cN_{\phi}$ and $\phi$ are said to be \emph{$\bA$-finite}.  In the case that $\phi$ is $\bA$-finite, we can choose a $\KK[t]$-basis $n_1, \dots, n_r \in \cN_{\phi}$, which determines a matrix $\Phi \in \Mat_r(\KK[t])$ with the property that
\begin{equation} \label{E:Phidefintro}
  \sigma \bsn = \Phi \bsn, \quad \bsn = (n_1, \dots, n_r)^{\tr}.
\end{equation}
If we can find a matrix $\Psi \in \GL_r(\TT)$ so that
\begin{equation} \label{E:RATintro}
  \Psi^{(-1)} = \Phi \Psi,
\end{equation}
where on the left we take Frobenius twists entry-wise, then $\phi$ is said to be \emph{rigid analytically trivial}. It was proved by Anderson (see \cite{And86}) that $\phi$ is rigid analytically trivial if and only if $\phi$ is uniformizable, i.e, the exponential function of $\phi$ is surjective. In~\cite{CP12} it was shown for a Drinfeld module~$\phi$ that the matrix $\Upsilon$ in \eqref{E:Upsilonintro} is nearly a rigid analytic trivialization (see \eqref{E:PsiinvDrinfeld}), which turns out to be a special case of a more general phenomenon for $t$-modules by applying a theorem of Hartl and Juschka~\cite{HartlJuschka20} (see Theorem~\ref{T:isodual}).

\begin{question} \label{Q:two}
From the standpoint of transcendence theory, rigid analytic trivializations are well-suited for the techniques of Frobenius difference equations investigated in~\cite{ABP04}, \cite{P08}. For a general uniformizable, abelian, and $\bA$-finite $t$-module~$\phi$, how do we determine explicit descriptions of periods, quasi-periods, logarithms, and quasi-logarithms in terms of systems of Frobenius difference equations, and especially in terms of a rigid analytic trivialization for~$\phi$?
\end{question}

As we have seen, this kind of description for Drinfeld modules is relatively straightforward, but as we discover the situation for general $t$-modules is more complicated.  Our results for periods and logarithms are governed by a fundamental exponentiation theorem of Anderson (Theorem~\ref{T:Exponentiation}) as well as explicit descriptions of rigid analytic trivializations in terms of a general theory of Anderson generating functions that we develop.  We describe these investigations in more detail in the next two sections, and fully in~\S\ref{S:Exp}--\S\ref{S:Quasi}.  Questions for periods and logarithms were also investigated by Juschka~\cite{Juschka10}, and results on periods in these directions have been obtained by Maurischat~\cite{Maurischat22}. A full account of Anderson's exponentiation theory was made recently by Hartl and Juschka~\cite{HartlJuschka20}.

Indeed there are subtle connections between the $t$-motive and dual $t$-motive of a $t$-module that control these identities, particularly through Theorem~\ref{T:isodual} of Hartl and Juschka.  Much can be proved for general abelian $\bA$-finite $t$-modules, and we obtain a complete picture in the case that $\phi$ is \emph{almost strictly pure}.  If $\phi$ is defined by
\[
 \phi_t = \rd\phi_t + B_1 \tau + \cdots + B_{\ell} \tau^{\ell}, \quad B_i \in \Mat_d(\KK),
\]
where $\rd\phi_t$ denotes simply the constant term in this expression, then $\phi$ is \emph{strictly pure} if $\det(B_{\ell}) \neq 0$ and $\phi$ is \emph{almost strictly pure} if the top coefficient of $\phi_{t^s}$ is invertible for some $s \geqslant 1$.  Almost strictly pure $t$-modules are pure in the sense of~\cite{And86} (see Remark~\ref{R:ASPimpliespure}), and they account for many $t$-modules of general interest, such as Drinfeld modules, tensor powers of the Carlitz module, tensor products of Drinfeld modules, etc., and they are automatically both abelian and $\bA$-finite.  Denis~\cite{Denis92} also showed that almost strictly pure $t$-modules carry canonical height functions that serve as analogues of the N\'eron-Tate height on abelian varieties.

Returning to the situation of Question~\ref{Q:one}, we consider that positive answers to Question~\ref{Q:two} demonstrate that quasi-periods and quasi-logarithms that are defined using nonabelian and non-$\bA$-finite quasi-periodic $t$-modules actually exist in the world of rigid analytic trivializations of abelian and $\bA$-finite $t$-modules.  Perhaps hyperderivatives of all of these quantities, which also are found in the nonabelian and non-$\bA$-finite $t$-modules of Brownawell and Denis, themselves also reside in the setting of rigid analytic trivializations.

\begin{question} \label{Q:three}
For a general uniformizable, abelian, and $\bA$-finite $t$-module defined over $K^{\sep}$, is it possible to identify the $\oK$-vector space spanned by the hyperderivatives with respect to~$\theta$ of its periods, quasi-periods, logarithms, and quasi-logarithms in terms of the rigid analytic trivialization of an abelian, $\bA$-finite, and uniformizable $t$-module?
\end{question}

In~\S\ref{S:Hyper} we obtain a satisfactory answer to this question for all uniformizable, abelian, and $\bA$-finite $t$-modules, which we recount in \S\ref{subS:hyperintro} below.  Our investigations begin with the \emph{prolongation $t$-modules} of Maurischat~\cite{Maurischat18}, whose rigid analytic trivializations are given in terms of hyperderivatives \emph{with respect to $t$} of the functions comprising the rigid analytic trivialization of the given $t$-module.  We provide a complete description of the quasi-periodic extensions of Maurischat's prolongations, and there is considerable interplay among the hyperderivatives of the constituent functions with respect to~$t$ evaluated at $t=\theta$ and hyperderivatives of the values themselves with respect to~$\theta$.

One memorandum is in order regarding the scope of this volume.  Although we have so far framed this work in terms of motivating problems from transcendental number theory, the aims of this project have been primarily to develop a comprehensive foundation on periods, quasi-periods, logarithms, and quasi-logarithms of Anderson $t$-modules, together with their hyperderivatives with respect to~$\theta$, in terms of abelian and $\bA$-finite $t$-modules.  Our vision has been focused on providing researchers explicit recipes for these values to make them amenable to current transcendence techniques.  To this end we address only Questions~\ref{Q:two} and~\ref{Q:three}, and aside from a passing result in Corollary~\ref{C:linindep}, there are no new transcendence results in this article.  However, as a realization of this principle, in~\cite{NPhD}, \cite{N21}, the first author has built on these techniques to determine all $\oK$-algebraic relations among hyperderivatives of periods, quasi-periods, logarithms, and quasi-logarithms of a Drinfeld module defined over $K^{\sep}$.

In what follows we survey our findings toward Question~\ref{Q:two} in \S\ref{subS:Andintro}--\S\ref{subS:RATintro}, and the complete investigations are contained in \S\ref{S:Exp}--\S\ref{S:Quasi}.  In \S\ref{subS:hyperintro} we summarize the main results of \S\ref{S:Hyper} and in particular obtain a positive answer to Question~\ref{Q:three}.

\subsection{Anderson's exponentiation theorem} \label{subS:Andintro}
We establish some terminology, which we review fully in \S\ref{subS:motives}, \S\ref{subS:dualtomodule}--\S\ref{subS:tframes}. For a matrix $C \in \Mat_{k \times \ell}(\laurent{\KK}{t})$ and $n \in \ZZ$, we define $C^{(n)}$ to be the matrix obtained by taking the $n$-th Frobenius twist of each entry of~$C$.  For $f = \sum_i c_i \tau^i \in \KK[\tau]$, we set $f^* \assign \sum_i c_i^{1/q^i} \sigma^i$, and as such the resulting map $\KK[\tau] \to \KK[\sigma]$ is an anti-isomorphism of rings.  We denote its inverse also by `$*$.' We extend this definition to a matrix $B = (b_{ij}) \in \Mat_{k \times \ell}(\KK[\tau])$ by setting $B^* \assign (b_{ij}^*)^{\tr} \in \Mat_{\ell \times k}(\KK[\sigma])$.

Let $\phi : \bA \to \Mat_d(\KK[\tau])$ be an $\bA$-finite $t$-module. Taking $\cN_{\phi} \assign \Mat_{1\times d}(\KK[\sigma])$, we make $\cN_{\phi}$ into a left $\KK[t,\sigma]$-module by setting
\[
  a \cdot \bsn \assign \bsn \phi_a^{*}, \quad a \in \bA,\ \bsn \in \cN_{\phi}.
\]
The module $\cN_{\phi}$ is the \emph{dual $t$-motive} of $\phi$. For a vector $\bsn = \sum \bsa_i \sigma^i \in \cN_{\phi}$, we set
\[
  \varepsilon_0(\bsn) \assign \bsa_0^{\tr}, \quad \varepsilon_1(\bsn) \assign \biggl( \sum_i \bsa_i^{(i)} \biggr)^{\tr}.
\]
As $\phi$ is assumed to be $\bA$-finite, the dual $t$-motive $\cN_{\phi}$ contains a $\KK[t]$-basis $\bsn_1, \dots, \bsn_r$ together with a matrix $\Phi \in \Mat_r(\KK[t])$ representing multiplication by~$\sigma$, as in \eqref{E:Phidefintro}.  The matrix $\Phi$ together with the $\KK$-linear isomorphism $\iota : \Mat_{1\times r}(\KK[t]) \to \cN_{\phi}$ defined by $\bsh \mapsto \bsh \cdot \bsn$ form a \emph{$t$-frame} of $\phi$.  Of particular utility are the identities for each $a \in \bA$ (see Lemma~\ref{L:epsilontframes}),
\[
  \rd\phi_{a} \bigl( \varepsilon_0(\iota(\bsalpha)) \bigr) = \varepsilon_0( \iota (a\cdot \bsalpha)), \quad
  \phi_a \bigl( \varepsilon_1(\iota(\bsalpha)) \bigr) = \varepsilon_1( \iota(a \cdot \bsalpha)).
\]
Anderson observed that there is a unique bounded $\KK$-linear map
\begin{equation}
  \cE_0: \bigl( \Mat_{1\times r}(\TT_\theta), \dnorm{\,\cdot\,}_{\theta} \bigr) \to \bigl( \KK^d, \inorm{\,\cdot\,} \bigr)
\end{equation}
so that $\cE_{0}|_{\Mat_{1\times r}(\KK[t])} = \varepsilon_0 \circ \iota$ (see Lemma~\ref{L:E0}), and we set $\cE_{1} \assign \varepsilon_1\circ \iota$.  The following result of Anderson is fundamental for our project (restated later as Theorem~\ref{T:Exponentiation}).

\begin{theoremintro}[{Anderson, see \cite[Thm.~2.5.21]{HartlJuschka20}}] \label{TI:Exponentiation}
Let $\phi : \bA \to \Mat_d(\KK[t])$ be an $\bA$-finite $t$-module with $t$-frame $(\iota, \Phi)$. Fix $\bsh \in \Mat_{1 \times r}(\KK[t])$, and suppose there exists $\bsg \in \Mat_{1 \times r}(\TT_\theta)$ such that
\[
\bsg^{(-1)}\Phi - \bsg = \bsh.
\]
Then
\[
\Exp_\phi\bigl( \cE_0(\bsg + \bsh) \bigr) = \cE_1(\bsh).
\]
\end{theoremintro}

If $\phi$ has a rigid analytic trivialization $\Psi \in \GL_r(\TT)$ as in~\eqref{E:RATintro}, and hence $\phi$ is also uniformizable, then in fact $\Psi$ is invertible over~$\TT_\theta$ by~\cite[Prop.~4.4.12]{ABP04}.  It follows from the equation $\Psi^{(-1)} = \Phi\Psi$ that if $\bsg$ is any row of $\Psi^{-1}$, then $\bsg^{(-1)} \Phi - \bsg = 0$, and we find that the period lattice $\Lambda_{\phi} \assign \ker \Exp_\phi$ satisfies
\begin{equation}
   \Lambda_{\phi} = \cE_0\bigl( \Mat_{1\times r}(\bA) \cdot \Psi^{-1} \bigr).
\end{equation}
Indeed the right-hand containment follows from Theorem~\ref{TI:Exponentiation}, and Anderson further proved the opposite containment (see Theorem~\ref{T:PeriodLattice}).  In addition to providing a way to capture elements of the period lattice of $\phi$ in terms of its rigid analytic trivialization, by finding appropriate vectors $\bsg$, $\bsh$, one can use Theorem~\ref{TI:Exponentiation} to provide formulas for arbitrary logarithms as well (see Remark~\ref{R:E0Logarithms}).

In order to apply Theorem~\ref{TI:Exponentiation} to obtain explicit formulas for periods and logarithms of $\phi$, there are two issues.
\begin{itemize}
\item[(i)] For given $\bsy$, $\bsalpha \in \KK^d$ with $\Exp_\phi(\bsy) = \bsalpha$, we seek $\bsg$, $\bsh$ as in Theorem~\ref{TI:Exponentiation} so that $\cE_0(\bsg + \bsh) = \bsy$ and $\cE_1(\bsh) = \bsalpha$.  Of course by Theorem~\ref{TI:Exponentiation}, the former equation implies the latter.
\item[(ii)] We need formulas for the map $\cE_0$ for advantageous descriptions of $\cE_0(\bsg + \bsh)$.
\end{itemize}
The first of these issues is addressed in the next section and in \S\ref{S:Quasi} using the theory of Anderson generating functions.  The second is resolved by Proposition~\ref{P:epsiota}.  Although the conditions in Proposition~\ref{P:epsiota} initially appear restrictive, we show in Remark~\ref{R:basechangered} that any $t$-module can be put in this desired form, perhaps after changes of bases over $\KK[\sigma]$ and $\KK[t]$ on $\cN_{\phi}$.  If $\phi$ is in this form, then we find positive integers $\ell_1, \dots, \ell_m$ that sum to $d$ ($m$ represents the number of Jordan blocks in $\rd\phi_t$ and the $\ell_i$'s are their dimensions), so that for $\bsg = (g_1, \dots, g_r) \in \Mat_{1\times r}(\TT_{\theta})$,
\begin{equation} \label{EI:E0}
  \cE_0(\bsg)
  = \left. \begin{pmatrix}
  \pd_{t}^{\ell_1 -1}(g_1) \\ \vdots \\
  \pd_{t}^1 (g_1) \\
  g_1 \\ \vdots \\
  \pd_{t}^{\ell_m-1}(g_m) \\ \vdots \\ \pd_{t}^{1}(g_m) \\ g_m
  \end{pmatrix} \right|_{t=\theta}.
\end{equation}
Thus $\cE_0$ is constructed through hyperderivatives of functions with respect to~$t$, evaluated at $t=\theta$.  Moreover, since Theorem~\ref{TI:Exponentiation} provides identities for periods and logarithms of $\phi$ through $\cE_0$, we see that these quantities already involve hyperderivatives, though with respect to $t$ and not yet with respect to $\theta$. This connection with hyperderivatives does not occur for Drinfeld modules, as they are $1$-dimensional. For tensor powers of the Carlitz module, this phenomenon was already observed by Anderson and Thakur~\cite{AndThak90} and was made explicit by Maurischat~\cite{Maurischat18}.  At the end of \S\ref{subS:E0calc} we provide the explicit descriptions of $\cE_0$ for Drinfeld modules, tensor powers of the Carlitz module, and strictly pure $t$-modules.

\subsection{The de Rham module and rigid analytic trivializations} \label{subS:RATintro}
The theory of biderivations, quasi-periodic functions, and the de Rham isomorphism for Drinfeld modules was developed by Anderson, Deligne, Gekeler, and Yu (see \cite{Brownawell93}, \cite{Gekeler89a}, \cite{Gekeler90}, \cite{Goss94}, \cite{Yu90}), and later this theory was extended by Brownawell and the second author to general $t$-modules~\cite{BP02} and in a different construction by Hartl and Juschka~\cite[Thm.~2.5.51]{HartlJuschka20}.

As opposed to Anderson's exponentiation theorem from the previous section, which centered on the dual $t$-motive of a $t$-module, the quasi-periodic theory associated to a $t$-module is built instead from its $t$-motive.  As for Drinfeld modules, the general theory of quasi-periods and quasi-logarithms can be realized through the appropriate extension of Anderson generating functions.  Ultimately through a result of Hartl and Juschka (Theorem~\ref{T:isodual}), we unify the theories of rigid analytic trivializations and quasi-periodic extensions of abelian and $\bA$-finite $t$-modules.

We let $\phi : \bA \to \Mat_d(\KK[\tau])$ be a uniformizable abelian $t$-module, and let $\cM_{\phi}$ be its $t$-motive.  Similar to the dual $t$-motive, we set $\cM_{\phi} \assign \Mat_{1\times d}(\KK[\tau])$, and we define the operation of $\KK[t]$ on $\cM_{\phi}$ by setting
\[
  a \cdot \bsm \assign \bsm \phi_a, \quad a \in \bA,\ \bsm \in \cM_{\phi}.
\]
A \emph{$\phi$-biderivation} is an $\FF_q$-linear function $\bsdelta : \bA \to \tau \cM_{\phi}$ such that
\[
  \bsdelta_{ab} = a(\theta)\bsdelta_b + \bsdelta_a \phi_b, \quad \forall\, a,\, b \in \bA.
\]
The $\KK$-vector space of biderivations is denoted $\Der(\phi)$, and each $\phi$-biderivation is uniquely determined by its value~$\bsdelta_t \in \Mat_{1\times d}(\KK[\tau]\tau)$.  If there exists $\bsu \in \Mat_{1\times d}(\KK[\tau])$ so that $\bsdelta_t = \bsu \phi_t - \theta \bsu$, then $\bsdelta$ is said to be an \emph{inner biderivation}, and it is \emph{strictly inner} if $\bsu \in \Mat_{1\times d}(\KK[\tau]\tau)$.  We let $\bsdelta^{(\bsu)}$ denote the inner biderivation associated to this~$\bsu$.  The $\KK$-subspaces of inner and strictly inner biderivations are denoted $\Der_{\inn}(\phi) \supseteq \Der_{\si}(\phi)$, and moreover the quotient
\[
  \rH_{\DR}^{1}(\phi) \assign \frac{\Der(\phi)}{\Der_{\si}(\phi)}
\]
is called the \emph{de Rham module} of $\phi$.  The de Rham module carries the structure of a $\KK[t]$-module, and its dimension as a $\KK$-vector space is~$r$.  For more details, see \S\ref{subS:biderivations}.

For each $\bsdelta \in \Der(\phi)$, there is a unique $\FF_q$-linear power series $\rF_{\bsdelta}(\bsz) \in \power{\KK}{z_1, \dots, z_d}$ so that
\begin{equation}
  \rF_{\bsdelta}(\rd \phi_a \cdot \bsz) = a(\theta) \rF_{\bsdelta}(\bsz) + \bsdelta_a \Exp_{\phi}(\bsz), \quad \forall\, a \in \bA.
\end{equation}
Much as for Drinfeld modules, using the theory in~\cite{BP02}, one can construct quasi-periodic extensions of $\phi$ whose exponential functions are given in terms of these quasi-periodic functions. \emph{Quasi-periods} of $\phi$ are the values $\rF_{\bsdelta}(\bslambda)$ for $\bslambda \in \Lambda_{\phi}$, and more generally \emph{quasi-logarithms} are the values $\rF_{\bsdelta}(\bsy)$ for $\bsy \in \KK^d$. If $\bsdelta = \bsdelta^{(\bsu)}$ is an inner biderivation as above, then
\begin{equation} \label{E:innerfneqintro}
  \rF_{\bsdelta^{(\bsu)}}(\bsz) = \bsu \Exp_{\phi}(\bsz) - \rd \bsu \cdot \bsz,
\end{equation}
and so quasi-periods and quasi-logarithms of inner biderivations are expressible simply in terms of periods and logarithms of $\phi$ themselves.

As is the case for Drinfeld modules, quasi-periods and quasi-logarithms for $\phi$ can be captured in terms of Anderson generating functions.  For $\bsy \in \KK^d$ we define $\cG_{\bsy} \in \TT^d$ by setting
\begin{equation}
  \cG_{\bsy} \assign \sum_{n=0}^\infty \Exp_\phi \bigl( (\rd\phi_t)^{-n-1} \bsy \bigr) t^n,
\end{equation}
and also as for Drinfeld modules, they satisfy properties that extend \eqref{E:AGFfneqintro}--\eqref{E:AGFResintro} (see Propositions~\ref{P:AGFres} and~\ref{P:AGFprops}).  The extension of \eqref{E:AGFquasiintro} requires some notation.  For matrices $\bsbeta = B_0 + B_1\tau + \cdots + B_{\ell}\tau^{\ell} \in \Mat_{m\times n}(\KK[\tau])$ and $M \in \Mat_{n\times s}(\TT)$, we set
\begin{equation}
  \twistop{\bsbeta}{M} \assign B_0 M + B_1 M^{(1)} + \cdots + B_{\ell} M^{(\ell)} \in \Mat_{m \times s}(\TT).
\end{equation}
The extension of \eqref{E:AGFquasiintro} is that for $\bsdelta \in \Der(\phi)$ and $\bsy \in \KK^d$
\begin{equation} \label{E:AGFquasidimintro}
  \twistop{\bsdelta_t}{\cG_{\bsy}}|_{t=\theta} = \rF_{\bsdelta}(\bsy).
\end{equation}
See Proposition~\ref{P:quasigen} for more details.  (A word about the utility of the $\twistop{\cdot}{\cdot}$ notation is that the left-hand entry is often taken from $\bsm \in \cM_{\phi}$.  For $f \in \KK[t]$, there is a need to distinguish $f \cdot \twistop{\bsm}{\cG_{\bsy}}$ and $\twistop{f \cdot \bsm}{\cG_{\bsy}}$, which this notation provides but that `$f\cdot \bsm(\cG_{\bsy})$' leaves ambiguous. For example, see Corollary~\ref{C:AGFtwistedmult}.)

We then define $\Upsilon \in \GL_r(\TT_{\theta})$ as in~\eqref{E:Upsilonintro} by choosing a $\KK[t]$-basis $\bsm_1, \dots, \bsm_r$ of $\cM_{\phi}$ and an $\bA$-basis $\bslambda_1, \dots, \bslambda_r$ of $\Lambda_{\phi}$, and then setting
\begin{equation}
  \Upsilon \assign \begin{pmatrix}
    \twistop{\tau \bsm_1}{\cG_{\bslambda_1}} & \cdots & \twistop{\tau \bsm_1}{\cG_{\bslambda_r}} \\
    \vdots & & \vdots\\
    \twistop{\tau \bsm_r}{\cG_{\bslambda_1}} & \cdots & \twistop{\tau \bsm_r}{\cG_{\bslambda_r}}
  \end{pmatrix}.
\end{equation}
As an element of $\GL_r(\TT_{\theta})$, its value at $t=\theta$ is non-zero, and $\Upsilon|_{t=\theta}$ is also the matrix representing the de Rham map in the following corollary (restated later as Corollary~\ref{C:deRhamiso}), which then must be an isomorphism.  This extends results of Gekeler~\cite{Gekeler89a} for Drinfeld modules (see also~\cite{Goss94}), and provides a specialized version of a related result in for Hodge-Pink structures for $t$-modules due to Hartl and Juschka~\cite{HartlJuschka20}.

\begin{corollaryintro}[de Rham isomorphism] \label{CI:deRhamiso}
Let $\phi : \bA \to \Mat_d(\KK[\tau])$ be a uniformizable abelian $t$-module.  The map
\[
  \DR : \rH^1_{\DR}(\phi) \to \Hom_{\bA}(\Lambda_{\phi},\KK),
\]
defined by $\DR([\bsdelta]) = \rF_{\bsdelta}|_{\Lambda_{\phi}}$, is an isomorphism.
\end{corollaryintro}

In the case of a Drinfeld module, the matrix $\Upsilon$ was used to determine a rigid analytic trivialization for $\phi$, and we find this connection to hold more generally.  If we assume that $\phi$ is uniformizable, abelian, and $\bA$-finite, then Hartl and Juschka construct another dual-$t$-motive $\cM_{\phi}^{\wedge}$ arising from the $t$-motive $\cM_{\phi}$, and they prove that there is an isomorphism of dual-$t$-motives,
\begin{equation} \label{E:xiintro}
  \xi : \cM_{\phi}^{\wedge} \to \cN_{\phi}.
\end{equation}
See Theorem~\ref{T:isodual} for complete details.  If we let $V \in \GL_r(\KK[t])$ denote the matrix representing $\xi$ with respect to the $\KK[t]$-basis of $\cM_{\phi}^{\wedge}$ induced by $\bsm_1, \dots, \bsm_r$ and some fixed $\KK[t]$-basis $\bsn_1, \dots, \bsn_r$ of $\cN_{\phi}$, then it follows that the matrix
\begin{equation} \label{E:PsiUpsilonVintro}
  \Psi \assign \bigl( \Upsilon^{\tr} V \bigr)^{-1} \in \GL_r(\TT_{\theta})
\end{equation}
is a rigid analytic trivialization for $\phi$ (see Proposition~\ref{P:PsiUpsilonV}).  Furthermore, for $\bsalpha$, $\bsy \in \KK^d$ with $\Exp_{\phi}(\bsy) = \bsalpha$, if we take
\begin{equation} \label{E:gyintro}
  \bsg_{\bsy} \assign -\bigl(\twistop{\tau \bsm_1}{\cG_{\bsy}}, \dots, \twistop{\tau \bsm_r}{\cG_{\bsy}} \bigr) \cdot V \in \Mat_{1 \times r}(\TT_{\theta})
\end{equation}
then there exists $\bsh_{\bsalpha} \in \Mat_{1\times r}(\KK[t])$ so that
\[
  \bsg_{\bsy}^{(-1)}\Phi - \bsg_{\bsy} = \bsh_{\bsalpha}.
\]
The explicit description of $\bsh_{\bsalpha}$ can be found in Lemma~\ref{L:gyha}.  Of particular interest is that now $\bsg_{\bsy}$ and $\bsh_{\bsalpha}$ can serve as inputs for Theorem~\ref{TI:Exponentiation}.  We note that for $\bsy = \bslambda_j$, the row vector $\bsg_{\bslambda_j}$ is the negative of the $j$-th row of $\Psi^{-1}=\Upsilon^{\tr}V$.

With only a little extra effort, we can pay attention to the field of definition of the $t$-module $\phi$ as well as its biderivations, quasi-periodic functions, and the isomorphism $\xi$ from~\eqref{E:xiintro}.  The above considerations then lead to the following theorem (which synthesizes Corollary~\ref{C:perquasiperspans} and Theorem~\ref{T:quasispancomplete}).

\begin{theoremintro} \label{TI:perquasiomnibus}
Let $\phi: \bA \to \Mat_d(\oK[\tau])$ be a uniformizable, abelian, and $\bA$-finite $t$-module of rank~$r$ defined over~$\oK$, and suppose we have chosen $\oK[t]$-bases $\{\bsm_1, \dots, \bsm_r\}$ and $\{\bsn_1, \dots, \bsn_r\}$ for $\cM_{\phi,\oK}$ and $\cN_{\phi,\oK}$ respectively.  Then for the rigid analytic trivialization $\Psi \in \GL_r(\TT_{\theta})$ for $\phi$ in~\eqref{E:PsiUpsilonVintro}, the following hold.
\begin{enumerate}
\item[(a)] If $\bsg_1, \dots, \bsg_r \in \Mat_{1\times r}(\TT_{\theta})$ are the rows of $\Psi^{-1} = \Upsilon^{\tr}V$, then letting $\tbslambda_j \assign \cE_{0}(\bsg_j)$ for $1\leqslant j \leqslant r$,
\[
  \Lambda_{\phi} = \bA \cdot \tbslambda_1 + \dots + \bA \cdot \tbslambda_r.
\]
\item[(b)] If $\bsdelta_1, \dots, \bsdelta_r \in \Der(\phi,\oK)$ represent a $\oK$-basis of $\rH_{\DR}^1(\phi,\oK)$, then
\[
  \Span{\oK} \bigl( \rF_{\bsdelta_i}(\tbslambda_j) : 1 \leqslant i, j \leqslant r \bigr)
  = \Span{\oK} \bigl( (\Upsilon^{\tr})|_{t=\theta} \bigr) = \Span{\oK} \bigl( (\Psi^{-1})|_{t=\theta}\bigr).
\]
\item[(c)] Suppose that $\bsy \in \KK^d$ and $\bsalpha \in \smash{\oK}^d$ are chosen so that $\Exp_{\phi}(\bsy) = \bsalpha$.  Then
\begin{align*}
  \Span{\oK} \bigl(1, \rF_{\bsdelta_1}(\bsy), \dots, \rF_{\bsdelta_r}(\bsy) \bigr) &=
  \Span{\oK} \bigl( \{1\} \cup  \{\rF_{\bsdelta}(\bsy) : \bsdelta \in \Der(\phi,\oK) \} \bigr)\\
   &= \Span{\oK} \bigl( \{1\} \cup \{ \bsg_{\bsy}|_{t=\theta}\} \bigr),
\end{align*}
where $\bsg_{\bsy} \in \Mat_{1\times r}(\TT_\theta)$ is given as in \eqref{E:gyintro}.
\end{enumerate}
\end{theoremintro}

Unlike in the case of \eqref{E:AGFquasiintro}, parts (b) and~(c) of this theorem do not necessarily account for the $\oK$-span of \emph{all} coordinates of logarithms and quasi-logarithms of $\phi$.  In dimensions $> 1$, the functional equation~\eqref{E:innerfneqintro} for inner biderivations yields only those coordinates of periods and logarithms that are \emph{tractable}, i.e., those that lie at the bottom of a Jordan block of $\rd\phi_t$ when $\rd\phi_t$ is in Jordan normal form.  Thus our answer so far to Question~\ref{Q:two} is not yet complete.  However, to settle this question and account for \emph{all} coordinates of periods and logarithms using these techniques, we apply this theory to Maurischat's prolongation $t$-modules, where these missing coordinates will appear.  See the next section and \S\ref{S:Hyper} for details.

One subtle and vexing point is that although Theorem~\ref{TI:Exponentiation} applies to the vectors $\bsg_{\bsy}$ and $\bsh_{\bsalpha}$, it is not known a priori that $\cE_0(\bsg_{\bsy} + \bsh_{\bsalpha}) = \bsy$.  In particular in the theorem, it would seem reasonable that $\cE_0(\bsg_j) = \bslambda_j$, but all we know for sure is that $\cE_0(\bsg_1), \dots, \cE_0(\bsg_r)$ form an $\bA$-basis of the period lattice.  The reason for this inexactness is that for a general $t$-module the definition of $V$ can be quite complicated.  Already for Drinfeld modules, $V$ arises somewhat mysteriously (see \eqref{E:VDrinfeld} and \cite[p.~133]{CP12}, \cite[p.~111]{HartlJuschka20}, \cite[p.~97]{Juschka10}). For almost strictly pure $t$-modules, the story has a satisfactory conclusion, and we provide an explicit description of $V$ in Corollary~\ref{C:VXBY}.  Moreover, the following identities hold.

\begin{propositionintro} \label{PI:Expgyha}
Let $\phi : \bA \to \Mat_d(\KK[\tau])$ be a uniformizable almost strictly pure $t$-module.  Let $\bsalpha$, $\bsy\in \KK^d$ be chosen so that $\Exp_{\phi}(\bsy) = \bsalpha$, and let $\bsg_{\bsy}$ and $\bsh_{\bsalpha}$ be given as in \eqref{E:gyintro}.  Then
\begin{enumerate}
\item[(a)] $\cE_0( \bsg_{\bsy}+\bsh_{\bsalpha}) = \bsy$,
\item[(b)] $\cE_1(\bsh_{\bsalpha}) = \bsalpha$.
\end{enumerate}
\end{propositionintro}

We close out \S\ref{S:Quasi} by giving precise accounts of these constructions and results for Drinfeld modules, tensor powers of the Carlitz module, and strictly pure $t$-modules in~\S\ref{subS:quasiperexamples}.  We also work out an almost strictly pure example in detail to highlight some of the different behaviors that can occur.

\subsection{Hyperderivatives of fundamental quantities} \label{subS:hyperintro}
In \S\ref{S:Hyper} we apply these previous techniques on exponentiation, rigid analytic trivializations, and the de Rham module to \emph{prolongation $t$-modules} of Maurischat~\cite{Maurischat18} so as to answer Question~\ref{Q:three}.  We fix a uniformizable, abelian, and $\bA$-finite $t$-module $\phi : \bA \to \Mat_d(K^{\sep}[\tau])$, defined by $\phi_t = \rd\phi_t + B_1 \tau + \cdots + B_{\ell} \tau^{\ell}$.  For $n \geqslant 0$, the $n$-th prolongation of $\phi$ is the $t$-module
\[
\rP_n\phi : \bA \to \Mat_{(n+1)d}(K^{\sep}[\tau])
\]
defined by
\begin{equation} \label{E:Pnphiintro}
({\rP_n\phi})_t  = \begin{pmatrix}
    \rd \phi_t &  0    & \cdots    & 0\\
    -\Id_d & \ddots & \ddots&\vdots\\
    \vdots & \ddots&\ddots&0\\
    0& \cdots & -\Id_d & \rd \phi_t
    \end{pmatrix}
    +\diag(B_1)\tau + \dots + \diag(B_\ell) \tau^\ell,
\end{equation}
where $\diag(B_i)$ is the block diagonal matrix with $n+1$ copies of $B_i$ down the diagonal.  Maurischat obtains this $t$-module by considering extensions of both the $t$-motive and dual $t$-motive of~$\phi$ using higher derivations and hyperderivatives (see \S\ref{subS:Prolongations}).  We have $\rP_0 \phi = \phi$, and for each $0 \leqslant h \leqslant n$, we have a natural exact sequence of $t$-modules,
\[
  0 \to \rP_{n-h-1}\phi \to \rP_n \phi \to \rP_h \phi \to 0.
\]
The exponential function for $\rP_n\phi$ is straightforward in that if $\bsz_0, \dots, \bsz_n$ represent vectors of variables for each of the $n+1$ blocks of $\rP_n\phi$, then
\[
  \Exp_{\rP_n\phi} \begin{pmatrix} \bsz_0 \\ \vdots \\ \bsz_n \end{pmatrix}
  = \begin{pmatrix} \Exp_{\phi}(\bsz_0) \\ \vdots \\ \Exp_{\phi}(\bsz_n) \end{pmatrix},
\]
and thus $\rP_n \phi$ shares its exponential function with the direct sum $\phi^{\oplus n}$.  However, if $\Upsilon$, $\Psi \in \GL_r(\TT_{\theta})$ are the matrices given in~\S\ref{subS:RATintro}, the former given in terms of Anderson generating functions and the latter being a rigid analytic trivialization for~$\phi$, then the corresponding matrices for $\rP_n \phi$ are more interesting.  Indeed Maurischat proved (see Theorem~\ref{T:prolongprops}) that
\[
  \Psi_{\rP_n\phi} \assign d_{t,n+1}[\Psi] \assign \begin{pmatrix}
  \Psi & \pd_t^1(\Psi) & \cdots & \pd_t^n (\Psi) \\
  0 & \Psi & \ddots & \vdots \\
  \vdots & \ddots & \ddots & \pd_t^1(\Psi) \\
  0 & \cdots & 0 & \Psi
  \end{pmatrix},
\]
where $\pd_t^j(\Psi)$ is obtained by applying $\pd_t^j$ entry-wise to~$\Psi$.  Moreover, we find that
\[
  \Upsilon_{\rP_n\phi} \assign d_{t,n+1}[\Upsilon^{\tr}]^{\tr} \assign \begin{pmatrix}
\Upsilon & 0 & \cdots & 0 \\
\pd_t^1(\Upsilon) & \Upsilon & \ddots & \vdots \\
\vdots & \ddots & \ddots & 0 \notag \\
\pd_t^n(\Upsilon) & \cdots & \pd_t^1(\Upsilon) & \Upsilon
  \end{pmatrix},
\]
is the corresponding system of Anderson generating functions for~$\rP_n\phi$ (see Proposition~\ref{P:UpsilonVpro}).  Thus both $\Psi_{\rP_n\phi}$ and $\Upsilon_{\rP_n\phi}$ are given in terms of hyperderivatives with respect to~$t$ of Anderson generating functions on~$\phi$.

On the other hand, the quantities in Question~\ref{Q:three} involve hyperderivatives with respect to~$\theta$, and connecting these two worlds together requires careful analysis using the chain rule for hyperderivatives.  One particularly useful construction, inspired by Brownawell and Denis~\cite{BrownawellDenis00} for Drinfeld modules, is that for any $\phi$-biderivation $\bsdelta \in \Der(\phi,\oK)$ and $s \geqslant 1$, there exists $\bsdelta^s \in \Der(\phi,\oK)$ so that $\bsdelta$ and $\bsdelta^s$ are in the same de Rham class and that $\deg_{\tau} \bsdelta^s_t \geqslant s$.  By taking $s$ large enough so that $q^s > n$, the hyperderivative calculations simplify by the $p$-th power rule for hyperderivatives (see Proposition~\ref{P:Hyperprops}), as we can restrict to using only $\pd_{\theta}^1, \dots, \pd_{\theta}^n$.  We obtain the following result that equates these two spaces of hyperderivatives with respect to~$t$ and~$\theta$ for quantities arising from~$\rP_n\phi$ (restated as Theorem~\ref{T:rathyperquasi0}).

\begin{theoremintro} \label{TI:rathyperquasi0}
Let $\phi$ be a uniformizable, abelian, and $\bA$-finite Anderson $t$-module defined over $K^\sep$ of rank $r$ and dimension $d$. Let $\bsy \in (K_{\infty}^{\sep})^d$ satisfy $\Exp_\phi(\bsy) \in (K^\sep)^d$, and let $\{\bsdelta_1, \dots, \bsdelta_r\}$ represent a $\oK$-basis of $\rH^1_{\DR}(\phi,\oK)$ defined over $K^\sep$.  For $n \geqslant 0$, the following hold.
\begin{enumerate}
\item[(a)] For the Anderson generating function $\cG_{\bsy} \in \TT^d$ for $\phi$ associated to $\bsy$, we have
\[
\Span{\oK} \biggl( \{1\} \cup \bigcup_{u=0}^n \bigcup_{v=1}^r \Bigl\{ \twistop{(\bsdelta_v)_t}{\pd_t^u(\cG_{\bsy})} \big|_{t=\theta} \Bigr\} \biggr)
= \Span{\oK} \biggl( \{1\} \cup \bigcup_{u=0}^n \bigcup_{v=1}^r \Bigl\{ \pd_\theta^u \bigl( \rF_{\bsdelta_v}(\bsy) \bigr) \Bigr\} \biggr).
\]
\item[(b)] Moreover, if $\{\bslambda_1, \dots, \bslambda_r\}$ is an $\bA$-basis of $\Lambda_\phi$, then
\[
\Span{\oK} \Bigl( d_{t, n+1}[\Psi]^{-1}\big|_{t=\theta} \Bigr)
= \Span{\oK} \biggl( \bigcup_{u=0}^n \bigcup_{v=1}^{r} \bigcup_{\ell=1}^r \Bigl\{ \pd_\theta^u \bigl(\rF_{\bsdelta_v}(\bslambda_\ell) \bigr) \Bigr\} \biggr),
\]
where $(\iota, \Phi,\Psi)$ is a rigid analytic trivialization for $\phi$.
\end{enumerate}
\end{theoremintro}

As mentioned in \S\ref{subS:RATintro}, the quantities $\rF_{\bsdelta_1}(\bsy), \dots, \rF_{\bsdelta_r}(\bsy)$ incorporate quasi-logarithms related to strictly reduced $\phi$-biderivations as well as the $\oK$-span of the tractable coordinates of~$\bsy$.  Theorem~\ref{TI:rathyperquasi0} then addresses the hyperderivatives of these quantities with respect to~$\theta$.  However, a full answer to Question~\ref{Q:three} should account for hyperderivatives of the non-tractable coordinates as well.  By analyzing the structure of $\rH_{\DR}^1(\rP_n\phi,\oK)$ and its associated quasi-periodic functions, we find, for $n$ sufficiently large, that hyperderivatives of \emph{all coordinates} can be obtained.  The fundamental principle is that non-tractable coordinates of $\phi$ can be explicitly related to tractable coordinates of $\rP_n\phi$ once $n$ is large enough.  For tensor powers of the Carlitz module this was first observed by Maurischat~\cite{Maurischat18}.

This result is broken down into two steps, the first for almost strictly pure $t$-modules (Theorem~\ref{T:hypernontract}) and the second for general $t$-modules (Theorem~\ref{T:hypernontractgen}).  We combine them together here, thus together with Theorem~\ref{TI:rathyperquasi0} we obtain a complete answer to Question~\ref{Q:three}.  See also Corollary~\ref{C:Jordannontract}.

\begin{theoremintro} \label{TI:hypernontract}
Let $\phi$ be a uniformizable, abelian, and $\bA$-finite Anderson $t$-module defined over $K^\sep$ of rank $r$ and dimension $d$. Let $\bsy \in (K_{\infty}^{\sep})^d$ satisfy $\Exp_\phi(\bsy) \in (K^\sep)^d$, and let $\bsdelta_1, \dots, \bsdelta_{r-m}$, $\bsgamma_1, \dots, \bsgamma_m$ represent a $\oK$-basis of $\rH^1_{\DR}(\phi,\oK)$ defined over $K^\sep$ so that $\bsgamma_1, \dots, \bsgamma_m$ provide a basis of inner biderivations modulo strictly inner ones as in~\eqref{E:basisSRphi1}.  Let $j \geqslant 0$, and choose $n \geqslant 0$ so that $(\rd\phi_t - \theta\Id_d)^{n-j}=0$.
\begin{enumerate}
\item[(a)]  Letting $\cG_{\bsy} \in \TT^d$ be the Anderson generating function for $\phi$ associated to~$\bsy$, we have
\begin{align*}
  \Span{\oK} \bigl( \pd_{\theta}^{j} (\bsy) \bigr) &\subseteq
  \Span{\oK} \biggl( \{1\} \cup \bigcup_{u=0}^n \bigcup_{v=1}^{r-m} \bigcup_{i=1}^m \Bigl\{ \twistop{(\delta_v)_t}{\pd_t^u(\cG_{\bsy})} \big|_{t=\theta}, \twistop{(\gamma_i)_t}{\pd_t^u(\cG_{\bsy})} \big|_{t=\theta} \Bigr\} \biggr) \\
  &=
  \Span{\oK} \biggl( \{1\} \cup \bigcup_{u=0}^n \bigcup_{v=1}^{r-m}\bigcup_{i=1}^{m} \Bigl\{ \pd_\theta^u \bigl( \rF_{\delta_v}(\bsy), \pd_\theta^u \bigl( \rF_{\gamma_i}(\bsy) \bigr) \Bigr\} \biggr).
\end{align*}
\item[(b)] In particular, if $\bslambda_1, \dots, \bslambda_r$ denote an $\bA$-basis of $\Lambda_{\phi}$.  Then for any $\bslambda \in \Lambda_{\phi}$,
\[
  \Span{\oK} \bigl( \pd_{\theta}^{j}(\bslambda) \bigr) \subseteq \Span{\oK}\Bigl( d_{t, n+1}[\Psi_{ \phi}]^{-1}\big|_{t=\theta} \Bigr) = \Span{\oK} \biggl( \bigcup_{u=0}^n \bigcup_{v=1}^{r} \bigcup_{\ell=1}^r \Bigl\{ \pd_\theta^u \bigl(\rF_{\bsdelta_v}(\bslambda_{\ell}) \bigr) \Bigr\} \biggr),
\]
where $\Psi_{\phi}$ is a rigid analytic trivialization for~$\phi$.
\end{enumerate}
\end{theoremintro}

These identities can be simplified and refined further by considering quasi-periodic functions on $\rP_n\phi$ directly.  For $\bsy \in (K_{\infty}^{\sep})^d$, if we let $(\bsy)_0 \in (K_{\infty}^{\sep})^{(n+1)d}$ be the block vector starting with~$\bsy$ and with $0$'s in its remaining entries, then we show (see Lemma~\ref{L:gyhapro}) that
\[
  \bsg_{(\bsy)_0} = \bigl( \bsg_{\bsy}, \pd_t^1(\bsg_{\bsy}), \dots, \pd_t^{n}(\bsg_{\bsy}) \bigr) \in \Mat_{1 \times (n+1)d}(\TT_\theta).
\]
Moreover, if $\bsdelta_1, \dots, \bsdelta_r$ represent a $\oK$-basis of $\rH_{\DR}^1(\phi,\oK)$ defined over $K^{\sep}$, then (see Corollary~\ref{C:rathpqgyhapr0})
\[
  \Span{\oK} \bigl( 1, \bsg_{\bsy}|_{t=\theta}, \pd_t^1(\bsg_{\bsy})|_{t=\theta}, \dots, \pd_t^n(\bsg_{\bsy})|_{t=\theta} \bigr) = \Span{\oK} \biggl( \{1\} \cup \bigcup_{u=0}^n \bigcup_{v=1}^r \Bigl\{ \pd_\theta^u \bigl( \rF_{\bsdelta_v}(\bsy) \bigr) \Bigr\} \biggr).
\]
Utilizing the trick of Brownawell and Denis, we fix $s$ so that $q^s > n$ and replace $\bsdelta_1, \dots, \bsdelta_r$ by $\bsdelta_1^s, \dots, \bsdelta_r^s$ chosen appropriately with the property that $\deg_{\tau} (\bsdelta^s_j)_t \geqslant s$.  We then obtain the following formula (see Corollary~\ref{C:hyperQPFAGF}): for $0 \leqslant u \leqslant n$ and $1 \leqslant v \leqslant r$,
\begin{equation}
  \pd_{\theta}^u \bigl( \rF_{\bsdelta_v^s}(\bsy) \bigr) =  \Bigl( \bigl\langle (\bsdelta_v^s)_t^{[u]} \bigm| \cG_{\bsy} \bigr\rangle + \bigl\langle (\bsdelta_v^s)_t^{[u-1]} \bigm| \pd_t^1(\cG_{\bsy}) \bigr\rangle + \dots + \bigl\langle (\bsdelta_v^s)_t \bigm| \pd_t^{u}(\cG_{\bsy}) \bigr\rangle \Bigr) \Big|_{t=\theta},
\end{equation}
where $(\bsdelta^s_{v})_t^{[u]} \in \Mat_{1\times d}(K^{\sep}[\tau])$ is obtained by applying the hyperdifferential operator $\pd_{\theta}^u$ to the coefficients of $(\bsdelta^s_v)_t$.  This formula distills in a simple way the connections between hyperderivatives of quasi-logarithms with respect to~$\theta$ and hyperderivatives of Anderson generating functions with respect to~$t$.

These concerns are worked out in more detail for Drinfeld modules, tensor powers of the Carlitz module, and strictly pure $t$-modules in~\S\ref{subS:Examplespro}.

\begin{acknowledgments}
The authors thank Dale Brownawell, Chieh-Yu Chang, Nathan Green, Urs Hartl, Andreas Maurischat, Federico Pellarin, and Jing Yu for many helpful discussions during the preparation of this manuscript. The authors further thank the referee for several suggestions that improved the clarity of arguments and exposition.
\end{acknowledgments}

\newpage
\section{Preliminaries} \label{S:Prelim}

\subsection{Notation}
The following notation will be used throughout the paper:

\begin{longtable}{p{3.75truecm}@{\hspace{5pt}$=$\hspace{5pt}}p{10truecm}}
$\FF_q$ & finite field with $q=p^n$ elements, $p$ a prime. \\
$K$ & $\FF_q(\theta)$, the rational function field in $\theta$ over $\FF_q$. \\
$K_\infty$ & $\laurent{\FF_q}{1/\theta}$, the completion of $K$ with respect to $\inorm{\,\cdot\,}$. \\
$\KK$ & the completion of an algebraic closure of $K_\infty$. \\
$\oK$ & the algebraic closure of $K$ inside $\KK$.  \\
$\bA$ & $\FF_q[t]$, the polynomial ring in $t$ over $\FF_q$, $t$ independent from $\theta$. \\
$\TT$ & the Tate algebra of $\power{\KK}{t}$ on the closed unit disk.\\
$\TT_\theta$ & the Tate algebra of $\power{\KK}{t}$ on the closed disk of radius $\inorm{\theta}$.\\
$F^{\sep}$ & a separable algebraic closure of a field $F$.\\
$\Mat_{m \times n}(R)$ & for a ring $R$, the left $R$-module of $m \times n$ matrices.\\
$\Mat_d(R)$ & $\Mat_{d\times d}(R)$.\\
$R^d$ & $\Mat_{d \times 1}(R)$.\\
$M^{\tr}$ & the transpose of a matrix $M$.\\
$\diag_d(M_1, \dots, M_s)$ & for square matrices $M_1, \dots, M_s$, whose dimensions sum to~$d$, the $d \times d$ block diagonal matrix with diagonal blocks $M_1, \dots, M_s$. \\
$\Span{L}(V)$ & for a subfield $L \subseteq \KK$ and a set of vectors $V$ with entries in $\KK$, the $L$-subspace of $\KK$ spanned by the entries of elements of $V$.\\
$(\bsm)_u$ & for $\bsm \in \Mat_{1\times d}(\power{\KK}{\tau})$ and $0 \leqslant u \leqslant n$, the row vector $(0, \dots, 0, \bsm, 0, \dots, 0) \in \Mat_{1\times (n+1)d}(\power{\KK}{\tau})$ where $\bsm$ occupies entries $du+1$ to $du+d$ and all other entries are~$0$.
\end{longtable}

Let $\FF_q$ be a finite field with $q$ elements, where $q$ is a power of a fixed prime $p$.  Let $K = \FF_q(\theta)$ be the rational function field in a single variable $\theta$ over $\FF_q$.  We let $K_\infty = \laurent{\FF_q}{1/\theta}$ be the completion of $K$ at its infinite place, and we let $\KK$ denote the completion of an algebraic closure of $K_{\infty}$.  We let $\inorm{\,\cdot\,}$ denote the absolute value on $\KK$, normalized so that $\inorm{\theta} = q$.  Finally we let $\oK$ denote the algebraic closure of $K$ in $\KK$.

We let $t$ be another variable independent from $\theta$, and we let $\bA = \FF_q[t]$ denote the polynomial ring in $t$.  We fix the canonical $\FF_q$-algebra homomorphism $\bA \to K$ taking $t \mapsto \theta$, thus making $\KK$ into an $\bA$-field (e.g., see~\cite[Ch.~3--4]{Goss}).  We let $\TT$ be the Tate algebra of rigid analytic functions on the closed unit disk of $\KK$,
\[
  \TT = \biggl\{ \sum a_i t^i \in \power{\KK}{t} : \inorm{a_i} \to 0 \biggr\},
\]
and we let $\TT_{\theta}$ denote the Tate algebra of functions converging on the disk of radius $\inorm{\theta}$,
\[
  \TT_{\theta} = \biggl\{ \sum a_i t^i \in \power{\KK}{t} : q^i \cdot \inorm{a_i} \to 0 \biggr\}.
\]
Furthermore, we let $\EE \subseteq \TT_{\theta} \subseteq \TT$ denote the ring of entire power series that converge on all of $\KK$ and whose coefficients lie in a finite extension of $K_{\infty}$, and we let $\MM$ be the fraction field of $\EE$, which is the field of meromorphic functions on $\KK$. For $f= \sum a_i t^i \in \TT$, we define the norm
\[
\dnorm{f} \assign \sup\limits_{i} \inorm{a_i}.
\]
For $f \in \TT_{\theta}$, we also define the norm
\[
\dnorm{f}_{\theta} \assign \sup\limits_{i} \Bigl( q^{i} \cdot \inorm{a_i}\Bigr).
\]
For more background on rigid analytic functions and Tate algebras, see \cite[\S 2]{FresnelvdPut}.

For an integer $n \in \ZZ$, we define Frobenius twisting $f \mapsto f^{(n)} : \laurent{\KK}{t} \to \laurent{\KK}{t}$ by setting for $f = \sum a_i t^i \in \laurent{\KK}{t}$,
\[
  f^{(n)} = \sum a_i^{q^n} t^i.
\]
Frobenius twisting is an automorphism of $\laurent{\KK}{t}$, and it induces automorphisms of $\oK(t)$, $\TT$, and $\EE$.  For a matrix $M = (m_{ij})$ with entries in $\laurent{\KK}{t}$, we define the twist $M^{(n)}$ entrywise by setting $(M^{(n)})_{ij} \assign m^{(n)}_{ij}$.  For matrices $M$, $N$ of appropriate sizes, we have $(MN)^{(n)} = M^{(n)} N^{(n)}$.

For an $\FF_q$-subalgebra $R \subseteq \KK$, we let $\tau : R \to R$ denote the $q$-th power Frobenius map, and we let $R[\tau]$ be the $R$-subalgebra of $\FF_q$-linear endomorphisms generated by $\tau$.  As such $R[\tau]$ is the ring of twisted polynomials in $\tau$ over $R$, subject to the relation, $\tau c = c^q \tau$ for any $c \in R$.  Moreover we let $\Mat_{m \times n}(R)[\tau] = \Mat_{m \times n}(R[\tau])$ denote the left $R[\tau]$-module of $m \times n$ matrices with entries in $R[\tau]$, which is also the same as the left $R[\tau]$-module of polynomials in $\tau$ with coefficients in $\Mat_{m \times n}(R)$.  In this way $\Mat_n(R)[\tau] = \Mat_{n\times n}(R)[\tau]$ forms a noncommutative ring subject to the relation,
\[
  \tau M = M^{(1)} \tau, \quad M \in \Mat_n(R).
\]
Moreover, the left $R$-module of column vectors $R^n=\Mat_{n \times 1}(R)$ is then a left $\Mat_n(R)[\tau]$-module by setting for $\beta = B_0 + B_1 \tau + \dots + B_{\ell} \tau^{\ell} \in \Mat_{n}(R)[\tau]$ and $\bsx \in R^n$,
\begin{equation} \label{E:tautwists}
  \beta \cdot \bsx \assign \beta(\bsx) \assign B_0\bsx + B_1 \bsx^{(1)} + \dots + B_{\ell} \bsx^{(\ell)}.
\end{equation}
We also set $\rd \beta \assign B_0$.  Similarly we can also form modules of power series in $\tau$ with matrix coefficients, $\power{\Mat_{m \times n}(R)}{\tau}$, and as such $\power{\Mat_n(R)}{\tau}$ forms a twisted power series ring.

Similarly if $\tau : R \to R$ is an automorphism, then we take $\sigma = \tau^{-1}$ on $R$ and we form the rings $R[\sigma]$ and $\Mat_n(R)[\sigma]$, where
\[
  \sigma M = M^{(-1)} \sigma, \quad M \in \Mat_n(R).
\]
Likewise, for $\gamma = C_0 + C_1 \sigma + \dots + C_\ell \sigma^{\ell} \in \Mat_n(R)[\sigma]$, we set $\rd \gamma \assign C_0$.

We define
\[
\twistop{\cdot}{\cdot} : \Mat_{m\times n}(\KK[\tau]) \times \Mat_{n\times s}(\TT) \to \Mat_{m \times s}(\TT)
\]
by setting for $\bsbeta = B_0 + B_1 \tau + \dots + B_{\ell} \tau^{\ell} \in \Mat_{m \times n}(\KK[\tau])$ and $M \in \Mat_{n \times s}(\TT)$,
\begin{equation} \label{E:twistop}
  \twistop{\bsbeta}{M} \assign B_0 M + B_1 M^{(1)} + \cdots + B_{\ell} M^{(\ell)}.
\end{equation}
The pairing $\twistop{\cdot}{\cdot}$ is biadditive, $\KK$-linear in the left entry, and $\bA$-linear in the right entry.  For appropriately sized matrices it is associative in that
\begin{equation} \label{E:twistopassoc}
  \twistop{\bsalpha \bsbeta}{\bsf} = \twistop{\bsalpha}{\twistop{\bsbeta}{\bsf}}.
\end{equation}
We note that if $\bsx \in \KK^n$, then $\twistop{\bsbeta}{\bsx} = \bsbeta(\bsx)$ defined in~\eqref{E:tautwists}, but we will see that this additional notation is helpful for distinguishing different operations of $\KK[t]$ on $t$-motives and on $\TT$ (e.g., see Corollary~\ref{C:AGFtwistedmult}).

\subsection{Anderson \texorpdfstring{$t$}{t}-modules} \label{subS:modules}
Let $R$ be an $\FF_q[\theta]$-subalgebra of $\KK$.  An \emph{Anderson $t$-module} over $R$ is an $\FF_q$-algebra homomorphism
\[
\phi: \bA \to \Mat_d(R[\tau])
\]
such that if $\phi_t = B_0 + B_1 \tau +\dots + B_\ell \tau^\ell$, $B_i \in \Mat_d(R)$,
then $\rd \phi_t = B_0 = \theta \Id_d + N$, where $\Id_d$ is the $d \times d$ identity matrix and $N$ is a nilpotent matrix. We say that $d$ is the \emph{dimension} of $\phi$, and a \emph{Drinfeld module} is a non-trivial Anderson $t$-module of dimension~$1$. The map $\phi$ is uniquely determined by its value $\phi_t$.  In this way $\phi$ defines an $\bA$-module structure on $R^d$ by setting
\[
  a * \bsx = \phi_a(\bsx), \quad a \in \bA,\ \bsx \in R^d.
\]
A morphism of $\mu : \phi \to \psi$ of $t$-modules is given by a matrix $\mu \in \Mat_{e \times d}(R[\tau])$, where $d$ is the dimension of $\phi$ and $e$ is the dimension of $\psi$, such that
\begin{equation}\label{E:morp}
  \mu \phi_a = \psi_a \mu, \quad \forall\, a \in \bA.
\end{equation}
Through an abuse of notation, we will let $\phi(R)$ denote $R^d$ together with the $\bA$-module structure induced by $\phi$, i.e., $\phi(R)$ denotes the `$R$-valued points of $\phi$.' For general information on $t$-modules, the reader is directed to Anderson's original paper~\cite{And86} or to other expository sources~\cite{BPrapid}, \cite{Goss}, \cite{Thakur}.

When $R = \KK$, Anderson~\cite{And86} showed that for a $t$-module $\phi$ as above, there exists a unique power series $\Exp_\phi = \sum_{i \geqslant 0} C_i \tau^i \in \power{\Mat_d(\KK)}{\tau}$ so that $C_0 = I_d$ and that for any $a \in \bA$, $\Exp_{\phi} \cdot\, {\rd} \phi_a = \phi_a \cdot \Exp_{\phi}$.  If we let $\bsz = (z_1, \dots, z_d)^{\tr} \in \KK[z_1, \dots, z_d]^d$ consist of independent variables over $\KK$, then the \emph{exponential series of $\phi$},
\begin{equation} \label{E:Expdef}
  \Exp_{\phi}(\bsz) = \sum_{i=0}^{\infty} C_i \bsz^{(i)} \in \power{\KK}{z_1, \dots, z_d}^d,
\end{equation}
defines an entire function $\Exp_{\phi} : \KK^d \to \KK^d$.  For any $a \in \bA$, we have the identity of power series,
\[
\Exp_{\phi}(\rd \phi_a \bsz) = \phi_a(\Exp_{\phi}(\bsz)).
\]

\begin{remark} \label{R:ExpoverL}
If $\phi$ is defined over a field $L$ with $K\subseteq L \subseteq \KK$, then $\Exp_{\phi}(\bsz) \in \power{L}{z_1, \dots, z_d}^d$. Indeed this is due to Anderson~\cite[Prop.~2.1.4, Lem.~2.1.6]{And86} and noted by Goss~\cite[Lem.~5.9.3]{Goss}. Goss's exposition in particular pays attention to the field of definition of the matrices $C_i$. It is important to point out that Anderson and Goss consider the case that $\phi$ is defined over a finite extension $L$ of $K_{\infty}$, but their arguments work equally well for any field of definition $L$ with $K \subseteq L \subseteq \KK$.
\end{remark}

If the function $\Exp_{\phi} : \KK^d \to \KK^d$ is surjective, then $\phi$ is said to be \emph{uniformizable}.  Drinfeld modules are always uniformizable, but a general $t$-module of dimension $\geqslant 2$ need not be.  As formal power series, the inverse of $\Exp_{\phi}(\bsz)$ is the \emph{logarithm series of~$\phi$}, $\Log_{\phi}(\bsz) \in \power{\KK}{z_1, \dots, z_d}^d$, which as a function on $\KK^d$ may have only a finite radius of convergence.  The kernel
\begin{equation}
\Lambda_{\phi} = \ker (\Exp_{\phi}) \subseteq \KK^d
\end{equation}
is a free, finitely generated, and discrete $\bA$-submodule through the action of $\rd \phi(\bA)$.  The $\bA$-module $\Lambda_\phi$ is called the \emph{period lattice} of $\phi$, and if $\phi$ is uniformizable, then we have an isomorphism of $\bA$-modules given by $\KK^d/\Lambda_{\phi} \cong \phi(\KK)$.

\subsection{\texorpdfstring{$t$}{t}-motives and dual \texorpdfstring{$t$}{t}-motives} \label{subS:motives}
In this section we define Anderson $t$-motives and dual $t$-motives.  These objects have gone through many iterations since Anderson's original definition of $t$-motives in~\cite{And86}.  Our definitions are given precisely below, but for more information on $t$-motives, the reader is directed to \cite{BPrapid}, \cite[Ch.~5]{Goss}, \cite{HartlJuschka20}.  For additional information on dual $t$-motives, one can consult \cite{ABP04}, where they were defined, and also \cite{BPrapid}, \cite{HartlJuschka20}.

We let $L$ be a field with $K \subseteq L \subseteq \KK$.  We define the ring $L[t,\tau]$ to be the ring of polynomials in $t$ and $\tau$, subject to the noncommuting relations,
\[
  tc=ct, \quad t\tau= \tau t, \quad \tau c = c^q \tau, \quad \forall\, c \in L.
\]
Similarly, if $L$ is perfect, we define the polynomial ring $L[t,\sigma]$ using the relations,
\[
  tc=ct, \quad t\sigma= \sigma t, \quad \sigma c = c^{1/q} \sigma, \quad \forall\, c \in L.
\]
Just as $\tau$ plays the role of the $q$-th power Frobenius endomorphism on $L$, $\sigma$ represents the inverse of $\tau$, and in fact we can embed these rings compatibly as subrings of $L[t,\tau,\sigma]$, in which $\sigma = \tau^{-1}$.  Furthermore, for any $f \in L[t]$, it follows that
\[
  \tau f = f^{(1)} \tau, \quad \sigma f = f^{(-1)} \sigma.
\]

We now fix an Anderson $t$-module $\phi : \bA \to \Mat_d(L[\tau])$ over $L$.  The \emph{$t$-motive} $\cM_{\phi}$ associated to $\phi$ is defined as follows.  We set $\cM_{\phi} \assign \Mat_{1 \times d}(L[\tau])$ and give it the unique structure of a left $L[t,\tau]$-module: $\cM_{\phi}$ is already a left $L[\tau]$-module, and for $L[t]$-module structure we set
\begin{equation}\label{E:tactionmotive}
t \cdot m \assign m \phi_t, \quad \forall\, m \in \cM_{\phi}.
\end{equation}
For any $a \in \bA$ and $m \in \cM_{\phi}$, it follows that $a \cdot m = m \phi_a$.  For any $m \in \cM_{\phi}$, straightforward computation reveals that
\[
  (t-\theta)^d \cdot m \in \cM_{\phi}\tau,
\]
and so, when $L$ is perfect, $\cM_{\phi}$ is a $t$-motive in the sense of \cite{BPrapid}, \cite[Ch.~5]{Goss}.  Morphisms of $t$-motives are simply morphisms of left $L[t,\tau]$-modules.  If $\mu : \phi \to \psi$ is a morphism of $t$-modules, then $\mu$ induces a morphism of $t$-motives $\mu^{\dagger} : \cM_{\psi} \to \cM_{\phi}$ given by
\begin{equation} \label{E:biderdagger}
  \mu^{\dagger}(m) = m \mu,
\end{equation}
and in this way the functor from $t$-modules to $t$-motives is contravariant.  The number $d = \rank_{L[\tau]} \cM_{\phi}$ is called the \emph{dimension} of $\cM_{\phi}$.  If $\cM_{\phi}$ is free and finitely generated as a left $L[t]$-module, then $r = \rank_{L[t]} \cM_{\phi}$ is called the \emph{rank} of $\cM_{\phi}$.  In this latter case, $\cM_{\phi}$ is a $t$-motive in the sense of Anderson~\cite{And86}, and $\phi$ itself is called an \emph{abelian $t$-module}.

We now assume that $L$ is a perfect field. Before defining dual $t$-motives we recall the anti-isomorphism $* : L[\tau] \to L[\sigma]$, as defined by Ore~\cite[\S 1.7]{Goss}, \cite{Ore33}, \cite{Poonen96}.  For $f = \sum_{i=0}^\ell c_i \tau^i \in L[\tau]$, we set
\[
  f^* \assign \sum_{i=0}^{\ell} c_i^{1/q^i} \sigma^i = \sum_{i=0}^{\ell} \sigma^i c_i.
\]
One easily checks that this provides an isomorphism of $\FF_q$-vector spaces, and for $f$, $g \in L[\tau]$ we have
\[
  (f g)^* = g^* f^*.
\]
To extend the $*$-operation to matrices, we define $* : \Mat_{k \times \ell} (L[\tau]) \to \Mat_{\ell \times k}(L[\sigma])$ in the following way.  If $B = (b_{ij}) \in \Mat_{k\times \ell}(L[\tau])$, then
\[
  B^* \assign \bigl( b_{ij}^* \bigr)^{\tr} = \bigl( b_{ji}^* \bigr).
\]
In this way if $B \in \Mat_{k \times \ell}(L[\tau])$ and $C \in \Mat_{\ell \times m}(L[\tau])$, then one checks that
\[
  (BC)^* = C^* B^*.
\]
We denote the inverse of $* : \Mat_{k \times \ell} (L[\tau]) \to \Mat_{\ell \times k}(L[\sigma])$ also by $* : \Mat_{\ell \times k}(L[\sigma]) \to \Mat_{k \times \ell}(L[\tau])$, and trivially we have $(B^*)^* = B$ for all $B \in \Mat_{k\times \ell}(L[\tau])$.

For our same $t$-module $\phi$ above, we now define its \emph{dual $t$-motive} $\cN_{\phi}$.  We set $\cN_{\phi} \assign \Mat_{1\times d}(L[\sigma])$, and we give $\cN_{\phi}$ the structure of a left $L[t,\sigma]$-module.  The module $\cN_{\phi}$ clearly has the structure of a left $L[\sigma]$-module, and we give it an $L[t]$-module structure by setting for $n \in \cN_{\phi}$,
\begin{equation}\label{E:tactiondual}
  t \cdot n \assign n \phi_t^*.
\end{equation}
Similar to the $t$-motive of $\phi$, we find that $(t-\theta)^d \cN_{\phi} \subseteq \sigma \cN_{\phi}$, and so $\cN_{\phi}$ is a dual $t$-motive in the sense of \cite[\S 4]{ABP04}.  Morphisms of dual $t$-motives are simply morphisms of left $L[t,\sigma]$-modules, and if $\mu : \phi \to \psi$ is a morphism of $t$-modules, then we obtain a morphism $\mu_{\dagger} : \cN_{\phi} \to \cN_{\psi}$ given by
\[
  \mu_{\dagger}(n) = n \mu^*.
\]
The functor from $t$-modules to dual $t$-motives is thus covariant.  The \emph{dimension} of $\cN_{\phi}$ is $d = \rank_{L[\sigma]}\cN_{\phi}$.  If $\cN_{\phi}$ is free and finitely generated as an $L[t]$-module, then $r = \rank_{L[t]} \cN_{\phi}$ is called the \emph{rank} of $\cN_{\phi}$, and in this case $\cN_{\phi}$ and $\phi$ itself are said to be \emph{$\bA$-finite}.

If an Anderson $t$-module $\phi$ over $L=\KK$ is uniformizable and $\cM_{\phi}$ is abelian, then Anderson~\cite[\S 2]{And86} showed that
\[
  \rank_{\KK[t]} \cM_{\phi} = \rank_{\bA} \Lambda_{\phi}.
\]
Likewise, if $\phi$ is uniformizable and $\cN_{\phi}$ is $\bA$-finite, then Anderson showed (see \cite[\S 2.5]{HartlJuschka20})
\[
  \rank_{\KK[t]} \cN_{\phi} = \rank_{\bA} \Lambda_{\phi}.
\]
Thus if $\phi$ is uniformizable, abelian, and $\bA$-finite, then all three of these quantities are the same and are called the rank of~$\phi$.

\begin{remark}
Up until now it was not known if an Anderson $t$-module is abelian if and only if it is $\bA$-finite~\cite{HartlJuschka20}. However, during the time this article has been under consideration for publication, it has been announced by Maurischat~\cite{Maurischat21a} that this is indeed the case. Taking Maurischat's work into account, henceforth throughout the paper one can replace `abelian and $\bA$-finite' with simply `abelian.' We have chosen to keep the terminology `abelian and $\bA$-finite' for emphasis.

Prior to Maurischat's work, the equivalence between abelian and $\bA$-finite was already known for Drinfeld modules and for \emph{strictly pure} and \emph{almost strictly pure} Anderson $t$-modules (see Hartl and Juschka~\cite[\S 2.5.2]{HartlJuschka20} and also \S\ref{subS:E0calc}, \S\ref{subS:Quasi} of the present paper).
\end{remark}

The notation $\cM_{\phi}$ and $\cN_{\phi}$ does not include the dependence on the base field $L$, but in future sections, where we usually take $L = K^{\sep}$, $L = \oK$, or $L = \KK$, this dependence will be clear from the context.

\subsection{Hyperderivatives and hyperdifferential operators}
For a field $F$, a transcendental variable $\theta$ over $F$, and $j \geqslant 0$, we recall the $F$-linear map $\pd_{\theta}^j: F[\theta]\to F[\theta]$ defined by setting $\pd_{\theta}^j(\theta^m) = \binom{m}{j} \theta^{m-j}$, where $\binom{m}{j} \in \ZZ$ is the usual binomial coefficient.  The map $\pd_{\theta}^j$ is called the \emph{$j$-th hyperdifferential operator with respect to~$\theta$} or the \emph{$j$-th hyperderivative}.  These operators satisfy several identities, notably the product rule
\begin{equation} \label{E:productrule}
\pd_\theta^j(fg) = \sum\limits_{i=0}^j \pd_\theta^i(f)\pd_\theta^{j-i}(g),
\end{equation}
and composition rule
\begin{equation} \label{E:comprule}
\pd_\theta^i(\pd_\theta^j(f))= \binom{i+j}{j}\pd_\theta^{i+j}(f).
\end{equation}
The product rule induces the unique extension $\pd_{\theta}^j : F(\theta) \to F(\theta)$, in which case the rule $\pd_{\theta}^j(\theta^m) = \binom{m}{j} \theta^{m-j}$ is valid even for $m < 0$ (by taking as usual $\binom{-\ell}{j} = (-1)^j \binom{\ell + j - 1}{j}$ for $\ell > 0$).  Similarly for any place $v$ of $F(\theta)$ there is a unique extension $\pd_{\theta}^j : F(\theta)_v^{\sep} \to F(\theta)_v^{\sep}$.  See \cite[\S 4]{Conrad00}, \cite[\S 7]{Brownawell99}, \cite[\S 2]{Jeong11}, \cite[Ch.~2]{PLogAlg}, for more details.  There are other formulas for hyperderivatives that we will use, summarized in the following proposition.

\begin{proposition}[{see Brownawell~\cite[\S 7]{Brownawell99}, \cite[Lem.~2.1]{BrownawellDenis00}, Jeong~\cite[\S 2]{Jeong11}}] \label{P:Hyperprops}
Let $F$ be a field, and let $v$ be a place of $F(\theta)$.  The hyperdifferential operators with respect to $\theta$, $\pd_\theta^j : F(\theta)_v^{\sep} \to F(\theta)_v^{\sep}$, $j \geqslant 0$, satisfy the following.
\begin{enumerate}
\item [(a)] For $f_1, \dots, f_s \in F(\theta)_v^{\sep}$ and $j \geqslant 0$,
\[
\pd_\theta^j(f_1 \cdots f_s) = \sum_{\substack{k_1, \dots, k_s \geqslant 0 \\ k_1 + \dots + k_s=j}} \pd_\theta^{k_1}(f_1) \cdots \pd_\theta^{k_s}(f_s).
\]
\item [(b)] If the characteristic of $F$ is $p > 0$, then for $f \in F(\theta)_v^{\sep}$, $n \geqslant 0$, and $j \geqslant 1$,
\[
\pd_\theta^j\bigl( f^{p^n} \bigr) =
\begin{cases}
      \bigl(\pd_\theta^k(f)\bigr)^{p^n} & \text{if $j=kp^n$,} \\
      0 & \text{if $p^n \nmid j$.}
\end{cases}
\]
\end{enumerate}
\end{proposition}

\begin{definition}\label{def:dmatrices}
For $f \in F(\theta)_v^{\sep}$ and $n \geqslant 1$, we define the \emph{$d$-matrix with respect to $\theta$}, $d_{\theta, n} [f] \in \Mat_n(F(\theta)_v^{\sep})$ to be the upper-triangular $n \times n$ matrix
\[
d_{\theta, n}[f] \assign \begin{pmatrix}
f & \pd_\theta^1(f) & \cdots & \cdots & \pd_\theta^{n-1}(f) \\
  & f & \pd_\theta^1(f) &  & \vdots \\
  &   & \ddots & \ddots &\vdots \\
  &   &   & \ddots & \pd_\theta^1(f)\\
  &   &   &  & f
  \end{pmatrix}.
\]
It is a straightforward consequence of the product rule~\eqref{E:productrule} that the map
\[
  d_{\theta,n} : F(\theta)_v^{\sep} \to \Mat_n \bigl( F(\theta)_v^{\sep} \bigr)
\]
is an $F$-algebra homomorphism, and in particular $d_{\theta,n}[g] \cdot d_{\theta,n}[f] = d_{\theta,n}[gf]$.  We also set
\[
\pd_{\theta,n}[f] \assign d_{\theta,n}[f] \begin{pmatrix} 0 \\ \vdots \\ 0 \\ 1 \end{pmatrix}
= \begin{pmatrix} \pd_{\theta}^{n-1}(f) \\ \vdots \\ \pd_{\theta}^{1}(f) \\ f \end{pmatrix}
\in \bigl( F(\theta)_v^{\sep} \bigr)^n.
\]
We see immediately that $d_{\theta,n}[g] \cdot \pd_{\theta,n}[f] = \pd_{\theta,n}[gf]$.
\end{definition}

It will also be convenient to work with \emph{partial hyperderivatives}, which for independent variables $\theta_1, \dots, \theta_m$ over $F$ are defined by $F$-linear maps
\[
  \pd_{\theta_i}^j : F(\theta_1, \dots \theta_m) \to F(\theta_1, \dots, \theta_m), \quad j \geqslant 0,
\]
so that $\pd_{\theta_k}^j (\theta_k^m) = \binom{m}{j} \theta_k^{m-j}$ and, when $k' \neq k$, $\pd_{\theta_{k'}}^j (\theta_k^m) = 0$.  In this way $\pd_{\theta_k}^j \circ \pd_{\theta_{k'}}^{j'} = \pd_{\theta_{k'}}^{j'} \circ \pd_{\theta_k}^j$, and these operators extend in the natural way to completions and separable closures.  See \cite[Ch.~2]{Okugawa} for more details.  We can also define $d$-matrices with respect to each variable as in Definition~\ref{def:dmatrices}, and in this case each $d_{\theta_i,n}$ is an $F$-algebra homomorphism.

For our purposes we will pay particular attention to the two-variable case,
\[
  \pd_{\theta}^j,\ \pd_{t}^j : F(\theta,t) \to F(\theta,t),
\]
for independent variables $\theta$ and $t$.  For $f(\theta,t) \in F(\theta,t)$, we set $f|_{t=\theta} = f(\theta,\theta)$, and we say that $f$ is regular at $t=\theta$ if $f|_{t=\theta}$ exists in $F(\theta)$.  If $f$ is regular at $t =\theta$, then we have the standard identity
\begin{equation} \label{E:Taylorseries}
  f = f(\theta,t) = \sum_{i=0}^{\infty} \pd_t^i(f)|_{t=\theta} \cdot (t-\theta)^i \in \power{F(\theta)}{t-\theta}.
\end{equation}
The operators $\pd_{t}^j$ extend in the natural way to $\TT$ and $\TT_{\theta}$, and the operators $\pd_{\theta}^j$ extend to $\TT \cap \power{K_{\infty}^{\sep}}{t}$ and $\TT_{\theta} \cap \power{K_{\infty}^{\sep}}{t}$.  Functions $f \in \TT_{\theta}$ are regular at $t=\theta$ and also satisfy \eqref{E:Taylorseries} (see~\cite[\S 1.2]{BosserPellarin08}, \cite[Cor.~2.7]{US98}). We state some additional rules we will use (see \cite[Ch.~2]{PLogAlg} for more general classes of identities).  The first is a straightforward application of Taylor series, but the second requires some clarification, though it is essentially a special case of a multivariable chain rule.

\begin{lemma}[{cf.\ \cite[\S 2.4]{PLogAlg}}] \label{L:chainrule}
Let $f \in \TT_{\theta} \cap \power{K_{\infty}^{\sep}}{t}$.  Then for $j \geqslant 0$, the following hold.
\begin{enumerate}
\item[(a)] $\displaystyle \pd_\theta^j(f)\big|_{t=\theta} = \sum_{i=0}^j (-1)^i \cdot \pd_{\theta}^{j-i} \bigl( \pd_t^i(f)\big|_{t=\theta} \bigr)$.
\item[(b)] $\displaystyle \pd_\theta^j \bigl( f|_{t=\theta} \bigr) = \sum_{i=0}^j \bigl(\pd_\theta^{j-i} \circ \pd_t^i(f) \bigr) \big|_{t =\theta}$.
\end{enumerate}
\end{lemma}

\begin{proof}
If we apply $\pd_\theta^j$ to both sides of~\eqref{E:Taylorseries}, we obtain
\[
  \pd_\theta^j(f) = \sum_{i=0}^{\infty}\sum_{k=0}^j \pd_{\theta}^{j-k} \bigl( \pd_t^i(f)\big|_{t=\theta} \bigr) \cdot (-1)^k \binom{i}{k} (t-\theta)^{i-k}.
\]
Substituting $t=\theta$, the only non-zero terms occur when $k=i$, and we obtain~(a).  We can then apply~(a) to the right-hand side of the proposed identity in~(b),
\[
  \sum_{i=0}^j \bigl( \pd_\theta^{j-i} \circ \pd_t^i(f) \bigr) \big|_{t =\theta} = \sum_{i=0}^j \sum_{k=0}^{j-i} (-1)^k \cdot\pd_{\theta}^{j-i-k} \bigl( \pd_t^k(\pd_t^i(f))\big|_{t=\theta} \bigr).
\]
Applying~\eqref{E:comprule} and reindexing the sum ($i \leftarrow i+k$), we obtain
\[
\sum_{i=0}^j \bigl( \pd_\theta^{j-i} \circ \pd_t^i(f) \bigr) \big|_{t =\theta} =
\sum_{i=0}^j \sum_{k=0}^i (-1)^k \binom{i}{k} \cdot \pd_{\theta}^{j-i} \bigl( \pd_t^i(f)\big|_{t=\theta} \bigr).
\]
The inner sum on the right is $0$ unless $i=0$, and~(b) follows.
\end{proof}

\newpage
\section{Exponentiation and rigid analytic trivializations} \label{S:Exp}

We recall results on rigid analytic trivializations for abelian and $\bA$-finite $t$-modules and review the exponentiation theorem of Anderson (Theorem~\ref{T:Exponentiation}). Originally proved by Anderson in unpublished work, Hartl and Juschka recorded accounts of Anderson's work in \cite{HartlJuschka20}.  See also~\cite[\S 2]{ChangGreenMishiba21}, \cite[\S 4]{GezmisP19}, \cite[\S 3.5]{GreenNgoDac20b}.  In \S\ref{subS:E0calc} we then provide a way to render Anderson's theorem amenable to explicit calculation in Proposition~\ref{P:epsiota}.

\subsection{From dual \texorpdfstring{$t$}{t}-motives to \texorpdfstring{$t$}{t}-modules} \label{subS:dualtomodule}
Let $L$ be an algebraically closed field with $K \subseteq L \subseteq \KK$, and suppose $\phi : \bA \to \Mat_d (L[\tau])$ is an $\bA$-finite $t$-module with dual $t$-motive $\cN_{\phi} = \Mat_{1\times d}(L[\sigma])$ as in \S\ref{subS:motives}.  One advantage of dual $t$-motives is the facility of passing back and forth from a $t$-module to its dual $t$-motive.

For $\bsn= \sum_{i=0}^{\ell} \bsa_i\sigma^i \in \cN_{\phi}$ with $\bsa_i \in \Mat_{1\times d}(L)$, we set
\begin{equation}\label{varepsilon0and1}
\varepsilon_0(\bsn)\assign \rd \bsn^\tr = \bsa_0^\tr, \quad
\varepsilon_1(\bsn)\assign \Biggl( \sum_{i=0}^{\ell} \bsa_i^{(i)} \Biggr)^\tr.
\end{equation}
We note that $\varepsilon_0 : \Mat_{1 \times d}(L[\sigma]) \to L^d$ is $L$-linear and that $\varepsilon_1 : \Mat_{1\times d}(L[\sigma]) \to L^d$ is $\FF_q$-linear.  The following lemma is due to Anderson, and a proof can be found in \cite{HartlJuschka20} (see also \cite[Lem.~1.1.21--22]{Juschka10}).

\begin{lemma}[{see Hartl-Juschka \cite[Prop.~2.5.8]{HartlJuschka20}}] \label{L:Nquotients}
Let $\phi : \bA \to \Mat_d(L[\tau])$ be an $\bA$-finite $t$-module, and let $a\in \bA$.
\begin{enumerate}
\item[(a)] We have the following commutative diagram with exact rows:
\[
\begin{tikzcd}
0 \arrow{r} & \cN_{\phi} \arrow{r}{\sigma(\cdot)} \arrow{d} {a(\cdot)} & \cN_{\phi} \arrow{r}{\varepsilon_0} \arrow{d}{a(\cdot)} & L^d \arrow{r} \arrow{d}{\rd\phi_a(\cdot)} &0 \\
0 \arrow{r} & \cN_{\phi} \arrow{r}{\sigma(\cdot)} & \cN_{\phi} \arrow{r}{\varepsilon_0} & L^d \arrow{r} & 0.
\end{tikzcd}
\]
\item[(b)] We have the following commutative diagram with exact rows:
\[
\begin{tikzcd}
0 \arrow{r} & \cN_{\phi} \arrow{r}{(\sigma-1)(\cdot)} \arrow{d} {a(\cdot)} & \cN_{\phi} \arrow{r}{\varepsilon_1} \arrow{d}{a(\cdot)} & L^d \arrow{r} \arrow{d}{\phi_a(\cdot)} &0 \\
0 \arrow{r} & \cN_{\phi} \arrow{r}{(\sigma-1)(\cdot)} & \cN_{\phi} \arrow{r}{\varepsilon_1} & L^d \arrow{r} & 0.
\end{tikzcd}
\]
\end{enumerate}
\end{lemma}

\begin{remark}
Lemma~\ref{L:Nquotients} shows that we have induced isomorphisms
\[
  \varepsilon_0 : \frac{\cN_{\phi}}{\sigma \cN_{\phi}} \iso \Lie(\phi)(L), \quad
  \varepsilon_1 : \frac{\cN_{\phi}}{(\sigma -1 ) \cN_{\phi}} \iso \phi(L),
\]
where the first is $L[t]$-linear, and the second is $\bA$-linear.  In particular we recover the $t$-module $\phi$ from $\cN_{\phi}$ via $\varepsilon_1$.
\end{remark}

\subsection{\texorpdfstring{$t$}{t}-frames} \label{subS:tframes}
Picking $\{\bsn_1, \dots, \bsn_r\}$ to be an $L[t]$-basis of~$\cN_{\phi}$, we let $\Phi \in \Mat_r(L[t])$ be the unique matrix such that
\[
\sigma \bsn = \Phi \bsn, \quad  \bsn = (\bsn_1, \dots, \bsn_r)^\tr.
\]
By convention we will say that $\bsn \in \Mat_{d\times 1}(\cN_{\phi})$ forms an $L[t]$-basis of $\cN_{\phi}$.

We define the map
\begin{equation} \label{E:iotadef}
\iota : \Mat_{1\times r}(L[t]) \to \cN_{\phi}
\end{equation}
by setting for $\bsalpha=(\alpha_1, \dots, \alpha_r) \in \Mat_{1\times r}(L[t])$,
\[
\iota(\bsalpha) = \bsalpha \cdot \bsn = \alpha_1 \bsn_1 + \dots + \alpha_r \bsn_r.
\]
The pair $(\iota, \Phi)$ is called a \emph{$t$-frame} for $\phi$.  The following proposition is due to Anderson, but for completeness we provide a direct proof (cf.~\cite[Lem.~4.4.2]{GezmisP19}).

\begin{proposition} \label{P:tframeprops}
Given a $t$-frame $(\iota, \Phi)$ for the $\bA$-finite $t$-module $\phi : \bA \to \Mat_d(L[\tau])$, the following properties hold.
\begin{enumerate}
\item[(a)] $\det \Phi = c(t-\theta)^d$ for some $c\in L^{\times}$.
\item[(b)] $\iota(\bsalpha^{(-1)}\Phi) = \sigma \iota(\bsalpha)$ for all $\bsalpha\in \Mat_{1\times r}(L[t])$.
\item[(c)] $\iota(t\bsalpha)=t \cdot \iota(\bsalpha) = \iota(\bsalpha)\phi_t^*$ for all $\bsalpha\in \Mat_{1\times r}(L[t])$.
\end{enumerate}
\end{proposition}

\begin{proof}
Parts (b) and (c) follow directly from the definitions of $\iota$ and the $t$- and $\sigma$-actions on $\cN_{\phi}$.  Recalling that $\cN_{\phi} = \Mat_{1\times d}(L[\sigma])$, part (b) then implies that
\begin{equation} \label{E:sigmaN}
  \sigma \cN_{\phi} = \iota \bigl( \Mat_{1\times r}(L[t]) \cdot \Phi \bigr).
\end{equation}
Since $\iota$ is an isomorphism of $L[t]$-modules by~(c), it follows that the Fitting ideal of $\cN_{\phi}/\sigma \cN_{\phi}$ over $L[t]$ is generated by $\det \Phi$.  On the other hand, the definition of $\cN_{\phi}$ as a dual $t$-motive implies that $\cN_{\phi}/\sigma \cN_{\phi}$ is a $d$-dimensional $L$-vector space annihilated by a power of $t-\theta$, and the result follows.
\end{proof}

One useful combination of Lemma~\ref{L:Nquotients} and Proposition~\ref{P:tframeprops} is the following lemma, whose proof is immediate.

\begin{lemma} \label{L:epsilontframes}
Given a $t$-frame $(\iota,\Phi)$ for the $\bA$-finite $t$-module $\phi : \bA \to \Mat_d(L[\tau])$,
we have for all $\bsalpha \in \Mat_{1 \times r}(L[t])$ and $a \in \bA$,
\[
  \rd\phi_{a} \bigl( \varepsilon_0(\iota(\bsalpha)) \bigr) = \varepsilon_0( \iota (a\cdot \bsalpha)), \quad
  \phi_a \bigl( \varepsilon_1(\iota(\bsalpha)) \bigr) = \varepsilon_1( \iota(a \cdot \bsalpha)).
\]
\end{lemma}

For an abelian $t$-module $\phi$, we make a companion construction to the $t$-frame as follows.  Let $\{\bsm_1, \dots, \bsm_r\}$ be an $L[t]$-basis of $\cM_{\phi}$, and take $\tPhi \in \Mat_r(L[t])$ to be the unique matrix so that
\[
  \sigma \bsm = \tPhi \bsm, \quad \bsm = (\bsm_1, \dots, \bsm_r)^{\tr}.
\]
Also by convention we will say that $\bsm \in \Mat_{d \times 1}(\cM_{\phi})$ forms an $L[t]$-basis of $\cM_{\phi}$.

Likewise, let $\jmath : \Mat_{1\times r}(L[t]) \to \cM_{\phi}$ be defined by $\jmath(\bsalpha) = \bsalpha \cdot \bsm = \alpha_1 \bsm_1 + \dots + \alpha_r \bsm_r$, where $\bsalpha = (\alpha_1, \dots, \alpha_r)$.  The following proposition is the companion to Proposition~\ref{P:tframeprops}, with essentially the same proof.

\begin{proposition} \label{P:abeliantframeprops}
For an abelian $t$-module $\phi : \bA \to \Mat_d(L[\tau])$, let $\jmath : \Mat_{1\times r}(L[t]) \to \cM_{\phi}$ and $\tPhi \in \Mat_r(L[t])$ be defined as above.  The following hold.
\begin{enumerate}
\item[(a)] $\det \tPhi = c(t-\theta)^d$ for some $c\in L^{\times}$.
\item[(b)] $\jmath(\bsalpha^{(1)}\tPhi) = \tau \jmath(\bsalpha)$ for all $\bsalpha\in \Mat_{1\times r}(L[t])$.
\item[(c)] $\jmath(t\bsalpha)=t \cdot \jmath(\bsalpha) = \jmath(\bsalpha)\phi_t$ for all $\bsalpha\in \Mat_{1\times r}(L[t])$.
\end{enumerate}
\end{proposition}

\subsection{Exponentiation via division towers}
We now fix an $\bA$-finite $t$-module $\phi : \bA \to \Mat_d(\KK[\tau])$ of dimension $d$.  We present here a construction of the exponential of $\phi$ via $t$-division towers of points due to Anderson.  The complete construction with proofs can be found in~\cite[\S 2.5.3]{HartlJuschka20}. For similar constructions in the context of Drinfeld modules over Tate algebras, see also~\cite[\S 4]{GezmisP19}.

For $X = (x_{ij}) \in \Mat_{m \times n}(\KK)$, we set $\inorm{X} \assign \max(\inorm{x_{ij}})$, making $\Mat_{m \times n}(\KK)$ into a complete normed vector space.

\begin{definition} \label{def:tconvergent}
For $\bsx \in \KK^d$ and a sequence $\{\bsx_n\}_{n=0}^\infty$ in $\KK^d$, consider the conditions
\begin{enumerate}
\item[(a)] $\phi_t(\bsx_{n+1}) = \bsx_n$ for all $n \geqslant 0$,
\item[(b)] $\phi_t(\bsx_0)=\bsx$,
\item[(c)] $\lim_{n \to \infty} \inorm{\bsx_n} = 0$ with respect to $\inorm{\,\cdot\,}$.
\end{enumerate}
If $\{\bsx_n\}_{n=0}^\infty$ satisfies (a) and (b), then it is called a \emph{$t$-division sequence above $\bsx$}, and if it satisfies all three conditions, then it is called a \emph{convergent $t$-division sequence above $\bsx$}.
\end{definition}

\begin{theorem}[{Anderson; see \cite[Thm.~2.5.20]{HartlJuschka20}, cf.\ \cite[Thm.~4.3.2]{GezmisP19}}] \label{T:Exptdivision}
Let $\phi: \bA \to \Mat_d(\KK[\tau])$ be an $\bA$-finite $t$-module, and let $\bsx \in \KK^d$.  There is a bijection
\[
G: \{ \bsu \in \KK^d : \Exp_\phi(\bsu) = \bsx \} \to  \{\textup{convergent $t$-division sequences above $\bsx$} \},
\]
defined by
\[
  G(\bsu) = \Bigl\{ \Exp_{\phi} \bigl( (\rd \phi_t)^{-1}(\bsu) \bigr), \Exp_{\phi} \bigl( (\rd \phi_t)^{-2}(\bsu) \bigr), \ldots \Bigr\}.
\]
Moreover, if $\{\bsx_n\}_{n=0}^\infty$ is a convergent $t$-division sequence above $\bsx$, then with respect to~$\inorm{\,\cdot\,}$,
\[
\lim_{n \to \infty}(\rd\phi_t)^{n+1} \bsx_n = \bsu,
\]
and we have
\[
\Exp_\phi(\bsu) = \bsx, \quad G(\bsu) = \{ \bsx_n \}_{n=0}^\infty.
\]
\end{theorem}

\begin{remark} \label{R:dphitinverse}
Because it will be useful in \S\ref{S:Quasi} and elsewhere, we should say a few words about why the map $G$ in Theorem~\ref{T:Exptdivision} is well-defined.  Indeed if we write $\rd \phi_t = \theta \Id_d + N$, where $N$ is a nilpotent matrix, then necessarily $N^d = 0$.  Thus
\[
  (\rd\phi_t)^{-1} = \frac{1}{\theta} \biggl( \Id_d - \frac{1}{\theta}\cdot N + \frac{1}{\theta^2} \cdot N^2 - \cdots + \frac{(-1)^{d-1}}{\theta^{d-1}}\cdot N^{d-1} \biggr).
\]
It follows that for $n \geqslant 1$,
\[
  \bigl\lvert (\rd \phi_t)^{-n} \bigr\rvert_{\infty} \leqslant \max_{0 \leqslant j \leqslant d-1} \bigl\{ \inorm{\theta}^{-n-j} \cdot \inorm{N}^j \bigr\},
\]
and so $\inorm{(\rd \phi_t)^{-n}} \to 0$ as $n \to \infty$.  Thus $G(\bsu)$ satisfies Definition~\ref{def:tconvergent}(c) as desired.
\end{remark}

\subsection{Exponentiation via Frobenius difference equations} \label{subS:AndersonThm}
Here we bring together the information from the preceding parts of this section to present Anderson's exponentiation theorem, which explicitly connects the exponential map to solutions of Frobenius difference equations and rigid analytic trivializations.  Throughout this section we fix an $\bA$-finite $t$-module $\phi: \bA \to \Mat_d(\KK[\tau])$ with $t$-frame $(\iota,\Phi)$.  We first observe the following lemma.

\begin{lemma}[{Anderson; see \cite[Prop.~2.5.8]{HartlJuschka20}, cf.\ \cite[Rem.~4.4.4]{GezmisP19}}] \label{L:E0}
There exists a unique bounded $\KK$-linear map
\[
\cE_0: \bigl( \Mat_{1\times r}(\TT_\theta), \dnorm{\,\cdot\,}_{\theta} \bigr) \to \bigl( \KK^d, \inorm{\,\cdot\,} \bigr)
\]
of normed vector spaces such that $\cE_0|_{\Mat_{1\times r}(\KK[t])} = \varepsilon_0 \circ \iota$.
\end{lemma}

For symmetry of notation, we furthermore let
\[
\cE_1 \assign \varepsilon_1 \circ \iota : \Mat_{1\times r}(\KK[t]) \to \KK^d.
\]
The following theorem is the main technique to express the exponentiation on $\phi$ in terms of solutions of Frobenius difference equations in the Tate algebra~$\TT_{\theta}$.

\begin{theorem}[{Anderson, see \cite[Thm.~2.5.21, Cor.~2.5.23]{HartlJuschka20}, cf.\ \cite[Thm.~4.4.6]{GezmisP19}}] \label{T:Exponentiation}
Let $\phi : \bA \to \Mat_d(\KK[\tau])$ be an $\bA$-finite $t$-module with $t$-frame $(\iota, \Phi)$. Fix $\bsh \in \Mat_{1 \times r}(\KK[t])$, and suppose there exists $\bsg \in \Mat_{1 \times r}(\TT_\theta)$ such that
\[
\bsg^{(-1)}\Phi - \bsg = \bsh.
\]
Then
\[
\Exp_\phi\bigl( \cE_0(\bsg + \bsh) \bigr) = \cE_1(\bsh).
\]
\end{theorem}

\begin{remark}
Anderson's argument proceeds along the following lines.  Suppose that in Theorem~\ref{T:Exponentiation} we have $\bsg = \sum_{i=0}^{\infty} \bsa_i t^i \in \Mat_{1 \times r}(\power{\KK}{t})$ with $\bsa_i \in \Mat_{1\times r}(\KK)$.  For $n \geqslant 0$, set
\[
  \bsg_{> n} \assign \sum_{i=n+1}^{\infty} \bsa_i t^i, \quad \bsg_{\leqslant n} \assign \sum_{i=0}^{n} \bsa_i t^i.
\]
Taking
\[
  \bsb_n \assign \frac{\bsg_{> n}^{(-1)} \Phi - \bsg_{>n}}{t^{n+1}}
  = \frac{\bsh + \bsg_{\leqslant n} - \bsg_{\leqslant n} \Phi^{(-1)}}{t^{n+1}} \in \Mat_{1 \times r}(\KK[t]),
\]
we see that the entries of the first expression are guaranteed to have no denominators, whereas the second has entries that are Laurent polynomials in $t$, and thus $\bsb_n \in \Mat_{1 \times r}(\KK[t])$.  Straightforward calculations using Lemma~\ref{L:Nquotients}(b) and Proposition~\ref{P:tframeprops} imply that $\{ \varepsilon_1(\iota(\bsb_n)) \}_{n=0}^{\infty}$ is a convergent $t$-division sequence above $\varepsilon_1(\iota(\bsh))$.  Furthermore, additional calculations show that with respect to $\inorm{\,\cdot\,}$,
\[
  \lim_{n\to \infty} (\rd \phi_t)^{n+1}\varepsilon_1(\iota(\bsb_n))
  = \lim_{n \to \infty} (\rd \phi_t)^{n+1} \varepsilon_0(\iota(\bsb_n))
  = \cE_0(\bsg +\bsh),
\]
and thus by Theorem~\ref{T:Exptdivision},
\[
  \Exp_{\phi}(\cE_0(\bsg+\bsh)) = \cE_1(\bsh).
\]
For complete details see \cite[Thm.~2.5.21, Cor.~2.5.23]{HartlJuschka20}, and for a similar construction see~\cite[Thm.~4.4.6]{GezmisP19}.
\end{remark}

\begin{remark} \label{R:expequivariance}
For $\bsg$ and $\bsh$ as in Theorem~\ref{T:Exponentiation}, we note that for $a \in \bA$, using the functional equation for $\Exp_{\phi}$ and Lemma~\ref{L:epsilontframes}, we have
\[
  \Exp_{\phi} \bigl( \rd \phi_a \cdot \cE_0(\bsg + \bsh) \bigr)
  = \Exp_{\phi} \bigl( \cE_0 ( a(\bsg + \bsh)) \bigr) = \phi_a \bigl( \cE_1(\bsh) \bigr) = \cE_1(a \bsh).
\]
\end{remark}

\begin{definition} \label{D:rat}
Given a $t$-frame $(\iota, \Phi)$ for the $\bA$-finite $t$-module $\phi$, suppose $\Psi \in \GL_r(\TT)$ satisfies
\begin{equation} \label{E:PsiPhi}
\Psi^{(-1)}=\Phi \Psi.
\end{equation}
Then we say that $(\iota, \Phi, \Psi)$ is a \emph{rigid analytic trivialization} of $\phi$.  By~\cite[Prop.~3.1.3]{ABP04} together with Proposition~\ref{P:tframeprops}(a), the entries of $\Psi$ are in fact entire functions in $\EE$.  Also $(\det \Psi)^{(-1)} = \det \Phi \cdot \det \Psi = c (t-\theta)^d \cdot \det \Psi$, for some $c \in \KK^{\times}$.  Since $\det \Psi \in \TT^{\times}$,
\[
  \det \Psi = \delta \gamma \cdot \Omega^d
\]
for $\gamma \in \KK$ with $\gamma^{(1-q)/q} = c$ and $\delta \in \FF_q^{\times}$, where $\Omega$ is defined in~\eqref{E:Omegadef}.  But $\Omega \in \TT_{\theta}^{\times}$, and so furthermore $\Psi \in \GL_{r}(\TT_{\theta})$.
\end{definition}

The following theorem provides the fundamental connection between uniformizability and rigid analytic trivializations, and it furnishes an explicit way to obtain the period lattice $\Lambda_{\phi}$.  Part~(a) is much in line with Anderson's original criteria for uniformizability in terms of $t$-motives (see \cite[Thm.~4]{And86}), and we demonstrate in Remark~\ref{R:PeriodLatticeRmk} how~(b) follows from the previous discussion.

\begin{theorem}[{Anderson; see \cite[Cor.~2.5.23, Thm.~2.5.32]{HartlJuschka20}, cf.\ \cite[Thm.~4.5.14]{GezmisP19}}] \label{T:PeriodLattice}
Let $\phi : \bA \to \Mat_d(\KK[\tau])$ be an $\bA$-finite $t$-module defined over $\KK$.
\begin{enumerate}
\item[(a)] $\phi$ is uniformizable if and only if it has a rigid analytic trivialization.
\item[(b)] If $(\iota, \Phi, \Psi)$ is a rigid analytic trivialization of $\phi$, then
\[
\Lambda_{\phi} = \cE_0 \bigl( \Mat_{1 \times r}(\bA) \cdot \Psi^{-1} \bigr) \subseteq \KK^d.
\]
\end{enumerate}
\end{theorem}

\begin{remark} \label{R:PeriodLatticeRmk}
We sketch the proof of~(b).  By~\eqref{E:PsiPhi} we see that if $\bsg$ is a row of $\Psi^{-1}$, then $\bsg^{(-1)} \Phi = \bsg$.  Theorem~\ref{T:Exponentiation} then implies that $\Exp_{\phi} \bigl( \cE_0(\bsg) \bigr) = 0$, and thus $\cE_0(\bsg) \in \Lambda_{\phi}$.  It then follows that
\[
  \cE_0 \bigl( \Mat_{1 \times r}(\bA) \cdot \Psi^{-1} \bigr) \subseteq \Lambda_{\phi}.
\]
To prove the opposite containment, we consider $\bslambda \in \Lambda_{\phi}$, and for $n \geqslant 0$ we set
\[
  \bslambda_n \assign (\rd \phi_t)^{-n-1} \bslambda.
\]
If $\bsv \in \iota^{-1}(\Mat_{1\times d}(\KK)) \subseteq \Mat_{1\times r}(\KK[t])$, then $\varepsilon_0(\iota(\bsv)) = \varepsilon_1(\iota(\bsv))$, so for each $n \geqslant 0$ we can pick unique $\bsh_n \in \iota^{-1}(\Mat_{1 \times d}(\KK))$ so that
\[
  \varepsilon_0(\iota(\bsh_n)) = \varepsilon_1(\iota(\bsh_n)) = \Exp_{\phi}(\bslambda_n).
\]
Because $\inorm{\bslambda_n} \to 0$, it follows from Lemma~\ref{L:E0} that $\dnorm{\bsh_n}_{\theta} \to 0$, and so for $n \gg 0$, $\dnorm{\bsh_n \cdot \Psi} \leqslant \dnorm{\bsh_n} \cdot \dnorm{\Psi} < 1$.  Taking such $n \gg 0$,
\begin{equation}
  \bsg_n = \sum_{i=1}^{\infty} (\bsh_n \cdot \Psi)^{(i)} \cdot \Psi^{-1} \in \Mat_{1\times r}(\TT_{\theta}),
\end{equation}
and we have $\bsg_n^{(-1)} \Phi - \bsg_n = \bsh_n$.  Thus, by Theorem~\ref{T:Exponentiation}
\[
  \Exp_{\phi}(\cE_0(\bsg_n + \bsh_n)) = \cE_1(\bsh_n) = \Exp_{\phi}(\bslambda_n).
\]
Since $\inorm{\cE_0(\bsg_n + \bsh_n)} \to 0$ as $n \to \infty$, the fact that $\Exp_{\phi}$ is an isometric embedding on sufficiently small open balls in $\KK^d$ (see~\cite[Lem.~2.5.4]{HartlJuschka20}) implies that for $n \gg 0$, we have
\[
\cE_0(\bsg_n+\bsh_n) = \bslambda_n.
\]
Fixing such an $N \gg 0$, Remark~\ref{R:expequivariance} then implies
\[
0 = \Exp_{\phi}(\bslambda) = \Exp_{\phi} \bigl( \rd \phi_{t^{N+1}} \cdot \bslambda_{N} \bigr)
= \Exp_{\phi} \bigl( \cE_0(t^{N+1}(\bsg_N + \bsh_N)) \bigr) = \varepsilon_1(\iota(t^{N+1} \bsh_N)).
\]
But then Lemma~\ref{L:Nquotients}(b) and Proposition~\ref{P:tframeprops}(b) imply there exists $\bsm \in \Mat_{1 \times r}(\KK[t])$ such that
\[
  \iota (t^{N+1} \bsh_N) = (\sigma-1)\cdot \iota(\bsm) = \iota(\bsm^{(-1)}\Phi - \bsm),
\]
and so $t^{N+1} \bsh_N = \bsm^{(-1)} \Phi - \bsm$.  Combining this with $\bsg_N^{(-1)} \Phi - \bsg_N = \bsh_N$ we see that
\[
  ( t^{N+1} \bsg_N - \bsm)^{(-1)} \Phi - (t^{N+1} \bsg_N - \bsm) = 0,
\]
and it follows that $t^{N+1} \bsg_N - \bsm$ is in the $\bA$-linear span of the rows of $\Psi^{-1}$ by~\cite[Lem.~4.4.12]{ABP04}.  Thus $\Lambda_{\phi} \subseteq \cE_0(\bA \cdot \Psi^{-1})$.
\end{remark}

\begin{remark} \label{R:E0Logarithms}
The argument in Remark~\ref{R:PeriodLatticeRmk} works just as well for $\bslambda$ replaced by any $\bsz \in \KK^d$.  That is, if $(\iota, \Phi, \Psi)$ is a rigid analytic trivialization for $\phi$, then for any $\bsz \in \KK^d$, it is possible to find $\bsg \in \Mat_{1 \times r}(\TT_{\theta})$ and $\bsh \in \Mat_{1 \times r}(\KK[t])$ so that $\bsg^{(-1)} \Phi - \bsg = \bsh$ and
\[
   \cE_0(\bsg + \bsh) = \bsz, \quad \cE_1(\bsh) = \Exp_{\phi}(\bsz).
\]
Thus if $\phi$ is rigid analytically trivial, then exponentiation of any point in $\KK^d$ is expressible as in Theorem~\ref{T:Exponentiation}.
\end{remark}

\subsection{Calculation of \texorpdfstring{$\cE_0$}{E0} and examples} \label{subS:E0calc}
Fix an $\bA$-finite $t$-module $\phi : \bA \to \Mat_d(\KK[\tau])$ as in the previous sections, and fix a $t$-frame $(\iota,\Phi)$ for $\phi$.  Then Theorem~\ref{T:Exponentiation} makes explicit calculation of $\cE_0 : \Mat_{1 \times r}(\TT_{\theta}) \to \KK^d$ of special interest.  To facilitate this it will help to investigate the effects of changing bases over both $\KK[t]$ and $\KK[\sigma]$ on the dual $t$-motive $\cN_{\phi}$.

We let $\bss_1, \dots, \bss_d$ denote the standard basis vectors in $\cN_{\phi} = \Mat_{1\times d}(\KK[\sigma])$.  For $P \in \GL_d(\KK[\sigma])$, we produce a new $\KK[\sigma]$-basis $\bss_1 P, \dots, \bss_d P$ for $\cN_{\phi}$.  Now multiplication by~$t$ on~$\cN_{\phi}$ is represented by right multiplication by $\phi_t^*$, and so with respect to our new basis, multiplication by~$t$ is represented by right multiplication by $P \phi_t^* P^{-1}$.  In this way, changing the $\KK[\sigma]$-basis of $\cN_{\phi}$ amounts to changing the coordinates for $\phi$ to an isomorphic $t$-module ${{\rho}} : \bA \to \Mat_d(\KK[\tau])$ defined by
\begin{equation} \label{E:changecoords}
  \rho_t = \bigl( P^{*} )^{-1}\cdot \phi_t \cdot P^{*},
\end{equation}
and notably $\rho_t^* = P \phi_t^* P^{-1}$.  Naturally, $P^{*} : \rho \to \phi$ is a $t$-module isomorphism.

Likewise, if we let $\bse_1, \dots, \bse_r$ denote the standard basis vectors of $\Mat_{1 \times r}(\KK[t])$ and we take $\bsn_1, \dots, \bsn_r \in \cN_{\phi}$ to be the $\KK[t]$-basis of $\cN_{\phi}$ used to define $\iota$ in~\eqref{E:iotadef}, then
\[
  \iota(\bse_i) = \bsn_i, \quad 1 \leqslant i \leqslant r.
\]
Picking another $\KK[t]$-basis $\bsn_1', \dots, \bsn_r'$ of $\cN_{\phi}$ is equivalent to picking $B \in \GL_r(\KK[t])$ so that $\bsn' = B \bsn$, where $\bsn = (\bsn_1, \dots, \bsn_r)^{\tr}$ and $\bsn' = (\bsn_1', \dots, \bsn_r')^{\tr}$.  The resulting $t$-frame $(\iota', \Phi')$ satisfies $\iota'(\bse_i) = \bsn_i'$ for each $i$, and
\begin{equation} \label{E:tframebasechange}
 \Phi' = B^{(-1)} \Phi B^{-1}.
\end{equation}
If $\Psi$ is a rigid analytic trivialization for $\Phi$, then every rigid analytic trivialization $\Psi'$ for~$\Phi'$ has the form $\Psi'=B\Psi\gamma$ for some $\gamma \in \GL_r(\FF_q[t])$. Thus we can arrange that
\begin{equation} \label{E:ratbasechange}
  \Psi' = B \Psi.
\end{equation}

Changing $\KK[t]$-bases for $\cN_{\phi}$ does not affect the $t$-module $\phi$, its exponential $\Exp_{\phi}$, nor its period lattice $\Lambda_\phi$, but it does affect how we use Theorem~\ref{T:Exponentiation} to calculate them.  We note from~\eqref{E:sigmaN} that we have a well-defined map of $\KK$-vector spaces,
\begin{equation} \label{E:epsiotafactor}
  \varepsilon_0 \circ \iota : \frac{\Mat_{1\times r}(\KK[t])}{\Mat_{1\times r}(\KK[t]) \cdot \Phi}
  \to \KK^d
\end{equation}
such that
\[
  \varepsilon_0(\iota(\bse_i)) = (\rd \bsn_i)^{\tr}, \quad 1 \leqslant i \leqslant r.
\]
Furthermore,
\begin{equation}
  \bigl(\varepsilon_0(\iota(\bse_1)), \dots, \varepsilon_0(\iota(\bse_r)) \bigr)
  = \begin{pmatrix} \rd \bsn_1 \\ \vdots \\ \rd \bsn_r \end{pmatrix}^\tr
  = \bigl( \rd \bsn \bigr)^{\tr} \in \Mat_{d\times r}(\KK),
\end{equation}
and likewise
\begin{equation} \label{E:epsiotabasechange}
  \bigl(\varepsilon_0(\iota'(\bse_1)), \dots, \varepsilon_0(\iota'(\bse_r)) \bigr)
  = \bigl( \rd(B\cdot \bsn) \bigr)^{\tr} \in \Mat_{d \times r}(\KK).
\end{equation}
We now demonstrate how to calculate $\varepsilon_0 \circ \iota$ and $\cE_0$ when the $\KK[t]$- and $\KK[\sigma]$-bases of $\cN_{\phi}$ are particularly convenient.  We then show in Remark~\ref{R:basechangered} how to reduce to this case for general~$\phi$.

\begin{proposition}[{cf.\ \cite[Thm.~3.3.5]{ChangGreenMishiba21}}] \label{P:epsiota}
Let $\phi : \bA \to \Mat_d(\KK[\tau])$ be an $\bA$-finite $t$-module with $t$-frame $(\iota, \Phi)$.  Suppose that the following conditions hold.
\begin{enumerate}
\item[(i)] There exists $C \in \GL_r(\KK[t])$ so that
\[
  D \assign C\Phi = \begin{pmatrix} (t-\theta)^{\ell_1} & & \\ & \ddots & \\ & & (t-\theta)^{\ell_r} \end{pmatrix}
\]
is a diagonal matrix.  Furthermore, for some $m$ with $1 \leqslant m \leqslant r$, we have $\ell_1, \dots, \ell_m > 0$, $\ell_{m+1} = \cdots = \ell_r = 0$, and $\ell_1 + \cdots + \ell_m = d$.
\item[(ii)] For $1 \leqslant i \leqslant m$ and $1 \leqslant j \leqslant \ell_i$,
\[
  \rd \iota\bigl( (t-\theta)^{j-1}\cdot \bse_{i} \bigr) = \bss_{\ell_1 + \cdots + \ell_i - j+1}.
\]
\end{enumerate}
Then for $\bsalpha = (\alpha_1, \dots, \alpha_r) \in \Mat_{1\times r}(\TT_{\theta})$,
\begin{equation} \label{E:E0}
  \cE_0(\bsalpha) =
  \left. \begin{pmatrix}
  \pd_{t,\ell_1}[\alpha_1] \\ \vdots \\ \pd_{t,\ell_m}[\alpha_m]
  \end{pmatrix}\right\rvert_{t=\theta}
  = \left. \begin{pmatrix}
  \pd_{t}^{\ell_1 -1}(\alpha_1) \\ \vdots \\
  \pd_{t}^1 (\alpha_1) \\
  \alpha_1 \\ \vdots \\
  \pd_{t}^{\ell_m-1}(\alpha_m) \\ \vdots \\ \pd_{t}^{1}(\alpha_m) \\ \alpha_m
  \end{pmatrix} \right|_{t=\theta}.
\end{equation}
\end{proposition}

\begin{proof}
By Lemma~\ref{L:E0} it suffices to show that $\varepsilon_0(\iota(\bsalpha))$ satisfies~\eqref{E:E0} when $\bsalpha \in \Mat_{1\times r} (\KK[t])$.  We observe that~\eqref{E:epsiotafactor} is unchanged if $\Phi$ is replaced by $C\Phi$ for $C \in \GL_r(\KK[t])$, and so by~(i),
\[
\frac{\Mat_{1\times r}(\KK[t])}{\Mat_{1\times r}(\KK[t]) \cdot \Phi}
= \frac{\Mat_{1\times r}(\KK[t])}{\Mat_{1\times r}(\KK[t]) \cdot D}
\cong \frac{\KK[t]}{(t-\theta)^{\ell_1}\KK[t]} \oplus \cdots \oplus
\frac{\KK[t]}{(t-\theta)^{\ell_m}\KK[t]}.
\]
The $m$ components on the right are generated by $\bse_1, \dots, \bse_m$.  For fixed $i$, $1 \leqslant i \leqslant m$, we see from~\eqref{E:Taylorseries} that
\[
  \alpha_i \equiv \alpha_i(\theta) + \pd_t^1(\alpha_i)|_{t=\theta} \cdot (t-\theta)
  + \cdots + \pd_t^{\ell_i-1}(\alpha_i)|_{t=\theta} \cdot (t-\theta)^{\ell_i-1} \pmod{(t-\theta)^{\ell_i}},
\]
and by~(ii) we find
\[
  \varepsilon_0 (\iota(\alpha_i \bse_i)) = \bigl( 0, \dots, 0, \pd_t^{\ell_i-1}(\alpha_i)|_{t=\theta}, \dots, \pd_t^1(\alpha_i)|_{t=\theta},\alpha_i(\theta), 0, \dots, 0 \bigr)^{\tr},
\]
where the non-zero entries are in places $\ell_{1}+ \cdots + \ell_{i-1}+1$ to $\ell_1 + \cdots + \ell_i$.  For $i > m$, (i) and~\eqref{E:epsiotafactor} imply that $\varepsilon_0(\iota(\alpha_i \bse_i)) = 0$.  Thus, $\varepsilon_0(\iota(\bsalpha)) = \varepsilon_0(\iota(\alpha_1 \bse_1)) + \cdots + \varepsilon_0(\iota(\alpha_m \bse_m))$, and the result follows.
\end{proof}

\begin{remark} \label{R:dphit}
Hypotheses~(i) and~(ii) in Proposition~\ref{P:epsiota} also imply that the matrix $\rd \phi_t$ is in Jordan normal form and consists of $m$ Jordan blocks of size $\ell_1, \dots, \ell_m$.  That is,
\begin{equation} \label{E:phitJordan}
  \rd \phi_t = \begin{pmatrix} d_{\theta,\ell_1}[\theta] & & \\ & \ddots & \\ & & d_{\theta,\ell_m}[\theta] \end{pmatrix}.
\end{equation}
Indeed by Lemma~\ref{L:Nquotients}(a) and~(ii), for $1 \leqslant i \leqslant m$ and $1\leqslant j < \ell_i$,
\begin{align*}
  \bss_{\ell_1 + \cdots + \ell_i - j + 1} \cdot (\rd \phi_t)^{\tr} &= \rd \iota\bigl( t \cdot (t-\theta)^{j-1} \bse_{i} \bigr) \\
  &= \rd \iota\bigl( ( (t-\theta)^{j} + \theta(t-\theta)^{j-1}) \bse_{i} \bigr) \\
  &= \bss_{\ell_1 + \cdots + \ell_i - j} + \theta \bss_{\ell_1 + \cdots + \ell_i - j+1}.
\end{align*}
Likewise, when $j = \ell_i$, we have $\bss_{\ell_1 + \cdots + \ell_{i-1} + 1} \cdot (\rd \phi_t)^{\tr} = \theta \bss_{\ell_1 + \cdots + \ell_{i-1} + 1}$.  The peculiar ordering in~(ii) is in part explained by the desire to put $\rd \phi_t$ in Jordan normal form.
\end{remark}

\begin{remark} \label{R:basechangered}
It remains now to demonstrate that any arbitrary $\bA$-finite $t$-module $\phi : \bA \to \Mat_d(\KK[\tau])$ can be put in the form of Proposition~\ref{P:epsiota}, at least after possibly changing the bases of $\cN_{\phi}$ over $\KK[t]$ and $\KK[\sigma]$.

Since $\KK[t]$ is a principal ideal domain, there are matrices $B$, $C \in \GL_r(\KK[t])$ so that
\[
  D \assign C \Phi B^{-1} = \begin{pmatrix} (t-\theta)^{\ell_1} & & \\ & \ddots & \\ & & (t-\theta)^{\ell_r} \end{pmatrix},
\]
as in the statement of Proposition~\ref{P:epsiota}(i) (see~\cite[\S 3.7]{JacobsonI}).  If we use $B$ to change our $\KK[t]$-basis on $\cN_{\phi}$ as in~\eqref{E:tframebasechange}, then taking $\Phi' = B^{(-1)} \Phi B^{-1}$, we see that
\[
  D = C\bigl( B^{(-1)} \bigr)^{-1} \Phi',
\]
and so the $t$-frame $(\iota',\Phi')$ satisfies~(i).  Without loss of generality we will assume that this change of basis has been made and that $(\iota,\Phi)$ itself satisfies~(i).

Now we will see that Proposition~\ref{P:epsiota}(ii) can also be satisfied after making a change of $\KK[\sigma]$-basis for $\cN_\phi$, and in fact we need only make a change of coordinates over~$\KK$.  First we let $\bsn_i \assign \iota(\bse_i)$ for $1 \leqslant i \leqslant r$.  We know already that $\rd\iota(\bse_i) = (\varepsilon_0(\iota(\bse_i)))^{\tr} = 0$ for $i > m$ by~(i).  Consider the $d$ vectors in~$\Mat_{1\times d}(\KK)$:
\begin{equation} \label{E:newbasis}
\begin{array}[t]{cccc}
\rd \bsn_1, &\rd \bsn_1 \cdot (\rd \phi_t^{*} - \theta \Id_d), & \ldots, & \rd \bsn_1 \cdot (\rd \phi_t^{*} - \theta \Id_d)^{\ell_1 - 1}, \\
\vdots & \vdots & & \vdots \\
\rd \bsn_m, &\rd \bsn_m \cdot (\rd \phi_t^{*} - \theta \Id_d), & \ldots, & \rd \bsn_m \cdot (\rd \phi_t^{*} - \theta \Id_d)^{\ell_m - 1}.
\end{array}
\end{equation}
For $1 \leqslant i \leqslant m$ and $1 \leqslant j \leqslant \ell_i$, a short calculation reveals
\[
  \rd \iota\bigl( (t-\theta)^{j-1} \cdot \bse_i \bigr) = \rd\bsn_i \cdot (\rd \phi_t^{*} - \theta \Id_d)^{j-1},
\]
and since $\{ (t-\theta)^{j-1} \cdot \bse_i : 1 \leqslant i \leqslant m,\ 1 \leqslant j \leqslant \ell_i \}$ is a $\KK$-basis of the domain of
\[
\rd \circ \iota : \frac{\Mat_{1\times r}(\KK[t])}{\Mat_{1\times r}(\KK[t]) \cdot \Phi} \iso \frac{\cN_{\phi}}{\sigma \cN_{\phi}},
\]
it follows that the vectors in~\eqref{E:newbasis} form a $\KK$-basis of $\cN_{\phi}/\sigma\cN_{\phi}$.  Let $P \in \GL_d(\KK)$ be the matrix whose rows are the vectors in~\eqref{E:newbasis} ordered in the following way:
\[
  P = \begin{pmatrix}
  \rd \bsn_1 \cdot (\rd \phi_t^{*} - \theta \Id_d)^{\ell_1 - 1} \\
  \vdots \\ \rd \bsn_1 \cdot (\rd \phi_t^{*} - \theta \Id_d) \\ \rd \bsn_1 \\
  \vdots \\
  \rd \bsn_m \cdot (\rd \phi_t^{*} - \theta \Id_d)^{\ell_m - 1} \\
  \vdots \\ \rd \bsn_m \cdot (\rd \phi_t^{*} - \theta \Id_d) \\ \rd \bsn_m
\end{pmatrix}.
\]
Now let $\rho : \bA \to \Mat_d(\KK[\tau])$ be the change of coordinates on $\phi$ given in \eqref{E:changecoords}, keeping in mind that since $P \in \GL_{d}(\KK)$, we have $P^* = P^{\tr}$.  The $t$-frame $\iota : \Mat_{1\times r}(\KK[t]) \to \Mat_{1 \times d}(\KK[\sigma])$ for $\phi$ induces a $t$-frame $\iota_{\rho} : \Mat_{1 \times r}(\KK[t]) \to \Mat_{1 \times d}(\KK[\sigma])$ given by
\[
  \iota_{\rho}(\bsalpha) = \iota(\bsalpha) P^{-1}.
\]
Indeed by~\eqref{E:changecoords},
\[
  \iota_{\rho}(t \bsalpha) = \iota(\bsalpha)\cdot \phi_t^* \cdot P^{-1} = \iota_{\rho}(\bsalpha) \rho_t^*.
\]
Now since $\iota_\rho(\bse_i) = \bsn_i \cdot P^{-1}$ for $1 \leqslant i \leqslant m$, from the definition of $P$ we have
\[
  \rd \iota_{\rho}(\bse_i) = \bss_{\ell_1 + \cdots + \ell_i}.
\]
More generally, for $1 \leqslant j \leqslant \ell_i$,
\begin{align*}
  \rd \iota_{\rho} \bigl( (t-\theta)^{j-1} \cdot \bse_i \bigr)
  &= \rd \bsn_i \cdot P^{-1} (\rd \rho_t^* - \theta \Id_d)^{j-1} \\
  &= \rd \bsn_i \cdot P^{-1} (P \rd \phi_t^* P^{-1} - \theta \Id_d)^{j-1} \\
  &= \rd \bsn_i \cdot P^{-1} \sum_{k=0}^{j-1} (-1)^k \binom{j-1}{k} \theta^{k} \cdot P \cdot \rd (\phi_t^*)^{j-1-k} \cdot P^{-1} \\
  &= \rd \bsn_i \cdot (\rd \phi_t^* - \theta \Id_d)^{j-1} \cdot P^{-1} \\
  &= \bss_{\ell_1 + \cdots + \ell_i - j+1},
\end{align*}
where again the last equality follows from the definition of $P$.  Thus we find that Proposition~\ref{P:epsiota}(ii) is satisfied for $\rho$.  Moreover, as the argument above shows if the $t$-module $\phi$ is defined over a perfect field $\FF_q \subseteq L \subseteq \KK$, then the basis $\bsn_1, \dots, \bsn_r$ can be chosen to lie in $\Mat_{1 \times r}(L[\sigma])$.
\end{remark}

\begin{remark} \label{R:tmotivebasis}
In a similar fashion to Remark~\ref{R:basechangered}, if $\phi$ is also an abelian $t$-module, then it is possible to change the $\KK[t]$-basis on the $t$-motive $\cM_{\phi}$ to a desired form.  That is, after possibly changing the $\KK[\tau]$-basis, we can find a $\KK[t]$-basis $\bsm_1, \dots, \bsm_r \in \cM_{\phi}$ so that if $\jmath: \Mat_{1\times r}(\KK[t]) \to \Mat_{1\times d}(\KK[\tau])$ and $\tPhi \in \Mat_r(\KK[t])$ are defined as in Proposition~\ref{P:abeliantframeprops}, then the following conditions hold.
\begin{enumerate}
\item[(i)] There exists $\tC \in \GL_r(\KK[t])$ so that
\[
  D \assign \tC\tPhi = \begin{pmatrix} (t-\theta)^{\ell_1} & & \\ & \ddots & \\ & & (t-\theta)^{\ell_r} \end{pmatrix}
\]
is a diagonal matrix and $\ell_1, \dots, \ell_r$ are the same as in Proposition~\ref{P:epsiota}.
\item[(ii)] For $1 \leqslant i \leqslant m$ and $1 \leqslant j \leqslant \ell_i$,
\[
  \rd \jmath\bigl( (t-\theta)^{j-1}\cdot \bse_{i} \bigr) = \bss_{\ell_1 + \cdots + \ell_{i-1} +j}.
\]
\end{enumerate}
The argument is almost exactly the same as the one in Remark~\ref{R:basechangered}, and we omit it.  Likewise if $\phi$ is defined over a field $\FF_q \subseteq L \subseteq \CC_{\infty}$ ($L$ not necessarily perfect), then the basis $\bsm_1, \dots, \bsm_r$ can be chosen in~$\Mat_{1\times r}(L[\tau])$.
\end{remark}

\begin{example} \label{Ex:Drinfeldmodule1}
\emph{Drinfeld modules.}  Let $\phi : \bA \to \KK[\tau]$ be a Drinfeld module of rank~$r$, defined by $\phi_t = \theta + b_1 \tau + \dots + b_r \tau^r$, with $b_r \neq 0$.  Then $\{1, \sigma, \ldots, \sigma^{r-1}\}$ forms a $\KK[t]$-basis of $\cN_{\phi} = \KK[\sigma]$ (cf.\ \cite[\S 1.5.6]{BPrapid}, \cite[\S 3.3]{CP12}, \cite[Lem.~5.4.1]{Goss}).  We see that with respect to this basis, multiplication by $\sigma$ on $\cN_{\phi}$ is represented by the companion matrix
\[
  \Phi = \begin{pmatrix}
    0 & 1 & \cdots & 0 \\
    \vdots & \vdots & \ddots & \vdots \\
    0 & 0 & \cdots & 1 \\
    (t - \theta)/b_r^{(-r)} & -b_1^{(-1)}/b_r^{(-r)} & \cdots & -b_{r-1}^{(-r+1)}/b_r^{(-r)}
    \end{pmatrix}.
\]
Taking
\[
  C = \begin{pmatrix}
  b_1^{(-1)} & \cdots & b_{r-1}^{(-r+1)} & b_r^{(-r)} \\
  1 & \cdots & 0 & 0 \\
  \vdots & \ddots & \vdots & \vdots \\
  0 & \cdots & 1 & 0
  \end{pmatrix},
\]
we see that
\[
C\Phi = \diag_r (t-\theta, 1, \dots, 1),
\]
and so Proposition~\ref{P:epsiota} applies.  Thus for $\bsalpha = (\alpha_1, \dots, \alpha_r) \in \Mat_{1 \times r}(\TT_{\theta})$, we have simply
\begin{equation}
  \cE_0(\bsalpha) = \alpha_1(\theta).
\end{equation}
\end{example}

\begin{example} \label{Ex:Carlitztensor1}
\emph{Carlitz tensor powers.}  For $n \geqslant 1$, we let $\phi : \bA \to \Mat_{n}(\KK[\tau])$ be the $n$-th Carlitz tensor power defined by
\[
  \phi_t = \begin{pmatrix}
   \theta & 1 & \cdots & 0 \\
    & \ddots & \ddots & \vdots \\
    & & \ddots & 1 \\
    & & & \theta
  \end{pmatrix}
  + \begin{pmatrix}
  0 & \cdots & \cdots & 0 \\
  \vdots & & & \vdots \\
  \vdots & & & \vdots \\
  1 & \cdots & \cdots & 0
  \end{pmatrix}
  \tau
\]
in~\cite[\S 1.4]{AndThak90}.  Then $\phi$ is uniformizable, abelian, and $\bA$-finite, and it has rank~$1$ and dimension~$n$.  Taking $\cN_{\phi} = \Mat_{1\times n}(\KK[\sigma])$, we find that $\bss_n = (0, \dots, 0, 1)$ generates $\cN_{\phi}$ over $\KK[t]$.  It follows that
\[
  \sigma \bss_n = (t-\theta)^n \bss_n
\]
(cf.\ \cite[\S 5.2, p.~427]{CPY19}, \cite[\S 5.8]{Goss}, \cite[\S 3.6]{PLogAlg}), and so $\Phi = (t-\theta)^n \in \Mat_1(\KK[t])$.  Proposition~\ref{P:epsiota} implies that for $\alpha \in \TT_{\theta}$,
\[
  \cE_0(\alpha) = \pd_{t,n}[\alpha]\big|_{t=\theta} = \left. \begin{pmatrix} \pd_t^{n-1}(\alpha) \\ \vdots \\ \pd_t^1(\alpha) \\ \alpha \end{pmatrix} \right|_{t=\theta}.
\]
\end{example}

\begin{example} \label{Ex:strictlypure1}
\emph{Strictly pure $t$-modules.}  Suppose that $\phi : \bA \to \Mat_d(\KK[\tau])$ is a $t$-module defined by
\[
  \phi_t = \rd \phi_t + B_1 \tau + \cdots + B_{\ell} \tau^{\ell}, \quad B_i \in \Mat_d(\KK),\ \det B_{\ell} \neq 0.
\]
We call $\phi$ a \emph{strictly pure} $t$-module, and it is pure in the sense of Anderson~\cite[\S 1.9]{And86} (see also Hartl and Juschka~\cite[\S 2.5.2, p.~112]{HartlJuschka20}).  Then $\phi$ is abelian  and $\bA$-finite (see by~\cite[Prop.~1.9.2]{And86}, \cite[Ex.~2.5.16(b)]{HartlJuschka20}), and for simplicity we assume that $\rd \phi_t$ is in Jordan normal form.  As
\[
  \bigl((B_{\ell}^{-1})^{(-\ell)}\bigr)^{\tr} \phi_t^* = \bigl((B_{\ell}^{-1})^{(-\ell)}\bigr)^{\tr} \rd \phi_t^{\tr} + \bigl((B_{\ell}^{-1})^{(-\ell)}\bigr)^{\tr} \bigl( B_1^{(-1)} \bigr)^{\tr} \sigma + \cdots + \Id_d \sigma^{\ell},
\]
it follows that $\{ \sigma^j \bss_i : 1 \leqslant i \leqslant d,\ 0 \leqslant j \leqslant \ell-1 \}$ is a $\KK[t]$-basis of $\cN_{\phi}$, and in particular $\phi$ has rank $r = \ell d$.  With respect to this basis, if we let $\gamma = (B_{\ell}^{(-\ell)})^{\tr}$, then multiplication by $\sigma$ is represented by
\begin{equation} \label{E:strictlypurePhi}
  \Phi =  \begin{pmatrix}
    0 & \Id_d & \cdots & 0 \\
    \vdots & \vdots & \ddots & \vdots \\
    0 & 0 & \cdots & \Id_d \\
    \gamma^{-1}(t\Id_d - \rd\phi_t^{\tr}) & -\gamma^{-1} (B_1^{(-1)})^{\tr} & \cdots & -\gamma^{-1} (B_{\ell-1}^{(-\ell+1)})^{\tr}
    \end{pmatrix}.
\end{equation}
If we let
\[
  \tC = \begin{pmatrix}
  (B_1^{(-1)})^{\tr} & \cdots & (B_{\ell-1}^{(-\ell+1)})^{\tr} & \gamma \\
  \Id_d & \cdots & 0 & 0 \\
  \vdots & \ddots & \vdots & \vdots \\
  0 & \cdots & \Id_d & 0
  \end{pmatrix},
\]
then
\begin{equation} \label{E:CtildePhi}
  \tC \Phi = \begin{pmatrix}
  t\Id_d - \rd\phi_t^{\tr} & 0 & \cdots & 0 \\
  0 & \Id_d & & 0 \\
  \vdots & \vdots & \ddots & \vdots \\
  0 & 0 & \cdots & \Id_d
  \end{pmatrix}.
\end{equation}
Suppose now that $\rd\phi_t$ consists of $m$ Jordan blocks of sizes $\ell_1, \dots, \ell_m$, as in~\eqref{E:phitJordan}.  If we take $\tU$, $\tV \in \Mat_{d}(\KK)$ to be block diagonal matrices $\tU = \diag_d(U_1, \dots, U_m)$, $\tV = \diag_d(V_1, \dots, V_m)$, where
\[
  U_i = \begin{pmatrix}
  0 & \cdots & 0 & 1 \\
  -1 & \cdots & 0 & (t-\theta)^{\ell_i - 1} \\
  \vdots & \ddots & \vdots & \vdots \\
  0 & \cdots & -1 & t-\theta
  \end{pmatrix},
  \quad
  V_i = \begin{pmatrix}
  1 & t-\theta & \cdots & (t-\theta)^{\ell_i-1} \\
   & 1 & \cdots & (t-\theta)^{\ell_i -2} \\
   & & \ddots & \vdots \\
   & & & 1
  \end{pmatrix}
\]
are both in $\GL_{\ell_{i}}(\KK[t])$, then a straightforward calculation reveals
\[
  \tV (t\Id_d - \rd \phi_t^{\tr}) \tU^{-1} = \begin{pmatrix} \diag_{\ell_1}((t-\theta)^{\ell_1}, 1, \dots, 1) & & \\ & \ddots & \\
  & & \diag_{\ell_m}((t-\theta)^{\ell_m}, 1, \dots, 1)
  \end{pmatrix},
\]
where we have used that
\begin{equation} \label{E:Uiinv}
  U_i^{-1} = \begin{pmatrix}
    (t-\theta)^{\ell_i-1} & -1 & \cdots & 0 \\
    \vdots & \vdots & \ddots & \vdots \\
    t-\theta & 0 & \cdots & -1 \\
    1 & 0 & \cdots & 0
    \end{pmatrix}.
\end{equation}
It is possible to reorder the coordinates further so that the right-hand side is exactly in the form of Proposition~\ref{P:epsiota}(i), but it will be easier to simply keep track of the entries with non-trivial powers of $t-\theta$.  If we let $U = \diag_r(\tU, \Id_d, \dots, \Id_d)$ and $V = \diag_r(\tV, \Id_d, \dots, \Id_d)$, then combining the preceding equation with~\eqref{E:CtildePhi}, we find
\[
  V\tC \Phi U^{-1} = \diag_r \bigl( (t-\theta)^{\ell_1}, 1, \dots, 1; \ldots; (t-\theta)^{\ell_m}, 1, \dots, 1; \Id_d; \ldots; \Id_d \bigr).
\]
In order to apply Proposition~\ref{P:epsiota}, we need to change the $\KK[t]$-basis on $\cN_{\phi}$ using $U$ as the change of basis matrix as in~\eqref{E:tframebasechange} to form a new $t$-frame $(\iota',\Phi')$.  If $\bsalpha = (\alpha_1, \dots, \alpha_r) \in \Mat_{1\times r}(\TT_\theta)$ represents an element of $\TT_{\theta} \otimes_{\KK[t]} \cN_{\phi}$ with respect to this new basis, then Proposition~\ref{P:epsiota} implies that
\begin{equation}
  \cE_{0, (\iota',\Phi')} (\bsalpha) = \left. \begin{pmatrix} \pd_{t,\ell_1} [ \alpha_1 ] \\ \pd_{t,\ell_2}[ \alpha_{\ell_1+1}] \\ \vdots \\ \pd_{t,\ell_m}[\alpha_{\ell_1 + \cdots + \ell_{m-1}+1}] \end{pmatrix} \right|_{t=\theta}.
\end{equation}
On the other hand if $\bsbeta = (\beta_1, \dots, \beta_r) \in \Mat_{1 \times r}(\TT_{\theta})$ represents an element of $\TT_{\theta} \otimes_{\KK[t]} \cN_{\phi}$ with respect to the original basis, then using~\eqref{E:Uiinv} we instead find
\begin{equation} \label{E:E0beta}
  \cE_{0, (\iota,\Phi)} (\bsbeta) =  \cE_{0,(\iota',\Phi')}(\bsbeta U^{-1}) \\
  = \left. \begin{pmatrix}
  {\displaystyle \pd_{t,\ell_1} \biggl[ \sum_{k=1}^{\ell_1} (t-\theta)^{\ell_1-k}\beta_k \biggr]} \\
  \vdots \\
  {\displaystyle \pd_{t,\ell_m} \biggl[ \sum_{k=1}^{\ell_m} (t-\theta)^{\ell_m-k}\beta_{\ell_1 + \cdots + \ell_{m-1}+k} \biggr]}
  \end{pmatrix} \right|_{t=\theta}.
\end{equation}
We observe this identity explicitly in Example~\ref{Ex:strictlypure2}, especially in \eqref{E:strictlypureE0gyha}--\eqref{E:strictlypurefinally}.
\end{example}

\newpage
\section{Biderivations and quasi-periodic extensions} \label{S:Quasi}

The theory of biderivations and quasi-periodic extensions for Drinfeld modules was initiated by Anderson, Deligne, Gekeler, and Yu (see~\cite{Gekeler89a}--\cite{Gekeler90}, \cite{Goss94}, \cite{Yu90}).  Yu~\cite{Yu90} established a rich transcendence theory revolving around quasi-periods of Drinfeld modules, and Brownawell~\cite{Brownawell93}, \cite{Brownawell96}, \cite{Brownawell01}, further investigated these considerations toward various linear independence results on quasi-periods and quasi-logarithms, using Yu's Theorem of the Sub-$t$-module~\cite{Yu97}.  Chang and the second author~\cite{CP11}, \cite{CP12}, extended these results to algebraic independence using techniques from~\cite{ABP04}, \cite{P08}.  The theory of quasi-periodic extensions and quasi-periodic functions was defined for abelian $t$-modules by Brownawell and the second author~\cite{BP02} in investigating geometric $\Gamma$-values (see also Hartl and Juschka~\cite[\S 2.5.7]{HartlJuschka20} for more general situations).  For biderivations and quasi-periodic function for Drinfeld modules over Tate algebras see~\cite{GezmisP19}.

Given that a quasi-periodic extension $\psi$ is an extension in the category of $t$-modules of an abelian $t$-module $\phi$ by a power of the additive group~$\Ga$, i.e., we have
\[
0 \to \Ga^s \to \psi \to \phi \to 0,
\]
it is necessarily the case that $\psi$ is \emph{not} itself abelian.  Thus the considerations of \S\ref{S:Exp} do not immediately seem to apply.  However, work of Anderson~(see \cite{Goss94}), Pellarin~\cite{Pellarin08}, and Chang and the second author~\cite{CP11}, \cite{CP12}, demonstrates that quasi-periods and quasi-logarithms for Drinfeld modules do occur as specializations of rigid analytic trivializations.  This phenomenon also appears in~\cite{BP02} for the abelian $t$-modules defined by Sinha~\cite{Sinha97} for special $\Gamma$-values.

One of our goals is to address Question~\ref{Q:two} in \S\ref{S:Intro} and to demonstrate how quasi-periods and quasi-logarithms for general abelian $t$-modules have an abelian theory themselves and can be realized as specializations of analytic functions as in Theorems~\ref{T:Exponentiation} and~\ref{T:PeriodLattice}.  Our main results are Corollary~\ref{C:perquasiperspans} and Theorem~\ref{T:quasispancomplete}, to which we anticipate that the transcendence methods of~\cite{ABP04}, \cite{P08}, can be applied, which is the subject of recent work of the first author~\cite{NPhD}, \cite{N21}. We note first off that these results rely crucially on a dual-$t$-motive isomorphism of Hartl and Juschka~\cite[Thm.~2.5.13]{HartlJuschka20} (see Theorem~\ref{T:isodual}).

\subsection{Biderivations for \texorpdfstring{$t$}{t}-modules and quasi-periodic functions} \label{subS:biderivations}
Fix an abelian $t$-module $\phi: \bA \to \Mat_d(\KK[\tau])$ of dimension $d$ and rank $r$ defined by
\[
\phi_t = \rd \phi_t + B_1\tau+ \dots + B_\ell\tau^\ell \in \Mat_d(\KK)[\tau],
\]
where $\rd \phi_t = \theta \Id_d + \rN_\phi$ with $\Id_d$ the $d\times d$ identity matrix and $\rN_\phi$ a nilpotent matrix.  We define the exponential function $\Exp_\phi$ and the $t$-motive $\cM_\phi = \Mat_{1 \times d}(\KK[\tau])$ for $\phi$ as in~\S\ref{subS:modules}--\S\ref{subS:motives}.  We recall some definitions and properties related to biderivations of $t$-modules.  For details the reader may refer to~\cite[\S 3]{BP02}, and in the case of Drinfeld modules see \cite{Brownawell93},  \cite{Gekeler89a}, \cite{Gekeler90}, \cite{Goss94}, \cite{Yu90}.

\begin{definition}
A \emph{$\phi$-biderivation} is an $\FF_q$-linear map $\delta: \bA \to \cM_{\phi} \tau$ satisfying
\[
\bsdelta_{ab}= a(\theta) \bsdelta_b + \bsdelta_a \phi_b
\]
for all $a$, $b \in \bA$.  The biderivation $\bsdelta$ is uniquely determined by its value $\bsdelta_t$ (see \cite[\S 3.1]{BP02}).  If $\phi$ is defined over a field $L$ and $\bsdelta_t \in \Mat_{1 \times d}(L[\tau] \tau)$, then we say $\bsdelta$ is \emph{defined over $L$}.
\end{definition}

We set
\[
\rN_\phi^\perp \assign \rN_\phi^\perp(\KK) = \{ \bsalpha \in \Mat_{1\times d}(\KK) : \bsalpha \rN_\phi = 0\}.
\]
Using that $\cM_\phi = \Mat_{1\times d}(\KK[\tau])$, let $\bsu \in \cM_\phi$ satisfy $\rd \bsu \in \rN_\phi^\perp$.  Then an \emph{inner biderivation} $\bsdelta^{(\bsu)}$ is defined by
\begin{equation} \label{E:innerdef}
\bsdelta^{(\bsu)}_a = \bsu \phi_a - a(\theta)\bsu, \quad \forall\, a \in \bA.
\end{equation}
If $\bsu \in \cM_{\phi}\tau$, then the $\phi$-biderivation $\delta^{(\bsu)}$ is said to be \emph{strictly inner}.

The set of $\phi$-biderivations forms a $\KK$-vector space, which we denote by $\Der(\phi)$.  The set of inner biderivations and the set of strictly inner biderivations are also $\KK$-vector subspaces, which we denote by $\Der_{\inn}(\phi)$ and $\Der_{\si}(\phi)$ respectively.  If $\mu : \phi \to \psi$ is a morphism of $t$-modules and $\bsdelta \in \Der(\psi)$, then there is an induced $\KK$-linear map $\mu^{\dagger} : \Der(\psi) \to \Der(\phi)$ given by
\[
  (\mu^{\dagger} \bsdelta)_a = \bsdelta_a \mu, \quad \forall\, a \in \bA,
\]
and in particular the maps $\phi^{\dagger}_a : \Der(\phi) \to \Der(\phi)$ for $a \in \bA$ make $\Der(\phi)$ into a left $\KK[t]$-module.  Thus $\Der(\,\cdot\,)$ is a contravariant functor from the category of abelian $t$-modules to the category of left $\KK[t]$-modules (see~\cite[\S 3.5]{BP02}).  Furthermore, if $\bsdelta^{(\bsu)}$ is an inner biderivation for $\psi$, then $\mu^{\dagger} \bsdelta^{(\bsu)} = \bsdelta^{(\bsu\mu)}$ is an inner biderivation for $\phi$, and it follows that $\Der_{\inn}(\,\cdot\,)$ and $\Der_{\si}(\,\cdot\,)$ are subfunctors of $\Der(\,\cdot\,)$.

The \emph{de Rham module} for $\phi$ is the left $\KK[t]$-module
\[
  \rH^1_{\DR}(\phi) \assign \Der(\phi) / \Der_{\si}(\phi),
\]
which defines a contravariant functor from abelian $t$-modules to finite $\KK[t]$-modules (see Proposition~\ref{P:biderprops}).  We further set $\Der_{0}(\phi) \assign \{\bsdelta^{(\bsu)} : \bsu \in \rN_{\phi}^\perp\} = \{\bsdelta^{(\bsu)} \in \Der_{\inn}(\phi) : \bsu \in \Mat_{1\times d}(\KK) \}$, and $\rH^1_{\sr}(\phi)\assign \Der(\phi)/\Der_{\inn}(\phi)$, the \emph{strictly reduced} $\phi$-biderivations.

As we are primarily interested in the analytic theory of quasi-periodic functions and quasi-periods, for smoother exposition we have been using $\KK$ as our base field.  However, the preceding discussion carries through for any base field~$L$ with $K \subseteq L \subseteq \KK$ over which $\phi$ is defined (see~\cite[\S 3.1]{BP02}), and $L$ need not be algebraically closed.  We let $\Der(\phi,L) \subseteq \Der(\phi)$ denote the $L$-vector space of $\phi$-biderivations that are defined over $L$, and likewise we let $\Der_{*}(\phi,L)$, $\rH^1_{\DR}(\phi,L) = \Der(\phi,L)/\Der_{\si}(\phi,L)$, etc. denote corresponding spaces defined over $L$.  The following proposition due to Brownawell and the second author summarizes fundamental relationships among these spaces.

\begin{proposition}[{\cite[\S 3.1]{BP02}}] \label{P:biderprops}
Let $L$ be a field with $K \subseteq L \subseteq \KK$.  Let $\phi : \bA \to \Mat_d(L[\tau])$ be an abelian $t$-module of dimension~$d$ and rank~$r$.  Let $\cM_{\phi,L} = \Mat_{1\times d}(L[\tau])$ be the $t$-motive of $\phi$ over $L$.  Then
\begin{enumerate}
\item[(a)] For each $\bsm \in \cM_{\phi,L}\tau$ there is a unique $\bsdelta_{\bsm} \in \Der(\phi,L)$ so that $(\bsdelta_{\bsm})_t = \bsm$, and the map
\[
  \bsm \mapsto \delta_{\bsm} : \cM_{\phi,L}\tau = \Mat_{1\times d}(L[\tau]\tau) \iso \Der(\phi,L),
\]
is an isomorphism.
\item[(b)] Under the isomorphism in~\textup{(a)},
\[
  \Der_{\inn}(\phi,L) \cong (t-\theta)(\Mat_{1\times d}(L[\tau]\tau) + \rN_{\phi}^{\perp}(L)), \quad
  \Der_{\si}(\phi,L) \cong (t-\theta)\Mat_{1\times d}(L[\tau]\tau).
\]
\item[(c)] As $L$-vector spaces, we have
\[
\Der_{\inn}(\phi,L) = \Der_0(\phi,L) \oplus \Der_{\si}(\phi,L), \quad \rH^1_{\DR}(\phi,L) \cong \Der_0(\phi,L) \oplus \rH^1_{\sr}(\phi,L).
\]
\item[(d)] We have $\dim_{L} \Der_0(\phi,L) = d - \rank \rN_{\phi}$, $\dim_{L} \rH^1_{\sr}(\phi,L) = r - d + \rank \rN_{\phi}$, and $\dim_{L} \rH^1_{\DR}(\phi,L) = r$.
\end{enumerate}
\end{proposition}

\begin{remark}
If $L$ is perfect then $\tau \cM_{\phi,L} = \cM_{\phi,L}\tau$, but if not, $\tau\cM_{\phi,L} \subsetneq \cM_{\phi,L}\tau = \Mat_{1\times d}(L[\tau]\tau)$.  As only the latter object is an $L$-vector space, we need to use $\cM_{\phi,L}\tau$ in this context.  When $L$ is perfect we will use the notation $\tau \cM_{\phi,L}$ without significant confusion.
\end{remark}

To each $\phi$-biderivation we can associate a unique quasi-periodic function, which is characterized by the following proposition.

\begin{proposition}[{\cite[Prop. 3.2.1]{BP02}}] \label{P:quasi}
For $\bsdelta \in \Der(\phi)$ there is a unique $\FF_q$-linear and entire power series
\[
\rF_{\bsdelta}(\bsz) = \sum_{i \geqslant 1} \bsc_i \cdot \bsz^{(i)} \in \power{\KK}{z_1, \dots, z_d},
\]
where $\bsc_i \in \Mat_{1 \times d}(\KK)$ and $\bsz = (z_1, \dots, z_d)^{\tr}$, such that
\begin{equation} \label{E:Fdeltafneq}
\rF_{\bsdelta}(\rd \phi_a \cdot \bsz) = a(\theta) \rF_{\bsdelta}(\bsz) + \bsdelta_a \Exp_\phi(\bsz), \quad \forall\, a \in \bA.
\end{equation}
\end{proposition}

We note that \eqref{E:Fdeltafneq} implies immediately that the restriction
\begin{equation} \label{E:Fdeltarestricted}
  \rF_{\bsdelta}|_{\Lambda_{\phi}} : \Lambda_{\phi} \to \KK
\end{equation}
is $\bA$-linear, where $a \in \bA$ operates on $\KK$ by multiplication by $a(\theta)$.

\begin{definition}
Fix an abelian $t$-module $\phi : \bA \to \Mat_d(\KK[\tau])$, and let $\bsdelta_1, \dots, \bsdelta_k$ be $\phi$-biderivations.  Then the \emph{quasi-periodic extension} of $\phi$ associated to $\bsdelta_1, \dots, \bsdelta_k$ is the $t$-module $\psi: \bA \rightarrow \Mat_{d+k}(\KK[\tau])$ defined by
\begin{equation} \label{E:QPext}
\psi_t = \left(
\begin{array}{c|c}
  \theta \Id_k & \begin{matrix} (\bsdelta_{1})_t \\ \vdots \\ (\bsdelta_{k})_t \end{matrix} \\
\hline
  0 & \phi_t
\end{array}
\right).
\end{equation}
As such $\psi$ is an extension of $\phi$ by $\Ga^k$, and there is an exact sequence of $t$-modules,
\[
  0 \to \Ga^{k} \to \psi \to \phi \to 0.
\]
The exponential function for $\psi$ is the entire function $\Exp_\psi: \KK^{k+d} \to \KK^{k+d}$ with
\begin{equation} \label{E:QPExp}
\Exp_{\psi} \begin{pmatrix} \bsx \\ \bsz \end{pmatrix} =
\Bigl( x_{1} + \rF_{\bsdelta_1}(\bsz), \dots, x_{k}+ \rF_{\bsdelta_k}(\bsz), \Exp_{\phi}(\bsz)^{\tr} \Bigr)^{\tr},
\end{equation}
where $\bsx = (x_1, \dots, x_k)^{\tr} \in \KK^k$.  This formula readily implies that if $\phi$ is uniformizable, then $\psi$ is uniformizable.  The periods of $\psi$ are of the form $(-\rF_{\bsdelta_1}(\bslambda), \dots, -\rF_{\bsdelta_k}(\bslambda), \smash{\bslambda}^{\tr} )^{\tr}$, where $\bslambda \in \Lambda_{\phi}$ is a period of $\phi$.  For $\bsdelta \in \Der(\phi)$, we set
\[
  \eta_{\bsdelta}(\bslambda) \assign \rF_{\delta}(\bslambda) \in \KK
\]
to be the \emph{quasi-period} with respect to $\bsdelta$ associated to the period $\bslambda$ of $\phi$, and thus
\begin{equation}
  \Lambda_{\psi} = \bigl\{ (-\eta_{\bsdelta_1}(\bslambda), \dots, -\eta_{\bsdelta_k}(\bslambda), \smash{\bslambda}^{\tr} )^{\tr} : \bslambda \in \Lambda_{\phi} \bigr\}.
\end{equation}
If $\bsdelta \in \Der_{\inn}(\phi)$, say $\bsdelta = \bsdelta^{(\bsu)}$ for $\bsu \in \Mat_{1\times d}(\KK[\tau])$ with $\rd\bsu \in \rN_\phi^{\perp}$, then by~\cite[Prop.~3.2.2]{BP02}
\begin{equation}\label{E:quasiinner}
  \rF_{\bsdelta^{(\bsu)}}(\bsz) = \bsu \Exp_{\phi}(\bsz) - \rd\bsu\cdot \bsz,
\end{equation}
and so for $\bslambda \in \Lambda_{\phi}$,
\begin{equation} \label{E:innerquasiper}
  \eta_{\bsdelta^{(\bsu)}}(\bslambda) = -\rd \bsu \cdot \bslambda.
\end{equation}
In particular, if $\bsdelta^{(\bsu)} \in \Der_{\si}(\phi)$, then for all $\bslambda \in \Lambda_\phi$, we have $\eta_{\bsdelta^{(\bsu)}}(\bslambda) = 0$.  A quick calculation reveals that
\begin{equation} \label{E:quasilogdef}
  \Log_{\psi} \begin{pmatrix} \bsx \\ \bsz \end{pmatrix} =
  \Bigl( x_1 - \rF_{\bsdelta_1}\bigl(\Log_{\phi}(\bsz) \bigr), \dots, x_k - \rF_{\bsdelta_k}\bigl(\Log_{\phi}(\bsz) \bigr), \Log_{\phi}(\bsz)^{\tr} \Bigr)^{\tr}.
\end{equation}
Thus for a biderivation $\bsdelta$, if $\bsy \in \KK^d$ satisfies $\Exp_{\phi}(\bsy) = \bsz$, then we say that $\rF_{\bsdelta}(\bsy)$ is a \emph{quasi-logarithm} of $\bsz$ (associated to $\bsy$).
\end{definition}

If $\phi$ is defined over a field $L$ with $K \subseteq L \subseteq \KK$, then by Remark~\ref{R:ExpoverL} we have $\Exp_{\phi}(\bsz) \in \power{L}{z_1,\dots, z_d}^{d}$, and so the functional equation in~\eqref{E:Fdeltafneq} sets up recursions on coefficients that imply
\[
  \rF_{\bsdelta}(\bsz) \in \power{L}{z_1, \dots, z_d}.
\]
Moreover the $L$-linear map
\begin{equation} \label{E:deltatoFdelta}
  \bsdelta \mapsto \rF_{\bsdelta}(\bsz) : \Der(\phi,L) \to \power{L}{z_1, \dots, z_d}
\end{equation}
is injective~\cite[p.~113]{BP02}.

\begin{proposition} \label{P:quasiperspans}
Let $L$ be a field with $K \subseteq L \subseteq \KK$.  Let $\phi : \bA \to \Mat_d(L[\tau])$ be a uniformizable abelian $t$-module of dimension~$d$ and rank~$r$, and let $s = d-\rank \rN_{\phi}$.  Let $\bsdelta_1, \dots, \bsdelta_r \in \Der(\phi,L)$ represent an $L$-basis of $\rH^1_{\DR}(\phi,L)$ such that $\bsdelta_1, \dots, \bsdelta_{s}$ represent an $L$-basis of the image of $\Der_0(\phi,L)$ in $\rH^1_{\DR}(\phi,L)$ and $\bsdelta_{s+1}, \dots, \bsdelta_r$ represent an $L$-basis of $\rH^1_{\sr}(\phi,L)$.
\begin{enumerate}
\item[(a)] For $\bslambda \in \Lambda_{\phi}$, we have
\begin{align}
\label{E:spanLai}
\Span{L} \bigl( \rF_{\bsdelta_1}(\bslambda), \dots, \rF_{\bsdelta_{r}}(\bslambda) \bigr)
  &= \Span{L} \bigl( \rF_{\bsdelta}(\bslambda) : \bsdelta \in \Der(\phi,L) \bigr) \\
  &= \Span{L} \bigl( \rF_{\bsdelta}(\bslambda) : [\bsdelta] \in \rH^1_{\DR}(\phi,L) \bigr); \notag \\
\label{E:spanLaii}
\Span{L} \bigl( \rF_{\bsdelta_1}(\bslambda), \dots, \rF_{\bsdelta_{s}}(\bslambda) \bigr)
  &= \Span{L} \bigl( \rF_{\bsdelta}(\bslambda) : \bsdelta \in \Der_0(\phi,L) \bigr) \\
  &= \rN_{\phi}^{\perp}(L) \cdot \bslambda. \notag
\end{align}
\item[(b)] For an $\bA$-basis $\{ \bslambda_1, \dots, \bslambda_r \}$ of $\Lambda_\phi$, we have
\begin{align}
\label{E:spanLbi}
\Span{L} \bigl( \rF_{\bsdelta_i}(\bslambda_j) : 1 \leqslant i,\, j \leqslant r \bigr)
  &= \Span{L} \bigl( \rF_{\bsdelta}(\bslambda) : \bsdelta \in \Der(\phi,L),\, \bslambda \in \Lambda_{\phi} \bigr) \\
  &= \Span{L} \bigl( \rF_{\bsdelta}(\bslambda) : [\bsdelta] \in \rH^1_{\DR}(\phi,L),\, \bslambda \in \Lambda_{\phi} \bigr); \notag \\
\label{E:spanLbii}
\Span{L} \bigl( \rF_{\bsdelta_i}(\bslambda_j) : 1 \leqslant i\, \leqslant s,\, 1 &\leqslant j \leqslant r \bigr) \\
  &= \Span{L} \bigl( \rF_{\bsdelta}(\bslambda) : \bsdelta \in \Der_0(\phi,L),\, \bslambda \in \Lambda_{\phi} \bigr) \notag \\
  &= \rN_{\phi}^{\perp}(L) \cdot \Lambda_{\phi}. \notag
\end{align}
\end{enumerate}
\end{proposition}

\begin{remark} \label{R:tractable}
In the case that $\rd\phi_t$ is in Jordan normal form, then $\rN_{\phi}^{\perp}(L) \cdot \Lambda_{\phi}$ comprises the $L$-linear span of all of the \emph{tractable coordinates} of periods in $\Lambda_{\phi}$ in the sense of~\cite[Def.~3.3.1]{ChangMishiba21} or~\cite[\S 1]{PR03} (or \emph{last coordinate logarithms} in~\cite[\S 2]{Yu91}).
\end{remark}

\begin{proof}[Proof of Proposition~\ref{P:quasiperspans}]
For~\eqref{E:spanLai}, since $\rF_{\bsdelta}(\bslambda) = 0$ for any $\bsdelta \in \Der_{\si}(\phi)$ by \eqref{E:innerquasiper}, it follows that
\[
\Span{L} (\rF_{\bsdelta}(\bslambda) : \bsdelta \in \Der(\phi,L)) = \Span{L} (\rF_{\bsdelta}(\bslambda) : [\bsdelta] \in \rH^1_{\DR}(\phi,L)).
\]
That these equal $\Span{L} (\rF_{\bsdelta_1}(\bslambda), \dots, \rF_{\bsdelta_{r}}(\bslambda))$ follows from the $L$-linearity of~\eqref{E:deltatoFdelta}.  Similarly, \eqref{E:spanLaii} follows from our choices of $\bsdelta_1, \dots, \bsdelta_{s}$ and the definition of $\Der_0(\phi,L)$.  Finally, since $\rF_{\bsdelta}|_{\Lambda_{\phi}} : \Lambda_{\phi} \to \KK$ is $\bA$-linear by~\eqref{E:Fdeltarestricted}, the identities in~(b) follow from~(a).
\end{proof}

If we take $L= K_{\infty}^{\sep}$, then the following lemma ensures that the periods and quasi-periods of $\phi$ also have coordinates in $K_{\infty}^{\sep}$.  In the case of Drinfeld modules, it was proved by Denis~\cite[p.~6]{Denis95} (see also~\cite[Pf.~of Cor.~2]{Brownawell99}), and our proof is similar.

\begin{lemma} \label{L:separable}
Suppose that $\phi : \bA \to \Mat_d(K_{\infty}^{\sep}[\tau])$ is an abelian $t$-module defined over $K_{\infty}^{\sep}$ and that $\bsdelta$ is a $\phi$-biderivation also defined over $K_{\infty}^{\sep}$.  For any $\bsx \in \KK^{d}$ such that $\Exp_{\phi}(\bsx) \in (K_{\infty}^{\sep})^d$, we have
\[
  \bsx \in (K_{\infty}^{\sep})^d, \quad \rF_{\bsdelta}(\bsx) \in K_{\infty}^{\sep}.
\]
\end{lemma}

\begin{proof}
Let $\bsu = \Exp_{\phi}(\bsx)$.  Let $L \subseteq K_{\infty}^{\sep}$ be a finite extension of $K_{\infty}$ that contains all of the entries of coefficients of $\phi_t$ and $\bsdelta_t$ as well as the entries of $\bsu$.  Then $L$ is a complete local field with respect to $\inorm{\,\cdot\,}$.  Also, as formal power series $\Exp_{\phi}(\bsz) \in \power{L}{z_1,\dots, z_d}^d$, and also its inverse $\Log_{\phi}(\bsz) \in \power{L}{z_1, \dots, z_d}^d$.  Now as in Remark~\ref{R:dphitinverse}, we see that $\inorm{(\rd \phi_t)^{-n} \bsx} \to 0$ as $n \to \infty$, and so for $n \gg 0$, $(\rd\phi_t)^{-n}\bsx$ lies within the domain of convergence of $\Log_{\phi}(\bsz)$.  Also, as in \cite[Lem.~2.5.4]{HartlJuschka20}, $\Exp_\phi(\bsz)$ is an isometric embedding on a ball of sufficiently small radius, so we can assume that
\[
  \bigl\lvert \Exp_{\phi} \bigl( (\rd\phi_t)^{-n}\bsx \bigr) \bigr\rvert_{\infty} = \bigl\lvert (\rd\phi_t)^{-n}\bsx \bigr\rvert_{\infty}.
\]
Because $L$ is complete, it follows that $\Exp_{\phi}((\rd\phi_t)^{-n}\bsx) \in L^d$, and likewise
\[
  (\rd \phi_t)^{-n} \bsx = \Log_{\phi} \bigl( \Exp_{\phi} \bigl( (\rd\phi_t)^{-n}\bsx \bigr) \bigr) \in L^d.
\]
Thus, $\bsx \in L^d$, and since $\rF_{\bsdelta}(\bsz) \in \power{L}{z_1, \dots, z_d}$, the completeness of $L$ implies that $\rF_{\bsdelta}(\bsx) \in L$.
\end{proof}

\begin{definition} \label{D:strictlyquasi}
If $\bsdelta_1, \dots, \bsdelta_k$ represent $\KK$-linearly independent classes in $\rH^1_{\sr}(\phi)$, we say that the extension $\psi$ is a \emph{strictly quasi-periodic extension} of $\phi$.  The following result will be useful for exchanging biderivations for ones in more convenient forms, while maintaining the same vector spaces spanned by their quasi-periods and quasi-logarithms.
\end{definition}

\begin{proposition}[{\cite[Prop.~3.3.3]{BP02}}]\label{P:strictqpextbasis}
Let $\bsdelta_1, \dots, \bsdelta_k$ and $\bsdelta_1', \dots, \bsdelta_k'$ be $\phi$-biderivations representing classes in $\rH^1_{\sr}(\phi)$.  Then the corresponding quasi-periodic extensions $\psi$ and $\psi'$ are isomorphic as $t$-module extensions of $\phi$ by $\Ga^k$ if and only if the biderivations generate the same $\KK$-subspace of $\rH^1_{\sr}(\phi)$.
\end{proposition}

\subsection{Anderson generating functions} \label{subS:AGF}
Originally defined by Anderson~\cite[\S 3.2]{And86} in the context of uniformizability of $t$-modules and rigid analytic trivializations, Anderson generating functions have proved to be a useful tool for investigating periods, quasi-periods, $L$-series, motivic Galois groups, and more for Drinfeld modules and $t$-modules (e.g., see~\cite{AndThak90}, \cite{AnglesPellarin15}, \cite{CP11}, \cite{CP12}, \cite{EP14}, \cite{GezmisP19}, \cite{Goss94}, \cite{GreenNgoDac20a}, \cite{GreenP18}, \cite{HartlJuschka20}, \cite{Maurischat22}, \cite{Pellarin12}, \cite{PellarinPerkins16}, \cite{Sinha97}). In this section we investigate properties of Anderson generating functions for arbitrary $t$-modules, some of which is inspired by computations of Green~\cite[\S 6]{Green22}, Pellarin~\cite[\S 4.2]{Pellarin08}, and the second author~\cite[\S 3.3]{PLogAlg}.

Fix an Anderson $t$-module $\phi : \bA \to \Mat_d(\KK[\tau])$, and let $\bsy \in \KK^d$.  The \emph{Anderson generating function} for $\phi$ with respect to $\bsy$ is the vector of power series,
\begin{equation} \label{E:AGFdef}
  \cG_{\bsy} \assign \sum_{n=0}^{\infty} \Exp_{\phi} \bigl( (\rd \phi_t)^{-n-1} \bsy \bigr) t^n \in \TT^d.
\end{equation}
A priori we have $\cG_{\bsy} \in \power{\KK}{t}^d$, but because $\inorm{(\rd \phi_t)^{-n-1}} \to 0$ as $n \to \infty$ by Remark \ref{R:dphitinverse}, we see that its components are in $\TT$.  This first lemma shows that $\cG_{\bsy}$ can be expressed as the sum of convenient rational functions.

\begin{lemma} \label{L:AGFsum}
Let $\phi : \bA \to \Mat_d(\KK[\tau])$ be an Anderson $t$-module, and let $\bsy \in \KK^d$.  Then recalling $\Exp_{\phi}(\bsz) = \sum_{i \geqslant 0} C_i \bsz^{(i)}$ in~\eqref{E:Expdef}, we have the identity
\begin{equation} \label{E:AGFsum}
  \cG_{\bsy} = \sum_{i=0}^{\infty} C_i \Bigl( \bigl( \rd \phi_t - t\Id_d\bigr)^{-1} \Bigr)^{(i)} \bsy^{(i)},
\end{equation}
which converges in $\TT^d$.
\end{lemma}

\begin{proof}
Since $\Exp_{\phi}(\bsz)$ is entire, it follows that $\inorm{C_i}^{1/q^i} \to 0$ as $i \to \infty$, and thus the right-hand side of~\eqref{E:AGFsum} has terms that go to~$0$ as $i \to \infty$ with respect to $\dnorm{\,\cdot\,}$ on $\TT^d$, and therefore the right-hand side converges in~$\TT^d$.  We then observe that because $\inorm{(\rd \phi_t)^{-n}} \to 0$ as $n \to \infty$ (Remark~\ref{R:dphitinverse}),
\begin{equation} \label{E:geomseries}
  \bigl( \rd \phi_t - t\Id_d\bigr)^{-1} = (\rd \phi_t)^{-1} \bigl(\Id_d - (\rd \phi_t)^{-1} t\bigr)^{-1}
  = (\rd \phi_t)^{-1} \sum_{n=0}^{\infty} (\rd \phi_t)^{-n} t^n.
\end{equation}
By substitution,
\begin{align*}
  \sum_{i=0}^{\infty} C_i \Bigl( \bigl( \rd \phi_t - t\Id_d\bigr)^{-1} \Bigr)^{(i)} \bsy^{(i)} &= \sum_{i=0}^{\infty} C_i \biggl( (\rd \phi_t)^{-1} \sum_{n=0}^{\infty} (\rd \phi_t)^{-n} t^n \biggr)^{(i)} \bsy^{(i)} \\
  &= \sum_{n=0}^{\infty} \biggl( \sum_{i=0}^{\infty} C_i \Bigl( (\rd \phi_t)^{-n-1} \bsy \Bigr)^{(i)} \biggr)t^n,
\end{align*}
which completes the proof once we justify interchanging the sum. This justification is due to Green~\cite[Pf.~of Prop.~6.3]{Green22}, whose proof transfers with only minor changes to our setting. In particular his analysis shows that
\[
\Bigl\lvert C_i \Bigl( (\rd \phi_t)^{-n-1} \bsy \Bigr)^{(i)} \Bigr| \to 0, \quad \textup{as $\max(i,n) \to \infty$},
\]
which permits interchanging the double sum by~\cite[\S 2.1.2]{Robert}.
\end{proof}

\begin{remark} \label{R:AGFconvergence}
Because $\det(\rd\phi_t - t \Id_d) = c(\theta-t)^d$, $c \in \KK^{\times}$, it follows that the terms of
\[
  \bigl(\Omega^{(-1)} \bigr)^d \cdot \cG_{\bsy} = \sum_{i=0}^{\infty} \bigl(\Omega^{(-1)} \bigr)^d \cdot C_i \Bigl( \bigl( \rd \phi_t - t\Id_d\bigr)^{-1} \Bigr)^{(i)} \bsy^{(i)}
\]
are regular on all of $\KK$.  Furthermore, since $\inorm{C_i}^{1/q^i} \to 0$, it follows that $\bigl(\Omega^{(-1)} \bigr)^d \cdot \cG_{\bsy} \in \EE^d$, and thus the entries of $\cG_{\bsy}$ are meromorphic functions in $\MM$ with possible poles only at $t = \theta$, $\theta^q$, $\theta^{q^2}, \ldots$.  Moreover, we find for each $i \geqslant 1$ that $\cG_{\bsy}^{(i)} \in \TT_{\theta^{q^{i-1}}}^d$.
\end{remark}

For meromorphic functions $\bsf = (f_1, \dots, f_d)^{\tr} \in \MM^d$ and $x \in \KK$, we define
\begin{equation} \label{E:Resdef}
\Res_{t=x} (\bsf) \assign \bigl( \Res_{t=x} (f_1), \dots, \Res_{t=x} (f_d) \bigr)^{\tr}.
\end{equation}
The following proposition presents one of the crucial properties of Anderson generating functions.

\begin{proposition}\label{P:AGFres}
Let $\phi : \bA \to \Mat_d(\KK[\tau])$ be an Anderson $t$-module, and let $\bsy \in \KK^d$.  Then
\[
  \Res_{t=\theta} (\cG_{\bsy}) = -\bsy.
\]
\end{proposition}

\begin{proof}
From \eqref{E:AGFsum}, we see that
\begin{equation} \label{E:GyRes}
  \Res_{t=\theta}(\cG_{\bsy}) = \Res_{t=\theta} \bigl( (\rd\phi_t - t\Id_d)^{-1}\bsy \bigr).
\end{equation}
Before proceeding to the general case, we begin by assuming that $\rd\phi_t$ is in Jordan normal form as in~\eqref{E:phitJordan}.  In that case,
\[
  \rd\phi_t - t\Id_d = \begin{pmatrix} d_{\theta,\ell_1}[\theta-t] & & \\ & \ddots & \\ & & d_{\theta,\ell_m}[\theta-t] \end{pmatrix},
\]
and so
\begin{equation} \label{E:dphitminustInv}
  (\rd\phi_t - t\Id_d)^{-1} = \begin{pmatrix} d_{\theta,\ell_1} \bigl[ (\theta-t)^{-1} \bigr] & & \\ & \ddots & \\ & & d_{\theta,\ell_m}\bigl[ (\theta-t)^{-1} \bigr] \end{pmatrix}.
\end{equation}
Now for $j \geqslant 0$, we have $\pd_{\theta}^{j} ( (\theta-t)^{-1}) = (-1)^j (\theta-t)^{-j-1}$, and so the off-diagonal entries of~\eqref{E:dphitminustInv} are either~$0$ or $\pm 1/(\theta-t)^{-k}$ for $k \geqslant 2$.
Thus writing $\bsy = (y_1, \dots, y_d)^{\tr}$, we have
\[
(\rd\phi_t - t\Id_d)^{-1}\bsy = \biggl( \frac{y_1}{\theta-t} + R_1(t), \dots, \frac{y_d}{\theta-t} + R_d(t) \biggr)^{\tr},
\]
where for each $j$,
\[
  R_j(t) \in \KK\cdot (t-\theta)^{-2} + \dots + \KK\cdot (t-\theta)^{-d}.
\]
(See also Remark~\ref{R:AGFprincipalpart} below.) Taking residues at $t=\theta$ we see that $\Res_{t=\theta} (\cG_{\bsy}) = -\bsy$ from~\eqref{E:GyRes}.

Now if $\rd \phi_t$ is not in Jordan normal form, we can choose a matrix $\rQ \in \GL_{d}(\KK)$ so that the resulting $t$-module $\rho : \bA \to \Mat_d(\KK[\tau])$ defined by $\rho_t = \rQ^{-1} \phi_t \rQ$ does have $\rd \rho_t$ in Jordan normal form.  Then
\begin{align*}
  \Res_{t=\theta} \bigl( (\rd\phi_t - t\Id_d)^{-1}\bsy \bigr) &= \Res_{t=\theta} \bigl( (\rQ (\rd\rho_t - t\Id_d) \rQ^{-1})^{-1}\bsy \bigr) \\
  &= \rQ \cdot \Res_{t=\theta} \bigl( (\rd \rho_t - t\Id_d)^{-1} \cdot \rQ^{-1}\bsy\bigr) \\
  &= \rQ \cdot (-\rQ^{-1} \bsy),
\end{align*}
where the last equality is the Jordan normal form case, and by~\eqref{E:GyRes} we are done.
\end{proof}

\begin{remark} \label{R:AGFprincipalpart}
When $\rd\phi_t$ is in Jordan normal form as in~\eqref{E:phitJordan}, we expand some of the details of the above proof for later use.  For any $\ell \geqslant 1$, a straightforward calculation yields
\[
  d_{\theta,\ell} \biggl[ \frac{1}{\theta-t} \biggr] = - \begin{pmatrix}
  \dfrac{1}{t-\theta} & \dfrac{1}{(t-\theta)^2} & \cdots & \dfrac{1}{(t-\theta)^{\ell}} \\
   & \dfrac{1}{t-\theta} & \ddots & \vdots \\
    & & \ddots & \dfrac{1}{(t-\theta)^2} \\
    & & & \dfrac{1}{t-\theta}
   \end{pmatrix}.
\]
For $1 \leqslant j \leqslant m$, we let $d_j \assign \ell_1 + \cdots + \ell_j$.  Thus if we write $\bsy = (y_1, \dots, y_d)^{\tr}$ and $\cG_{y} = (g_1, \dots, g_d)^{\tr}$, then by \eqref{E:GyRes} we have identities in $\laurent{\KK}{t-\theta}$,
\begin{align} \label{E:giprinparts}
  g_{d_{j-1}+1} &= -\frac{y_{d_{j}}}{(t-\theta)^{\ell_j}} - \frac{y_{d_{j}-1}}{(t-\theta)^{\ell_j-1}} - \cdots - \frac{y_{d_{j-1}+1}}{t-\theta} +  O(1), \\
  g_{d_{j-1}+2} &= -\frac{y_{d_j}}{(t-\theta)^{\ell_j-1}} - \frac{y_{d_j-1}}{(t-\theta)^{\ell_j-2}} - \cdots - \frac{y_{d_{j-1}+2}}{t-\theta} +  O(1), \notag \\
  &\phantom{=} \qquad \vdots \notag \\
  g_{d_j} &= -\frac{y_{d_j}}{t-\theta} +  O(1), \notag
\end{align}
where `$O(1)$' represents $O((t-\theta)^0)$.
\end{remark}

Another important property of Anderson generating functions is their compatibility with the $t$-module structure of $\phi$.  Recall the notation $\twistop{\bsbeta}{\bsf}$ from~\eqref{E:twistop}.

\begin{proposition} \label{P:AGFprops}
Let $\phi : \bA \to \Mat_d(\KK[\tau])$ be an Anderson $t$-module, and let $\bsy \in \KK^d$.  The following hold.
\begin{enumerate}
\item[(a)] For every $a \in \bA$, $\twistop{\phi_a}{\cG_{\bsy}} = \cG_{\rd\phi_a\cdot \bsy}$.
\item[(b)] We have $\twistop{\phi_t}{\cG_{\bsy}} = t \cG_{\bsy} + \Exp_{\phi}(\bsy)$.
\item[(c)] If $\bslambda \in \Lambda_{\phi}$, then for every $a \in \bA$, $\twistop{\phi_a}{\cG_{\bslambda}} = a(t) \cG_{\bslambda}$.
\end{enumerate}
\end{proposition}

\begin{proof}
Each part follows from standard arguments (e.g., see~\cite[\S 4.2]{Pellarin08}).  For part~(a),
\[
  \twistop{\phi_a}{\cG_{\bsy}} = \sum_{n=0}^{\infty} \phi_a \bigl( \Exp_{\phi} \bigl( (\rd \phi_t)^{-n-1} \bsy \bigr) \bigr) t^n = \sum_{n=0}^{\infty} \Exp_{\phi} \bigl( (\rd \phi_t)^{-n-1} \rd \phi_a \bsy \bigr) t^n.
\]
Taking this same equation with $a=t$, after reindexing the sum we find
\[
  \twistop{\phi_t}{\cG_{\bsy}} = \sum_{n=0}^{\infty} \Exp_{\phi} \bigl( (\rd \phi_t)^{-n}\bsy \bigr) t^n
  = \Exp_{\phi}(\bsy) + \sum_{n=0}^{\infty} \Exp_{\phi} \bigl( (\rd \phi_t)^{-n-1}\bsy \bigr) t^{n+1},
\]
and~(b) follows.  Part~(c) follows from successive applications of~(b).
\end{proof}

For $\bsm \in \Mat_{1\times d}(\KK[\tau]) = \cM_{\phi}$ and $f \in \KK[t]$, it would seem possible that $\twistop{f \cdot \bsm}{\cG_{\bsy}} \mayeq f \cdot \twistop{\bsm}{\cG_{\bsy}}$, but this is not usually the case.  In the left expression we are multiplying $f$ against an element of the $t$-motive $\cM_{\phi}$, whereas on the right we are multiplying $f$ against a vector in~$\TT$, and these actions may not produce the same results.  However, the above proposition allows us to determine the discrepancy.

\begin{corollary} \label{C:AGFtwistedmult}
Let $\phi : \bA \to \Mat_d(\KK[\tau])$ be an Anderson $t$-module, and let $\bsy \in \KK^d$.  Let $\bsm \in \Mat_{1\times d}(\KK[\tau])$, and suppose $f = \sum_{i=0}^s c_i t^i \in \KK[t]$.
\begin{enumerate}
\item[(a)] Then
\[
  \twistop{f \cdot \bsm}{\cG_{\bsy}} = f\cdot \twistop{\bsm}{\cG_{\bsy}} + \sum_{i=1}^{s} \sum_{j=0}^{i-1} c_i t^{i-1-j} \cdot \bsm \phi_{t^j} \bigl( \Exp_{\phi}(\bsy) \bigr).
\]
\item[(b)] In particular, if $\bslambda \in \Lambda_{\phi}$, then $\twistop{f\cdot \bsm}{\cG_{\bslambda}} = f \cdot \twistop{\bsm}{\cG_{\bslambda}}$.
\end{enumerate}
\end{corollary}

\begin{proof}
By induction on Proposition~\ref{P:AGFprops}(b), it follows that for $i \geqslant 0$,
\begin{equation}
  \twistop{\phi_{t^i}}{\cG_{\bsy}} = t^i \cG_{\bsy} + \sum_{j=0}^{i-1} t^{i-1-j} \phi_{t^j}\bigl( \Exp_{\phi}(\bsy) \bigr).
\end{equation}
Thus,
\begin{align*}
  \twistop{f \cdot \bsm}{\cG_{\bsy}} &= \biggl\langle \sum_{i=0}^s c_i \bsm \phi_{t^i} \biggm| \cG_{\bsy} \biggr\rangle \\
  &= \sum_{i=0}^s c_i \biggl\langle \bsm \biggm| t^i \cG_{\bsy} + \sum_{j=0}^{i-1} t^{i-1-j} \phi_{t^j}(\Exp_{\phi}(\bsy)) \biggr\rangle \\
  &= \sum_{i=0}^s c_i t^i \cdot \twistop{\bsm}{\cG_{\bsy}} + \sum_{i=1}^s \sum_{j=0}^{i-1} c_i t^{i-1-j} \bsm \phi_{t^j} \bigl( \Exp_{\phi}(\bsy) \bigr),
\end{align*}
and since the first term is $f \cdot \twistop{\bsm}{\cG_{\bsy}}$, we are done.
\end{proof}

\subsection{The de Rham pairing} \label{subS:deRhampairing}
Using Anderson generating functions as defined in \S\ref{subS:AGF}, we form a pairing
\[
  ( \cdot \, , \cdot ) : \tau \cM_{\phi} \times \Lambda_{\phi} \to \TT_{\theta}
\]
defined by
\[
  ( \bsalpha \tau,\bslambda ) = \twistop{\bsalpha \tau}{\cG_{\bslambda}} = \bigl\langle \bsalpha \bigm| \cG_{\bslambda}^{(1)} \bigr\rangle.
\]
By Remark~\ref{R:AGFconvergence}, we see that $(\Omega^{(-1)})^d \cdot \cG_{\bslambda} \in \EE^d$, and hence $\Omega^d\cdot \cG_{\bslambda}^{(1)} \in \EE^d$. Since $\Omega \in \TT_{\theta}^{\times}$ by~\eqref{E:Omegadef}, it follows that $\cG_{\bslambda}^{(1)}$ has entries in $\TT_{\theta}$, and thus this pairing is well-defined.
We verify immediately that $( \cdot\,, \cdot )$ is biadditive, and for any $a \in \bA$, Corollary~\ref{C:AGFtwistedmult} implies that
\[
  ( a \cdot \bsalpha \tau, \bslambda) = \twistop{a\cdot \bsalpha \tau}{\cG_{\bslambda}} = a(t) \cdot \twistop{\bsalpha \tau}{\cG_{\bslambda}}.
\]
Furthermore by direct calculation,
\[
( a \cdot \bsalpha \tau,\bslambda ) = ( \bsalpha \tau \phi_a,\bslambda ) = \twistop{\bsalpha \tau}{\twistop{\phi_a}{\cG_{\bslambda}}}, \quad
( \bsalpha \tau, \rd \phi_a \bslambda ) = \bigl\langle \bsalpha \tau \bigm| \cG_{\rd \phi_a\cdot \bslambda} \bigr\rangle,
\]
and so Proposition~\ref{P:AGFprops} then implies that
\[
  (a \cdot \bsalpha \tau,\bslambda ) = ( \bsalpha \tau, \rd \phi_a \bslambda ) = a(t) ( \bsalpha \tau, \bslambda),
\]
making $( \cdot\,, \cdot )$ $\bsA$-bilinear.  We check further that it is left $\KK[t]$-linear in the left entry.

Now using the identification of $\Der(\phi) \cong \tau\cM_{\phi}$ from Proposition~\ref{P:biderprops}(a), we have an induced $\bA$-bilinear pairing
\begin{equation} \label{E:predeRhampairing}
 (\cdot\, ,\cdot ) : \Der(\phi) \times \Lambda_{\phi} \to \TT_{\theta}.
\end{equation}
given by
\begin{equation}
(\bsdelta,\bslambda) = \twistop{\bsdelta_t}{\cG_{\bslambda}}.
\end{equation}
Moreover, the identification $\Der_{\si}(\phi) \cong (t-\theta)\tau\cM_{\phi}$ implies that if $\bsalpha \tau = (t-\theta)\bsbeta \tau$, then
\[
( \bsalpha \tau, \bslambda ) = ( (t-\theta)\bsbeta \tau, \bslambda)
= (t-\theta) ( \bsbeta\tau,\bslambda)
= (t-\theta) \twistop{\bsbeta \tau}{\cG_{\bslambda}}.
\]
Therefore, we have an induced $\bA$-bilinear pairing
\begin{equation}
  [ \cdot\, , \cdot ] : \rH^1_{\DR}(\phi) \times \Lambda_{\phi} \to \KK,
\end{equation}
defined by
\begin{equation}
  [\bsdelta, \bslambda ] = (\bsdelta,\bslambda)|_{t=\theta} = \twistop{\bsdelta_t}{\cG_{\bslambda}}|_{t=\theta},
\end{equation}
which we call the \emph{de Rham pairing}.  The connection with quasi-periods and quasi-logarithms is the following, inspired by calculations of Anderson~\cite[\S 2.6]{Goss94}, Gekeler~\cite[Eq.~(5.3)]{Gekeler89b}, and Pellarin~\cite[\S 4.2]{Pellarin08}.

\begin{proposition} \label{P:quasigen}
For an abelian $t$-module $\phi : \bA \to \Mat_d(\KK[\tau])$ and $\phi$-biderivation $\bsdelta \in \Der(\phi)$, the following hold.
\begin{enumerate}
\item[(a)] For any $\bsy \in \KK^d$,
\[
  \twistop{\bsdelta_t}{\cG_{\bsy}}|_{t=\theta} = \rF_{\bsdelta}(\bsy).
\]
\item[(b)] In particular, for $\bslambda \in \Lambda_{\phi}$ we have
\[
[\bsdelta, \bslambda ] = \twistop{\bsdelta_t}{\cG_{\bslambda}}|_{t=\theta} = \rF_{\bsdelta}(\bslambda) = \eta_{\bsdelta}(\bslambda).
\]
\end{enumerate}
\end{proposition}

\begin{proof}
Part~(b) follows immediately from (a) and the definition of $[ \cdot \, , \cdot]$.  For~(a), consider the series
\[
  \cF_{\bsdelta}(\bsz) = \biggl\langle \bsdelta_t \biggm| \sum_{n=0}^{\infty} \Exp_{\phi} \bigl( (\rd \phi_t)^{-n-1} \bsz \bigr) t^n \biggr\rangle \in \power{\TT_{\theta}}{z_1, \dots, z_d}^d.
\]
Then
\begin{align*}
  \cF_{\bsdelta}( \rd \phi_t \cdot \bsz ) - t \cF_{\bsdelta}(\bsz)
  &= \sum_{n=0}^{\infty} \bsdelta_t \bigl( \Exp_{\phi} \bigl( (\rd \phi_t)^{-n} \bsz \bigr) \bigr) t^n
  -  \sum_{n=1}^{\infty} \bsdelta_t \bigl( \Exp_{\phi} \bigl( (\rd \phi_t)^{-n} \bsz \bigr) \bigr) t^n \\
  &= \bsdelta_t \bigl( \Exp_{\phi}(\bsz)).
\end{align*}
This last expression is an element of $\power{\KK}{z_1, \dots, z_d}^d$, and so specializing at $t=\theta$, we have
\begin{equation} \label{E:cFdeltafneq}
  \cF_{\bsdelta}( \rd \phi_t \cdot \bsz )|_{t=\theta} - \theta \cF_{\bsdelta}(\bsz)|_{t=\theta} = \bsdelta_t \Exp_{\phi}(\bsz).
\end{equation}
Now since $\bsdelta_t \in \tau\cM_{\phi}$, it follows that $\cF_{\bsdelta}(\bsz)|_{t=\theta} \in \power{\KK}{z_1, \dots, z_d}$ is $\FF_q$-linear and has no linear terms.  By~\eqref{E:cFdeltafneq}, we see that $\cF_{\bsdelta}(\bsz)|_{t=\theta}$ satisfies~\eqref{E:Fdeltafneq} for $a = t$.  Since this identity determines $\rF_{\bsdelta}(\bsz)$ uniquely~\cite[Pf.~of Prop.~3.2.1]{BP02}, we have $\cF_{\bsdelta}(\bsz)|_{t=\theta} = \rF_{\bsdelta}(\bsz)$.  We are done after substituting $\bsz = \bsy$, since $\cF_{\bsdelta}(\bsz)|_{\bsz = \bsy} = \twistop{\bsdelta_t}{\cG_{\bsy}}$.
\end{proof}

Now assume further that $\phi$ is uniformizable.  Let $\{ \bsm_1, \dots, \bsm_r \}$ be a $\KK[t]$-basis of $\cM_{\phi}$, and let $\{ \bslambda_1, \dots, \bslambda_r \}$ be an $\bA$-basis of $\Lambda_{\phi}$.  With respect to these bases, the pairing $(\cdot\, ,\cdot ): \tau \cM_{\phi} \times \Lambda_{\phi} \to \TT_{\theta}$ is represented by the matrix in $\Mat_r(\TT_\theta)$,
\begin{equation} \label{E:Upsilondef}
  \Upsilon \assign \left\langle \begin{pmatrix}
  \tau \bsm_1 \\ \vdots \\ \tau \bsm_r
  \end{pmatrix} \Biggm|
  \bigl( \cG_{\bslambda_1}, \ldots, \cG_{\bslambda_r} \bigr) \right\rangle
  =
  \begin{pmatrix}
  \twistop{\tau \bsm_1}{\cG_{\bslambda_1}} & \cdots & \twistop{\tau \bsm_1}{\cG_{\bslambda_r}} \\
  \vdots & & \vdots \\
  \twistop{\tau \bsm_r}{\cG_{\bslambda_1}} & \cdots & \twistop{\tau \bsm_r}{\cG_{\bslambda_r}}
  \end{pmatrix}.
\end{equation}
From Proposition~\ref{P:abeliantframeprops}, we can pick $\tPhi \in \Mat_d(\KK[t])$ so that $\tau \bsm = \tPhi \bsm$, where $\bsm = (\bsm_1, \dots, \bsm_r)^{\tr}$, and we obtain
\begin{multline} \label{E:Upsilonfneq1}
  \Upsilon^{(-1)} = \twistop{\bsm}{(\cG_{\bslambda_1}, \dots, \cG_{\bslambda_r})} \\
  = \bigl\langle \tPhi^{-1} \cdot \tau\bsm \bigm| (\cG_{\bslambda_1}, \dots, \cG_{\bslambda_r}) \bigr\rangle
  = \tPhi^{-1} \bigl\langle \tau \bsm \bigm| (\cG_{\bslambda_1}, \dots, \cG_{\bslambda_r}) \bigr\rangle = \tPhi^{-1} \Upsilon,
\end{multline}
where in the third equality we have applied Corollary~\ref{C:AGFtwistedmult}.

\begin{lemma} \label{L:detUpsilon}
With notation as above, $\det \Upsilon \neq 0$.
\end{lemma}

\begin{proof}
Each column $\psi$ of $\Upsilon$ is a solution of the $\tau$-difference equation $\psi^{(1)} = \tPhi^{(1)} \psi$, and so by standard difference algebra arguments, if the columns of $\Upsilon$ are linearly independent over $\FF_q(t)$, then they are linearly independent over the fraction field~$\LL$ of~$\TT$ (cf.~\cite[Lem.~3.3.7, 4.1.4]{P08}).  Take $\bse_1, \dots, \bse_r$ to be the standard basis vectors in $\TT^{r}$ and suppose we have $a_1, \dots, a_r \in \bA$ that produce a vanishing linear combination of columns, $a_1 \Upsilon \bse_1 + \cdots + a_r \Upsilon \bse_r = 0$.  Letting $\bslambda = \rd \phi_{a_1} \bslambda_1 + \cdots + \rd \phi_{a_r} \bslambda_r$, we then have
\[
  \twistop{\tau \bsm_i}{\cG_{\bslambda}} = 0, \quad \forall\, i,\ 1 \leqslant i \leqslant r.
\]
Since $\tau \bsm_1, \dots, \tau \bsm_r$ form a $\KK[t]$-basis of $\tau\cM_{\phi}$ and since $\twistop{\,\cdot}{\cG_{\bslambda}}$ is $\KK[t]$-linear by Corollary~\ref{C:AGFtwistedmult}, it follows that for any $\bsalpha \in \tau \cM_{\phi}$, we have $\twistop{\bsalpha}{\cG_{\bslambda}} = 0$.  If the $i$-th entry of $\cG_{\bslambda}$ is non-zero, then so is the $i$-th entry of $\twistop{(0, \dots, 0, \tau, 0, \dots, 0)}{\cG_{\bslambda}}$, where the $\tau$ is placed in the $i$-th entry of the row vector on the left.  Therefore, it must be that $\cG_{\bslambda} = 0$ identically.  By Proposition~\ref{P:AGFres}, we must have $\bslambda = 0$, which implies that $a_1 = \cdots = a_r = 0$.
\end{proof}

Building on this lemma, we find that even more is true and that $\Upsilon \in \GL_r(\TT_\theta)$ (cf.\ \cite[Prop.~6.2.4]{GezmisP19}, \cite[Prop.~3.3.9]{P08}).

\begin{proposition} \label{P:Upsiloninvertible}
Let $\phi : \bA \to \Mat_d(\KK[\tau])$ be a uniformizable abelian $t$-module of rank~$r$, and let $\Upsilon \in \Mat_r(\TT_\theta)$ be defined as in~\eqref{E:Upsilondef}.  Then $\Upsilon \in \GL_r(\TT_\theta)$.
\end{proposition}

\begin{proof}
By~\eqref{E:Upsilonfneq1}, we see that $\Upsilon = \tPhi \Upsilon^{(-1)}$, and if we set $\Theta = \Upsilon^{\tr}$, then $\Theta = \Theta^{(-1)}\tPhi^{\tr}$.  Since $\Theta \in \Mat_r(\TT_{\theta})$, $\det \Theta \neq 0$ by Lemma~\ref{L:detUpsilon}, and $\det (\tPhi^{\tr}) = c(t-\theta)^d$ for some $c \in \KK^{\times}$, the argument in the proof of \cite[Prop.~3.3.9(c), pp.~140--141]{P08} exactly applies to prove that
\[
  b \det \Theta = \gamma \cdot \Omega^{-d}, \quad b^{(1-q)/q} = c,\ \gamma \in \FF_q^{\times}.
\]
Since $\Omega \in \TT_{\theta}^{\times}$, we see that $\det \Upsilon = \det \Theta \in \TT_{\theta}^{\times}$.
\end{proof}

Let $\bsdelta_1, \dots, \bsdelta_r \in \Der(\phi)$ be $\phi$-biderivations chosen so that $(\bsdelta_i)_t = \tau \bsm_i$ for each $i$.  Then by Proposition~\ref{P:quasigen}(b),
\begin{equation} \label{E:Upsilonspecial}
  \Upsilon|_{t=\theta} = \begin{pmatrix} \rF_{\bsdelta_1}(\bslambda_1) & \cdots & \rF_{\bsdelta_1}(\bslambda_r) \\
  \vdots & & \vdots\\
   \rF_{\bsdelta_r}(\bslambda_1) & \cdots & \rF_{\bsdelta_r}(\bslambda_r) \end{pmatrix},
\end{equation}
and $\Upsilon|_{t=\theta}$ represents the de Rham pairing.  (Since $\rH^1_{\DR}(\phi) \cong \tau \cM_{\phi} / (t-\theta) \tau \cM_{\phi}$, it follows that $\bsdelta_1, \dots, \bsdelta_r$ represent a $\KK$-basis of $\rH^1_{\DR}(\phi)$.)  Moreover, Proposition~\ref{P:quasiperspans} implies immediately the following.

\begin{proposition} \label{P:Upsilonkbarspan}
Let $\phi : \bA \to \Mat_d(\oK[\tau])$ be a uniformizable abelian $t$-module of rank $r$ defined over $\oK$, and let $\bsm_1, \dots, \bsm_r \in \Mat_{1\times d}(\oK[\tau])$ be a $\oK[t]$-basis of $\cM_{\phi,\oK}$.  Letting $\Upsilon \in \Mat_r(\TT_\theta)$ be defined as in~\eqref{E:Upsilondef}, we have
\[
  \Span{\oK} \bigl( \Upsilon|_{t=\theta} \bigr) = \Span{\oK} \bigl( \rF_{\bsdelta}(\bslambda) : \bsdelta \in \Der(\phi,\oK),\, \bslambda \in \Lambda_{\phi} \bigr).
\]
\end{proposition}

By Remark~\ref{R:tractable}, after a change of basis the entries of $\Upsilon|_{t=\theta}$ consist of tractable periods of $\phi$ and strictly reduced quasi-periods of $\phi$.   By applying Yu's Theorem of the Sub-$t$-module~\cite[Thm.~3.3]{Yu97}, it is straightforward to adapt the proofs of \cite[Thm.~5.2.1, Cor.~5.2.2]{BP02} to prove the following $\oK$-linear independence result in the case that $\phi$ is \emph{simple}, i.e., $\phi$ has no proper non-trivial sub-$t$-modules.

\begin{corollary} \label{C:linindep}
Let $\phi : \bA \to \Mat_d(\oK[\tau])$ be a uniformizable abelian $t$-module of rank~$r$ defined over $\oK$, and assume further that $\phi$ is simple.  Let $\bsm_1, \dots, \bsm_r \in \Mat_{1\times d}(\oK[\tau])$ be a $\oK[t]$-basis of $\cM_{\phi,\oK}$, and let $\Upsilon \in \Mat_r(\TT_\theta)$ be defined as in~\eqref{E:Upsilondef}.  Then
\[
  \dim_{\oK} \bigl( \Span{\oK} \bigl( \Upsilon|_{t=\theta} \bigr) \bigr) = \frac{r^2}{s},
\]
where $s = [\End(\phi):\bA]$.
\end{corollary}

We end this section on the de Rham isomorphism for uniformizable abelian $t$-modules, which was proved by Gekeler~\cite[Thm.~5.14]{Gekeler89a} in the case of Drinfeld modules, using the theory of biderivations and quasi-periodic functions.  Anderson gave a different proof using rigid analytic trivializations and Anderson generating functions (see~\cite[\S 1.5]{Goss94}).  Our investigations provide a proof in the case of uniformizable abelian $t$-modules using a combination of these methods (cf.~\cite[Thm.~5.3.3]{GezmisP19}).  In the related context of Hodge-Pink structures for $t$-modules, see also the intrinsically same result of Hartl and Juschka~\cite[Thm.~2.5.51]{HartlJuschka20}.

\begin{corollary}[de Rham isomorphism] \label{C:deRhamiso}
Let $\phi : \bA \to \Mat_d(\KK[\tau])$ be a uniformizable abelian $t$-module.  The map
\[
  \DR : \rH^1_{\DR}(\phi) \to \Hom_{\bA}(\Lambda_{\phi},\KK),
\]
defined by $\DR([\bsdelta]) = \rF_{\bsdelta}|_{\Lambda_{\phi}}$, is an isomorphism.
\end{corollary}

\begin{proof}
We recall from~\eqref{E:Fdeltarestricted} that for $\bsdelta \in \Der(\phi)$ the map $\rF_{\bsdelta}|_{\Lambda_{\phi}} : \Lambda_{\phi} \to \KK$ is $\bA$-linear.  Furthermore, for $\bslambda \in \Lambda_{\phi}$, Proposition~\ref{P:quasigen}(b) implies
\[
  \DR([\bsdelta]) = \rF_{\bsdelta}(\bslambda) = [ \bsdelta, \bslambda ].
\]
Thus the matrix $\Upsilon|_{t=\theta}$ in \eqref{E:Upsilonspecial} represents $\DR$ with respect to our chosen bases.  By Proposition~\ref{P:Upsiloninvertible}, we see that $\det(\Upsilon|_{t=\theta}) \neq 0$.
\end{proof}

\subsection{Rigid analytic trivializations, quasi-periods, and quasi-logarithms} \label{subS:RAT}
As in previous sections we fix an abelian $t$-module $\phi : \bA \to \Mat_d(\KK[\tau])$.  We saw in Theorems~\ref{T:Exponentiation} and~\ref{T:PeriodLattice} that if $\phi$ is also $\bA$-finite, then $\Exp_{\phi}(\bsz)$ and $\Lambda_{\phi}$ can be effectively described and that $\Lambda_{\phi}$ can be determined precisely if $\phi$ has a rigid analytic trivialization.  In this section we demonstrate how to use the construction of the matrix $\Upsilon \in \GL_r(\TT_{\theta})$ in \eqref{E:Upsilondef} together with a theorem of Hartl and Juschka to create a rigid analytic trivialization $(\iota, \Phi, \Psi)$ for~$\phi$.  We then modify these constructions to develop tools for characterizing the $\oK$-linear spans of tractable period coordinates, quasi-periods, and quasi-logarithms for~$\phi$.

We have seen in \S\ref{subS:motives}, that we can associate to $\phi$ both its $t$-motive $\cM_\phi$ and dual $t$-motive $\cN_{\phi}$.  Additionally there is a third object $\cM_{\phi}^{\wedge}$, defined by Hartl and Juschka~\cite[Prop.~2.4.3]{HartlJuschka20} (see below), that is also a dual $t$-motive.  Hartl and Juschka~\cite[Thm.~2.5.13]{HartlJuschka20} proved that if $\phi$ is both abelian and $\bA$-finite, then through an intricately defined isomorphism,
\[
  \cM_{\phi}^{\wedge} \cong \cN_{\phi}
\]
as dual $t$-motives (see Theorem~\ref{T:isodual}). This isomorphism provides the key to expressing quasi-periods and quasi-logarithms in terms of rigid analytic trivializations.

\begin{proposition}[{Hartl-Juschka~\cite[Prop.~2.4.3, Thm.~2.5.13]{HartlJuschka20}}] \label{P:dualmotive}
Let $\phi : \bA \to \Mat_{d}(\KK[\tau])$ be an abelian $t$-module with $t$-motive $\cM_{\phi}$.  If we take
$\cM_{\phi}^\wedge \assign \Hom_{\KK[t]}(\tau \cM_{\phi}, \KK[t]\,dt)$, then we make $\cM_{\phi}^{\wedge}$ into a left $\KK[t,\sigma]$-module by setting
\begin{equation} \label{E:sigmawedge}
(\sigma \cdot \bsmu)(\tau\bsm) \assign \bigl( \bsmu(\tau \cdot \tau\bsm) \bigr)^{(-1)}, \quad \bsmu \in \cM_{\phi}^{\wedge},\ \tau\bsm \in \tau\cM_{\phi}.
\end{equation}
Moreover, under this definition $\cM_{\phi}^{\wedge}$ is an $\bA$-finite dual $t$-motive of dimension~$d$ and rank~$r$.
\end{proposition}

\begin{remark}
For clearer exposition we have mentioned that $\cM_{\phi}^{\wedge}$ is a dual $t$-motive here, but the most efficient way of explaining its finite generation as an $\KK[\sigma]$-module is to use \cite[Thm.~2.5.13]{HartlJuschka20}, stated below as Theorem~\ref{T:isodual}, which shows that it injects into a finitely generated $\KK[\sigma]$-module.
\end{remark}

If $\bsm = (\bsm_1, \dots, \bsm_r)^{\tr}$ forms a $\KK[t]$-basis of $\cM$, then $\tau \bsm = (\tau \bsm_1, \dots, \tau \bsm_r)^{\tr}$ is a $\KK[t]$-basis of $\tau \cM_{\phi}$, and we can take $\bsmu = (\bsmu_1, \dots, \bsmu_r)^{\tr}$ with $\bsmu_i \in \cM_{\phi}^{\wedge}$ to be the corresponding dual basis.  If $\tPhi \in \Mat_{d}(\KK[t])$ represents left multiplication by $\tau$ on $\cM_{\phi}$ with respect to~$\bsm$, then
\begin{equation}
  \tau \cdot \tau \bsm = \tau \tPhi \bsm = \tPhi^{(1)} \tau\bsm,
\end{equation}
and so $\tPhi^{(1)}$ represents left multiplication by $\tau$ on $\tau \cM_{\phi}$.  From this we determine the matrix representing left multiplication by $\sigma$ on $\cM_{\phi}^{\wedge}$.  For $\bsc$, $\bsu \in \Mat_{1\times r}(\KK[t])$, we can take $\bsc \cdot \bsmu$ and $\bsu \cdot \tau \bsm$ for arbitrary elements of $\cM_{\phi}^{\wedge}$ and $\tau \cM_{\phi}$ respectively.  Then easily
\[
  (\bsc \cdot \bsmu)(\bsu \cdot \tau \bsm) = \bsu \cdot \bsc^{\tr}\, dt,
\]
and by definition of multiplication by~$\sigma$ on $\cM_{\phi}^{\wedge}$,
\begin{multline*}
  \bigl( \sigma (\bsc \cdot \bsmu) \bigr)(\bsu \cdot \tau \bsm) = (\bsc \cdot \bsmu) \bigl( \tau (\bsu \cdot \tau \bsm) \bigr)^{(-1)}
  = (\bsc \cdot \bsmu) \bigl( \bsu^{(1)} \cdot \tPhi^{(1)} \tau \bsm \bigr)^{(-1)} \\
  = \bigl( \bsu^{(1)} \tPhi^{(1)} \bsc^{\tr} \bigr)^{(-1)} \,dt
  = \bsu \tPhi (\bsc^{(-1)})^{\tr}\, dt
  = \bsu \bigl( \bsc^{(-1)} \tPhi^{\tr} \bigr)^{\tr}\,dt.
\end{multline*}
Thus $\sigma(\bsc \cdot \bsmu) = \bsc^{(-1)} \tPhi^{\tr} \bsmu$, and so
\begin{equation} \label{E:sigmamultwedge}
  \sigma \bsmu = \tPhi^{\tr} \bsmu.
\end{equation}
That is, $\tPhi^{\tr}$ represents left multiplication by $\sigma$ on $\cM_{\phi}^{\wedge}$ with respect to~$\bsmu$.

A natural question is whether $\cM_{\phi}^{\wedge}$ is related or even isomorphic to the dual $t$-motive $\cN_{\phi}$.  This question is answered by a theorem of Hartl and Juschka that constructs an isomorphism between $\cM_{\phi}^{\wedge}$ and $\cN_{\phi}$ in the case that $\phi$ is $\bA$-finite.  Before stating their theorem, we define a residue map on $\cM_{\phi}^{\wedge}$.  Letting $z = 1/t$, there is an evident injection of $\KK[t]$-modules,
\begin{equation} \label{E:kappadef}
  \kappa : \cM_{\phi}^{\wedge} \hookrightarrow \Hom_{\laurent{\KK}{z}} \bigl( \laurent{\KK}{z} \otimes_{\KK[t]} \tau\cM_{\phi}, \laurent{\KK}{z}\,dz \bigr).
\end{equation}
For $\bsmu \in \cM_{\phi}^{\wedge}$ and $\tau\bsm \in \laurent{\KK}{z} \otimes_{\KK[t]} \cM_{\phi}\tau$, if $\kappa(\bsmu)(\tau \bsm) = \sum c_n z^n\,dz$, we set
\begin{equation} \label{E:Resinftydef}
  \Res_{\infty} \bigl( \kappa(\bsmu)(\tau \bsm) \bigr) = c_{-1}.
\end{equation}
Also, as $\det \tPhi^{(1)} \in \laurent{\KK}{z}^{\times}$, we note that the induced map
\begin{equation} \label{E:tauauto}
  \tau : \laurent{\KK}{z} \otimes_{\KK[t]} \tau \cM_{\phi} \iso \laurent{\KK}{z} \otimes_{\KK[t]} \tau \cM_{\phi}
\end{equation}
is an automorphism.  Let $\bss_1, \dots, \bss_d \in \Mat_{1\times d}(\KK[\tau])$ be the standard $\KK[\tau]$-basis of~$\cM_{\phi}$.

\begin{theorem}[{Hartl-Juschka~\cite[Thm.~2.5.13]{HartlJuschka20}}] \label{T:isodual}
Let $\phi : \bA \to \Mat_{d}(\KK[\tau])$ be an abelian $t$-module with $t$-motive $\cM_{\phi}$ and dual $t$-motive $\cN_{\phi}$.  Define
\[
\xi : \cM_{\phi}^{\wedge} \to \cN_{\phi}
\]
by
\begin{equation} \label{E:xidef}
  \xi(\bsnu) \assign -\sum_{i \geqslant 0} \Bigl( \Res_{\infty} \bigl( \kappa(\bsnu) ( \tau^{-i-1} \cdot \tau \bss_{j}) \bigr) \Bigr)_{1 \leqslant j \leqslant d} \cdot \sigma^i.
\end{equation}
Then the following hold.
\begin{enumerate}
\item[(a)] $\xi$ is an injective morphism of left $\KK[t,\sigma]$-modules.
\item[(b)] $\xi$ is an isomorphism of left $\KK[t,\sigma]$-modules if and only if $\phi$ is $\bA$-finite.
\end{enumerate}
\end{theorem}

\begin{remark}
The map $\tau^{-1}$ used in \eqref{E:xidef} is the inverse of the isomorphism $\tau$ given in \eqref{E:kappadef}. It is not immediately clear that the formula in~\eqref{E:xidef} is well-defined, in that it is necessary to show that $\Res_{\infty} (\kappa(\bsnu)(\tau^{-i-1} \cdot \tau \bss_j)) = 0$ for $i \gg 0$.  Hartl and Juschka prove this by decomposing elements of $\cM_{\phi}$ in terms of its associated $z$-isocrystals, which behave like Dieudonn\'{e} modules in this context (see also~\cite[\S 5]{Taelman09}).  Except for the surjectivity in~(b), the remaining assertions of the theorem follow in a straightforward manner.  The proof of surjectivity follows from a fairly elaborate diagram of isomorphisms (see~\cite[Cor.~2.5.14]{HartlJuschka20}).  We note that we have incorporated an extra negative sign when compared to Hartl's and Juschka's formula, only for convenience so that the formulas match up with those of \cite{CP11}, \cite{CP12}.  We consider examples in \S\ref{subS:quasiperexamples}.
\end{remark}

For the remainder of the section, we assume that $\phi : \bA \to \Mat_d(\KK[\tau])$ is uniformizable, abelian, and  $\bA$-finite.  We let $\bsm = (\bsm_1, \dots, \bsm_r)^{\tr}$, $\bsn = (\bsn_1, \dots, \bsn_r)^{\tr}$, and $\bsmu = (\bsmu_1, \dots, \bsmu_r)^{\tr}$ denote $\KK[t]$-bases of $\cM_{\phi}$, $\cN_{\phi}$, and $\cM_{\phi}^{\wedge}$ respectively.  Then
\[
  \tau \bsm = \tPhi \bsm, \quad \sigma \bsn = \Phi \bsn, \quad \sigma \bsmu = \tPhi^{\tr} \bsmu.
\]
We let $\{\bslambda_1, \dots, \bslambda_r\} \subseteq \Lambda_{\phi}$ be a fixed choice of $\bA$-basis of $\Lambda_{\phi}$.  For the isomorphism $\xi : \cM_{\phi}^{\wedge} \to \cN_{\phi}$ defined in Theorem~\ref{T:isodual}, we let $V \in \GL_r(\KK[t])$ be given so that
\begin{equation} \label{E:Vdef}
  \xi(\bsmu) = \bigl( \xi(\bsmu_1), \dots, \xi(\bsmu_r) \bigr)^{\tr}
  = V \bsn.
\end{equation}
It follows that
\begin{equation} \label{E:Visomorphism}
  \tPhi^{\tr} V = V^{(-1)} \Phi.
\end{equation}
Letting $\Upsilon \in \GL_r(\TT_{\theta})$ be defined as in~\eqref{E:Upsilondef} and Proposition~\ref{P:Upsiloninvertible}, we recall that $\Upsilon^{(-1)} = \tPhi^{-1} \Upsilon$ from~\eqref{E:Upsilonfneq1}.  We then see that
\begin{multline}
  \Bigl( \bigl( \Upsilon^{\tr} V \bigr)^{-1} \Bigr)^{(-1)}
  = \Bigl( \bigl( \Upsilon^{\tr} \bigr)^{(-1)} V^{(-1)} \Bigr)^{-1}
  = \Bigl( \Upsilon^{\tr} \bigl( \tPhi^{\tr} \bigr)^{-1} V^{(-1)} \Bigr)^{-1} \\
  = \bigl( \Upsilon^{\tr} V \Phi^{-1} \bigr)^{-1}
  = \Phi \bigl(\Upsilon^{\tr} V \bigr)^{-1},
\end{multline}
and thus we have proved the following proposition.

\begin{proposition} \label{P:PsiUpsilonV}
Let $\phi : \bA \to \Mat_d(\KK[\tau])$ be an abelian, $\bA$-finite, and uniformizable $t$-module.  With notation as above, if we let
\begin{equation} \label{E:PsiUpsilonV}
\Psi \assign \bigl( \Upsilon^{\tr} V \bigr)^{-1} \in \GL_r(\TT_{\theta}),
\end{equation}
then $\Psi^{(-1)} = \Phi \Psi$, making $\Psi$ a rigid analytic trivialization for $\Phi$.
\end{proposition}

As a companion to Proposition~\ref{P:Upsilonkbarspan}, we can combine Theorem~\ref{T:PeriodLattice} and Proposition~\ref{P:PsiUpsilonV} to obtain a more developed picture relating periods and quasi-periods to rigid analytic trivializations of the dual $t$-motive as follows.

\begin{corollary} \label{C:perquasiperspans}
Let $\phi: \bA \to \Mat_d(\oK[\tau])$ be a uniformizable, abelian, and $\bA$-finite $t$-module of rank~$r$ defined over~$\oK$, and suppose we have chosen $\oK[t]$-bases $\bsm$ and $\bsn$ for $\cM_{\phi,\oK}$ and $\cN_{\phi,\oK}$ respectively.  Then for the rigid analytic trivialization $\Psi \in \GL_r(\TT_{\theta})$ for $\phi$ in Proposition~\ref{P:PsiUpsilonV}, the following hold.
\begin{enumerate}
\item[(a)] If $\bsg_1, \dots, \bsg_r \in \Mat_{1\times r}(\TT_{\theta})$ are the rows of $\Psi^{-1} = \Upsilon^{\tr}V$, then letting $\bslambda_j \assign \cE_{0}(\bsg_j)$ for $1\leqslant j \leqslant r$,
\[
  \Lambda_{\phi} = \bA \cdot \bslambda_1 + \dots + \bA\cdot \bslambda_r.
\]
\item[(b)] If $\bsdelta_1, \dots, \bsdelta_r \in \Der(\phi,\oK)$ represent a $\oK$-basis of $\rH_{\DR}^1(\phi,\oK)$, then
\[
  \Span{\oK} \bigl( \rF_{\bsdelta_i}(\bslambda_j) : 1 \leqslant i, j \leqslant r \bigr)
  = \Span{\oK} \bigl( (\Upsilon^{\tr})|_{t=\theta} \bigr) = \Span{\oK} \bigl( (\Psi^{-1})|_{t=\theta}\bigr).
\]
\end{enumerate}
\end{corollary}

\begin{proof}
As mentioned in the preceding paragraph, much of the proof follows directly from Theorem~\ref{T:PeriodLattice} and Proposition~\ref{P:PsiUpsilonV}.  The first equality in~(b) follows from \eqref{E:Upsilonspecial}, while the second follows from the facts that $\Psi^{-1} = \Upsilon^{\tr}V$ and $V \in \GL_r(\oK[t])$.
\end{proof}

\begin{remark}
Part~(a) of this corollary makes perfect sense if $\phi$ is simply defined over~$\KK$, but part~(b) becomes trivial if the field of definition $\oK$ is replaced by $\KK$.
\end{remark}

This next construction is the extension of the one in \cite[\S 4.2]{CP12} to uniformizable, abelian, and $\bA$-finite, $t$-modules.

\begin{lemma} \label{L:gyha}
Let $\phi : \bA \to \Mat_d(\KK[\tau])$ be a uniformizable, abelian, and $\bA$-finite $t$-module. Let $\bsalpha$, $\bsy \in \KK^d$ be chosen so that $\Exp_{\phi}(\bsy) = \bsalpha$, and write $\tPhi = \sum_{i=0}^h \tU_i t^i$ with $\tU_i \in \Mat_r(\KK)$.  If we let
\begin{alignat}{2} \label{E:galphadef}
\bsg_{\bsy} &\assign -\twistop{\tau \bsm}{\cG_{\bsy}}^{\tr}\cdot V &&\in \Mat_{1 \times r}(\TT_{\theta}), \\
\label{E:halphadef}
\bsh_{\bsalpha} &\assign \sum_{i=1}^h \sum_{j=0}^{i-1} t^{i-1-j} \bigl\langle \tU_i \bsm \bigm| \phi_{t^{j}}(\bsalpha) \bigr\rangle^{\tr} \cdot V && \in \Mat_{1 \times r}(\KK[t]),
\end{alignat}
then
\[
  \bsg_{\bsy}^{(-1)} \Phi - \bsg_{\bsy} = \bsh_{\bsalpha}.
\]
\end{lemma}

\begin{proof}
We note first that
\begin{multline*}
\bsg_{\bsy} = -\twistop{\tPhi \bsm}{\cG_{\bsy}}^{\tr}\cdot V
= -\biggl\langle \sum_{i=0}^h \tU_i t^i \cdot \bsm \biggm| \cG_{\bsy} \biggr\rangle^{\tr}\cdot V  \\
= -\sum_{i=0}^h \Bigl\langle \tU_i \bsm \phi_{t^i} \Bigm| \cG_{\bsy} \Bigr\rangle^{\tr}\cdot V
= -\sum_{i=0}^h \Bigl\langle \tU_i \bsm \Bigm| \twistop{\phi_{t^i}}{\cG_{\bsy}}\Bigr\rangle^{\tr}\cdot V,
\end{multline*}
where in the last equality we have applied~\eqref{E:twistopassoc}.  We then apply Corollary~\ref{C:AGFtwistedmult}(a):
\begin{align} \label{E:galphacalc}
\bsg_{\bsy} &= -\sum_{i=0}^h \biggl\langle \tU_i \bsm \biggm| t^i \cG_{\bsy} + \sum_{j=0}^{i-1} t^{i-1-j} \phi_{t^j}(\bsalpha) \biggr\rangle^{\tr}\cdot V \\
  &=-\sum_{i=0}^h \bigl\langle \tU_i \bsm \bigm| t^i \cG_{\bsy} \bigr\rangle^{\tr}\cdot V
  -\sum_{i=1}^h \sum_{j=0}^{i-1} t^{i-1-j} \bigl\langle \tU_i \bsm \bigm| \phi_{t^j}(\bsalpha) \bigr\rangle^{\tr} \cdot V. \notag
\end{align}
On the other hand,
\begin{equation} \label{E:galphatwistPhi}
  \bsg_{\bsy}^{(-1)} \Phi = -\twistop{\bsm}{\cG_{\bsy}}^{\tr}\cdot V^{(-1)} \Phi
  = -\twistop{\bsm}{\cG_{\bsy}}^{\tr} \tPhi^{\tr}V,
\end{equation}
where the first equality follows from the definition of $\twistop{\cdot}{\cdot}$ and the second from~\eqref{E:Visomorphism}.  Continuing this we see that
\[
  \bsg_{\bsy}^{(-1)} \Phi = -\bigl( \tPhi \twistop{\bsm}{\cG_{\bsy}} \bigr)^{\tr}\cdot V
  = -\sum_{i=0}^h \bigl( t^i\tU_i \twistop{\bsm}{\cG_{\bsy}} \bigr)^{\tr} \cdot V
  = -\sum_{i=0}^h \twistop{\tU_i \bsm}{t^i\cG_{\bsy}}^{\tr} \cdot V,
\]
and combining this with \eqref{E:halphadef} and~\eqref{E:galphacalc} we are done.
\end{proof}

The key purpose of Lemma~\ref{L:gyha} is to manufacture vectors of functions to which Theorem~\ref{T:Exponentiation} can apply, as we see in the following calculations (cf.~\cite[\S 4.3]{CP11}, \cite[\S 4--5]{CP12}, \cite[\S 4.1]{N21}).

\begin{definition} \label{def:xipairing}
When $\phi$ is both abelian and $\bA$-finite, Hartl and Juschka~\cite[Question 2.5.15]{HartlJuschka20} make the following construction.  The definition of $\cM_{\phi}^{\wedge} = \Hom_{\KK[t]}(\tau\cM_{\phi},\KK[t]\,dt)$ in Proposition~\ref{P:dualmotive} and the isomorphism $\xi: \cM_{\phi}^{\wedge} \to \cN_{\phi}$ in Theorem~\ref{T:isodual} provide a perfect $\KK[t]$-bilinear pairing
\[
  \{ \cdot \,, \cdot \}: \tau \cM_{\phi} \times \cN_{\phi} \to \KK[t]
\]
defined by $\{ \tau \bsa,\bsb\} = \xi^{-1}(\bsb)(\tau\bsa)/dt$.  We note that $\{ \tau\cdot \tau\bsa,\bsb\} = \{\tau\bsa,\sigma\bsb\}^{(1)}$, which induces compatibility between the $\KK[\tau]$ and $\KK[\sigma]$ actions.  For $\KK[t]$-bases $\bsm$ for $\cM_{\phi}$ and $\bsn$ for $\cN_{\phi}$ and for $\bsu$, $\bsw \in \Mat_{1 \times r}(\KK[t])$, we find by~\eqref{E:Vdef} that
\begin{equation} \label{E:xipairing}
  \{ \bsu \tau\bsm, \bsw\bsn \} = \bsw V^{-1} \bsu^{\tr}.
\end{equation}
We can extend this pairing in the evident fashion to
\[
  \{ \cdot\,, \cdot \}: \tau \cM_{\phi} \times (\TT_{\theta}\otimes_{\KK[t]} \cN_{\phi}) \to \TT_{\theta},
\]
and we have the following proposition  and corollary which connects this pairing to quasi-logarithms and quasi-period respectively.
\end{definition}

\begin{remark}
Hartl and Juschka pose the question of how to make this pairing explicit for general abelian and $\bA$-finite $t$-modules, which amounts to explicitly identifying~$V$, and they demonstrate how to do this for Drinfeld modules and strictly pure $t$-modules~\cite[Ex.~2.5.16]{HartlJuschka20} (for Drinfeld modules see also~\cite[\S 3.4]{CP12}).  In the next section we explore this question in detail for almost strictly pure $t$-modules.  In the proceeding proposition we connect it to quasi-logarithms.
\end{remark}

\begin{proposition} \label{P:quasiloggy}
Let $\phi : \bA \to \Mat_d(\KK[\tau])$ be a uniformizable, abelian and $\bA$-finite $t$-module of rank~$r$, and let $\bsm$ and $\bsn$ be $\KK[t]$-bases for $\cM_{\phi}$ and $\cN_{\phi}$ respectively.  Let $\bsalpha$, $\bsy\in \KK^d$ be chosen so that $\Exp_{\phi}(\bsy) = \bsalpha$, and let $\bsg_{\bsy} \in \Mat_{1\times r}(\TT_{\theta})$ be given as in Lemma~\ref{L:gyha}.
\begin{enumerate}
\item[(a)] For any $\bsdelta \in \Der(\phi)$, if $\bsdelta_t = \bsu \cdot \tau \bsm$, with $\bsu = \sum_{i=0}^s c_i t^i$, $c_i \in \Mat_{1\times r}(\KK)$, then
\[
  \rF_{\bsdelta}(\bsy) = - \{ \bsdelta_t, \bsg_{\bsy} \bsn \}|_{t=\theta} + \sum_{i=1}^s \sum_{j=0}^{i-1} c_i \theta^{i-1-j} \twistop{\tau \bsm}{\phi_{t^j}(\bsalpha)}.
\]
\item[(b)] In particular, taking $\bsdelta_1, \dots, \bsdelta_r \in \Der(\phi)$ with $(\bsdelta_j)_t = \tau \bsm_{j}$ for $1 \leqslant j \leqslant r$, so that $\bsdelta_1, \dots, \bsdelta_r$ represents a $\KK$-basis of $\rH_{\DR}^1(\phi)$, we have
\[
  \rF_{\bsdelta_j}(\bsy) = -\{ (\bsdelta_j)_t, \bsg_{\bsy} \bsn \}|_{t=\theta}.
\]
\end{enumerate}
\end{proposition}

\begin{proof}
According to \eqref{E:xipairing} and the definition of $\bsg_{\bsy}$ in Lemma~\ref{L:gyha},
\[
  \{ \bsdelta_t, \bsg_{\bsy}\bsn \} = \{ \bsu\tau\bsm, \bsg_{\bsy}\bsn \} = \bsg_{\bsy} V^{-1} \bsu^{\tr} = -\twistop{\tau\bsm}{\cG_{\bsy}}^{\tr} \cdot \bsu^{\tr}
  = -\biggl( \sum_{i=0}^s c_i t^i \twistop{\tau \bsm}{\cG_{\bsy}} \biggr)^{\tr}.
\]
By Corollary~\ref{C:AGFtwistedmult}(a) and by noting that the resulting terms on the right-hand side are all $1\times 1$ matrices, it follows that
\[
  \{ \bsdelta_t, \bsg_{\bsy} \bsn\} = -\twistop{\bsu \cdot \tau \bsm}{\cG_{\bsy}} + \sum_{i=1}^s \sum_{j=0}^{i-1} c_i t^{i-1-j} \twistop{\tau \bsm}{\phi_{t^j}(\bsalpha)}.
\]
Part~(a) then follows upon evaluation at $t=\theta$ by Proposition~\ref{P:quasigen}(a), and part~(b) is an immediate consequence of~(a).
\end{proof}

\begin{corollary} \label{C:xipairingquasiper}
Let $\phi : \bA \to \Mat_d(\KK[\tau])$ be a uniformizable, abelian, and $\bA$-finite $t$-module, and suppose that we have chosen $\KK[t]$-bases $\bsm$ and $\bsn$ for $\cM_{\phi}$ and $\cN_{\phi}$ respectively.  Let $\bsdelta \in \Der(\phi)$ and let $\bslambda \in \Lambda_{\phi}$, and let $\bsg_{\bslambda} \assign -\twistop{\tau \bsm}{\cG_{\bslambda}}^{\tr}\cdot V \in \Mat_{1\times r}(\TT_{\theta})$.  Then
\[
\{ \bsdelta_t, \bsg_{\bslambda}\bsn \} = -\twistop{\bsdelta_t}{\cG_{\bslambda}},
\]
and in particular,
\[
  \{ \bsdelta_t, \bsg_{\bslambda}\bsn \}|_{t=\theta} = -\rF_{\bsdelta}(\bslambda).
\]
\end{corollary}

\begin{remark}
The notation in Corollary~\ref{C:xipairingquasiper} might suggest that $\cE_0(\bsg_{\bslambda}) = \bslambda$ for each $\bslambda \in \Lambda_{\phi}$.  And since $\bsg_{\bslambda}^{(-1)} \Phi = \bsg_{\bslambda}$, it follows from Theorem~\ref{T:Exponentiation} that $\cE_0(\bsg_{\bslambda}) \in \Lambda_{\phi}$.  However, to verify if it equals $\bslambda$ exactly is a bit subtle.  Likewise, Proposition~\ref{P:quasiloggy} together with Theorem~\ref{T:Exponentiation} might indicate that $\cE_0(\bsg_{\bsy} + \bsh_{\bsalpha}) = \bsy$ and $\cE_1(\bsh_{\bsalpha}) = \bsalpha$.  We prove these identities for almost strictly pure $t$-modules in \S\ref{subS:Quasi}.  See Proposition~\ref{P:Expgyha}.
\end{remark}

Specializing now to the case that $\phi$ is defined over $\oK$, the constructions above yield precise information about the $\oK$-span of quasi-logarithms for $\bsy$ with $\Exp_{\phi}(\bsy) \in \oK^d$ and the evaluation~$\bsg_{\bsy}|_{t=\theta}$.

\begin{theorem} \label{T:quasispancomplete}
Let $\phi : \bA \to \Mat_d(\oK[\tau])$ be a uniformizable, abelian, and $\bA$-finite $t$-module defined over~$\oK$, and let $\bsm$ and $\bsn$ be $\oK[t]$-bases for $\cM_{\phi,\oK}$ and $\cN_{\phi,\oK}$ respectively.  Let $\bsdelta_1, \dots, \bsdelta_r \in \Der(\phi,\oK)$ represent a $\oK$-basis of $\rH_{\DR}^1(\phi,\oK)$.  Finally, suppose that $\bsy \in \KK^d$ and $\bsalpha \in \smash{\oK}^d$ are chosen so that $\Exp_{\phi}(\bsy) = \bsalpha$.  Then
\begin{align*}
\Span{\oK} \bigl(1, \rF_{\bsdelta_1}(\bsy), \dots, \rF_{\bsdelta_r}(\bsy) \bigr)
&= \Span{\oK} \bigl( \{1\} \cup \{ \rF_{\bsdelta}(\bsy) : \bsdelta \in \Der(\phi,\oK) \} \bigr) \\
&= \Span{\oK} \bigl( \{1\} \cup \{ \bsg_{\bsy}|_{t=\theta} \} \bigr),
\end{align*}
where $\bsg_{\bsy} \in \Mat_{1\times r}(\TT_\theta)$ is given as in Lemma~\ref{L:gyha}.
\end{theorem}

\begin{proof}
For $\bsdelta = \bsdelta^{(\bsu)} \in \Der_{\si}(\phi,\oK)$, where $\bsu \in \Mat_{1 \times d}(\oK[\tau]\tau)$, it follows from~\eqref{E:quasiinner} that $\rF_{\bsdelta}(\bsy) = \bsu \bsalpha \in \smash{\oK}$.  This implies the first equality.  On the other hand, using~\eqref{E:xipairing} with Proposition~\ref{P:quasiloggy}(b), we see that if $\bsepsilon_1, \dots, \bsepsilon_r \in \Der(\phi,\oK)$ are chosen so that for each $j$ we have $(\bsepsilon_j)_t = \tau \bsm_j$, then
\[
  \rF_{\bsepsilon_j}(\bsy) = \bsg_{\bsy} V^{-1} \bse_{j}^{\tr}\big|_{t=\theta}.
\]
Since $\phi$ is defined over $\oK$, we have $V \in \GL_r(\oK[t])$, and so
\begin{equation}
  \Span{\oK} \bigl( \rF_{\bsepsilon_1}(\bsy), \dots, \rF_{\bsepsilon_r}(\bsy) \bigr)
  = \Span{\oK} \bigl( \bsg_{\bsy}|_{t=\theta} \bigr),
\end{equation}
and the result follows by converting $\bsepsilon_1, \dots, \bsepsilon_r$ to $\bsdelta_1, \dots, \bsdelta_r$ in $\rH_{\DR}^1(\phi,\oK)$.
\end{proof}

\begin{remark} \label{R:tractablequasi}
If $\rd\phi_t$ is in Jordan normal form, then as we saw in Remark~\ref{R:tractable}, it follows from \eqref{E:quasiinner} that for an inner biderivation $\bsdelta$ defined over~$\oK$, $\rF_{\bsdelta}(\bsy)$ is a $\oK$-linear combination of $1$ and the tractable coordinates of $\bsy$.  Thus the $\oK$-vector space defined in Theorem~\ref{T:quasispancomplete} consists of the $\oK$-linear span of $1$, the tractable coordinates of~$\bsy$, and the strictly reduced quasi-logarithms associated to $\bsy$.
\end{remark}

\subsection{Periods and logarithms for almost strictly pure \texorpdfstring{$t$}{t}-modules} \label{subS:Quasi}
Here we account how to use rigid analytic trivializations $(\iota,\Phi,\Psi)$ and Anderson generating functions to explicitly construct periods and logarithms for a class of Anderson $t$-modules we call \emph{almost strictly pure $t$-modules}.  We let $\phi : \bA \to \Mat_d(\KK[\tau])$ be a $t$-module defined over a field $L$ with $K \subseteq L \subseteq \KK$, with
\begin{equation} \label{E:phitB}
  \phi_t = \rd\phi_t + B_1 \tau + \dots + B_w \tau^w, \quad d\phi_t,\ B_i \in \Mat_d(L).
\end{equation}
Recall from Example~\ref{Ex:strictlypure1} that $\phi$ is strictly pure if $B_w \in \GL_d(L)$, and then in this case $\phi$ is pure, abelian, and $\bA$-finite.  On the other hand, we will say that $\phi$ is almost strictly pure if there is some $s \geqslant 1$ so that
\begin{equation} \label{E:phits}
  \phi_{t^s} = \rd\phi_{t^s} + A_1 \tau + \cdots + A_{\ell} \tau^{\ell},
\end{equation}
with $A_{\ell} \in \GL_r(L)$. Almost strictly pure $t$-modules are strictly pure as Anderson $\FF_q[t^s]$-modules, and since $\KK[t]$ is a free and finitely generated $\KK[t^s]$-module, it follows that they are also pure, abelian, and $\bA$-finite as Anderson $t$-modules. An almost strictly pure $t$-module need not be uniformizable, but throughout this section we will assume this to be true of $\phi$.

\begin{remark} \label{R:ASPimpliespure}
Goss~\cite[Rem.~2.2.3]{Goss94}, \cite[Rem.~5.5.5]{Goss}, alludes to the purity of $\phi$, but for completeness we indicate the justification for this. We recall from our discussion in Example~\ref{Ex:strictlypure1} that $\phi$ is pure as Anderson $\FF_q[t^s]$-module of rank $\ell d$. From the definition of purity~\cite[\S 1.9]{And86}, \cite[\S 5.5]{Goss}, this means that if we let $\laurent{\cM_{\phi}}{t^{-s}} \assign \cM_{\phi} \otimes_{\KK[t^s]} \laurent{\KK}{t^{-s}}$, equipped as a left $\KK[t^s,\tau]$-module with $\tau$ acting diagonally, then there is a $\power{\KK}{t^{-s}}$-lattice $W \subseteq \laurent{\cM_{\phi}}{t^{-s}}$ such that
\begin{equation} \label{E:latticeidentity1}
  \tau^{\ell d} W = (t^s)^d W.
\end{equation}
The precise definition of $W$ can be obtained by using the argument in~\cite[Ex.~2.2.2]{Goss94}, \cite[Cor.~5.5.4]{Goss}. Now
\[
  \laurent{\cM_{\phi}}{t^{-1}} \assign \cM_{\phi} \otimes_{\KK[t]} \laurent{\KK}{t^{-1}} \cong \laurent{\cM_{\phi}}{t^{-s}} \otimes_{\laurent{\KK}{t^{-s}}} \laurent{\KK}{t^{-1}},
\]
and we let $W' \assign W \otimes_{\power{\KK}{t^{-s}}} \power{\KK}{t^{-1}}$. It follows that $W'$ is a $\power{\KK}{t^{-1}}$-lattice in $\laurent{\cM_{\phi}}{t^{-1}}$. Furthermore, \eqref{E:latticeidentity1} implies that
\begin{equation} \label{E:latticeidentity2}
  \tau^{\ell d} W' = t^{sd} W',
\end{equation}
and so $\phi$ is pure of weight $s/\ell$ by~\cite[\S 1.9]{And86}, \cite[\S 5.5]{Goss}.
\end{remark}

\begin{remark}
The class of almost strictly pure $t$-modules contains many $t$-modules of common interest, including Drinfeld modules, tensor powers of the Carlitz module (see~\cite[Prop.~1.6.1]{AndThak90} or \cite[\S 3.5]{PLogAlg}), and tensor products of Drinfeld modules (through formulas in~\cite[\S 2]{Hamahata93}).  The prolongation $t$-modules of Maurischat~\cite{Maurischat18} for an almost strictly pure $t$-module are themselves almost strictly pure (see Theorem~\ref{T:prolmod}).  The methods in this section and later can also be applied to iterated extensions of almost strictly pure $t$-modules, such as the $t$-modules associated to multiple zeta values and multiple polylogarithms~\cite{Chang14}--\cite{ChangMishiba21}, \cite{CPY19}, \cite{GreenNgoDac20b}.
\end{remark}

As noted above, $\rank_{\KK[t^s]}(\cN_{\phi}) = \ell d$ as in Example~\ref{Ex:strictlypure1}, and so
\[
r = \rank_{\KK[t]}(\cN_{\phi}) = \rank_{\KK[t]}(\cM_{\phi}) = \frac{\ell d}{s},
\]
which in particular agrees with~\eqref{E:latticeidentity2}. We can assume that $\ell \geqslant w$ in \eqref{E:phitB} and \eqref{E:phits}, because by taking a high enough multiple of $s$ this inequality is assured.  We will also use the conventions that $B_0 = \rd\phi_t$ and for $i < 0$ or $i > w$, $B_i=0$.  For $i \geqslant 0$ and $1 \leqslant j \leqslant d$, we set
\[
  \bsm'_{id+j} \assign \tau^i \bss_j \in \cM_{\phi}, \quad \bsn'_{id+j} \assign \sigma^i \bss_j \in \cN_{\phi}.
\]
In this way, $\bsm' \assign (\bsm'_1, \dots, \bsm'_{\ell d})^{\tr}$ and $\bsn' \assign (\bsn'_1, \dots, \bsn'_{\ell d})^{\tr}$ are $\KK[t^s]$-bases for $\cM_{\phi}$ and $\cN_{\phi}$ respectively (cf.~Examples~\ref{Ex:strictlypure1} and~\ref{Ex:strictlypure2}).  If we fix $\KK[t]$-bases
\begin{equation} \label{E:MNbases}
\bsm = (\bsm_1, \dots, \bsm_r)^{\tr} \in (\cM_{\phi})^r, \quad
\bsn = (\bsn_1, \dots, \bsn_r)^{\tr} \in (\cN_{\phi})^r,
\end{equation}
then there are unique matrices $X$, $Y \in \Mat_{\ell d \times r}(\KK[t])$ so that
\begin{equation} \label{E:XYdef}
  \bsm' = X\bsm, \quad \bsn' = Y\bsn.
\end{equation}
Now we let $\bsmu = (\bsmu_1, \dots, \bsmu_r)^{\tr} \in (\cM_{\phi}^{\wedge})^r$ be the dual $\KK[t]$-basis of $\tau \bsm$.  The definition of $X$ then implies
\begin{equation} \label{E:mumprimes}
  \begin{pmatrix} \bsmu_1(\tau \bsm_1') & \cdots & \bsmu_1(\tau \bsm_{\ell d}') \\
  \vdots & & \vdots \\
  \bsmu_r(\tau \bsm_1') & \cdots & \bsmu_r(\tau \bsm_{\ell d}')
  \end{pmatrix}
  = \bigl( X^{(1)} \bigr)^{\tr}\, dt.
\end{equation}
Recalling $\kappa$ from~\eqref{E:kappadef}, we see that we have a commuting diagram
\[
\begin{tikzcd}
\cM_{\phi}^{\wedge} = \Hom_{\KK[t]}(\tau \cM_{\phi}, \KK[t]\,dt)
\arrow[d, hook]
\arrow[r, hook, "\kappa"]
&
\Hom_{\laurent{\KK}{z}}\bigl(\laurent{\KK}{z} \otimes_{\KK[t]} \tau \cM_{\phi}, \laurent{\KK}{z}\,dz \bigr)
\arrow[d, hook]
\\
(\cM_{\phi}^{\wedge})' = \Hom_{\KK[t^s]}(\tau \cM_{\phi}, \KK[t]\,dt)
\arrow[r, hook, "\kappa'"]
&
\Hom_{\laurent{\KK}{z^s}}\bigl(\laurent{\KK}{z} \otimes_{\KK[t]} \tau \cM_{\phi}, \laurent{\KK}{z}\,dz \bigr),
\end{tikzcd}
\]
where the vertical arrows are inclusions and $\kappa'$ is the natural embedding indicated.  We note that $(\cM_{\phi}^{\wedge})'$ is a free $\KK[t]$-module of rank $rs=\ell d$ and that we can choose a $\KK[t]$-basis $\bsmu_1', \dots, \bsmu_{\ell d}'$ so that for $1 \leqslant i, j  \leqslant \ell d$, we have $\bsmu_i'(\tau \bsm_j') = \delta_{ij}$.  It follows from \eqref{E:mumprimes} that
\begin{equation} \label{E:muprimetomu}
  \bsmu = \bigl( X^{(1)} \bigr)^{\tr} \bsmu'.
\end{equation}
If we consider the map $\xi : \cM_{\phi}^{\wedge} \to \cN_{\phi}$ in~\eqref{E:xidef}, we note that we can further define $\xi' : (\cM_{\phi}^{\wedge})' \to \cN_{\phi}$ by setting
\begin{equation} \label{E:xiprimedef}
  \xi'(\bsnu') =  -\sum_{i \geqslant 0} \Bigl( \Res_{\infty} \bigl( \kappa'(\bsnu') ( \tau^{-i-1} \cdot \tau \bss_{j}) \bigr) \Bigr)_{1 \leqslant j \leqslant d} \cdot \sigma^i.
\end{equation}
It is straightforward to check from the commutative square above and the definitions of $\xi$ and $\xi'$ that for $\bsnu \in \cM_{\phi}^{\wedge}$ we have $\xi'(\bsnu) = \xi(\bsnu)$.

\begin{proposition} \label{P:xiprimemuprime}
As an element of $\Mat_{\ell d \times d}(\KK[\sigma])$, we have
\[
  \xi'(\bsmu') = \begin{pmatrix} \xi'(\bsmu_1') \\ \vdots \\ \xi'(\bsmu_{\ell d}') \end{pmatrix}
  = \sum_{i \geqslant 0} \Bigl( B_{i+1}^{(-i)}, \dots, B_{i + \ell}^{(-i)} \Bigr)^{\tr} \cdot \sigma^i,
\]
and so
\[
  \xi'(\bsmu') = B^{\tr} \bsn',
\]
where
\begin{equation} \label{E:Bdef}
  B \assign \begin{pmatrix}
  B_1 & B_2 & \cdots & B_{\ell-1} & B_{\ell} \\
  B_2^{(-1)} & B_3^{(-1)} & \cdots & B_{\ell}^{(-1)}& \\
  \vdots & \vdots & \reflectbox{$\ddots$} & & \\
  B_{\ell-1}^{(-\ell+2)} & B_{\ell}^{(-\ell + 2)} & & &\\
  B_{\ell}^{(-\ell+1)} & & & &
  \end{pmatrix}.
\end{equation}
\end{proposition}

\begin{proof}
We recall the convention that $B_i = 0$ for $i > w$ in~\eqref{E:phitB}.  We adopt further the convention that if $\alpha$, $\beta \in \Mat_{u\times v}(\laurent{\KK}{z})$, then $\alpha = \beta + O(z^m)$ indicates that each entry satisfies $\alpha_{ij} = \beta_{ij} + O(z^m)$.  For $\gamma \in \Hom_{\KK[t^s]}(\tau \cM_{\phi}, \KK[t]\,dt)$ and $\alpha \in \Mat_{u\times v}(\tau \cM_{\phi})$, we set $\gamma(\alpha) \in \Mat_{u\times v}(\KK[t]\,dt)$ so that $(\gamma(\alpha))_{ij} = (\gamma(\alpha_{ij}))$ (and similarly for $\Hom_{\laurent{\KK}{z^s}}(\laurent{\KK}{z} \otimes_{\KK[t]} \tau \cM_{\phi}, \laurent{\KK}{z}\,dz)$).

We first note that similar to Example~\ref{Ex:strictlypure1}, as elements of $(\tau\cM_{\phi})^d$ we have from \eqref{E:phitB} that (recall $\ell \geqslant w$)
\begin{equation} \label{E:bsstomprime}
  ( t \Id_d - \rd \phi_t ) \cdot \bss = ( B_1, \dots, B_{\ell} ) \cdot \tau \bsm',
\end{equation}
where $\bss = (\bss_1, \dots, \bss_d)^{\tr}$.  From this it follows that
\begin{equation} \label{E:bssintermsofz}
  \bss = \bigl( ( t\Id_d - \rd\phi_t )^{-1}B_1, \dots, (t\Id_d - \rd\phi_t)^{-1}B_{\ell} \bigr) \cdot \tau \bsm' \in \Mat_{d\times 1}(\laurent{K}{z} \otimes_{\KK[t]} \tau \cM_{\phi}).
\end{equation}
For $i \geqslant 0$, we claim that as an element of $( \laurent{\KK}{z} \otimes_{\KK[t]} \tau \cM_{\phi})^d$, we have
\begin{equation} \label{E:taubssintermsofz}
  \tau^{-i-1} \cdot \tau \bss
  = \Bigl( zB_{i+1}^{(-i)} + O(z^2), \dots, zB_{i+\ell}^{(-i)} + O(z^2) \Bigr) \cdot \tau \bsm',
\end{equation}
where we have used~\eqref{E:tauauto} to invert $\tau$.

Let $i=0$ and suppose that
\begin{equation*}\label{E:JNF}
    \rJ = \begin{pmatrix}
d_{\theta,\ell_1}[\theta] & &  \\
 & \ddots & \\
& &  d_{\theta,\ell_m}[\theta] \end{pmatrix},
\end{equation*}
where $\ell_1, \dots, \ell_m > 0$ and $\ell_1 + \cdots + \ell_m = d$, is the Jordan normal form of $\rd \phi_{t}$. If $\rQ\in \GL_d(\KK)$ such that $\rd\phi_{t}=\rQ \rJ \rQ^{-1}$, then we see that
\begin{equation}\label{E:JNFt}
(t\Id_d - \rd\phi_t)^{-1} = \rQ^{-1}(t\Id_d - \rJ)^{-1}\rQ,
\end{equation}
where
\[
  t\Id_d - \rJ = \begin{pmatrix}
  d_{\theta,\ell_1}[t-\theta] & & \\ & \ddots & \\ & & d_{\theta,\ell_m}[t-\theta].
  \end{pmatrix}.
\]
Using that $z=1/t$, this implies that
\[
  (t\Id_d - \rJ)^{-1} = \begin{pmatrix}
  d_{\theta,\ell_1} \biggl[ \dfrac{z}{1-\theta z} \biggr] & & \\
  & \ddots & \\
  & & d_{\theta,\ell_m} \biggl[ \dfrac{z}{1-\theta z} \biggr]
  \end{pmatrix}.
\]
Moreover, since for $j \geqslant 0$,
\[
  \pd_\theta^j \biggl( \frac{z}{1-\theta z} \biggr) = \frac{z^j}{(1-\theta z)^j},
\]
for any constant matrices $C_1, C_2 \in \Mat_{e}(\KK)$,
\[
 C_1 \cdot d_{\theta,e} \biggl[ \frac{z}{1-\theta z} \biggr] \cdot C_2 = zC_1C_2 + O(z^2).
\]
It then follows from \eqref{E:bssintermsofz} and \eqref{E:JNFt} that
\begin{equation} \label{E:i0case}
  \tau^{-1} \cdot \tau \bss = \bigl( zB_1 + O(z^2), \dots, zB_{\ell} + O(z^2) \bigr)\cdot \tau \bsm',
\end{equation}
which concludes the $i=0$ case of \eqref{E:taubssintermsofz}.

Now suppose that $i \geqslant 1$, and suppose that we have shown
\[
  \tau^{-(i-1)-1} \cdot \tau \bss = \bigl( zB_{i}^{(-i+1)} + O(z^2), \dots, zB_{i-1+\ell}^{(-i+1)} + O(z^2) \bigr) \cdot \tau \bsm'.
\]
Then
\begin{align*}
  \tau^{-i-1} \cdot \tau \bss &= \bigl( z B_{i}^{(-i)} + O(z^2), \dots, z B_{i-1+\ell}^{(-i)} +O(z^2) \bigr)
  \cdot \tau \bsm' \\
  &= \begin{aligned}[t]
  \bigl( zB_i^{(-i)} + O(z^2) \bigr) \bss &{}+ \bigl( zB_{i+1}^{(-i)} + O(z^2) \bigr)  \tau \bss \\
  &{}+ \cdots + \bigl( zB_{i-1+\ell}^{(-i)} + O(z^2) \bigr) \tau^{\ell-1}\bss.
  \end{aligned}
\end{align*}
Substituting in the $i=0$ case from~\eqref{E:i0case}, we obtain
\begin{align*}
  \tau^{-i-1} \cdot \tau \bss &= \begin{aligned}[t]
    \bigl(z B_i^{(-i)} + O(z^2) \bigr) & \bigl( (zB_1 + O(z^2)) \tau \bss + \cdots + (zB_{\ell} + O(z^2))\tau^{\ell} \bss \bigr) \\
    &{} + \bigl( z B_{i+1}^{(-i)} + O(z^2) \bigr) \tau \bss + \cdots + \bigl( z B_{i-1+\ell}^{(-i)} + O(z^2) \bigr) \tau^{\ell-1}\bss
    \end{aligned}
    \\
  &= \bigl( z B_{i+1}^{(-i)} + O(z^2) \bigr) \tau \bss + \cdots + \bigl( z B_{i-1+\ell}^{(-i)} + O(z^2) \bigr) \tau^{\ell-1}\bss + \bigl( z B_{i+\ell}^{(-i)} + O(z^2) \bigr) \tau^{\ell}\bss,
\end{align*}
where we have used that the first collection of terms in the first equality ultimately contains only $O(z^2)$ terms and the final $B_{i+\ell}^{(-i)}$ can be added on because in fact $B_{i+\ell}=0$ as $i \geqslant 1$.  Then \eqref{E:taubssintermsofz} follows and the claim is proved.

For $1 \leqslant k \leqslant \ell d$ and $i \geqslant 0$, we then consider the $1\times d$ vector,
\begin{align*}
\Bigl( \Res_{\infty} &\bigl( \kappa'(\bsmu_k')(\tau^{-i-1} \cdot \tau \bss_j)\bigr) \Bigr)_{1 \leqslant j \leqslant d} \\
  &= \Res_{\infty} \Bigl( \kappa'(\bsmu_k') \bigl( (\tau^{-i-1} \cdot \tau\bss)^{\tr} \bigr) \Bigr) \\
  &= \Res_{z=0} \Bigl( \kappa' (\bsmu_k') \bigl( \bigl( \bigl( z B_{i+1}^{(-i)} + O(z^2), \dots, z B_{i+\ell}^{(-i)} + O(z^2) \bigr)^{\tr} \bigr) \tau \bsm' \bigr) \Bigr) \\
  &= \Res_{z=0} \Bigl( \kappa'(\bsmu_k') \bigl( \bigl( (zB_{i+1}^{(-i)} + O(z^2))\tau \bss
  + \cdots + (z B_{i+\ell}^{(-i)} + O(z^2))\tau^{\ell}\bss \bigr)^{\tr} \bigr) \Bigr).
\end{align*}
We claim that it follows from this calculation that
\begin{align} \label{E:Resmatrix}
  R &\assign \Bigl( \Res_{\infty} \bigl( \kappa'(\bsmu_k')(\tau^{-i-1} \cdot \tau \bss_j)\bigr) \Bigr)_{\substack{1 \leqslant k \leqslant \ell d \\ 1 \leqslant j \leqslant d}}\\
  &= - \Bigl( B_{i+1}^{(-i)}, \dots, B_{i+1}^{(-i)} \Bigr)^{\tr} \in \Mat_{\ell d \times d}(\KK). \notag
\end{align}
Indeed from the preceding calculation, we have the $(k,j)$-entry,
\begin{align*}
  R_{kj} &= \Res_{z=0}\Bigl( \kappa'(\bsmu_k') \bigl( \textup{row $j$ of $(z B_{i+1}^{(-i)} + O(z^2), \dots, zB_{i+\ell}^{(-i)} + O(z^2))$} \tau\bsm' \bigr) \Bigr) \\
  &= \Res_{z=0} \biggl( \bigl( \textup{$(j,k)$-entry of $\bigl(zB_{i+1}^{(-i)}, \dots, zB_{i+\ell}^{(-i)} \bigr)$} \bigr) \biggl( -\frac{dz}{z^2} \biggr) \biggr),
\end{align*}
and \eqref{E:Resmatrix} follows.  Combining \eqref{E:xiprimedef} and \eqref{E:Resmatrix} the formula for $\xi'(\bsmu')$ in the statement of the proposition is immediate.
\end{proof}

Continuing with the notation in \eqref{E:MNbases}--\eqref{E:Bdef}, we demonstrate how to choose $V \in \GL_r(\KK[t])$ as in~\eqref{E:Vdef} so that $\xi(\bsmu) = V\bsn$.  We fix $\Phi$, $\tPhi \in \Mat_r(\KK[t])$ so that $\tau \bsm = \tPhi \bsm$, $\sigma \bsn = \Phi \bsn$, and $\sigma \bsmu = \tPhi^{\tr} \bsmu$.

\begin{corollary} \label{C:VXBY}
We continue with the notation above.
\begin{enumerate}
\item[(a)] We have $\xi(\bsmu) = V\bsn$, where
\[
  V = ( X^{(1)})^{\tr} B^{\tr} Y \in \GL_r(\KK[t]).
\]
\item[(b)] Furthermore,
\[
  V^{(-1)} \Phi = \tPhi^{\tr} V = X^{\tr}C^{\tr}Y,
\]
where
\begin{equation} \label{E:Cdef}
  C \assign \begin{pmatrix}
  t\Id_d - \rd\phi_t & 0 & \cdots & 0 & 0 \\
  0 & B_2^{(-1)} & \cdots & B_{\ell-1}^{(-1)} & B_{\ell}^{(-1)} \\
  0 & B_3^{(-2)} & \cdots & B_{\ell}^{(-2)} & \\
  \vdots & \vdots & \reflectbox{$\ddots$} & & \\
  0 & B_{\ell}^{(-\ell+1)} & & &
  \end{pmatrix}.
\end{equation}
\end{enumerate}
\end{corollary}

\begin{proof}
For (a), we calculate
\[
\xi(\bsmu) = \xi'(\bsmu) = \xi' \bigl( (X^{(1)})^{\tr} \bsmu' \bigr)
= (X^{(1)})^{\tr} \xi'(\bsmu')
= (X^{(1)})^{\tr} B^{\tr} \bsn'
= (X^{(1)})^{\tr} B^{\tr} Y \bsn,
\]
where in the second equality we have used~\eqref{E:muprimetomu}, in the fourth Proposition~\ref{P:xiprimemuprime}, and in the last~\eqref{E:XYdef}.  Because $\xi$ is a $\KK[t]$-isomorphism from Theorem~\ref{T:isodual} it follows that $V = (X^{(1)})^{\tr} B^{\tr} Y$ and a fortiori that this matrix is invertible over $\KK[t]$.

Using \eqref{E:bsstomprime}, a short calculation reveals that $B\tau \bsm' = C \bsm'$, and from~\eqref{E:XYdef} it follows that $BX^{(1)}\tPhi \bsm = CX \bsm$, and thus
\[
  BX^{(1)}\tPhi = CX.
\]
Similarly we obtain $(B^{(-1)})^{\tr} \sigma \bsn' = C^{\tr}\bsn'$, and from this we see $(B^{(-1)})^{\tr} Y^{(-1)}\Phi \bsn = C^{\tr}Y \bsn$ so that
\[
  (B^{(-1)})^{\tr} Y^{(-1)}\Phi = C^{\tr}Y.
\]
Therefore using part~(a),
\[
  V^{(-1)} \Phi = X^{\tr} (B^{(-1)})^{\tr} Y^{(-1)}\Phi = X^{\tr}C^{\tr}Y = \tPhi^{\tr}(X^{(1)})^{\tr} B^{\tr} Y = \tPhi^{\tr}V.
\]
Of course the first equality in~(b) is~\eqref{E:Visomorphism}, but we recover it here as well.
\end{proof}

We now demonstrate how to use Anderson generating functions together with Anderson's Theorem~\ref{T:Exponentiation} to produce identities of exponentials, logarithms, and quasi-logarithms of arbitrary points on our almost strictly pure $t$-module~$\phi$.  We give a proof for the general case below, and an alternate proof with some restrictions in \S\ref{subS:quasiperexamples}. The previous constructions and notations from this section remain in force.

\begin{proposition} \label{P:Expgyha}
Let $\phi : \bA \to \Mat_d(\KK[\tau])$ be a uniformizable almost strictly pure $t$-module.  Let $\bsalpha$, $\bsy\in \KK^d$ be chosen so that $\Exp_{\phi}(\bsy) = \bsalpha$, and let $\bsg_{\bsy}$ and $\bsh_{\bsalpha}$ be given as in Lemma~\ref{L:gyha}.  Then
\begin{enumerate}
\item[(a)] $\cE_0( \bsg_{\bsy}+\bsh_{\bsalpha}) = \bsy$,
\item[(b)] $\cE_1(\bsh_{\bsalpha}) = \bsalpha$.
\end{enumerate}
\end{proposition}

\begin{proof}
Suppose that ${\bsf_n}$ is a sequence in $\Mat_{1\times r}(\KK[t])$ converging to $\bsf \in \Mat_{1 \times r}(\TT_{\theta})$.  For each $n$, Proposition~\ref{P:tframeprops}(b) implies $\iota(\bsf_n\Phi) = \sigma\iota(\bsf_n^{(1)})$, and so
\[
  (\varepsilon_0 \circ \iota)\bigl( \bsf_n \Phi \bigr) = 0.
\]
Since $\bsf_n \to \bsf$ in $\Mat_{1\times r}(\TT_{\theta})$ it follows for such a sequence that
\begin{equation} \label{E:E0limit}
  \cE_0\bigl( \bsf \Phi \bigr) = 0.
\end{equation}
Now by Lemma~\ref{L:gyha} combined with~\eqref{E:galphatwistPhi}, we have
\begin{align} \label{E:E0simplify}
  \cE_0 (\bsg_{\bsy} + \bsh_{\bsalpha}) &= -\cE_0 \bigl( \twistop{\bsm}{\cG_{\bsy}}^{\tr} \cdot V^{(-1)}\Phi \bigr), \\
\intertext{and using~\eqref{E:E0limit} together with Lemma~\ref{L:AGFsum} and Remark~\ref{R:AGFconvergence},}
   &= -\cE_0\bigl( \twistop{\rd\bsm}{\cG_{\bsy}}^{\tr}\cdot V^{(-1)}\Phi \bigr) \notag \\
   &= -\cE_0 \Bigl( \bigl\langle \rd \bsm \bigm| (\rd\phi_t - t\Id_d)^{-1} \cdot \bsy \bigr\rangle^{\tr} \cdot V^{(-1)}\Phi \Bigr). \notag
\end{align}
Recalling~\eqref{E:XYdef} where $\bsm' = X\bsm$, write $X = \sum_{i=0}^v X_i t^i$ with $X_i \in \Mat_{\ell d \times r}(\KK)$, and then multiply in $\cM_{\phi}$ to find
\begin{align*}
  X\cdot \bsm &= \sum_{i=0}^v X_i\, \bsm\, \phi_{t^i} \in \Mat_{\ell d \times d}(\KK[\tau]) \\
  &\Rightarrow \quad
  \rd \bsm'= \rd(X \cdot \bsm) = \sum_{i=0}^v X_i\, \rd \bsm\, \rd\phi_{t^i} \in \Mat_{\ell d \times d}(\KK).
\end{align*}
On the other hand, by direct multiplication
\[
  X\,\rd\bsm = \sum_{i=0}^v X_i\,\rd\bsm\,t^i \in \Mat_{\ell d \times d}(\KK[t]).
\]
Combining these together and using Corollary~\ref{C:VXBY}(b), we find
\begin{multline*}
-\bigl\langle \rd \bsm \bigm| (\rd\phi_t - t\Id_d)^{-1} \cdot \bsy \bigr\rangle^{\tr} \cdot V^{(-1)}\Phi
= \bsy^{\tr} \, (t\Id_d - \rd\phi_t^{\tr})^{-1} \cdot \rd\bsm^{\tr} X^{\tr}C^{\tr}Y \\
= \bsy^{\tr}\, (t\Id_d - \rd\phi_t^{\tr})^{-1} \biggl( \bigl(\rd \bsm' \bigr)^{\tr} + \sum_{i=0}^v \bigl( t^i\Id_d - \rd\phi_{t^i}^{\tr}\bigr)\,\rd\bsm^{\tr}\, X_i^{\tr}\biggr) C^{\tr} Y,
\end{multline*}
and so
\begin{multline} \label{E:E0gyhapenultimate}
  \cE_0 (\bsg_{\bsy} + \bsh_{\bsalpha}) = \cE_0\Bigl( \bsy^{\tr}\, \bigl(t\Id_d - \rd\phi_t^{\tr}\bigr)^{-1} (\rd \bsm')^{\tr}C^{\tr}Y\Bigr) \\ {}+ \cE_0\biggl( \bsy^{\tr}\, \bigl(t\Id_d - \rd\phi_t^{\tr}\bigr)^{-1} \sum_{i=0}^v \bigl( t^i\Id_d - \rd\phi_{t^i}^{\tr}\bigr)\,\rd\bsm^{\tr}\, X_i^{\tr} C^{\tr} Y\biggr).
\end{multline}
The latter term is $0$, as $D \assign \bsy^{\tr}\, (t\Id_d - \rd\phi_t^{\tr})^{-1} \sum_{i=0}^v ( t^i\Id_d - \rd\phi_{t^i}^{\tr})\,\rd\bsm^{\tr}\, X_i^{\tr} \in \Mat_{1 \times \ell d}(\KK[t])$, and
\begin{align*}
\cE_0( D C^{\tr} Y) = \varepsilon_0(\iota(D C^{\tr} Y)) = \varepsilon_0(D C^{\tr} Y\bsn) = \varepsilon_0(D C^{\tr}\bsn') = \varepsilon_0 \Bigl( D \bigl( t\Id_d -\rd\phi_t^{\tr}\bigr)\bss \Bigr) = 0.
\end{align*}
Here in the fourth equality we used \eqref{E:Cdef} and the fact that $\bsn' = (\bss,\sigma \bss, \dots, \sigma^{\ell-1}\bss)^{\tr}$, and the final equality follows from the dual $t$-motivic version of~\eqref{E:bsstomprime}.  At last we note that $(\rd\bsm')^{\tr} C^{\tr} Y\bsn = (\rd\bsm')^{\tr}C^{\tr}\bsn' \equiv (t\Id_d - \rd\phi_t^{\tr})\bss \pmod{\sigma\cN_{\phi}}$, and so from \eqref{E:E0},
\[
\cE_0(\bsg_{\bsy}+\bsh_{\bsalpha}) = \varepsilon_0\bigl( \bsy^{\tr}\bss\bigr) = \bsy,
\]
which proves~(a). Part~(b) follows then from Theorem~\ref{T:Exponentiation} and Lemma~\ref{L:gyha}.  However, for completeness we give a direct proof.  From Corollary~\ref{C:AGFtwistedmult}(a), we have
\begin{align*}
  CX\twistop{\bsm}{\cG_{\bsy}} &= C\twistop{X\bsm}{\cG_{\bsy}} - \sum_{i=1}^v\sum_{j=0}^{i-1} t^{i-1-j} CX_i \twistop{\bsm}{\phi_{t^j}(\bsalpha)}, \\
  &=
    \twistop{CX\bsm}{\cG_{\bsy}} - \begin{pmatrix} \Id_d & 0 \\ 0 & 0 \end{pmatrix} \twistop{X\bsm}{\bsalpha}
    - \sum_{i=1}^v\sum_{j=0}^{i-1} t^{i-1-j} CX_i \twistop{\bsm}{\phi_{t^j}(\bsalpha)}
  \\
  &=\twistop{CX\bsm}{\cG_{\bsy}} - \bigl( \bsalpha^{\tr}, 0, \dots, 0\bigr)^{\tr}
    - \sum_{i=1}^v\sum_{j=0}^{i-1} t^{i-1-j} CX_i \twistop{\bsm}{\phi_{t^j}(\bsalpha)},
\end{align*}
where in the last equality we have used that $X\bsm = \bsm'$.  Similarly, using $B$ from~\eqref{E:Bdef},
\[
BX^{(1)} \twistop{\tau \bsm}{\cG_{\bsy}} = \twistop{BX^{(1)}\tau \bsm}{\cG_{\bsy}} - \sum_{i=1}^v\sum_{j=0}^{i-1} t^{i-1-j} BX_i^{(1)} \twistop{\tau \bsm}{\phi_{t^j}(\bsalpha)}.
\]
Letting $E$ and $F$ denote the final double sums in these two identities, then using them together with Corollary~\ref{C:VXBY} and~\eqref{E:galphatwistPhi}, we have
\begin{align} \label{E:bshalphadecomp}
  \bsh_{\bsalpha} &= \bsg_{\bsy}^{(-1)}\Phi - \bsg_{\bsy} \\
  &= -\twistop{\bsm}{\cG_{\bsy}}^{\tr}X^{\tr}C^{\tr}Y + \twistop{\tau \bsm}{\cG_{\bsy}}^{\tr} (X^{(1)})^{\tr} B^{\tr}Y \notag \\
  &= \bigl( \bsalpha^{\tr},0,\dots, 0\bigr)Y - \twistop{CX\bsm}{\cG_{\bsy}}^{\tr}Y + E^{\tr}Y
    + \twistop{BX^{(1)}\tau\bsm}{\cG_{\bsy}}^{\tr}Y - F^{\tr}Y. \notag
\end{align}
Now by Corollary~\ref{C:VXBY} we find
\[
  \twistop{BX^{(1)} \tau \bsm}{\cG_{\bsy}}^{\tr}Y = \twistop{Y^{\tr} B X^{(1)} \tPhi \bsm}{\cG_{\bsy}}^{\tr}
  = \twistop{V^{\tr}\tPhi \bsm}{\cG_{\bsy}}^{\tr} = \twistop{CX\bsm}{\cG_{\bsy}}^{\tr}Y.
\]
Since $Y\bsn = \bsn'$, we see that
\begin{multline*}
  \sigma \bigl(F^{\tr} Y\bsn \bigr) = \sigma(F\bsn')
    = \sum_{i=1}^v\sum_{j=0}^{i-1} t^{i-1-j} \twistop{\bsm}{\phi_{t^j}(\bsalpha)}X_i^{\tr}(B^{(-1)})^{\tr}\sigma \bsn' \\
    = \sum_{i=1}^v\sum_{j=0}^{i-1} t^{i-1-j} \twistop{\bsm}{\phi_{t^j}(\bsalpha)}X_i^{\tr}C^{\tr} \bsn'
    = E^{\tr}Y\bsn.
\end{multline*}
Combining these last identities with \eqref{E:bshalphadecomp} to calculate $\iota(\bsh_{\bsalpha}) = \bsh_{\bsalpha}\bsn$ and using that $\varepsilon_1$ is trivial on $(\sigma-1)\cN_{\phi}$, we find that $\cE_1(\bsh_{\bsalpha}) = \varepsilon_1(\iota(\bsh_{\bsalpha})) = \bsalpha$.
\end{proof}

\subsection{Examples} \label{subS:quasiperexamples}
We now turn to explicit calculations by fixing an almost strictly pure $t$-module $\phi : \bA \to \Mat_d(\KK[\tau])$, maintaining the same notation as in \S\ref{subS:Quasi}.  Then later in the section we consider further specialized cases to demonstrate how various results in \S\ref{S:Quasi} are realized for common classes of $t$-modules.

So far we have not imposed any restrictions on the $\KK[t]$-bases $\bsm$ and $\bsn$, but to facilitate calculations with logarithms and quasi-logarithms, we assume that $\bsn$ is chosen as in the statement of Proposition~\ref{P:epsiota} and that $\bsm$ is chosen as in Remark~\ref{R:tmotivebasis}.  In particular, if we let $d_i \assign \ell_1 +\dots + \ell_i$ for $1 \leqslant i \leqslant m$, then
\begin{align} \label{E:specialbases}
\bsm_1 &\equiv \bss_1 \pmod{\tau \cM_{\phi}} & \bsn_1 &\equiv \bss_{d_1} \pmod{\sigma \cN_{\phi}}, \\
\bsm_2 &\equiv \bss_{d_1+1} \pmod{\tau \cM_{\phi}} & \bsn_2 &\equiv \bss_{d_2} \pmod{\sigma \cN_{\phi}}, \notag \\
{\vdots \;\;} &  {\qquad \vdots} & {\vdots \;\;} & {\qquad \vdots} \notag \\
\bsm_m & \equiv \bss_{d_{m-1}+1} \pmod{\tau \cM_{\phi}} & \bsn_m &\equiv \bss_{d_m} \pmod{\sigma \cN_{\phi}}. \notag
\end{align}
We also have $\bsn_m \equiv \bss_d \pmod{\sigma \cN_{\phi}}$, and if we let $d_0 \assign 0$, then $\bsm_1 \equiv \bss_{d_0+1} \pmod{\tau \cM_{\phi}}$.  We then write $X$ and $Y$ from~\eqref{E:XYdef} as
\begin{equation} \label{E:XYblocks}
  X = \begin{pmatrix} X' \\ X'' \end{pmatrix}, \quad
  Y = \begin{pmatrix} Y' \\ Y'' \end{pmatrix},
\end{equation}
where $X'$, $Y' \in \Mat_{d \times r}(\KK[t])$ and $X''$, $Y'' \in \Mat_{(\ell-1)d \times r}(\KK[t])$.  By the conditions of Proposition~\ref{P:epsiota}, Remark~\ref{R:tmotivebasis}, and \eqref{E:specialbases}, we find
\[
X' = \mbox{\scriptsize $\left( \begin{array}{c|c}
    \begin{matrix}
    1 & & \\
    t-\theta & & \\
    \vdots & & \\
    (t-\theta)^{\ell_1-1} & & \\
    & \ddots & \\
    & & 1 \\
    & & t-\theta \\
    & & \vdots \\
    & & (t-\theta)^{\ell_m-1}
    \end{matrix}
   & \text{\Large{$*$}}
  \end{array}
  \right)$},\
Y' = \mbox{\scriptsize $\left( \begin{array}{c|c}
    \begin{matrix}
    (t-\theta)^{\ell_1-1} & & \\
    \vdots & & \\
    t-\theta & & \\
    1 & & \\
    & \ddots & \\
    & & (t-\theta)^{\ell_m-1} \\
    & & \vdots \\
    & & t-\theta \\
    & & 1
    \end{matrix}
   & \text{\Large{$*$}}
  \end{array}
  \right)$}.
\]
Furthermore these choices of bases imply that for $1 \leqslant j \leqslant m$, the $j$-th columns of $X''$ and $Y''$ are both divisible by $(t-\theta)^{\ell_j}$.  The following lemma is central for future computations.

\begin{lemma} \label{L:VPhi}
For the choices of $\KK[t]$-bases for $\cM_{\phi}$ and $\cN_{\phi}$ in~\eqref{E:specialbases}, the upper-left $m\times m$ block of $V^{(-1)} \Phi = \tPhi^{\tr}V$ is
\[
  \begin{pmatrix}
  (t-\theta)^{\ell_1} & & \\
  & \ddots & \\
  & & (t-\theta)^{\ell_m}
  \end{pmatrix}
  +
  \biggl( O\Bigl( (t-\theta)^{\ell_k + \ell_j} \Bigr) \biggr),
\]
where $k$ and $j$ refer to the $(k,j)$-entry.
\end{lemma}

\begin{remark} \label{R:VPhi}
We can specify more about the entries of $V^{(-1)}\Phi = \tPhi^{\tr}V$ as follows.  The conditions of Proposition~\ref{P:epsiota} say that the row space of $\Phi$ over $\KK[t]$ is the same as the row space of the diagonal matrix $D$ in the proposition.  As the row spaces of $V^{(-1)}\Phi$ and $\Phi$ are the same, it follows that for $1 \leqslant j \leqslant m$, the $j$-th column is divisible by $(t-\theta)^{\ell_j}$.  Using the conditions in Remark~\ref{R:tmotivebasis}, it similarly follows from analyzing $\tPhi^{\tr}V$ that its $j$-th row is divisible by $(t-\theta)^{\ell_j}$.  However, within the upper-left $m \times m$ block, the conditions of the lemma are in general stronger.
\end{remark}

\begin{proof}[Proof of Lemma~\ref{L:VPhi}]
We let $C'' \in \Mat_{(\ell-1)d \times (\ell-1)d}(\KK)$ be chosen so that in \eqref{E:Cdef},
\[
  C = \begin{pmatrix}
    t\Id_d - \rd\phi_t & 0 \\
    0 & C''
    \end{pmatrix}.
\]
Then by Corollary~\ref{C:VXBY}(b),
\begin{align} \label{E:VPhisum}
V^{(-1)}\Phi &= \bigl( (X')^{\tr}, (X'')^{\tr} \bigr)
  \begin{pmatrix} t\Id_d - \rd\phi_t^{\tr} & 0 \\ 0 & (C'')^{\tr} \end{pmatrix}
  \begin{pmatrix} Y' \\ Y'' \end{pmatrix} \\
  &= (X')^{\tr} \bigl(t\Id_d - \rd\phi_t^{\tr} \bigr) Y' + (X'')^{\tr} (C'')^{\tr} Y''. \notag
\end{align}
A straightforward calculation using the explicit descriptions of $X'$ and $Y'$ above yields that
\[
 (X')^{\tr} \bigl(t\Id_d - \rd\phi_t^{\tr} \bigr) Y'
 = \left( \begin{array}{c|c}
   \begin{matrix}
   (t-\theta)^{\ell_1} & & \\
   & \ddots & \\
   & & (t-\theta)^{\ell_m}
   \end{matrix}
   & \textup{\Large{$*$}} \\ \hline
   & \\
   \textup{\Large{$*$}} & \textup{\Large{$*$}} \\[10pt]
 \end{array}
 \right).
\]
On the other hand, as noted above for $1 \leqslant j \leqslant m$, the $j$-th columns of $X''$ and $Y''$ are divisible by $(t-\theta)^{\ell_j}$.  It follows that each entry in the $j$-th column of $(C'')^{\tr}Y''$ is divisible by $(t-\theta)^{\ell_j}$.  Likewise, we see that for $1 \leqslant k \leqslant m$, the $(k,j)$-entry of $(X'')^{\tr} (C'')^{\tr} Y''$ is divisible by $(t-\theta)^{\ell_k+\ell_j}$.  These two calculations together with the sum in~\eqref{E:VPhisum} complete the proof.
\end{proof}

\begin{proof}[Alternate proof of Proposition~\ref{P:Expgyha}]
Letting $\bsalpha$, $\bsy \in \KK^d$ be chosen so that $\Exp_{\phi}(\bsy) = \bsalpha$, we verify Proposition~\ref{P:Expgyha} directly from Lemma~\ref{L:VPhi} having first chosen bases $\bsm$ and $\bsn$ as in~\eqref{E:specialbases}.  By~\eqref{E:E0simplify}, we see that
\begin{equation} \label{E:E0simplifyagain}
\cE_0(\bsg_{\bsy} + \bsh_{\bsalpha}) = -\cE_0 \Bigl( \bigl( \rd \bsm \cdot (\rd\phi_t - t\Id_d)^{-1} \cdot \bsy \bigr)^{\tr} \cdot V^{(-1)}\Phi \Bigr).
\end{equation}
From the choice of basis $\bsm$ in~\eqref{E:specialbases}, we see that
\[
  \rd \bsm = \begin{pmatrix}
   \bss_1 \\ \bss_{d_1+1} \\ \vdots \\ \bss_{d_{m-1}+1} \\ \bzero \\ \vdots \\  \bzero
   \end{pmatrix} \in \Mat_{r \times d}(\FF_q),
\]
where $\bzero = (0, \dots, 0) \in \Mat_{1\times d}(\FF_q)$. Recalling the calculations of Remark~\ref{R:AGFprincipalpart}, if we write
\[
  \bsy = \bigl( y_{1,\ell_1}, \dots, y_{\ell_1,\ell_1}, y_{1,\ell_2}, \dots, y_{\ell_m,\ell_m} \bigr)^{\tr},
\]
then
\begin{multline*}
  \bigl(\rd \bsm \cdot ( \rd\phi_t - t\Id_d)^{-1} \cdot \bsy \bigr)^{\tr} \\
  = -\biggl( \frac{y_{1,\ell_1}}{t-\theta} + \cdots + \frac{y_{\ell_1,\ell_1}}{(t-\theta)^{\ell_1}},
  \ldots, \frac{y_{1,\ell_m}}{t-\theta} + \cdots + \frac{y_{\ell_m,\ell_m}}{(t-\theta)^{\ell_m}}, 0, \ldots, 0 \biggr).
\end{multline*}
We now combine this with \eqref{E:E0simplifyagain}, and using Proposition~\ref{P:epsiota} to calculate $\cE_0$, together with Lemma~\ref{L:VPhi}, we see that $\cE_0(\bsg_{\bsy} + \bsh_{\bsalpha}) = \bsy$ as desired.
\end{proof}

\begin{example} \label{Ex:Drinfeldmodule2}
\emph{Drinfeld modules.}  Here we recover identities for Drinfeld modules in~\cite[\S 3--4]{CP11}, \cite[\S 3--4]{CP12} (see also~\cite[Ex.~2.5.16]{HartlJuschka20} and~\cite[\S 4.2]{Pellarin08}).  We let $\phi : \bA \to \KK[\tau]$ be a Drinfeld module of rank~$r$, defined by $\phi_t = \theta + b_1 \tau + \cdots + b_r\tau^r$, with $b_r \neq 0$, and we continue with the notation in Example~\ref{Ex:Drinfeldmodule1}.  Moreover, we let $\bsm \assign (1, \tau, \ldots, \tau^{r-1})^{\tr} \in (\cM_{\phi})^r$ and $\bsn \assign (1, \sigma, \ldots, \sigma^{r-1})^{\tr} \in (\cN_{\phi})^r$ be $\KK[t]$-bases, for which multiplication by $\tau$ and $\sigma$ are represented by
\[
  \tPhi = \begin{pmatrix}
    0 & 1 & \cdots & 0 \\
    \vdots & \vdots & \ddots & \vdots \\
    0 & 0 & \cdots & 1 \\
    (t - \theta)/b_r & -b_1/b_r & \cdots & -b_{r-1}/b_r
    \end{pmatrix}
\]
and
\[
  \Phi = \begin{pmatrix}
    0 & 1 & \cdots & 0 \\
    \vdots & \vdots & \ddots & \vdots \\
    0 & 0 & \cdots & 1 \\
    (t - \theta)/b_r^{(-r)} & -b_1^{(-1)}/b_r^{(-r)} & \cdots & -b_{r-1}^{(-r+1)}/b_r^{(-r)}
    \end{pmatrix}
\]
respectively.  The matrices $X$, $Y \in \Mat_{r}(\KK[t])$ in \eqref{E:XYdef} are both simply the identity matrix $\Id_r$, and so by~\eqref{E:Bdef} and Corollary~\ref{C:VXBY}, we have
\begin{equation} \label{E:VDrinfeld}
  V = \begin{pmatrix}
  b_1 & b_2^{(-1)} & \cdots & b_{r-1}^{(-r+2)} & b_{r}^{(-r+1)} \\
  b_2 & b_3^{(-1)} & \cdots & b_{r}^{(-r+2)}& \\
  \vdots & \vdots & \reflectbox{$\ddots$} & & \\
  b_{r-1} & b_{r}^{(-1)} & & &\\
  b_{r} & & & &
  \end{pmatrix}
\end{equation}
and
\begin{equation}
  V^{(-1)}\Phi = \tPhi^{\tr}V = \begin{pmatrix}
  t - \theta & 0 & \cdots & 0 & 0 \\
  0 & b_2^{(-1)} & \cdots & b_{r-1}^{(-r+2)} & b_{r}^{(-r+1)} \\
  0 & b_3^{(-1)} & \cdots & b_{r}^{(-r+2)} & \\
  \vdots & \vdots & \reflectbox{$\ddots$} & & \\
  0 & b_{r}^{(-1)} & & &
  \end{pmatrix},
\end{equation}
which exhibits the form in Lemma~\ref{L:VPhi}.  One should compare the formula for $V$ with~\cite[p.~133]{CP12}, \cite[Eq.~(2.5.11)]{HartlJuschka20} and the one for $V^{(-1)}\Phi$ with~\cite[Eq.~(2.5.12)]{HartlJuschka20}.

We fix an $\bA$-basis $\lambda_1, \dots, \lambda_r$ of $\Lambda_\phi$.  Applying Proposition~\ref{P:PsiUpsilonV}, we see that we can pick a rigid analytic trivialization $\Psi \in \GL_{r}(\TT_{\theta})$ of $\phi$ so that
\begin{equation} \label{E:PsiinvDrinfeld}
  \Psi^{-1} = \Upsilon^{\tr} V = \begin{pmatrix}
  \cG_{\lambda_1}^{(1)} & \cdots & \cG_{\lambda_1}^{(r)} \\
  \vdots & & \vdots \\
  \cG_{\lambda_r}^{(1)} & \cdots & \cG_{\lambda_r}^{(r)}
  \end{pmatrix}
  \cdot V,
\end{equation}
which matches with~\cite[Eq.~(3.4.6)]{CP12}.

Let $y$, $\alpha \in \KK$ satisfy $\Exp_{\phi}(y) = \alpha$.  Then by Lemma~\ref{L:gyha} together with Proposition~\ref{P:AGFprops},
\begin{align} \label{E:gyDrinfeld}
\bsg_y &= -\bigl( \cG_y^{(1)}, \dots, \cG_y^{(r)} \bigr) \cdot V \\
 &= -\Bigl(
 (t-\theta)\cG_y + \alpha,
 b_2^{(-1)} \cG_y^{(1)} + \dots + b_r^{(-1)} \cG_y^{(r-1)},
 \ldots,
 b_r^{(-r+1)} \cG_y^{(1)} \Bigr). \notag
\end{align}
Using the formula for $\tPhi$ above with Lemma~\ref{L:gyha},
\begin{align}
\bsh_{\alpha} &= \left( \begin{pmatrix} 0 & \cdots & 0 \\ \vdots & & \vdots \\ 1/b_r & \cdots & 0 \end{pmatrix}
\begin{pmatrix} 1 \\ \vdots \\ \tau^{r-1} \end{pmatrix} \cdot \alpha \right)^{\tr}\cdot V \\
&= \biggl( 0, \dots, 0, \frac{\alpha}{b_r} \biggr) \cdot V \notag \\
&= (\alpha, 0, \dots, 0),\notag
\end{align}
and we note that these formulas for $\bsg_y$ and $\bsh_{\alpha}$ also agree with \cite[p.~136]{CP12}.  We verify directly from Example~\ref{Ex:Drinfeldmodule1}, Proposition~\ref{P:AGFprops}, and~\eqref{E:gyDrinfeld} that
\[
  \cE_0\bigl( \bsg_{y} + \bsh_{\alpha} \bigr) = y, \quad \cE_1 \bigl( \bsh_{\alpha} \bigr) = \cE_0 \bigl( \bsh_{\alpha} \bigr) = \alpha,
\]
which agrees with Proposition~\ref{P:Expgyha}.

As we know from Gekeler~\cite{Gekeler89a} (see also~\cite{Brownawell93}), we can find $\delta_0$, $\delta_1, \dots, \delta_{r-1} \in \Der(\phi)$ representing a $\KK$-basis of the de Rham module $\rH^1_{\DR}(\phi)$, where
\[
  \big( \delta_0\bigr)_t = \phi_t - \theta = b_1 \tau + \dots + b_r \tau^r,
\]
and for $1 \leqslant i \leqslant r-1$,
\[
  \big( \delta_i \bigr)_t = \tau^i.
\]
Thus $\Der_0(\phi) = \KK \delta_0$ and $\rH^1_{\sr}(\phi) = {\KK}[\delta_1]\oplus \cdots \oplus \KK [\delta_{r-1}] \subseteq \rH^1_{\DR}(\phi)$.  For $1 \leqslant j \leqslant r-1$, it is evident from Proposition~\ref{P:quasiloggy} that
\[
  \rF_{\delta_j}(y) = \cG_{y}^{(j)}\big|_{t=\theta}.
\]
Furthermore, $(\delta_0)_t = (b_1, \dots, b_r)\cdot \tau\bsm$, and by Propositions~\ref{P:AGFprops} and~\ref{P:quasiloggy},
\[
  \rF_{\delta_0}(y) = \bigl(b_1\cG_y^{(1)} + \cdots + b_r\cG_y^{(r)} \bigr)\big|_{t=\theta}
  = \bigl( (t-\theta)\cG_y + \alpha \bigr)\big|_{t=\theta} = \alpha - y,
\]
as expected by~\eqref{E:quasiinner}.  In this way, we can verify Corollary~\ref{C:perquasiperspans} and Theorem~\ref{T:quasispancomplete} by direct evaluation in the case of Drinfeld modules defined over~$\oK$.
\end{example}

\begin{example} \label{Ex:Carlitztensor2}
\emph{Carlitz tensor powers.}
For $n \geqslant 1$, we let $\phi : \bA \to \Mat_n(\KK[\tau])$ be the $n$-th tensor power of the Carlitz module and continue with the notation in Example~\ref{Ex:Carlitztensor1}.  The $t$-motive $\cM_{\phi}$ is rank~$1$ over $\KK[t]$ and is generated by $\bsm = \bss_1 = (1,0, \dots, 0) \in \Mat_{1\times n}(\KK[\tau])$, and likewise the dual $t$-motive is generated by $\bsn = \bss_n = (0,\dots, 0,1) \in \Mat_{1\times n}(\KK[\sigma])$.  It is readily apparent that
\[
  \Phi = \tPhi = (t-\theta)^n, \quad V = 1.
\]
Recall the functions $\Omega \in \TT_{\theta}^{\times}$ and $\omega \in \TT^{\times}$ from \eqref{E:Omegadef} and \eqref{E:omegadef} respectively.  It follows that $\Psi = (-1)^n\Omega^n$ is a rigid analytic trivialization for $\phi$ and that
\begin{equation} \label{E:Pidef}
  \Pi = \cE_0\bigl( (-1)^n\Omega^{-n} \bigr) = \begin{pmatrix}
  (-1)^n\pd_t^{n-1}(\Omega^{-n})|_{t=\theta} \\ \vdots \\
  (-1)^n \pd_t^1(\Omega^{-n})|_{t=\theta} \\ \tpi^n \end{pmatrix}
\end{equation}
generates $\Lambda_{\phi}$ as an $\bA$-module.  Letting
\[
  \cG_{\Pi} \rassign (\gamma_1, \dots, \gamma_n)^{\tr} \in \TT^d,
\]
it follows from Proposition~\ref{P:PsiUpsilonV} and the fact that $V=1$ that
\[
  \twistop{\tau\bss_1}{\cG_{\Pi}} = \gamma_1^{(1)} = \frac{(-1)^n}{\Omega^n} = (-1)^n \bigl(\omega^{(1)} \bigr)^n.
\]
From Proposition~\ref{P:AGFprops} we see that
\begin{equation} \label{E:GPifneq}
  t\cG_{\Pi} - \twistop{\phi_t}{\cG_{\Pi}} = \begin{pmatrix}
  (t-\theta)\gamma_1 -\gamma_2 \\ \vdots \\ (t-\theta) \gamma_{n-1} - \gamma_n \\ (t-\theta)\gamma_n - \gamma_1^{(1)} \end{pmatrix}
  = 0,
\end{equation}
and so for $1 \leqslant i \leqslant n$,
\[
  \gamma_i = (-1)^n (t-\theta)^i \omega^n.
\]
Moreover, these identities coincide with the ones in \cite[\S 2.5]{AndThak90}.

Taking $\bsy$, $\bsalpha \in \KK^d$ that satisfy $\Exp_{\phi}(\bsy) = \bsalpha$, we let $\cG_{\bsy} \rassign (g_1, \dots, g_n)^{\tr} \in \TT^d$, and we see from Lemma~\ref{L:gyha} that
\[
  \bsg_{\bsy} = -\twistop{\tau \bss_1}{\cG_{\bsy}} = -g_1^{(1)}.
\]
If $\bsalpha = (\alpha_1, \dots, \alpha_n)^{\tr}$, then since $\tPhi = (t-\theta)^n$, Lemma~\ref{L:gyha} implies that
\begin{align*}
  \bsh_{\bsalpha} &= \sum_{i=1}^n \sum_{j=0}^{i-1} (-1)^{n-i} \binom{n}{i} \theta^{n-i} t^{i-1-j} \bss_1 \cdot \phi_{t^j}(\bsalpha) \\
  &= \sum_{i=1}^n \sum_{j=0}^{i-1} (-1)^{n-i} \binom{n}{i} \theta^{n-i} t^{i-1-j} \sum_{k=0}^j \binom{j}{k} \theta^{j-k} \alpha_{k+1},
\end{align*}
where the final inner sum is by induction on the first coordinate of $\phi_{t^j}(\bsalpha)$.  By rearranging the sum and  using the binomial expansion for $(t-\theta)^n/(t-\theta)^{k+1} = (t-\theta)^{n-k-1}$, we find that
\[
  \bsh_{\bsalpha} = \alpha_n + (t-\theta)\alpha_{n-1} + \cdots + (t-\theta)^{n-1}\alpha_1.
\]
By the calculations in Example~\ref{Ex:Carlitztensor1}, we see that $\cE_1(\bsh_{\bsalpha}) = \bsalpha$ as predicted.  Furthermore, by Proposition~\ref{P:AGFprops} we have $t \cG_{\bsy} - \twistop{\phi_t}{\cG_{\bsy}} + \bsalpha = 0$, and similar to~\eqref{E:GPifneq} we have
\[
  \Bigl( (t-\theta)g_1 - g_2 + \alpha_1, \ldots, (t-\theta)g_{n-1} - g_n + \alpha_{n-1},(t-\theta)g_n - g_1^{(1)} + \alpha_n \Bigr)^{\tr} = 0.
\]
It follows that for $2 \leqslant i \leqslant n$,
\begin{equation} \label{E:giinterim}
  g_i = (t-\theta)^{i-1} g_1 + \alpha_{i-1} + (t-\theta)\alpha_{i-2} + \cdots + (t-\theta)^{i-2}\alpha_1,
\end{equation}
and moreover that
\begin{equation} \label{E:Ctensorgyha}
  \bsg_{\bsy} + \bsh_{\bsalpha} = -g_1^{(1)} + \alpha_n + (t-\theta)\alpha_{n-1} + \cdots + (t-\theta)^{n-1}\alpha_1 = -(t-\theta)^n g_1.
\end{equation}
By Example~\ref{Ex:Carlitztensor1},
\[
  \cE_0\bigl( (t-\theta)^n g_1 \bigr) = \Bigl( \pd_t^{n-1}\bigl((t-\theta)^n g_1\bigr), \dots, \pd_t^1\bigl( (t-\theta)^n g_1 \bigr), (t-\theta)^n g_1 \Bigr)^{\tr}\Big|_{t=\theta},
\]
and for $1 \leqslant i \leqslant n$, we see from \eqref{E:giprinparts} that
\begin{equation} \label{E:Ctensordiff}
  \pd_t^{n-i}\bigl( (t-\theta)^n g_1 \bigr)\big|_{t=\theta} = -y_i.
\end{equation}
It then follows that
\[
  \cE_0\bigl( \bsg_{\bsy} + \bsh_{\bsalpha} \bigr) = \bsy, \quad \cE_1\bigl( \bsh_{\bsalpha} \bigr) = \cE_0\bigl( \bsh_{\bsalpha} \bigr) = \bsalpha,
\]
where for the second equality we use that $\deg_t \bsh_{\bsalpha} \leqslant n$ in Example~\ref{Ex:Carlitztensor1}, and
which agrees with Proposition~\ref{P:Expgyha}.

Now $\rH_{\DR}^1(\phi) \cong \Der_0(\phi)$ is $1$-dimensional over $\KK$, and since $\rN^{\perp}(\phi) = \KK\bss_n$, it is generated by $\bsdelta \assign \bsdelta^{(\bss_n)}$ so that
\[
  \bsdelta_t = (\tau, 0, \dots, 0) = \tau \bss_1.
\]
In this way if $\Exp_{\phi}(\bsy) = \bsalpha$ as above, then by Proposition~\ref{P:quasiloggy}, \eqref{E:Ctensorgyha}, and~\eqref{E:Ctensordiff},
\[
  \rF_{\bsdelta}(\bsy)  = -\bsg_{\bsy}|_{t=\theta} = g_1^{(1)}\big|_{t=\theta} = \alpha_n - (t-\theta)^n g_1\big|_{t=\theta} = \alpha_n - y_n,
\]
as expected by~\eqref{E:quasiinner}.
\end{example}

\begin{example} \label{Ex:strictlypure2}
\emph{Strictly pure $t$-modules.}
We let $\phi : \bA \to \Mat_d(\KK[\tau])$ be a uniformizable strictly pure $t$-module defined by
\[
\phi_t = \rd \phi_t + B_1 \tau + \dots + B_{\ell}\tau^{\ell}, \quad \det(B_{\ell}) \neq 0,
\]
which has dimension $d$ and rank $r = \ell d$, and we continue with the notation of Example~\ref{Ex:strictlypure1}.  In many respects strictly pure $t$-modules behave like a block version of Drinfeld modules in Example~\ref{Ex:Drinfeldmodule2}, but with a few twists.  We let
\begin{align*}
\bsm &\assign (\tau^j \bss_i : 1 \leqslant i \leqslant d,\ 0 \leqslant j \leqslant \ell-1)^{\tr} \in (\cM_{\phi})^r, \\
\bsn &\assign (\sigma^j \bss_i : 1 \leqslant i \leqslant d, 0 \leqslant j \leqslant \ell-1)^{\tr} \in (\cN_{\phi})^r,
\end{align*}
be $\KK[t]$-bases of $\cM_{\phi}$ and $\cN_{\phi}$ respectively.  Multiplication by $\sigma$ on $\cN_{\phi}$ is represented by the matrix $\Phi \in \Mat_r(\KK[t])$ in \eqref{E:strictlypurePhi}, and multiplication by $\tau$ on $\cM_{\phi}$ is represented by
\begin{equation} \label{E:tPhistrictlypure}
\tPhi =  \begin{pmatrix}
    0 & \Id_d & \cdots & 0 \\
    \vdots & \vdots & \ddots & \vdots \\
    0 & 0 & \cdots & \Id_d \\
    B_{\ell}^{-1}(t\Id_d - \rd\phi_t) & -B_{\ell}^{-1} B_1 & \cdots & -B_{\ell}^{-1} B_{\ell-1}
    \end{pmatrix}.
\end{equation}
Since $X$ and $Y$ in~\eqref{E:XYdef} are by definition both the identity matrix, it follows from Corollary~\ref{C:VXBY} that
\begin{equation}
  V = \begin{pmatrix}
  B_1^{\tr} & \bigl( B_2^{(-1)} \bigr)^{\tr} & \cdots & \bigl( B_{\ell-1}^{(-\ell+2)} \bigr)^{\tr} & \bigl( B_{\ell}^{(-\ell+1)} \bigr)^{\tr} \\
  B_2^{\tr} & \bigl( B_3^{(-1)} \bigr)^{\tr} & \cdots & \bigl( B_{\ell}^{(-\ell+2)} \bigr)^{\tr} & \\
  \vdots & \vdots & \reflectbox{$\ddots$} & & \\
  B_{\ell-1}^{\tr} & \bigl( B_{\ell}^{(-1)} \bigr)^{\tr} & & &\\
  B_{\ell}^{\tr} & & & &
  \end{pmatrix}
\end{equation}
and
\begin{equation} \label{E:VPhistrictlypure}
  V^{(-1)}\Phi = \tPhi^{\tr}V = \begin{pmatrix}
  t\Id_d - \rd\phi_t^{\tr} & 0 & \cdots & 0 & 0 \\
  0 & \bigl( B_2^{(-1)} \bigr)^{\tr} & \cdots & \bigl( B_{\ell-1}^{(-\ell+2)} \bigr)^{\tr} & \bigl( B_{\ell}^{(-\ell+1)} \bigr)^{\tr} \\
  0 & \bigl( B_3^{(-1)} \bigr)^{\tr} & \cdots & \bigl( B_{\ell}^{(-\ell+2)} \bigr)^{\tr} & \\
  \vdots & \vdots & \reflectbox{$\ddots$} & & \\
  0 & \bigl( B_{\ell}^{(-1)} \bigr)^{\tr} & & &
  \end{pmatrix},
\end{equation}
both of which coincide with~\cite[Ex.~2.5.16(b)]{HartlJuschka20}.

Fixing an $\bA$-basis $\bslambda_1, \dots, \bslambda_r \in \Lambda_{\phi}$ also determines Anderson generating functions $\cG_{\bslambda_1}, \dots, \cG_{\bslambda_r} \in \TT^d$, and by Proposition~\ref{P:PsiUpsilonV} we have a rigid analytic trivialization $\Psi \in \GL_r(\TT_{\theta})$ of $\phi$ with
\begin{equation}
  \Psi^{-1} = \Upsilon^{\tr} V = \begin{pmatrix}
  \bigl(\cG_{\bslambda_1}^{(1)}\bigr)^{\tr} & \cdots & \bigl(\cG_{\bslambda_1}^{(\ell)}\bigr)^{\tr} \\
  \vdots & & \vdots \\
  \bigl(\cG_{\bslambda_r}^{(1)}\bigr)^{\tr} & \cdots & \bigl(\cG_{\bslambda_r}^{(\ell)}\bigr)^{\tr}
  \end{pmatrix}
  \cdot V,
\end{equation}
which is simply a block matrix version of~\eqref{E:PsiinvDrinfeld}.

Letting $\bsy$, $\bsalpha \in \KK^d$ satisfy $\Exp_{\phi}(\bsy) = \bsalpha$, it follows from Proposition~\ref{P:AGFprops} and Lemma~\ref{L:gyha} that
\begin{align} \label{E:gystrictlypure}
  \bsg_{\bsy} &= - \Bigl( \bigl( \cG_{\bsy}^{(1)} \bigr)^{\tr}, \dots, \bigl( \cG_{\bsy}^{(\ell)} \bigr)^{\tr} \Bigr) \cdot V  \\
  &= \begin{aligned}[t]
  -\Bigl( \bigl( (t\Id_d - \rd\phi_t)\cG_{\bsy} +\bsalpha \bigr)^{\tr}, \bigl( B_2^{(-1)}\cG_{\bsy}^{(1)} + \cdots + {} &B_{\ell}^{(-1)}\cG_{\bsy}^{(\ell-1)}\bigr)^{\tr},\\
   &\ldots, \bigl( B_{\ell}^{(-\ell+1)} \cG_{\bsy}^{(1)} \bigr)^{\tr} \Bigr).
  \end{aligned} \notag
\end{align}
Moreover, Lemma~\ref{L:gyha} combined with~\eqref{E:tPhistrictlypure} implies that
\begin{align} \label{E:strictlypureha}
  \bsh_{\bsalpha} &= \left( \begin{pmatrix} 0 & \cdots & 0 \\
  \vdots & & \vdots \\
  B_{\ell}^{-1} & \cdots & 0 \end{pmatrix}
  \begin{pmatrix} \bss \\ \vdots \\ \tau^{\ell-1}\bss \end{pmatrix}
  \cdot \bsalpha \right)^{\tr} \cdot V \\
  &= \Bigl( 0, \ldots, 0, B_{\ell}^{-1}\bsalpha \Bigr)^{\tr} \cdot V \notag \\
  &= \bigl( \bsalpha^{\tr}, 0, \dots 0 \bigr), \notag
\end{align}
and it is straightforward to verify, using \eqref{E:strictlypurePhi}, Proposition~\ref{P:AGFprops}, Lemma~\ref{L:gyha}, and \eqref{E:VPhistrictlypure}, that
\[
  \bsg_{\bsy} + \bsh_{\bsalpha} = \bsg_{\bsy}^{(-1)} \Phi.
\]
Now if $\cG_{\bsy} = (g_1, \dots, g_d)^{\tr} \in \TT^d$, then we see that
\begin{multline*}
  \bigl( (t\Id_d - \rd\phi_t)\cG_{\bsy} \bigr)^{\tr} \\ = \Bigl( (t-\theta)g_1 - g_2, (t-\theta)g_2 - g_3, \ldots, (t-\theta)g_{\ell_1}, (t-\theta)g_{\ell_1+1} - g_{\ell_1+2}, \ldots, (t-\theta) g_{\ell_m} \Bigr),
\end{multline*}
and applying~\eqref{E:E0beta}, it follows that
\begin{equation} \label{E:strictlypureE0gyha}
  \cE_0\bigl(\bsg_{\bsy} + \bsh_{\bsalpha} \bigr) = -\begin{pmatrix}
  \pd_{t,\ell_1} \bigl[ (t-\theta)^{\ell_1} g_1 \bigr] \big|_{t=\theta} \\
  \vdots \\
  \pd_{t,\ell_j} \bigl[ (t-\theta)^{\ell_j} g_{d_{j-1}+1} \bigr] \big|_{t=\theta} \\
  \vdots \\
  \pd_{t,\ell_m} \bigl[ (t-\theta)^{\ell_m} g_{d_{m-1}+ 1} \bigr] \big|_{t=\theta}
  \end{pmatrix},
\end{equation}
where we recall $d_j$ from \eqref{E:specialbases}. For each $1 \leqslant j \leqslant m$ and $0 \leqslant i \leqslant \ell_j-1$, it follows from \eqref{E:giprinparts} that
\begin{equation}
  \pd_t^{\ell_j-i} \bigl( (t-\theta)^{\ell_j} g_{d_{j-1}+1} \bigr)\big|_{t=\theta}
   = -y_{d_{j-1}+i}.
\end{equation}
Thus by~\eqref{E:strictlypureE0gyha},
\begin{equation} \label{E:strictlypurefinally}
  \cE_0 \bigl(\bsg_{\bsy} + \bsh_{\bsalpha} \bigr) = \bsy,
\end{equation}
as expected.  Similarly, using~\eqref{E:E0beta} and~\eqref{E:strictlypureha} and taking $\bsalpha \rassign (\alpha_1, \dots, \alpha_d)^{\tr}$,
\[
  \cE_1 \bigl( \bsh_{\bsalpha} \bigr) = \cE_0 \bigl( \bsh_{\bsalpha} \bigr)
  =\begin{pmatrix}
  {\displaystyle \pd_{t,\ell_1} \biggl[ \sum_{i=1}^{\ell_1} (t-\theta)^{\ell_1-i} \alpha_i \biggr]\bigg|_{t=\theta} } \\
  \vdots\\
  {\displaystyle \pd_{t,\ell_m} \biggl[ \sum_{i=1}^{\ell_m} (t-\theta)^{\ell_m-i} \alpha_{d_{m-1}+i} \biggr]\bigg|_{t=\theta} }
  \end{pmatrix}
  =\bsalpha.
\]

For quasi-periods and quasi-logarithms, we assume that $\phi$ is defined over $\oK$ for simplicity. If we consider the de Rham module $\rH_{\DR}^1(\phi,\oK) \cong \Der_0(\phi,\oK) \oplus \rH_{\sr}^1 (\phi,\oK)$, then
\[
  \Der_0(\phi,\oK) = \oK \bsdelta^{(\bss_{d_1})} \oplus \cdots \oplus \oK \bsdelta^{(\bss_{d_m})}
\]
and for $1 \leqslant j \leqslant m$,
\[
 \bigl( \bsdelta^{(\bss_{d_j})} \bigr)_t = \bss_{d_j} B_1 \tau + \cdots + \bss_{d_j} B_{\ell}\tau^\ell.
\]
We note that if we define $\bsdelta_{i,j} \in \Der(\phi,\oK)$ by
\[
  \bigl( \bsdelta_{i,j} \bigr)_t \assign \bss_i \tau^j, \quad 1 \leqslant i \leqslant d,\, 1 \leqslant j \leqslant \ell,
\]
then
\[
\rH_{\DR}^1(\phi,\oK)= \bigoplus_{\substack{1 \leqslant i \leqslant d \\ 1 \leqslant j \leqslant \ell}} \oK \bigl[\bsdelta_{i,j}\bigr].
\]
Moreover, if we define $\bsgamma_i \in \Der(\phi,\oK)$ so that
\[
  \bigl(\bsgamma_{i} \bigr)_t \assign \bss_i B_{\ell} \tau^{\ell}, \quad 1 \leqslant i \leqslant d,
\]
then taking $B_{\ell} = (b_{ik})$,
\[
  \bsgamma_i = \sum_{k=1}^d b_{ik} \bsdelta_{k,\ell},
\]
and in particular
\[
  \rH_{\sr}^1(\phi,\oK) = \bigoplus_{\substack{1 \leqslant i \leqslant d \\ 1 \leqslant j \leqslant \ell-1}} \oK \bigl[\bsdelta_{i,j}\bigr] \oplus \bigoplus_{i \notin \{ d_1, \dots, d_m \}} \oK \bigl[ \bsgamma_{i} \bigr].
\]
For $\bsy \in \KK^d$, $\bsalpha \in \oK^d$, and $\cG_{\bsy} = (g_1, \dots, g_d)^{\tr} \in \TT^d$ as above, it is clear from Proposition~\ref{P:quasigen} that for all $i$, $j$,
\begin{equation}
  \rF_{\bsdelta_{i,j}}(\bsy) = g_i^{(j)}\big|_{t=\theta}.
\end{equation}
By Proposition~\ref{P:AGFprops}, $(t\Id_d -\rd\phi_t) \cG_{\bsy} - B_1 \cG_{\bsy}^{(1)} - \cdots - B_{\ell} \cG_{\bsy}^{(\ell)} + \bsalpha = 0$,
and so for quasi-logarithms associated to $\bsgamma_1, \dots, \bsgamma_d$, this identity together with Proposition~\ref{P:quasigen} implies
\begin{align*}
  \begin{pmatrix} \rF_{\bsgamma_1}(\bsy) \\ \vdots \\ \rF_{\bsgamma_d}(\bsy) \end{pmatrix}
  &= B_{\ell} \cG_{\bsy}^{(\ell)}\big|_{t=\theta} \\
  &= \bigl( (t\Id_d - \rd\phi_t)\cG_y + \bsalpha - B_1 \cG_{\bsy}^{(1)} - \cdots - B_{\ell-1} \cG_{\bsy}^{(\ell-1)} \bigr)\big|_{t=\theta} \\
  &= \bigl( (t\Id_d - \rd\phi_t)\cG_y \bigr)|_{t=\theta} + \bsalpha - B_1
  \begin{pmatrix} \rF_{\bsdelta_{1,1}}(\bsy) \\ \vdots \\ \rF_{\bsdelta_{d,1}}(\bsy) \end{pmatrix}
  - \cdots - B_{\ell-1} \begin{pmatrix} \rF_{\bsdelta_{1,\ell-1}}(\bsy) \\ \vdots \\ \rF_{\bsdelta_{d,\ell-1}}(\bsy) \end{pmatrix}.
\end{align*}
From this we see that for $1 \leqslant i \leqslant d$,
\[
  \rF_{\bsgamma_{i}}(\bsy) \in \begin{cases}
  \Span{\oK} \bigl((t-\theta)g_i|_{t=\theta}, \rF_{\bsdelta_{1,1}}(\bsy), \dots, \rF_{\bsdelta_{d,\ell-1}}(\bsy) \bigr) & \textup{if $i \in \{ d_1, \dots, d_m\}$,} \\
  \Span{\oK} \bigl( ((t-\theta)g_i - g_{i+1})|_{t=\theta}, \rF_{\bsdelta_{1,1}}(\bsy), \dots, \rF_{\bsdelta_{d,\ell-1}}(\bsy) \bigr) & \textup{otherwise.}
  \end{cases}
\]
By considering the expression for $\bsg_{\bsy}$ in~\eqref{E:gystrictlypure}, we see that
\begin{align*}
  \Span{\oK} \bigl( &\{1\} \cup \{ \rF_{\bsdelta}(\bsy) : \ [\bsdelta] \in \rH_{\DR}^1(\phi)\} \bigr) \\
  &= \Span{\oK} \bigl( \{1\} \cup \{ \rF_{\bsdelta_{i,j}}(\bsy) : 1 \leqslant i \leqslant d,\, 1 \leqslant j \leqslant \ell-1 \} \cup \{ \rF_{\bsgamma_1}(\bsy), \dots, \rF_{\bsgamma_d}(\bsy) \} \bigr) \\
  &= \Span{\oK} \bigl( \{1 \} \cup \{ \bsg_{\bsy}|_{t=\theta}\} \bigr),
\end{align*}
which is the conclusion of Theorem~\ref{T:quasispancomplete} in this case.
Furthermore, in the case that $i=d_j$ for some $1\leqslant j \leqslant m$, we have that $(t-\theta)g_i|_{t=\theta} = -y_i$.  On the other hand, if $i \neq d_j$ for any $1\leqslant j \leqslant m$, then by~\eqref{E:giprinparts},
\[
  \bigl( (t-\theta)g_i - g_{i+1} \bigr)\big|_{t=\theta} = \Res_{t=\theta} \biggl( g_i - \frac{g_{i+1}}{t-\theta} \biggr).
\]
Thus for our uniformizable strictly pure $t$-module~$\phi$, if we let
\[
  y_i' \assign \begin{cases}
    -y_i & \textup{if $i \in \{ d_1, \dots, d_m \}$,} \\
    {\displaystyle -y_i - \Res_{t=\theta}\biggl( \frac{g_{i+1}}{t-\theta} \biggr)} & \textup{otherwise,}
  \end{cases}
\]
then
\begin{multline}
  \Span{\oK} \bigl( \{1\} \cup \{ \rF_{\bsdelta}(\bsy) : [\bsdelta] \in \rH_{\DR}^1(\phi) \} \bigr) \\
  = \Span{\oK} \bigl(1, y_1', \dots, y_d', \rF_{\bsdelta_{1,1}}(\bsy), \dots, \rF_{\bsdelta_{d,\ell-1}}(\bsy) \bigr).
\end{multline}
\end{example}

\begin{example} \label{Ex:ASP}
\emph{An almost strictly pure $t$-module.}
In this section we investigate an almost strictly pure $t$-module that is not strictly pure.  In general, the structure of such $t$-modules is more difficult to characterize explicitly.  Tensor powers of the Carlitz module provide well-behaved examples, but as indicated in Corollary~\ref{C:VXBY} and Lemma~\ref{L:VPhi}, there can be much variation depending on the presentation of the matrices $X$ and $Y$ in~\eqref{E:XYblocks}.

Let $\phi : \bA \to \Mat_2(\KK[\tau])$ be the $t$-module defined by
\begin{equation}
  \phi_t = \begin{pmatrix} \theta & 1 \\ 0 & \theta \end{pmatrix}
  + \begin{pmatrix} 1 & 0 \\ 1 & 1 \end{pmatrix} \tau
  + \begin{pmatrix} 0 & 1 \\ 0 & 0 \end{pmatrix} \tau^2.
\end{equation}
We calculate that
\begin{align}
  \phi_{t^2} &= \begin{pmatrix} \theta^2 & 2\theta \\ 0 & \theta^2 \end{pmatrix}
  + \begin{pmatrix} \theta^q + \theta + 1 & 2 \\ \theta^q + \theta & \theta^q + \theta + 1 \end{pmatrix} \tau
  + \begin{pmatrix} 1 & \theta^{q^2} + \theta \\ 2 & 1 \end{pmatrix}\tau^2
  + \begin{pmatrix} 1 & 2 \\ 0 & 1 \end{pmatrix} \tau^3, \\
  &\rassign A_0 + A_1 \tau + A_2 \tau^2 + A_3 \tau^3. \notag
\end{align}
and since $A_3$ is invertible, we see that the induced map $\phi : \FF_q[t^2] \to \Mat_2(\KK[\tau])$ is a strictly pure `$t^2$-module.'  As a $t^2$-module, it has rank $6=2\cdot 3$ (and dimension~$2$).  Thus
\[
  \rank_{\KK[t]}(\cM_{\phi}) = \rank_{\KK[t]}(\cN_{\phi}) = 3,
\]
and $\phi$ is pure of weight $\frac23$ in the sense of~\cite[\S 1.10]{And86} (and weight $-\frac23$ in the sense of \cite[Def.~2.5.27]{HartlJuschka20}).  We do not know whether or not $\phi$ is uniformizable.

As $\KK[t^2]$-modules, $\cM_{\phi}$ has the basis $\bsm' = (\bsm_1', \dots, \bsm_6')^{\tr}$ with
\begin{gather*}
  \bsm_1' = (1,0), \quad \bsm_2' = (0,1), \quad \bsm_3'  = (\tau,0), \\
  \bsm_4' = (0, \tau), \quad \bsm_5' = (\tau^2,0), \quad \bsm_6' = (0, \tau^2),
\end{gather*}
and likewise $\cN_{\phi}$ has the basis $\bsn' = (\bsn_1', \dots, \bsn_6')^{\tr}$ given by
\begin{gather*}
  \bsn_1' = (1,0), \quad \bsn_2' = (0,1), \quad \bsn_3'  = (\sigma,0), \\
  \bsn_4' = (0, \sigma), \quad \bsn_5' = (\sigma^2,0), \quad \bsn_6' = (0, \sigma^2).
\end{gather*}
One checks that $\{\bsm_1', \bsm_2',\tau\bsm_1'\}$ and $\{\bsn_1', \bsn_2',\sigma\bsn_2'\}$ form $\KK[t]$-bases of $\cM_{\phi}$ and $\cN_{\phi}$ respectively, but they are not in the form of~\eqref{E:specialbases} (in this example $m=1$ and $\ell_1=d=2$, but $\bsm_2' \notin \tau\cM_{\phi}$ and similarly for $\bsn$).  We let
\begin{gather*}
  \bsm_1 \assign \bsm_1' = (1,0),\quad \bsm_2 \assign -(t-\theta)\bsm_1' + \bsm_2' = (-\tau,-\tau^2), \quad \bsm_3 \assign \bsm_3' = (\tau, 0), \\
  \bsn_1 \assign \bsn_2' = (0,1), \quad \bsn_2 \assign -(t-\theta)\bsn_2' + \bsn_1' = (-\sigma^2,-\sigma), \quad \bsn_3 \assign \bsn_4' = (0,\sigma),
\end{gather*}
and we check that $\bsm = (\bsm_1, \bsm_2,\bsm_3)^{\tr}$ and $\bsn = (\bsn_1, \bsn_2, \bsn_3)^{\tr}$ form $\KK[t]$-bases of $\cM_{\phi}$ and $\cN_{\phi}$ respectively that conform to~\eqref{E:specialbases}.  With respect to these bases, we have $\tau\bsm = \tPhi \bsm$ and $\sigma \bsn = \Phi \bsn$, where
\[
  \tPhi = \begin{pmatrix}
  0 & 0 & 1 \\
  (t-\theta)^2 & t-\theta & -t+\theta^q-1 \\
  (t-\theta^q)(t-\theta)^2 & (t-\theta)(t-\theta^q)+1 & -t+\theta^q+1
  \end{pmatrix}
\]
and
\[
  \Phi = \begin{pmatrix}
  0 & 0 & 1 \\ (t-\theta)^2 & t-\theta & -t+\theta^{(-1)} - 1 \\
  (t-\theta)^2(t-\theta^{(-1)}) & (t-\theta)(t-\theta^{(-1)}) + 1 &
  -t+\theta^{(-1)}+1
  \end{pmatrix}.
\]
For the matrices $X$ and $Y$ in~\eqref{E:XYdef} and~\eqref{E:XYblocks}, we find
\[
  X = \begin{pmatrix}
  1 & 0 & 0 \\
  t-\theta & 1 & 0 \\
  0 & 0 & 1 \\
  (t-\theta)^2 & t-\theta & -1 \\
  (t-\theta)^2(t-\theta^q) & (t-\theta)(t-\theta^q) +1 & -t+\theta^q + 1 \\
  0 & -1 & -1
  \end{pmatrix}
\]
and
\[
  Y = \begin{pmatrix}
  t-\theta & 1 & 0 \\
  1 & 0 & 0 \\
  (t-\theta)^2 & t-\theta & -1 \\
  0 & 0 & 1\\
  0 & -1 & -1 \\
  (t-\theta)^2(t-\theta^{(-1)}) & (t-\theta)(t-\theta^{(-1)}) + 1 & -t + \theta^{(-1)} + 1
  \end{pmatrix}.
\]
The matrix $B^{\tr}$ from \eqref{E:Bdef} is
\[
  B^{\tr} = \begin{pmatrix}
  1 & 1 & 0 & 0 & 0 & 0 \\
  0 & 1 & 1 & 0 & 0 & 0 \\
  0 & 0 & 0 & 0 & 0 & 0 \\
  1 & 0 & 0 & 0 & 0 & 0 \\
  0 & 0 & 0 & 0 & 0 & 0 \\
  0 & 0 & 0 & 0 & 0 & 0
  \end{pmatrix},
\]
and so Corollary~\ref{C:VXBY} implies that
\begin{align*}
  V &= \bigl(X^{(1)}\bigr)^{\tr} B^{\tr} Y\\
    &= \begin{pmatrix}
  (2t-\theta^q-\theta)((t-\theta)(t-\theta^q)+1)+1 & (2t-\theta^q-\theta)(t-\theta^q)+1 & -(t-\theta^q) \\
  (2t-\theta^q-\theta)(t-\theta)+1 & 2t-\theta^q-\theta & -1 \\
  -(t-\theta) & -1 & 0
  \end{pmatrix}. \notag
\end{align*}
We should remark that in the previous examples considered in this section the matrix $V$ has entries in $\KK$, but examples like this one demonstrate the necessity for entries in $\KK[t]$.  Also the matrix~$C^{\tr}$ in \eqref{E:Cdef} is
\[
  C^{\tr} = \begin{pmatrix}
  t-\theta & 0 & 0 & 0 & 0 & 0 \\
  -1 & t-\theta & 0 & 0 & 0 & 0 \\
  0 & 0 & 0 & 0 & 0 & 0 \\
  0 & 0 & 1 & 0 & 0 & 0 \\
  0 & 0 & 0 & 0 & 0 & 0 \\
  0 & 0 & 0 & 0 & 0 & 0
  \end{pmatrix},
\]
and so from Corollary~\ref{C:VXBY}
\begin{align} \label{E:VPhi23}
  V^{(-1)}\Phi = \tPhi^{\tr} V &= X^{\tr}C^{\tr}Y \\
  &= \begin{pmatrix}
  (t-\theta)^2(1+(t-\theta)^2) & (t-\theta)^3 & -(t-\theta)^2 \\
  (t-\theta)^3 & (t-\theta)^2-1 & -(t-\theta) \\
  -(t-\theta)^2 & -(t-\theta) & 1
  \end{pmatrix}, \notag
\end{align}
which matches the expectations from Lemma~\ref{L:VPhi} and Remark~\ref{R:VPhi}.  If we take $\bsy = (y_1,y_2)^{\tr}$, $\bsalpha = (\alpha_1,\alpha_2)^{\tr} \in \KK^{2}$ with $\Exp_{\phi}(\bsy) = \bsalpha$, and we let $\cG_{\bsy} = (g_1,g_2)^{\tr} \in \TT^2$, then by Lemma~\ref{L:VPhi},
\begin{align*}
  \bsg_{\bsy} + \bsh_{\bsalpha} = \bsg_{\bsy}^{(-1)} \Phi &= - \twistop{\bsm}{\cG_{\bsy}}^{\tr} \cdot V^{(-1)}\Phi \\
  &= -\left\langle \begin{pmatrix} (1,0) \\ (-\tau,-\tau^2) \\ (\tau,0) \end{pmatrix} \Biggm| \begin{pmatrix} g_1 \\ g_2 \end{pmatrix} \right\rangle^{\tr} \cdot V^{(-1)}\Phi \\
  &= -\Bigl( g_1, -g_1^{(1)}-g_2^{(2)}, g_1^{(1)} \Bigr) \cdot V^{(-1)}\Phi.
\end{align*}
By Proposition~\ref{P:epsiota} and~\eqref{E:VPhi23}, we see
\[
  \cE_0 \bigl(\bsg_{\bsy} + \bsh_{\bsalpha} \bigr) = -\pd_{t,2}\Bigl[ g_1 (t-\theta)^2(1+(t-\theta)^2) - \bigl(g_1^{(1)}+g_2^{(2)} \bigr)(t-\theta)^3 - g_1^{(1)}(t-\theta)^2 \Bigr],
\]
but in evaluating these functions and their derivative at $t=\theta$, the extra powers of $t-\theta$ cause much of the expression to vanish.  Finally using \eqref{E:giprinparts} we have
\[
  g_1 = -\frac{y_2}{(t-\theta)^2} - \frac{y_1}{t-\theta} + O(1),
\]
and after a quick calculation the above expression simplifies as
\[
  \cE_0\bigl( \bsg_{\bsy} + \bsh_{\bsalpha} \bigr)  = \bsy
\]
as anticipated.
\end{example}

\newpage
\section{\for{toc}{Hyperderivatives of periods, quasi-periods, logarithms, and quasi-logarithms} \except{toc}{Hyperderivatives of periods, quasi-periods,\texorpdfstring{\\}{} logarithms, and quasi-logarithms}} \label{S:Hyper}

In this section we investigate hyperderivatives of periods, quasi-periods, logarithms and quasi-logarithms of uniformizable abelian and $\bA$-finite Anderson $t$-modules defined over $K_{\infty}^{\sep}$ or $K^{\sep}$, and we explore solutions to Question~\ref{Q:three} in \S\ref{S:Intro}. Hyperdifferential operators have become increasingly important tools for understanding the arithmetic of function fields (e.g., see \cite{AnglesNgoDacTavares20}, \cite{BosserPellarin08}, \cite{BosserPellarin09}, \cite{Brownawell96}--\cite{BrownawellDenis00}, \cite{ChangGreenMishiba21}, \cite{Conrad00}, \cite{Denis93}--\cite{Denis00}, \cite{Jeong00}, \cite{Jeong11}, \cite{Maurischat18}, \cite{Maurischat19a}, \cite{Maurischat21b}--\cite{N21}, \cite{P22}, \cite{PLogAlg}, \cite{PZ17}, \cite{PellarinPerkins18}, \cite{Petrov15}, \cite{US98}, \cite{Vincent15}). Given a Drinfeld module over $K_{\infty}^{\sep}$, Brownawell and Denis~\cite{Brownawell99}, \cite{BrownawellDenis00}, constructed $t$-modules whose exponential functions comprise hyperderivatives of the exponential function and quasi-periodic functions of the original $t$-module.  These investigations grew out of earlier constructions of Denis for the Carlitz module~\cite{Denis93}--\cite{Denis00}.  Most importantly, logarithms of the Brownawell-Denis $t$-modules are themselves hyperderivatives, with respect to $\theta$, of logarithms and quasi-logarithms of the underlying Drinfeld module.

However, the Brownwell-Denis $t$-modules are not abelian or $\bA$-finite, and neither are they quasi-periodic extensions of $\phi$.  Thus these $t$-modules somewhat mysteriously occupy a state in which the machinery of \S\ref{S:Exp}--\S\ref{S:Quasi} does not readily apply.

Ultimately our main tools will be the prolongations of $t$-modules defined by Maurischat~\cite{Maurischat18}.  It turns out that quasi-periodic extensions of Maurischat's prolongations have periods that are expressible in terms of hyperderivatives of logarithms and quasi-logarithms of the underlying $t$-module.  Among other properties that are well-suited to working with the hyperderivatives in question, prolongation $t$-modules of abelian and $\bA$-finite Anderson $t$-modules are themselves abelian and $\bA$-finite, and so the techniques of the previous sections will be a central part of our investigations.

In the next section we review the Brownawell-Denis theory for comparison to our subsequent results.  We will see that the information about logarithms and quasi-logarithms contained in Brown\-a\-well-Denis modules can be interpreted in terms of the quasi-periodic extensions of Maurischat's prolongations.  One additional consideration is that the Brown\-a\-well-Denis theory is inherently defined by taking derivatives \emph{with respect to~$\theta$}, whereas Maurischat's prolongations are defined through hyperderivatives \emph{with respect to~$t$}. One goal of our investigations is to clarify how these two sides are expressly linked (see Theorem~\ref{T:rathyperquasi0}, Corollary~\ref{C:Jordannontract}, and Theorems~\ref{T:hypernontract} and~\ref{T:hypernontractgen}).

\subsection{Hyperderivative \texorpdfstring{$t$}{t}-modules of Brownawell and Denis}\label{subS:BD}
We review the $t$-mod\-ules of Brownawell and Denis \cite{BrownawellDenis00} and their exponential functions, which are constructed from hyperderivatives of the exponential series and quasi-periodic functions of a Drinfeld module.  Throughout we fix a Drinfeld module $\phi: \bA \to K_{\infty}^\sep[\tau]$ of rank $r$ defined by
\[
\phi_t = \theta + \kappa_1 \tau + \dots + \kappa_r\tau^r,
\]
such that $\kappa_r \neq 0$, and we let $\exp_\phi: \KK \to \KK$ be its exponential function.

To simplify notation, for $g = \sum_i \alpha_i z^i \in \laurent{K_{\infty}^{\sep}}{z}$ and $j \geqslant 0$, we write
\begin{equation} \label{E:bracketdef}
  g^{[j]} \assign \pd_{\theta}^j(g) = \sum_i \pd_{\theta}^j(\alpha_i) z^i,
\end{equation}
where we consider $z$ to be an independent variable from $\theta$.  Similarly for $\nu = \sum_i \nu_i \tau^i \in K_{\infty}^{\sep}[\tau]$, we set $\nu^{[j]} = \sum_i \nu_{i}^{[j]} \tau^i$.  We extend this to all $j \in \ZZ$ by setting $g^{[j]}=0$ and $\nu^{[j]}=0$ for $j < 0$.

Brownawell's and Denis' main construction is the following.  Fixing $n \geqslant 1$, for $1 \leqslant i, k \leqslant n+1$, we set
\[
  \cD[i,k] \assign \sum_{j=0}^r \pd_{\theta}^{n-i+1 - q^j(n-k+1)}(\kappa_j) \tau^j \in K_{\infty}^{\sep}[\tau].
\]
We then define a $t$-module $\rho_n : \bA \to \Mat_{n+1}(K_{\infty}^{\sep}[\tau])$ by setting
\begin{equation}
(\rho_n)_t \assign \begin{pmatrix} \cD[1,1] & \cdots & \cD[1,n+1] \\ \vdots & & \vdots \\
  \cD[n+1,1] & \cdots & \cD[n+1,n+1] \end{pmatrix}.
\end{equation}
After some short calculations, we find that $(\rho_n)_t$ has the following upper-triangular form:
\begin{equation} \label{E:BrDerho}
(\rho_n)_t = \begin{pmatrix}
  \theta & 1 & \cdots & \cdots & \cdots & \phi_t^{[n]} \\
  & \theta & \ddots & & \sum_{j=0}^r \pd_{\theta}^{i-q^j k}(\kappa_j) \tau^j  & \vdots \\
  & &  \ddots & \ddots & & \vdots \\
  & & & \theta & 1 & \phi_t^{[2]} \\
  & & & & \theta & \phi_t^{[1]} \\
  & & & & & \phi_t
\end{pmatrix},
\end{equation}
where the general sum appears in row $n-i+1$ and column $n-k+1$.  We note that we have reordered the coordinates from Brownawell's and Denis' original construction so that $(\rho_n)_t$ is upper-triangular (cf.~\cite[\S 2]{BrownawellDenis00}).

\begin{remark}
From this definition, we see that $\rho_n$ is an iterated extension of $\phi$ by the trivial Drinfeld module $\Ga$.  However, since $\rd (\rho_n)_t$ is a full Jordan block of size $n+1$, we see that $\rho_n$ is not isomorphic to any quasi-periodic extension of $\phi$ as in~\eqref{E:QPext}.
\end{remark}

Taking $\exp_\phi(z) = \sum c_h z^{q^h} \in \power{K_{\infty}^{\sep}}{\tau}$, the following fundamental result provides the exponential function of $\rho_n$.

\begin{theorem}[{Brownawell-Denis~\cite[Cor.~2.2]{BrownawellDenis00}}]  \label{T:BrDe}
For $n \geqslant 1$, the unique exponential function for $\rho_n$ is given by
\[
\Exp_{\rho_n} (\bsz) = \begin{pmatrix}
  \exp_\phi^{[n]}(y) + \sum\limits_{q^h k \leqslant n} \pd_{\theta}^{n-q^h k}(c_h)z_k^{q^h}\\
  \vdots\\
   \exp_\phi^{[i]}(y) + \sum\limits_{q^h k \leqslant i} \pd_{\theta}^{i-q^h k}(c_h) z_k^{q^h}\\
  \vdots\\
  \exp_\phi^{[1]}(y) + z_1\\
   \exp_\phi(y)
    \end{pmatrix},
\]
where $\bsz = (z_n, \dots, z_1,y)^\tr$.
\end{theorem}

As a result of this theorem, we can quickly analyze the period lattice $\Lambda_{\rho_n}$ of $\rho_n$.  Indeed if $\bslambda = (\lambda_n, \dots, \lambda_1, \lambda)^{\tr} \in \Lambda_{\rho_n}$, then
\[
  \lambda \in \Lambda_{\phi}, \quad \lambda_1 = -\exp_{\phi}^{[1]}(\lambda),
\]
and more generally for $1 \leqslant i \leqslant n$,
\begin{equation} \label{E:lambdasum}
  \lambda_i = -\exp_{\phi}^{[i]}(\lambda) - \sum_{\substack{q^h k \leqslant i \\ k \neq i}}
  \pd_{\theta}^{i-q^h k}(c_h) \lambda_k^{q^h}.
\end{equation}
Thus each $\lambda_i$ is inductively and uniquely determined by $\lambda \in \Lambda_{\phi}$, and so $\Lambda_{\rho_n}$ has rank~$r$ as an $\bA$-module.  These calculations simplify greatly, and we find in the following corollary that in fact $\lambda_i = \pd_{\theta}^i(\lambda)$. Moreover, the above considerations can be extended to the following result on hyperderivatives of logarithms due to Brownawell and Denis.

\begin{corollary}[{Brownawell-Denis~\cite[Pf.~of Thm.~4.2]{BrownawellDenis00}}] \label{C:BrDelogs}
Let $\alpha$, $y \in K_{\infty}^{\sep}$ be chosen so that $\exp_{\phi}(y) = \alpha$.  Then for $n \geqslant 1$,
\[
  \Exp_{\rho_n} \begin{pmatrix} \pd_{\theta}^n(y) \\ \vdots \\ \pd_{\theta}^1(y) \\ y \end{pmatrix}
  = \begin{pmatrix} \pd_{\theta}^n(\alpha) \\ \vdots \\ \pd_{\theta}^1(\alpha) \\ \alpha \end{pmatrix}.
\]
In particular, the period lattice $\Lambda_{\rho_n}$ of $\rho_n$ satisfies
\[
  \Lambda_{\rho_n} = \left\{ \begin{pmatrix} \pd_{\theta}^n(\lambda) \\ \vdots \\ \pd_{\theta}^1(\lambda) \\ \lambda \end{pmatrix} \Biggm| \lambda \in \Lambda_{\rho} \right\}.
\]
\end{corollary}

\begin{proof}
Let $\bsy \assign (\pd_{\theta}^{n}(y), \dots, \pd_{\theta}^1(y),y)^{\tr}$ and $\bsalpha \assign (\alpha_n, \dots, \alpha_1, \alpha)^{\tr} \in (K_{\infty}^{\sep})^{n+1}$ be given so that $\Exp_{\rho_n}(\bsy) = \bsalpha$.  Then by Theorem~\ref{T:BrDe}, for $1 \leqslant i \leqslant n$,
\begin{align*}
  \alpha_i &= \exp_{\phi}^{[i]}(y) + \sum_{q^h k \leqslant i} \pd_{\theta}^{i-q^h k}(c_h) \bigl( \pd_{\theta}^k(y) \bigr)^{q^h} \\
  &= \exp_{\phi}^{[i]}(y) + \sum_{q^h k \leqslant i} \pd_{\theta}^{i-q^h k}(c_h) \pd_{\theta}^{q^h k} \bigl( y^{q^h} \bigr),
\end{align*}
where the second equality is an application of the $p$-th power rule for hyperderivatives (see Proposition~\ref{P:Hyperprops}(b)).  We note that in these sums, $k$ is implicitly at least~$1$, and so $1 \leqslant q^h k$ in all cases.  We choose $h_0$ so that $q^{h_0} \leqslant i < q^{h_0+1}$ and then continuing the calculation, while using again the $p$-th power rule, we have
\[
  \alpha_i = \exp_{\phi}^{[i]}(y) + \sum_{h=0}^{h_0} \Bigl( \pd_{\theta}^i \bigl( c_h y^{q^h} \bigr) - \pd_{\theta}^i(c_h) y^{q^h} \Bigr)
  = \pd_{\theta}^i \bigl( \exp_{\phi}(y) \bigr)
  = \pd_{\theta}^i(\alpha).
\]
This proves the first part of the corollary, and the second part follows from the uniqueness imparted by~\eqref{E:lambdasum} on the coefficients of the period in~$\Lambda_{\rho_n}$ whose last entry is a given $\lambda \in \Lambda_{\phi}$.
\end{proof}

Brownawell's and Denis' constructions further extend to hyperderivatives of quasi-periodic functions and quasi-logarithms.  For the remainder of the section, we fix $\phi$-biderivations $\delta_1, \dots, \delta_{\ell} \in \Der(\phi,K_{\infty}^{\sep})$, and again fixing $n \geqslant 1$, we assume that for each $j$, we have $\deg_{\tau} (\delta_j)_t > n$.  Brownawell and Denis~\cite[Thm.~2.3]{BrownawellDenis00} define a $t$-module $\trho_n : \bA \to \Mat_{(n+1)(\ell+1)}(K_{\infty}^{\sep}[\tau])$ by setting
\begin{equation} \label{E:BrDepsi}
 (\trho_n)_t \assign \left( \begin{array}{c|c}
   \begin{matrix}
   & & & & & \\
   d_{\theta,n+1}[\theta] & & & & & \\
   & & & & & \\
   & \ddots & & & & \\
   & & \ddots & & \\
   & & & \ddots & \\
   & & & & & \\
   & & & & & d_{\theta,n+1}[\theta] \\
   & & & & &
   \end{matrix}
   & \begin{array}{c|c}
   \textup{\Large{$0$\ }} &
     \begin{matrix}
       (\delta_{\ell})_t^{[n]} \\ \vdots \\ (\delta_{\ell})_t^{[1]} \\ (\delta_{\ell})_t \\
       \vdots \\ (\delta_{1})_t^{[n]} \\ \vdots \\ (\delta_{1})_t^{[1]} \\ (\delta_{1})_t
     \end{matrix}
   \end{array} \\ \hline
   & \\
   \textup{\Large{$0$}} & (\rho_n)_t \\[10pt]
   \end{array}
   \right).
\end{equation}
Similar to $\rho_n$, the $t$-module $\trho_n$ is not abelian or $\bA$-finite, and as an extension of $\rho_n$, it is not quasi-periodic.  However, its exponential map possesses particularly useful properties for hyperdifferentiation with respect to~$\theta$. For each $j$ with $1 \leqslant j \leqslant \ell$, we let $\rF_{j}(z) \assign \rF_{\delta_j}(z)$.

\begin{theorem}[{Brownawell-Denis~\cite[Thm.~2.3]{BrownawellDenis00}}] \label{T:BrDepsi}
For $n \geqslant 1$, the unique exponential function for $\trho_n$ is given by
\[
\Exp_{\trho_n} \begin{pmatrix} \bsx \\ \bsz \end{pmatrix} = \begin{pmatrix}
  \rF_{\ell}^{[n]}(y)+x_{\ell,n}\\
  \vdots\\
   \rF_{\ell}^{[1]}(y)+x_{\ell,1}\\
  \rF_{\ell}(y)+x_{\ell,0}\\
  \vdots\\
  \rF_1^{[n]}(y)+x_{1,n}\\
  \vdots\\
  \rF_1^{[1]}(y) + x_{1,1}\\
  \rF_1(y)+x_{1,0}\\[5pt] \hline \\[-10pt]
  \Exp_{\rho_n}(\bsz)
    \end{pmatrix}
\]
where $\bsx = (x_{\ell,n}, \dots, x_{\ell,0}, \dots, x_{1,n}, \dots, x_{1,0})^{\tr}$ and $\bsz = (z_n, \dots, z_1, y)^\tr$.
\end{theorem}

\begin{remark}
We note that the definition of $\trho_n$ in~\eqref{E:BrDepsi} is slightly different from what is given in~\cite[Thm.~2.3]{BrownawellDenis00}, where more of the entries of the column coming from biderivations are~$0$.  We believe that this is a simple typographical error, as later in the paper the authors reduce to the case that their biderivations do behave in the way that is presented there.  Especially one should compare with~\cite[p.~50--1]{Brownawell99}, where the construction aligns with what is in Theorem~\ref{T:BrDepsi}.
\end{remark}

As in Corollary~\ref{C:BrDelogs}, we obtain the following result on hyperderivatives of quasi-logarithms, also implicit in \cite[Thm.~4.2]{BrownawellDenis00}.

\begin{corollary}[{Brownawell-Denis \cite[Thm.~4.2]{BrownawellDenis00}}] \label{C:BrDequasilogs}
Let $\alpha$, $y \in K_{\infty}^{\sep}$ be chosen so that $\exp_{\phi}(y) = \alpha$.  Then for $n \geqslant 1$,
\[
  \Exp_{\trho_n} \begin{pmatrix}
  -\pd_{\theta}^{n}(\rF_{\ell}(y)) \\ \vdots \\ -\pd_{\theta}^1(\rF_{\ell}(y)) \\ -\rF_{\ell}(y) \\ \vdots \\
  -\pd_{\theta}^{n}(\rF_{1}(y)) \\ \vdots \\ -\pd_{\theta}^1(\rF_{1}(y)) \\ -\rF_{1}(y) \\
  \pd_{\theta}^{n}(y) \\ \vdots \\ \pd_{\theta}^{1}(y) \\ y
  \end{pmatrix}
  =
  \begin{pmatrix}
  0 \\ \vdots \\ \vdots \\ 0 \\ \pd_{\theta}^n(\alpha) \\ \vdots \\ \pd_{\theta^1}(\alpha) \\ \alpha
  \end{pmatrix}.
\]
\end{corollary}

\begin{proof}
The proof is essentially immediate from Corollary~\ref{C:BrDelogs}.  The key identity is that, by our choice that $\deg_{\tau} (\delta_j)_t > n$ for each~$j$, it follows that for $1 \leqslant i \leqslant n$ we have $\pd_{\theta}^i(\rF_{j}(y)) = \rF_j^{[i]}(y)$ for each~$j$.
\end{proof}

\subsection{Prolongations of \texorpdfstring{$t$}{t}-motives, dual \texorpdfstring{$t$}{t}-motives, and Anderson \texorpdfstring{$t$}{t}-modules} \label{subS:Prolongations}
In this section, we briefly review the prolongations of $t$-motives, dual $t$-motives, and Anderson $t$-modules due to Maurischat~\cite{Maurischat18}. We then apply the constructions of \S\ref{S:Exp}--\S\ref{S:Quasi} to these new $t$-modules to find connections with hyperderivatives of periods, logarithms, quasi-periods, and quasi-logarithms.

\begin{definition}\label{D:promotive}
For a left $\KK[t, \tau]$-module $\cM$ and $n \geqslant 0$, the \emph{$n$-th prolongation} of $\cM$ is the left $\KK[t, \tau]$-module $\rP_n\cM$ which is generated by the symbols $D_i \bsm$, for $i = 0$, $1, \dots, n$ and $\bsm \in \cM$, subject to the relations
\begin{enumerate}
\item[(a)] $D_i( \bsm_1+ \bsm_2) = D_i \bsm_1+D_i \bsm_2$,
\item[(b)] $D_i(a\cdot \bsm) = \sum_{i_1+i_2=i} \pd_t^{i_1}(a)\cdot D_{i_2}\bsm$,
\item[(c)] $\tau^{k}(a\cdot D_i \bsm)  =a^{(k)}D_i(\tau^{k}\bsm)$,
\end{enumerate}
for all $\bsm$, $\bsm_1$, $\bsm_2 \in \cM$ and $a\in \KK[t]$.
\end{definition}

\begin{definition} \label{D:dualtpro}
For a left $\KK[t, \sigma]$-module $\cN$ and $n \geqslant 0$, the \emph{$n$-th prolongation} of $\cN$ is the left $\KK[t, \sigma]$-module $\rP_n\cN$ which is generated by the symbols $D_i \bsh$, for $i = 0$, $1, \dots, n$ and $\bsh \in \cN$, subject to the relations
\begin{enumerate}
\item[(a)] $D_i(\bsh_1+ \bsh_2) = D_i \bsh_1+D_i \bsh_2$,
\item[(b)] $D_i(a\cdot \bsh) = \sum_{i_1+i_2=i} \pd_t^{i_1}(a)\cdot D_{i_2}\bsh$,
\item[(c)] $\sigma^{k}(a\cdot D_i\bsh)  =a^{(-k)}D_i(\sigma^{k}\bsh)$,
\end{enumerate}
for all $\bsh$, $\bsh_1$, $\bsh_2 \in \cN$ and $a\in \KK[t]$.
\end{definition}

Although we use the same notation `$D_i \bsm$' and `$D_i \bsh$' in both types of prolongations, the reader should be able to distinguish them by the context.  The primary result on prolongations of $t$-motives and dual $t$-motives is due to Maurischat.

\begin{theorem}[{Maurischat \cite[Thm.~3.4, Thm.~3.6, Prop.~4.2]{Maurischat18}}] \label{T:prolongprops} \hfill
\begin{enumerate}
\item[(a)] If $\cM$ is a $t$-motive, then the $n$-th prolongation $\rP_n \cM$ is also a $t$-motive for all $n \geqslant 0$. If $\cM$ is abelian, then so is $\rP_n \cM$. Moreover, if $\cM$ is rigid analytically trivial as a $t$-motive, in the sense of Anderson~\cite[\S 2.3]{And86}, then so is $\rP_n\cM$.
\item[(b)] If $\cN$ is an $\bA$-finite dual $t$-motive, then the $n$-th prolongation $\rP_n \cN$ is also an $\bA$-finite dual $t$-motive for all $n \geqslant 0$.
\item[(c)] If $\bsn = (\bsn_1, \dots, \bsn_r)^\tr$ is a $\KK[t]$-basis of $\cN$ and $\Phi\in \Mat_r(\KK[t])$ is the unique matrix such that $\sigma \bsn = \Phi \bsn$, then
\[
\bsD_n \bsn \assign (D_n\bsn^\tr, D_{n-1}\bsn^\tr, \dots, D_0\bsn^\tr)^\tr \in (\rP_{n}\cN)^{(n+1)r},
\]
with $D_i\bsn \assign (D_i\bsn_1, \dots, D_i\bsn_r)^\tr \in (\rP_{n} \cN)^{r}$ for each $i$, forms a $\KK[t]$-basis of $\rP_n \cN$, and
\[
\sigma \bsD_n \bsn = d_{t,n+1}[\Phi] \cdot \bsD_n \bsn.
\]
Moreover, if $\cN$ is rigid analytically trivial with rigid analytic trivialization $\Psi$, i.e., $\Psi^{(-1)} = \Phi \Psi$, then $\rP_n\cN$ is also rigid analytically trivial with rigid analytic trivialization $d_{t, n+1}[\Psi]$.
\end{enumerate}
\end{theorem}

\begin{remark}\label{R:basispromotive}
If $\bsm = (\bsm_1, \dots, \bsm_r)^\tr$ is a $\KK[t]$-basis of an abelian $t$-motive $\cM$, then
\begin{equation}\label{E:basisM2}
\tbsD_n \bsm \assign (D_0\bsm^\tr, D_{1}\bsm^\tr, \dots, D_n\bsm^\tr)^\tr \in \cM^{(n+1)r},
\end{equation}
where $D_i\bsm \assign (D_i\bsm_1, \dots, D_i\bsm_r)^\tr \in \cM^r$ for each $i$, forms a $\KK[t]$-basis of $\rP_n \cM$.  Moreover, if $\tau \bsm = \tPhi \bsm$, then
\begin{equation}
\tau \tbsD_n \bsm = \begin{pmatrix}
\tPhi & 0 & \cdots & 0 \\
\pd_t^1(\tPhi) & \tPhi & \ddots & \vdots \\
\vdots & \ddots & \ddots & 0\\
\pd_t^n(\tPhi) & \cdots & \pd_t^1(\tPhi) & \tPhi
\end{pmatrix} \cdot \tbsD_n \bsm.
\end{equation}
Similarly, if $\bss = (\bss_1, \dots, \bss_d)^\tr$ is a $\KK[\tau]$-basis of a $t$-motive $\cM$, then
\begin{equation}\label{E:basisM2std}
\tbsD_n \bss \assign (D_0\bss^\tr, D_{1}\bss^\tr, \dots, D_n\bss^\tr)^\tr \in (\rP_n \cM)^{(n+1)d},
\end{equation}
with $D_i\bss \assign (D_i\bss_1, \dots, D_i\bss_d)^\tr \in \cM^d$ for each $i$, forms a $\KK[\tau]$-basis of $\rP_n \cM$.
\end{remark}

\begin{definition}
Letting $\phi$ be an Anderson $t$-module of dimension $d$ and $\cM_\phi$ its corresponding $t$-motive, we then define the \emph{$n$-th prolongation} $\rP_n \phi$ of $\phi$ to be the Anderson $t$-module associated to the $n$-th prolongation $\rP_n \cM_{\phi}$ of $\cM_\phi$.
\end{definition}

\begin{theorem}[{Maurischat~\cite[Thm.~5.2]{Maurischat18}}] \label{T:prolmod}
Let $\phi: \bA \to \Mat_d(\KK[\tau])$ be an Anderson $t$-module of dimension $d$ such that
\[
\phi_t = \rd \phi_t + B_1\tau + \dots + B_\ell\tau^\ell.
\]
Then the $n$-th prolongation $\rP_n \phi$ of $\phi$ is of dimension $(n+1)d$ and is given by
\[
({\rP_n\phi})_t  = \begin{pmatrix}
    \rd \phi_t &  0    & \cdots    & 0\\
    -\Id_d & \ddots & \ddots&\vdots\\
    \vdots & \ddots&\ddots&0\\
    0& \cdots & -\Id_d & \rd \phi_t
    \end{pmatrix}
    +\diag(B_1)\tau + \dots + \diag(B_\ell) \tau^\ell,
\]
where for each $i$, $\diag(B_i) \assign \diag_{(n+1)d}(B_i, \dots, B_i)$.
\end{theorem}

\begin{proof}
For the convenience of the reader, we include a brief proof.  Let $\{\bss_1, \dots, \bss_d\} \subseteq \Mat_{1 \times d}(\KK[\tau])$ be a $\KK[\tau]$-basis of the $t$-motive $\cM_{\phi}$, and let $\bss \assign (\bss_1, \dots, \bss_d)^{\tr}$. Then by \eqref{E:tactionmotive} we have
\[
t\cdot \bss = (t\cdot \bss_1, \dots, t \cdot \bss_d)^{\tr} = (\bss_1 \phi_t, \dots, \bss_d \phi_t)^{\tr} = \phi_t \in \Mat_d(\KK[\tau]).
\]
Using Definition~\ref{D:promotive}(c), we identify $\rP_n\cM_{\phi}$ with $\Mat_{1 \times (n+1)d}(\KK[\tau])$, and by~\eqref{E:basisM2std} it follows that under this identification
\[
 \tbsD_n \bss = (D_0 \bss^{\tr}, \dots, D_n \bss^{\tr})^{\tr}, \quad D_i\bss^{\tr} = (D_i\bss_1, \dots, D_i\bss_d),
\]
comprise the standard basis vectors of $\Mat_{1 \times (n+1)d}(\KK[\tau])$.  By Definition~\ref{D:promotive}(b), for $1 \leqslant i \leqslant n$, $1 \leqslant j \leqslant d$, we have
\[
t \cdot D_0(\bss_j) = D_0(t\cdot \bss_j) = D_0(\bss_j \phi_t)
\]
and
\[
t\cdot D_i(\bss_j) = D_i(t \cdot \bss_j) - D_{i-1}(\bss_j) = D_i(\bss_j \phi_t) - D_{i-1}(\bss_j).
\]
Using Definition~\ref{D:promotive}(c) and writing in matrix form, we obtain
\[
t \cdot \tbsD_n \bss = \begin{pmatrix}
    \phi_t &  0    & \cdots    & 0\\
    -\Id_d & \ddots & \ddots&\vdots\\
    \vdots & \ddots&\ddots&0\\
    0& \cdots & -\Id_d &\phi_t
    \end{pmatrix} \in \Mat_{(n+1)d}(\KK[\tau]),
\]
from which the desired formula for $(\rP_n \phi)_t$ follows.
\end{proof}

\begin{remark} \label{R:strictlypro}
If $\phi$ is strictly pure (resp.\ almost strictly pure), then Theorem~\ref{T:prolmod} implies immediately that $\rP_n\phi$ is also strictly pure (resp.\ almost strictly pure).
\end{remark}

Given a $t$-module $\phi$, its $n$-th prolongation $\rP_n \phi$ is the Anderson $t$-module associated to the $t$-motive $\rP_n \cM_{\phi}$, as outlined in Theorem~\ref{T:prolmod}.  One might ask how $\rP_n\phi$ is related to the prolongation dual $t$-motive $\rP_n \cN_{\phi}$.  Luckily, as we see in the following proposition, these constructions are compatible and commute: that is, the (dual) $t$-motive of the prolongation of a $t$-module is the same as the prolongation of its associated (dual) $t$-motive.

\begin{proposition}\label{P:promandmofpro}
Let $\phi : \bA \to \Mat_d(\KK[\tau])$ be an Anderson $t$-module, and let $\cM_{\phi}$ and $\cN_{\phi}$ be its associated $t$-motive and dual $t$-motive.  For $n \geqslant 0$, the following hold.
\begin{enumerate}
\item[(a)] $\cM_{\rP_n \phi} = \rP_n \cM_{\phi}$.
\item[(b)] $\cN_{\rP_n \phi} = \rP_n \cN_{\phi}$.
\end{enumerate}
\end{proposition}

\begin{proof}
Part (a) is immediate from the construction of $\rP_n \phi$ and Theorem~\ref{T:prolmod}.  Part~(b) follows from the argument in Theorem~\ref{T:prolmod} in reverse order.  Indeed as in that proof, if we let $\bss \assign (\bss_1, \dots, \bss_d)^{\tr}$, with $\bss_j \in \Mat_{1\times d} (\KK[\sigma])$ the $j$-th standard basis vector, then
\begin{equation} \label{E:Dsymbols}
\bsD_n \bss = (D_n \bss^{\tr}, \dots, D_0 \bss^{\tr})^{\tr}
\end{equation}
can be identified with the standard basis vectors of $\rP_n \cN_{\phi} = \Mat_{1 \times (n+1)d}(\KK[\sigma])$.  Likewise for $1 \leqslant i \leqslant n$, $1 \leqslant j \leqslant d$ and by the definition of $\rP_n \cN_{\phi}$, we have
\[
  t\cdot D_i(\bss_j) = D_i(t\cdot \bss_j) - D_{i-1}(\bss_j) = D_i(\bss_j \phi_t^{*}) - D_{i-1}(\bss_j)
\]
and
\[
  t\cdot D_0(\bss_j) = D_0(t\cdot \bss_j) = D_0(\bss_j \phi_t^{*}).
\]
But then using Definition~\ref{D:dualtpro}(c) and writing in matrix form, we find
\[
  t\cdot \bsD_n \bss = \begin{pmatrix}
    \phi_t^* &  -\Id_d    & \cdots    & 0\\
    0 & \ddots & \ddots&\vdots\\
    \vdots & \ddots&\ddots& -\Id_d \\
    0& \cdots & 0 &\phi_t^*
    \end{pmatrix} \in \Mat_{(n+1)d}(\KK[\sigma]),
\]
and it follows that the $t$-module associated to $\rP_n \cN_{\phi}$ is $\rP_n \phi$.
\end{proof}

For an Anderson $t$-module $\phi$ and associated $t$-motive $\cM_{\phi}$, the $0$-th prolongation $P_0\cM_{\phi}$ is naturally isomorphic to $\cM_{\phi}$ via the map $D_0 \bsm \mapsto \bsm$.  Additionally, as is shown in \cite[Rem.~3.2]{Maurischat18}, for $0 \leqslant h \leqslant n$ we have an exact sequence of $t$-motives,
\begin{equation}\label{E:motivespro}
0 \to \rP_h\cM_{\phi} \xrightarrow{\bsi} \rP_{n}\cM_{\phi} \xrightarrow{\bspr} \rP_{n-h-1}\cM_{\phi} \to 0,
\end{equation}
where for $\bsm \in \cM_{\phi}$,
\[
\bspr(D_i \bsm) = \begin{cases}
  D_{i-h-1}\bsm & \textup{if $i> h$,} \\
  0 & \textup{if $i \leqslant h$.}
  \end{cases}
\]
In the context of the $\KK[t]$- and $\KK[\tau]$-bases of $\rP_{n} \cM_{\phi}$ in~\eqref{E:basisM2} and~\eqref{E:basisM2std}, we see that $\rP_h \cM_{\phi}$ injects as the first $h+1$ blocks of $\rP_n \cM_{\phi}$ and that $\bspr$ is the projection onto the last $n-h$ blocks of $\rP_n \cM_{\phi}$.  This induces an exact sequence of $t$-modules,
\begin{equation} \label{E:tmodprSES}
0 \to \rP_{n-h-1} \phi \xrightarrow{\bspr^*} \rP_n \phi \xrightarrow{\bsi^*} \rP_h \phi \to 0,
\end{equation}
and relative to the description of $\rP_n$ in Theorem~\ref{T:prolmod}, $\bspr^*$ injects $\rP_{n-h-1}\phi$ into the final $n-h$ blocks of $\rP_n \phi$ and $\bsi^*$ projects onto the first $h+1$ blocks of $\rP_n \phi$.

\begin{remark}
We now turn to the exponential function of $\rP_n \phi$, which turns out to be relatively straightforward. Indeed the exponential function for $\rP_n\phi$ is exactly the same as the exponential function for $\phi^{\oplus (n+1)}$, and so $\rP_n\phi$ and $\phi^{\oplus (n+1)}$ share the same period lattice $\Lambda_{\phi}^{n+1}$.  However, this does not violate the faithfulness of the functor $\phi \mapsto \Lambda_\phi$, as proved by Anderson~\cite[Cor.~2.12.2]{And86}, because the $\bA$-module structures on $\Lambda_{\rP_n\phi}$ and $\Lambda_{\phi}^{n+1}$ are different.
\end{remark}

\begin{proposition} \label{P:Exppro}
The unique exponential function of ${\rP_n\phi}$ is given by
\[
\Exp_{\rP_n\phi}(\bsz) = \begin{pmatrix}
\Exp_\phi(\bsz_0)\\
\vdots \\
\Exp_\phi(\bsz_n)
\end{pmatrix}
\]
where $\bsz = (\bsz_0^{\tr}, \dots, \bsz_n^{\tr})^\tr$ with each $\bsz_i \in \KK^d$. In particular, $\Lambda_{\rP_n\phi} = \Lambda_{\phi}^{n+1}$.
\end{proposition}

\begin{proof}
If we let $E(\bsz) \assign (\Exp_{\phi}(\bsz_0)^{\tr}, \dots, \Exp_{\phi}(\bsz_n)^{\tr})^{\tr}$, then it is a straightforward consequence of the functional equation $\Exp_{\phi}(\rd\phi_t z) = \phi_t \Exp_\phi(z)$ that
\[
E \bigl( \rd({\rP_n\phi})_t\bsz \bigr) = ({\rP_n\phi})_t E(\bsz).
\]
Thus $E(\bsz) = \Exp_{\rP_n \phi}(\bsz)$.
\end{proof}

For an $\bA$-finite $t$-module $\phi : \bA \to \Mat_d(\KK[\tau])$, we fix a $\KK[t]$-basis $\bsn = (\bsn_1, \dots, \bsn_r)^{\tr}$ of $\cN_{\phi}$ with corresponding $t$-frame $(\iota_{\phi},\Phi)$.  Then as in Theorem~\ref{T:prolongprops}(c),
\[
\bsD_n \bsn \assign (D_n\bsn^\tr, D_{n-1}\bsn^\tr, \dots, D_0\bsn^\tr)^\tr
\]
is a $\KK[t]$-basis of $\rP_n \cN_{\phi}$, and we have the associated $t$-frame $(\iota_{\rP_n \phi},d_{t,n+1}[\Phi])$, where $\iota_{\rP_n \phi} : \Mat_{1\times (n+1)r}(\KK[t]) \to \rP_n \cN_{\phi}$ is defined in the following way.  For $\bsalpha = (\bsalpha_0, \dots, \bsalpha_n)$, where each $\bsalpha_u = (\alpha_{u,0}, \dots, \alpha_{u,n}) \in \Mat_{1\times (n+1)}(\KK[t])$, we have
\begin{equation} \label{E:iotaPnphidef}
\iota_{\rP_n\phi}(\bsalpha) = \sum_{u=0}^n \sum_{i=1}^r \alpha_{u,i} \cdot D_{n-u} \bsn_i.
\end{equation}
If we identify $\rP_n \cN_{\phi}$ with $\Mat_{1 \times (n+1)d}(\KK[\sigma])$, then for each $\bsh \in \cN_{\phi} = \Mat_{1 \times d}(\KK[\sigma])$ and each $0 \leqslant j \leqslant n$,
\begin{equation}\label{E:dualofproandprodual}
  D_{j} (\bsh) = (0, \dots, 0, \bsh, 0, \dots, 0) \in \Mat_{1 \times (n+1)d}(\KK[\sigma]),
\end{equation}
where $\bsh$ appears in block $n-j+1$.  Using this identification, we can extend $\varepsilon_0 \circ \iota_{\rP_n \phi}$ to $\cE_{0,\rP_n} : \Mat_{1 \times (n+1)r}(\TT_{\theta}) \to \KK^{(n+1)d}$ as in \S\ref{subS:tframes}.  Moreover, we have the following relationship between $\cE_{0,\rP_n \phi}$ and $\cE_{0,\phi}$ for $\phi$.

\begin{proposition} \label{P:epsiotapro}
Let $\phi : \bA \to \Mat_d(\KK[\tau])$ be an $\bA$-finite $t$-module with $t$-frame $(\iota_{\phi}, \Phi)$. Then for $\bsalpha = (\bsalpha_0, \dots, \bsalpha_n) \in \Mat_{1\times (n+1)r}(\TT_{\theta})$, we have
\begin{equation}\label{E:e0pro}
  \cE_{0,\rP_n \phi} (\bsalpha) = \sum_{u=0}^{n} \sum_{v=0}^{n-u} D_{n-u-v} \Bigl( \cE_{0,\phi} \bigl( (-1)^{v}\pd_t^v(\bsalpha_u)\bigr)^{\tr} \Bigr)^\tr.
\end{equation}
\end{proposition}

\begin{remark}\label{R:sigmaK}
To ease notation, for the right-hand side of \eqref{E:e0pro} we use the canonical embeddings of $\Mat_{1\times d}(\KK)$ into $\cN_{\phi}=\Mat_{1\times d}(\KK[\sigma])$ and $\Mat_{1\times (n+1)d}(\KK)$ into $\rP_n \cN_{\phi}= \Mat_{1\times (n+1)d}(\KK[\sigma])$, which enables us to apply $D_j$ to elements in $\Mat_{1\times d}(\KK)$ for each $0 \leqslant j \leqslant n$ and obtain elements in $\Mat_{1\times (n+1)d}(\KK)$. 
\end{remark}

\begin{proof}[Proof of Proposition~\ref{P:epsiotapro}]
By Lemma~\ref{L:E0} it suffices to show that $\varepsilon_0(\iota_{\rP_n\phi}(\bsalpha))$ satisfies the desired identity when $\bsalpha \in \Mat_{1\times (n+1)r} (\KK[t])$.  By~\eqref{E:iotaPnphidef}, we have
\begin{equation} \label{E:iotaPnphistart}
\iota_{\rP_n\phi}(\bsalpha) = \sum_{i=1}^r(\alpha_{0,i}\cdot D_{n}\bsn_i + \alpha_{1,i}\cdot D_{n-1}\bsn_i + \dots + \alpha_{n,i}\cdot D_{0}\bsn_i).
\end{equation}
We note that Definition~\ref{D:dualtpro}(b) implies that for any $\alpha \in \KK[t]$, $\bsh \in \cN_{\phi}$, and $j \geqslant 0$,
\begin{equation} \label{E:Djrecursion}
  \alpha \cdot D_j (\bsh) = \sum_{v=0}^j (-1)^v D_{j-v}( \pd_t^{v}(\alpha) \cdot \bsh).
\end{equation}
Indeed, by Definition~\ref{D:dualtpro}(b),
\begin{align*}
  \alpha \cdot D_j(\bsh) &= D_j(\alpha \cdot \bsh) - \sum_{k=1}^j \pd_t^k(\alpha) D_{j-k}(\bsh) \\
  &= D_j(\alpha \cdot \bsh) - \sum_{k=1}^j \sum_{\ell=0}^{j-k} (-1)^{\ell} D_{j-k-\ell} \bigl(
    \pd_t^{\ell}(\pd_t^k(\alpha))\cdot \bsh \bigr), \quad \textup{(induction hypothesis)} \\
  &= D_j(\alpha \cdot \bsh) - \sum_{v=1}^j \sum_{k=1}^v (-1)^{v-k} \binom{v}{k} D_{j-v}( \pd_t^v(\alpha) \cdot \bsh),
  \quad \textup{($v \leftarrow k + \ell$)}
\end{align*}
and \eqref{E:Djrecursion} follows as the inner sum is $(-1)^{v+1}$.  Now for $u=0, \dots, n$, \eqref{E:Djrecursion} then implies
\begin{align}\label{E:iotaproe}
\sum_{i=1}^r \alpha_{u,i}\cdot D_{n-u}\bsn_i &= \sum_{i=1}^r \sum_{v=0}^{n-u} D_{n-u-v} \bigl((-1)^v\pd_t^v(\alpha_{u,i})\cdot \bsn_i \bigr)\\
&=\sum_{v=0}^{n-u} D_{n-u-v}\bigl(\iota_{\phi}((-1)^v\pd_t^v(\bsalpha_u)) \bigr).
\notag
\end{align}
By the embeddings in Remark~\ref{R:sigmaK}, we have $\varepsilon_0(D_j(\bsh)) = D_j(\varepsilon_0(\bsh)^{\tr})^{\tr}$ for each $\bsh \in \Mat_{1\times d}(\KK[\sigma])$. Combining this with~\eqref{E:iotaPnphistart} and \eqref{E:iotaproe}, we find
\[
  \varepsilon_0 \circ \iota_{\rP_n\phi}(\bsalpha) = \sum_{u=0}^{n} \sum_{v=0}^{n-u} D_{n-u-v} \Bigl( \varepsilon_{0}\circ \iota_{\phi} \bigl((-1)^v\pd_t^v(\bsalpha_u) \bigr)^{\tr} \Bigr)^{\tr},
\]
as desired.
\end{proof}

\begin{remark}\label{R:epsiotapro}
Using the calculations in the proof of Proposition~\ref{P:epsiotapro}, we obtain the following similar result for $\bsalpha = (\bsalpha_0, \dots, \bsalpha_n) \in \Mat_{1 \times (n+1)r}(\KK[t])$:
\[
  \cE_{1,\rP_n\phi}(\bsalpha) = \sum_{u=0}^{n} \sum_{v=0}^{n-u} D_{n-u-v}\Bigl( \cE_{1,\phi} \bigl( (-1)^v\pd_t^v(\bsalpha_u) \bigr)^{\tr} \Bigr)^{\tr},
\]
where the map $\cE_{1,\rP_n\phi}$ on the left-hand side corresponds to the one in \S\ref{subS:AndersonThm} for  $\rP_n\cN_{\phi}$, and the one on the right-hand side corresponds to the one for $\cN_{\phi}$.
\end{remark}

For $\bsm \in \Mat_{1\times d}(\power{\KK}{\tau})$ and $0 \leqslant u \leqslant n$, set
\begin{equation} \label{E:Dumotivenotation}
(\bsm)_{u} \assign (0, \dots, 0, \bsm, 0, \dots, 0) \in \Mat_{1\times (n+1)d}(\power{\KK}{\tau}),
\end{equation}
where $\bsm$ occupies entries $du+1$ through $du +d$ and all other entries are $0$. The following proposition exhibits an $L$-basis for the de Rham module $\rH^1_{\DR}(\rP_n \phi,L)$, which is recursively built out of lower level prolongations using~\eqref{E:motivespro}.

\begin{proposition} \label{P:deRhambasispro}
Let $\phi : \bA \to \Mat_d(L[\tau])$ be an abelian Anderson $t$-module of dimension $d$ and rank $r$ defined over a field $L$ with $K \subseteq L \subseteq \KK$, and suppose $\{ \bsdelta_1, \dots, \bsdelta_r\}$ represents an $L$-basis of $\rH_{\DR}^1(\phi, L)$.  For $0 \leqslant u \leqslant n$ and $1 \leqslant v \leqslant r$ define $\bsdelta_{u,v} \in \Der(\rP_n \phi,L)$ by setting
\[
  (\bsdelta_{u,v})_t \assign ((\bsdelta_v)_t)_u \in \Mat_{1\times (n+1)d}(L[\tau]\tau) = \rP_n \cM_{\phi,L} \tau.
\]
Then
\[
\bigcup_{u=0}^{n} \{\bsdelta_{u,1}, \dots, \bsdelta_{u,r}\}
\]
represents an $L$-basis of $\rH_{\DR}^1(\rP_n \phi, L)$.
\end{proposition}

\begin{proof}
We note from Proposition~\ref{P:biderprops} that
\[
\rH_{\DR}^1(\rP_n \phi, L) = \Der(\rP_n \phi, L)/\Der_{\si}(\rP_n \phi, L) \cong  \rP_n\cM_{\phi,L} \tau/(t-\theta) \rP_n\cM_{\phi,L} \tau,
\]
and so~\eqref{E:motivespro} induces an exact sequence on de Rham modules for $0 \leqslant h \leqslant n-1$,
\begin{equation} \label{E:DRSESpro}
0 \to \rH_{\DR}^1(\rP_h \phi,L) \xrightarrow{\bsi_*} \rH_{\DR}^1(\rP_n \phi,L) \xrightarrow{\bspr_*} \rH_{\DR}^1(\rP_{n-h-1}\phi, L) \to 0,
\end{equation}
where $\bspr_*$ is induced by
\[
(\bspr_* (\bsdelta))_t \assign \bspr(\bsdelta_t).
\]
It follows that an $L$-basis of $\rH_{\DR}^1(\rP_n \phi, L)$ is obtained from the union of an $L$-basis of $\bsi_*(\rH_{\DR}^1(\rP_h \phi, L))$ and a preimage of an $L$-basis of $\rH_{\DR}^1(\rP_{n-1-h} \phi, L)$ under $\bspr_*$.

We now proceed by induction on $n$.  When $n=0$, there is nothing to prove as $\rP_0 \cM_{\phi,L}$ is naturally isomorphic to $\cM_{\phi,L}$.  When $n=1$, we take $h=0$ in \eqref{E:DRSESpro}, and we find for $1 \leqslant v \leqslant r$,
\[
  \bsi_*(\bsdelta_v) = \bsdelta_{0,v}, \quad
  \bspr_*(\bsdelta_{1,v}) = \bsdelta_v.
\]
Thus, $\{\bsdelta_{0,v},\bsdelta_{1,v} \}_{v=1}^r$ consists of the union of an $L$-basis of $\bsi_*(H_{\DR}^1(\rP_0 \phi, L)$ and a preimage of an $L$-basis of $H_{\DR}^1(\rP_0\phi, L)$ under $\bspr_*$, and the result follows.

The general induction step is similar.  Suppose that the result is true for the $n-1$ case, and let $h=0$ in \eqref{E:DRSESpro}.  Then for $1 \leqslant u \leqslant n$, $1 \leqslant v \leqslant r$, the induction hypothesis implies
\[
  \bsi_*(\bsdelta_v) = \bsdelta_{0,v}, \quad \bspr_*(\bsdelta_{u,v}) = \bsdelta_{u-1,v}.
\]
Therefore, $\{ \bsdelta_{u,v} \}_{0 \leqslant u \leqslant n,\, 1 \leqslant v \leqslant r}$ represents an $L$-basis of $\rH^1_{\DR}(\rP_n\phi,L)$.
\end{proof}

\begin{remark} \label{R:innernotpreserved}
A word of caution is warranted here.  For $1 \leqslant u \leqslant n$, the $L$-linear map
\[
  [\bsdelta_v] \mapsto [\bsdelta_{u,v}] : \rH_{\DR}^1(\phi,L) \to \rH_{\DR}^1(\rP_n\phi,L)
\]
defined in Proposition~\ref{P:deRhambasispro} does not arise from well-defined $L$-linear maps $\Der_0(\phi,L) \to \Der_0(\rP_n\phi,L)$ or $\rH_{\sr}^1(\phi,L) \to \rH_{\sr}^1(\rP_n\phi,L)$.  Indeed even if $\bsdelta_v$ is an inner biderivation for $\phi$, then $\bsdelta_{u,v}$ need not be inner for~$\rP_n\phi$.  One reason is that this assignment is not induced by a $t$-module morphism.  When $u=0$, the map does preserve these subspaces, as it is induced by the morphism $\rP_n \phi \to \phi$ which projects onto the first block.
\end{remark}

\subsection{Hyperderivatives of logarithms and quasi-logarithms via prolongations} \label{subS:logquasilog}
In this section, we study how hyperderivatives of entries of periods, quasi-periods, logarithms and quasi-logarithms of abelian and $\bA$-finite Anderson $t$-modules can be obtained as the entries of the period matrix and other solutions of Frobenius difference equations attached to prolongations~\cite{Maurischat18}. Let $\phi : \bA \to \Mat_d(K^{\sep}[\tau])$ be an abelian and $\bA$-finite Anderson $t$-module of dimension $d$ and rank $r$ defined over $K^{\sep}$. We note that if $\bsy \in (K_{\infty}^{\sep})^d$, then $\Exp_{\phi}(\bsy) \in (K_{\infty}^{\sep})^d$, as the coefficients of $\Exp_{\phi}$ and the entries of $\bsy$ all lie in a finite separable extension of $K_{\infty}$, which is necessarily complete.  In particular, if we take the Anderson generating function $\cG_{\bsy} \in \TT^d$ as in~\eqref{E:AGFdef}, we have
\[
  \cG_{\bsy} \in \bigl(\power{K_{\infty}^{\sep}}{t}\bigr)^d \cap \TT^d,
\]
and thus for any $j \geqslant 0$, we can find hyperderivatives $\pd_{t}^j(\cG_{\bsy})$ and $\pd_{\theta}^{j}(\cG_{\bsy})$. Recall further from Lemma~\ref{L:separable} that if $\bsy \in \KK^d$ satisfies $\Exp_{\phi}(\bsy) \in (K_{\infty}^{\sep})^d$, then
\[
  \bsy \in (K_{\infty}^{\sep})^d, \quad \rF_{\bsdelta}(\bsy) \in K_{\infty}^{\sep},\ \forall\,\bsdelta \in \Der(\phi).
\]
We will make use of this throughout so that we can take hyperderivatives with respect to~$\theta$ of all of these quantities associated to~$\phi$.  The first main theorem of this section is the following.  It shows that the $\oK$-linear span of hyperderivatives \emph{with respect to $\theta$} of all quasi-logarithms of $\bsalpha$ associated to~$\bsy$, together with~$1$, is the same as the $\oK$-linear span of prescribed combinations of hyperderivatives \emph{with respect to $t$} of Anderson generating functions, evaluated at $t=\theta$, together with~$1$.

\begin{theorem} \label{T:rathyperquasi0}
Let $\phi$ be a uniformizable, abelian, and $\bA$-finite Anderson $t$-module defined over $K^\sep$ of rank $r$ and dimension $d$. Let $\bsy \in (K_{\infty}^{\sep})^d$ satisfy $\Exp_\phi(\bsy) \in (K^\sep)^d$, and let $\{\bsdelta_1, \dots, \bsdelta_r\}$ represent a $\oK$-basis of $\rH^1_{\DR}(\phi,\oK)$ defined over $K^\sep$.  For $n \geqslant 0$, the following hold.
\begin{enumerate}
\item[(a)] For the Anderson generating function $\cG_{\bsy} \in \TT^d$ for $\phi$ associated to $\bsy$, we have
\[
\Span{\oK} \biggl( \{1\} \cup \bigcup_{u=0}^n \bigcup_{v=1}^r \Bigl\{ \twistop{(\bsdelta_v)_t}{\pd_t^u(\cG_{\bsy})} \big|_{t=\theta} \Bigr\} \biggr)
= \Span{\oK} \biggl( \{1\} \cup \bigcup_{u=0}^n \bigcup_{v=1}^r \Bigl\{ \pd_\theta^u \bigl( \rF_{\bsdelta_v}(\bsy) \bigr) \Bigr\} \biggr).
\]
\item[(b)] Moreover, if $\{\bslambda_1, \dots, \bslambda_r\}$ is an $\bA$-basis of $\Lambda_\phi$, then
\[
\Span{\oK} \Bigl( d_{t, n+1}[\Psi]^{-1}\big|_{t=\theta} \Bigr)
= \Span{\oK} \biggl( \bigcup_{u=0}^n \bigcup_{v=1}^{r} \bigcup_{\ell=1}^r \Bigl\{ \pd_\theta^u \bigl(\rF_{\bsdelta_v}(\bslambda_\ell) \bigr) \Bigr\} \biggr),
\]
where $(\iota, \Phi,\Psi)$ is a rigid analytic trivialization for $\phi$.
\end{enumerate}
\end{theorem}

\begin{remark} \label{R:quasikbarspans}
As we see in Proposition~\ref{P:quasiperspans} and Theorem~\ref{T:quasispancomplete} (see Remarks~\ref{R:tractable} and~\ref{R:tractablequasi}), if $\rd\phi_t$ is in Jordan normal form, then the values in the above theorem account for all derivatives with respect to $\theta$ (up to the $n$-th derivative) of tractable coordinates of~$\bsy$ and strictly reduced quasi-logarithms associated to~$\bsy$.  The reader may wonder if it is possible to also account for the non-tractable coordinates of~$\bsy$.  This will be the subject of the second main theorem of this section (Theorem~\ref{T:hypernontract}), in the case of almost strictly pure $t$-modules, and for general $t$-modules in the next section (Theorem~\ref{T:hypernontractgen}).
\end{remark}

To prove Theorem~\ref{T:rathyperquasi0} we require some preliminary results.  The following lemma, inspired by a result of Brownawell and Denis~\cite[Lem.~5.1]{BrownawellDenis00}, shows that for any $s \in \NN$, every class in the de Rham module $\rH_{\DR}^1(\phi,K^{\sep})$ has a representative $\bsdelta^s$ such that $\deg_\tau(\bsdelta^s_t) \geqslant s$.

\begin{lemma}[{cf.~\cite[Lem.~5.1]{BrownawellDenis00}}] \label{L:degdeRham}
For every $\bsdelta \in \Der(\phi,K^{\sep})$ and $s \in \NN$, there is $\bsdelta^s \in \Der(\phi,K^{\sep})$ such that \textup{(a)} $\deg_\tau(\bsdelta^s_t) \geqslant s$, and \textup{(b)} $\bsdelta$ and $\bsdelta^s$ are in the same class in $\rH_{{\DR}}^1(\phi, K^{\sep})$.
\end{lemma}

\begin{proof}
Suppose we have $1 \leqslant \ell < s$ so that
\[
\bsdelta_t = \sum_{j \geqslant \ell}\bsb_{j} \tau^j = \sum_{j \geqslant \ell}(b_{j,1}, \dots, b_{j,d}) \tau^j \in \cM_{\phi, K^{\sep}} \tau, \quad \bsb_{\ell} \neq \mathbf{0}.
\]
If we let $\bsU \assign \bsb_{\ell}( (\rd \phi_t)^{(\ell)} - \theta\Id_d)^{-1} \tau^{\ell} \in \Mat_{1\times d}(K^{\sep})\tau^{\ell}$, then \eqref{E:innerdef} implies that $\bsdelta^{(\bsU)}$ is a strictly inner $\phi$-biderivation, and one verifies that
\[
\bsdelta^{(\bsU)}_t = \bsb_{\ell} \tau^{\ell} + \textup{higher degree terms in $\tau$.}
\]
Thus if we set $\bsdelta^{\ell+1} \assign  \bsdelta-\bsdelta^{(\bsU)}$, then
\[
\deg_\tau \bigl(\bsdelta^{\ell+1}_t \bigr) \geqslant \ell+1.
\]
As $\bsdelta^{(\bsU)} \in \Der_{\si}(\phi,K^{\sep})$, it follows that $\bsdelta^{\ell+1}$ and $\bsdelta$ are in the same class in $\rH^1_{\DR}(\phi,K^{\sep})$.  Proceeding by induction on~$\ell$ we are done.
\end{proof}

\begin{remark}\label{R:degreebidhypersep}
If we let $\bsdelta^{s} \in \Der(\phi,K^{\sep})$ be the biderivation obtained in the proof of Lemma~\ref{L:degdeRham}, then Proposition~\ref{P:biderprops}(b) implies that
\begin{equation} \label{E:deltasdecomp}
  \bsdelta_t^s = \bsdelta_t + \sum_{j=\ell}^{s-1} (t-\theta)\cdot \bsa_j \tau^j,
\end{equation}
for some $\bsa_j \in \Mat_{1 \times d}(K^{\sep})$. From the functional equation~\eqref{E:Fdeltafneq} for quasi-periodic functions, it is straightforward to verify that the first non-zero term of $\rF_{\bsdelta^s}(\bsz)$ has degree at least $q^{\deg_\tau(\bsdelta^{s}_{t})}$ in each of the variables $z_1, \dots, z_d$.  Since $\bsdelta_t^s-\bsdelta_t = \sum_{j=\ell}^{s-1} (t-\theta)\cdot \bsa_j\tau^j \in \Der_{\si}(\phi, K^{\sep})$, we see from~\eqref{E:quasiinner} that
\[
\rF_{\bsdelta^s-\bsdelta}(\bsz) = \rF_{\delta^{(\bsU)}}(\bsz) = \sum_{j=\ell}^{s-1}\bsa_j \bigl(\Exp_\phi(\bsz) \bigr)^{(j)}.
\]
Thus as vectors of power series in $\power{K^{\sep}}{z_1, \dots, z_d}^d$, we have
\begin{equation}\label{eq:F0}
\rF_{\bsdelta}(\bsz) = \rF_{\bsdelta^s}(\bsz) + \sum_{j=\ell}^{s-1} \bsa_j \bigl( \Exp_\phi(\bsz) \bigr)^{(j)},
\end{equation}
which will drive much of the proof of Theorem~\ref{T:rathyperquasi0}.
\end{remark}

We now fix $n \geqslant 0$, and consider the $n$-th prolongation $\rP_n\phi$ of $\phi$.  Similar to~\eqref{E:Dumotivenotation}, for $\bsx \in \KK^d$ and $0 \leqslant u \leqslant n$ we set
\[
(\bsx)_{u} \assign (0, \dots, 0, \bsx^{\tr}, 0, \dots, 0)^{\tr} = \begin{pmatrix}
  0\\
  \vdots\\
  0\\
  \bsx\\
  0\\
  \vdots\\
  0
\end{pmatrix}\in \KK^{(n+1)d},
\]
where $\bsx$ occupies entries $du+1$ through $du+d$ and all other entries are $0$.  Let $\{ \bslambda_1, \dots, \bslambda_r \}$ be an $\bA$-basis of the period lattice $\Lambda_{\phi}$.  By Proposition~\ref{P:Exppro} and the form of $\rd(\rP_n\phi)_t$ in Theorem~\ref{T:prolmod}, we see that an $\bA$-basis of $\Lambda_{\rP_n \phi}$ is
\begin{equation}\label{E:latticebasispro}
\bigl\{ (\bslambda_v)_u \mid 1 \leqslant v \leqslant r,\, 0 \leqslant u \leqslant n \bigr\}.
\end{equation}
By \eqref{E:AGFdef}, for $\bsy \in \KK^d$ the Anderson generating function for $\rP_n\phi$ with respect to $(\bsy)_u$ is
\begin{equation} \label{E:AGFpro1}
    \fG_{(\bsy)_u} = \sum_{w=0}^{\infty} \Exp_{n} \bigl( (\rd (\rP_n\phi)_t)^{-w-1} (\bsy)_u \bigr) t^w.
\end{equation}
Suppose that
\begin{equation}\label{E:Jordanpro1}
    \rJ = \begin{pmatrix}
d_{\theta,\ell_1}[\theta] & &  \\
 & \ddots & \\
& &  d_{\theta,\ell_m}[\theta] \end{pmatrix},
\end{equation}
where $\ell_1, \dots, \ell_m > 0$ and $\ell_1+ \dots+ \ell_m = d$, is the Jordan normal form of $\rd \phi_{t}$. We pick $\rQ\in \GL_d(K^{\sep})$ such that $\rd\phi_{t}=\rQ \rJ \rQ^{-1}$. Note that for $h \in \ZZ$  and $\ell \geqslant 1$ we have $d_{\theta, \ell}[\theta]^h = d_{\theta, \ell}[\theta^h]$ and thus, we obtain
\[\rd \phi_t^h = \rQ \rJ^h \rQ^{-1} = \rQ\begin{pmatrix}
d_{\theta,\ell_1}[\theta^h] & &  \\
 & \ddots & \\
& &  d_{\theta,\ell_m}[\theta^h] \end{pmatrix}\rQ^{-1}.
\]
We observe that in Theorem~\ref{T:prolmod}, the subdiagonal $d \times d$ blocks of $\rd(\rP_n\rho)_t$ are $-\Id_d = -\rQ\pd_\theta^1(\rJ)\rQ^{-1}$ and that $0=(-1)^c \rQ\pd_\theta^c(\rJ)\rQ^{-1}$ for $c \geqslant 2$, that is,
\[
\rd(\rP_n\phi)_t= \begin{pmatrix}
\rd\phi_t &&&&&\\
-\rQ\pd_\theta^1(\rJ)\rQ^{-1} & \rd\phi_t &&&&\\
(-1)^2\rQ\pd_\theta^2(\rJ)\rQ^{-1} & -\rQ\pd_\theta^1(\rJ)\rQ^{-1} & \rd\phi_t&&&\\
\vdots&&\ddots&\ddots&&\\
\vdots &&&\ddots&\ddots&\\
(-1)^n\rQ \pd_\theta^{n}(\rJ)\rQ^{-1} &\dots&\dots&\dots&-\rQ\pd_\theta^1(\rJ)\rQ^{-1}&\rd\phi_t
\end{pmatrix}.
\]
By the product rule of hyperderivatives we have for $0\leqslant u \leqslant n$,
\[
  \sum_{\substack{v_1, \dots, v_u \geqslant 0 \\ v_1 + \dots + v_u=u}}(\rQ\pd_{\theta}^{v_1}(\rJ^h)\rQ^{-1})(\rQ\pd_{\theta}^{v_2}(\rJ^h)\rQ^{-1})\cdots (\rQ\pd_{\theta}^{v_u}(\rJ^h)\rQ^{-1})  = \rQ\pd_{\theta}^{u}(\rJ^h)\rQ^{-1}.
\]
Thus for $h\in \ZZ$ we obtain
\[
(\rd(\rP_n\phi)_t)^{h}= \begin{pmatrix}
\rd\phi_t^{h} &&&&&\\
-\rQ\pd_\theta^1(\rJ^{h})\rQ^{-1} & \rd\phi_t^{h} &&&&\\
(-1)^2\rQ\pd_\theta^2(\rJ^{h})\rQ^{-1} & -\rQ\pd_\theta^1(\rJ^{h})\rQ^{-1} & \rd\phi_t^{h}&&&\\
\vdots&&\ddots&\ddots&&\\
\vdots &&&\ddots&\ddots&\\
(-1)^n\rQ \pd_\theta^{n}(\rJ^{h})\rQ^{-1} &\dots&\dots&\dots&-\rQ\pd_\theta^1(\rJ^{h})\rQ^{-1}&\rd\phi_t^{h}
\end{pmatrix}.
\]
Note that for $i$, $w \geqslant 0$, we have $(-1)^i\pd_\theta^i(\theta^{-w-1})=\binom{w+i}{i}\theta^{-w-1-i}$. Moreover, since $d_{\theta, \ell}[\theta]^h = d_{\theta, \ell}[\theta^h]$, a short calculation using \eqref{E:Jordanpro1} shows that
\[
(-1)^i\pd_\theta^i(\rJ^{-w-1}) = \binom{w+i}{i} \begin{pmatrix}
d_{\theta, \ell_1}[\theta^{-w-1-i}] &&\\
&\ddots&\\
&&d_{\theta, \ell_m}[\theta^{-w-1-i}]
\end{pmatrix} = \binom{w+i}{i}\rJ^{-w-1-i}.
\]
Then for $\bsy \in \KK^d$ we can use Proposition~\ref{P:Exppro} to see that the Anderson generating function of $\rP_n\phi$ associated to $(\bsy)_u$ is
\begin{align*}
    \fG_{(\bsy)_u} &= \begin{pmatrix}
    0\\
    \vdots\\
    0\\
    \sum\limits_{w=0}^{\infty}\Exp_{\phi} (\rd\phi_t^{-w-1}\bsy)t^w\\
    \sum\limits_{w=0}^{\infty}\binom{w+1}{1}\Exp_{\phi} (\rd\phi_t^{-w-1-1}\bsy)t^w\\
    \sum\limits_{w=0}^{\infty}\binom{w+2}{2}\Exp_{\phi} (\rd\phi_t^{-w-1-2}\bsy)t^w\\
    \vdots\\
    \sum\limits_{w=0}^{\infty}\binom{w+(n-u)}{n-u}\Exp_{\phi} (\rd\phi_t^{-w-1-(n-u)}\bsy)t^w
    \end{pmatrix}\\
    &= \begin{pmatrix}
    0\\
    \vdots\\
    0\\
    \sum\limits_{w=0}^{\infty}\Exp_{\phi} (\rd\phi_t^{-w-1}\bsy)t^w\\
    \sum\limits_{w=1}^{\infty}\Exp_{\phi} (\rd\phi_t^{-w-1}\bsy) \binom{w}{1}t^{w-1}\\
    \sum\limits_{w=2}^{\infty}\Exp_{\phi} (\rd\phi_t^{-w-1}\bsy \binom{w}{2})t^{w-2}\\
    \vdots\\
    \sum\limits_{w=n-u}^{\infty}\Exp_{\phi} (\rd\phi_t^{-w-1}\bsy) \binom{w}{n-u}t^{w-(n-u)}
    \end{pmatrix}.
\end{align*}
Thus we have proved the following proposition.

\begin{proposition} \label{P:qpfunctionpro}
For $\bsy \in \KK^d$, let $\cG_{\bsy} \in \TT^d$ be the Anderson generating function for $\phi$ with respect to $\bsy$.  Then for $0 \leqslant u \leqslant n$,
\[
\fG_{(\bsy)_u} = \begin{pmatrix}
0 \\ \vdots \\ 0 \\ \cG_{\bsy} \\ \pd_t^1(\cG_{\bsy}) \\ \vdots \\ \pd_t^{n-u}(\cG_{\bsy})
\end{pmatrix} \in \TT^{d(n+1)},
\]
where $\cG_{\bsy}$ occupies entries $du+1$ through $du+d$.
\end{proposition}

\begin{proof}[Proof of Theorem~\ref{T:rathyperquasi0}]
Let $\bsm= (\bsm_1, \dots, \bsm_r)^{\tr}$ be a $K^{\sep}[t]$-basis of $\cM_{\phi, K^{\sep}}$, which then extends to a $\oK[t]$-basis of $\cM_{\phi,\oK}$.  Letting $\{\bsdelta_1, \dots, \bsdelta_r\}$ represent a $\oK$-basis of $\rH^1_{\DR}(\phi, \oK)$ defined over $K^\sep$, Proposition~\ref{P:deRhambasispro} shows that
\[
    \bigcup_{u=0}^{n} \{\bsdelta_{u,1}, \dots, \bsdelta_{u,r} \} \hookrightarrow \rH^1_{\DR}(\rP_n \phi,\oK)
\]
represents a $\oK$-basis of $\rH^1_{\DR}(\rP_n \phi,\oK)$.  By assumption (and Lemma~\ref{L:separable} for that matter) $\bsy \in (K_\infty^\sep)^d$, and we set $(\bsy)_0 \assign  (\bsy^{\tr}, 0, \dots, 0)^\tr \in (K_{\infty}^\sep)^{(n+1)d}$. By Proposition~\ref{P:qpfunctionpro},
\begin{equation}\label{E:proquasiperfnc0}
\fG_{(\bsy)_0} = \bigl( \cG_{\bsy}^{\tr}, \pd_t^1(\cG_{\bsy})^{\tr}, \dots, \pd_t^n(\cG_{\bsy})^{\tr} \bigr)^\tr.
\end{equation}
Suppose that $n= wq^{s-1}+\alpha$ for some $s\in \NN$, $0\leqslant w < q$ and $0\leqslant \alpha <q^{s-1}$. That is, $s = \lfloor \log_q(n)\rfloor + 1$ if $n \geqslant 1$ and $s=1$ otherwise. Then for $1 \leqslant v \leqslant r$, let $\bsdelta_v^s$ be taken as in Lemma~\ref{L:degdeRham}. For $0 \leqslant u \leqslant n$ and $1 \leqslant v \leqslant r$ define $\bsdelta_{u,v}^{s} \in \Der(\rP_n \phi,\oK)$ by setting
\[
  (\bsdelta_{u,v}^{s})_t \assign ((\bsdelta_v^{s})_t)_u \in  \rP_n \cM_{\phi,\oK} \tau,
\]
as in Proposition~\ref{P:deRhambasispro}. Since $\{\bsdelta_{1}^{s}, \dots, \bsdelta_{r}^{s}\}$ represents a $\oK$-basis of $\rH_{\DR}^{1}(\phi, \oK)$, by Proposition~\ref{P:deRhambasispro} we obtain that $\bigcup_{u=0}^{n} \{\bsdelta_{u,1}^{s}, \dots, \bsdelta_{u,r}^{s}\}$ is a $\oK$-basis of $\rH_{\DR}^{1}(\rP_n\phi, \oK)$.
Then by Proposition~\ref{P:quasigen}(a) and Theorem~\ref{T:quasispancomplete} we see that
\begin{multline}\label{E:quasiAGFhyper}
\cW \assign \Span{\oK}\Bigl( \{1 \} \cup \Bigl\{ \twistop{\bsepsilon_t}{\fG_{(\bsy)_0}}\big|_{t=\theta} : \bsepsilon \in \Der(\rP_n\phi,\oK) \Bigr\} \Bigr)\\
=\Span{\oK} \biggl( \{1\}\cup \bigcup_{u=0}^n\bigcup_{v=1}^r \Bigl\{ \twistop{(\bsdelta_{u,v}^{s})_t}{\fG_{(\bsy)_0}}\big|_{t=\theta} \Bigr\} \biggr).
\end{multline}
By Proposition~\ref{P:deRhambasispro} and \eqref{E:proquasiperfnc0},
\begin{equation} \label{E:hypertbasisdert}
\twistop{(\bsdelta_{u,v}^{s})_t}{\fG_{(\bsy)_0}}= \twistop{(\bsdelta_{v}^{s})_t}{\pd_t^{u}(\cG_{\bsy})} = \pd_t^u \bigl( \twistop{(\bsdelta_v^s)_t}{\cG_y} \bigr),
\end{equation}
and therefore,
\begin{equation}\label{E:hypertbasisAGF}
\cW = \Span{\oK} \biggl( \{1 \} \cup \bigcup_{u=0}^n\bigcup_{v=1}^r \Bigl\{ \twistop{(\bsdelta_v^{s})_t}{\pd_t^u(\cG_{\bsy})}\big|_{t=\theta} \Bigr\} \biggr).
\end{equation}
We claim that
\begin{equation} \label{E:hypertbasistheta}
\cW = \Span{\oK} \biggl( \{1 \} \cup \bigcup_{u=0}^n \bigcup_{j \geqslant s} \Bigl\{ \pd_{\theta}^u\bigl( \cG_{\bsy}^{(j)} \big|_{t=\theta} \bigr) \Bigr\} \biggr).
\end{equation}
First, for $0 \leqslant u \leqslant n$ and $j \geqslant s$, Lemma~\ref{L:chainrule}(b) implies
\begin{align}\label{E:hyperthetat}
\pd_{\theta}^{u} \bigl( \cG_{\bsy}^{(j)}\big|_{t=\theta} \bigr)
&= \sum_{u_1+u_2=u} \pd_{\theta}^{u_1} \circ \pd_t^{u_2} \bigl(\cG_{\bsy}^{(j)}\bigr) \big|_{t=\theta} \\
&= \sum_{u_1+u_2=u} \pd_{\theta}^{u_1}\bigl( \pd_t^{u_2}(\cG_{\bsy})^{(j)} \bigr) \big|_{t=\theta}
= \pd_{t}^{u} \bigl(\cG_{\bsy} \bigr)^{(j)} \big|_{t=\theta}
= \pd_{t}^{u} \bigl(\cG_{\bsy}^{(j)} \bigr) \big|_{t=\theta}, \notag
\end{align}
where the second and the fourth equalities hold since Frobenius twisting commutes with differentiation with respect to $t$, and the third equality follows from Proposition~\ref{P:Hyperprops}(b) (where we need that $j \geqslant s$). For $1 \leqslant v \leqslant r$ fixed, if we suppose $(\bsdelta_v^s)_t = \sum_{j \geqslant s}\bsb_{j} \tau^j \in \cM_{\phi, \oK} \tau$, then
\begin{equation} \label{E:Fdeltavs}
\twistop{(\bsdelta_v^s)_t}{\cG_{\bsy}}|_{t=\theta} = \sum_{j\geqslant s}\bsb_{j}\cG_{\bsy}^{(j)}\big|_{t=\theta} = \rF_{\bsdelta_v^s}(\bsy),
\end{equation}
where the second equality (which we save for later in the proof) follows from Proposition~\ref{P:quasigen}(a).  Now for $0 \leqslant u \leqslant n$ and $1 \leqslant v \leqslant r$, it follows that
\begin{align*}
  \twistop{(\bsdelta_v^s)_t}{\pd_t^u(\cG_{\bsy})}\big|_{t=\theta}
  &= \pd_t^u \biggl( \biggl\langle \sum_{j \geqslant s} \bsb_j \tau^j \biggm| \cG_{\bsy} \biggr\rangle \biggr) \bigg|_{t=\theta} && \textup{(by \eqref{E:hypertbasisdert})}\\
  &= \sum_{j \geqslant s} \bsb_j \bigl( \pd_t^u \bigl( \cG_{\bsy}^{(j)} \bigr)\bigr)\big|_{t=\theta} &&  \\
  &= \sum_{j \geqslant s} \bsb_j \pd_{\theta}^u \bigl( \cG_{\bsy}^{(j)} \big|_{t=\theta} \bigr) && \textup{(by \eqref{E:hyperthetat}).}
\end{align*}
Thus by \eqref{E:hypertbasisAGF}, we see that the left-hand side of \eqref{E:hypertbasistheta} is contained in the right.  To show the reverse containment, for $0 \leqslant u \leqslant n$, $j \geqslant s$, and $1 \leqslant i \leqslant d$, we choose $\bsepsilon \in \Der(\rP_n \phi,\oK)$ so that $\bsepsilon_t = (\bss_i \tau^j)_u$, where as usual $\bss_i$ is the $i$-th standard basis vector on $\Mat_{1\times d}(\oK[\tau])$.  Then by~\eqref{E:proquasiperfnc0} and~\eqref{E:hyperthetat},
\[
  \twistop{\bsepsilon_t}{\fG_{(\bsy)_0}}\big|_{t=\theta} = \bss_i \cdot \pd_t^u\bigl( \cG_{\bsy}^{(j)} \bigr)\big|_{t=\theta} = \bss_i \cdot \pd_{\theta}^u \bigl( \cG_{\bsy}^{(j)}\big|_{t=\theta} \bigr),
\]
and hence \eqref{E:hypertbasistheta} holds.

With this in hand we further claim that
\begin{multline}\label{E:thetaandthypereq}
\cW
= \Span{\oK} \biggl( \{1 \} \cup \bigcup_{u=0}^n \bigcup_{v=1}^r \Bigl\{ \twistop{(\bsdelta_{v}^{s})_t}{\pd_t^u(\cG_{\bsy})}\big|_{t=\theta} \Bigr\} \biggr) \\
= \Span{\oK} \biggl( \{1 \} \cup \bigcup_{u=0}^n\bigcup_{v=1}^r \Bigl\{ \pd_\theta^u(\rF_{\bsdelta_{v}^{s}}(\bsy)) \Bigr\} \biggr).
\end{multline}
Indeed~\eqref{E:hypertbasistheta} and~\eqref{E:Fdeltavs} imply immediately that the right-hand side (of the second equality) is contained in the left.  For the opposite containment, fix $j \geqslant s$.  Given $1 \leqslant i \leqslant d$, let $\varpi_{i,j} \in \Der(\phi, \oK)$ be chosen so that $(\varpi_{i,j})_t =  \bss_i \tau^j$.  By Proposition~\ref{P:quasigen}(a), $\rF_{\varpi_{i,j}}(\bsy) = \twistop{(\varpi_{i,j})_t}{\cG_{\bsy}}|_{t=\theta} = \bss_i \cG_{\bsy}^{(j)}|_{t=\theta}$, and so \eqref{E:hyperthetat} implies
\begin{align} \label{E:derQuasiAGFstandard}
\pd_{\theta}^{u} \bigl(\cG_{\bsy}^{(j)}\big|_{t=\theta} \bigr) = \pd_{t}^u \bigl( \cG_{\bsy}^{(j)} \bigr) \big|_{t=\theta} &= \bigl( \twistop{(\varpi_{1,j})_t}{\pd_t^u(\cG_{\bsy})}, \dots, \twistop{(\varpi_{d,j})_t}{\pd_t^u(\cG_{\bsy})} \bigr)^{\tr} \big|_{t=\theta} \\
&= \bigl( \pd_{\theta}^u(\rF_{\varpi_{1,j}}(\bsy)), \dots, \pd_{\theta}^u(\rF_{\varpi_{d,j}}(\bsy)) \bigr)^{\tr}.
\notag
\end{align}
Thus~\eqref{E:thetaandthypereq} holds by~\eqref{E:hypertbasistheta}.

We now turn to our original biderivations $\bsdelta_1, \dots, \bsdelta_r$.  By the same arguments as used for~\eqref{E:quasiAGFhyper} and~\eqref{E:hypertbasisAGF}, we obtain
\begin{multline} \label{E:hypertbasisAGFdeg}
\cW  = \Span{\oK} \Bigl( \{1\} \cup \Bigl\{ \twistop{\bsepsilon_t}{\fG_{(\bsy)_0}}\big|_{t=\theta} : \bsepsilon \in \Der(\rP_n\phi,\oK) \Bigr\} \Bigr) \\
= \Span{\oK} \biggl(\{1 \} \cup \bigcup_{u=0}^n\bigcup_{v=1}^r \Bigl\{ \twistop{(\bsdelta_{v})_t}{\pd_t^u(\cG_{\bsy})}\big|_{t=\theta} \Bigr\} \biggr).
\end{multline}
By \eqref{eq:F0}, we have $\pd_{\theta}^{u}(\rF_{\bsdelta^s}(\bsy)) = \pd_{\theta}^{u}(\rF_{\bsdelta}(\bsy)) -\pd_{\theta}^{u}(\sum_{j=1}^{s-1}\bsa_j(\Exp_\phi(\bsy))^{(j)})$. Since $\Exp_{\phi}(\bsy) \in (K^{\sep})^d$ and since each $\bsa_j \in \Mat_{1\times d}(K^{\sep})$ by Remark~\ref{R:degreebidhypersep}, it follows that
\begin{align*}
\Span{\oK} \biggl( \{1\} \cup \bigcup_{u=0}^n\bigcup_{v=1}^r \Bigl\{ \pd_\theta^u(\rF_{\bsdelta_v}(\bsy)) \Bigr\} \biggr)
&= \Span{\oK} \biggl( \{1\} \cup \bigcup_{u=0}^n\bigcup_{v=1}^r \Bigl\{ \pd_\theta^u(\rF_{\bsdelta^s_v}(\bsy)) \Bigr\} \biggr) = \cW,
\end{align*}
where the last equality is simply \eqref{E:thetaandthypereq}.  Part~(a) of the theorem then follows from \eqref{E:hypertbasisAGFdeg}.

To prove part~(b), note that $\Psi_{\rP_n\phi} = d_{t,n+1}[\Psi]$ by Theorem~\ref{T:prolongprops}(c). Then by combining  Corollary~\ref{C:perquasiperspans}(b) with~\eqref{E:thetaandthypereq} and~\eqref{E:hypertbasisAGFdeg}, and taking $\bsy = \bslambda_{\ell}$ for $1 \leqslant \ell \leqslant r$,  we have
\[
\Span{\oK} \Bigl( d_{t, n+1}[\Psi]^{-1}\big|_{t=\theta} \Bigr) = \Span{\oK} \biggl(\bigcup_{u=0}^n \bigcup_{v=1}^{r} \bigcup_{\ell=1}^r \Bigl\{ \pd_\theta^u(\rF_{\bsdelta_v^s}(\bslambda_\ell)) \Bigr\} \biggr).
\]
By \eqref{eq:F0}, we see that $\rF_{\bsdelta_v^s}(\bslambda_{\ell}) = \rF_{\bsdelta_v}(\bslambda_{\ell})$, and part~(b) follows.
\end{proof}

\begin{remark}
In~\cite{BrownawellDenis00}, Brownawell and Denis encountered a similar situation in proving $\oK$-linear independence of hyperderviatives of quasi-periods and quasi-logarithms with respect to~$\theta$ for Drinfeld modules.  Our proof of Theorem~\ref{T:rathyperquasi0} was inspired by their arguments and discussion in~\cite[\S 5]{BrownawellDenis00}.
\end{remark}

The main principle behind Theorem~\ref{T:rathyperquasi0} is to use Corollary~\ref{C:perquasiperspans} and Theorem~\ref{T:quasispancomplete} to analyze $\oK$-linear spans of quasi-periods and quasi-logarithms in aggregate and to identify these spaces as the $\oK$-linear spans of both the values of hyperderivatives of Anderson generating functions $\cG_{\bsy}$ with respect to $t$ at $t=\theta$ as well as the hyperderivatives of qua\-si-log\-a\-rithms with respect to~$\theta$.  As mentioned in Remark~\ref{R:quasikbarspans}, this theorem only immediately accounts for hyperderivatives of tractable coordinates of a given period or logarithm.  It is a natural question to ask how non-tractable coordinates and their hyperderivatives fit into this framework.  After some preliminary results to reflect on the findings in \S\ref{subS:deRhampairing}--\S\ref{subS:Quasi}, we show in Theorem~\ref{T:hypernontract} that these coordinates and their hyperderivatives appear in the same spaces as in Theorem~\ref{T:rathyperquasi0}, though with perhaps an increased value of~$n$.

We continue with our previous situation with (i) a uniformizable, abelian, and $\bA$-finite $t$-module $\phi : \bA \to \Mat_d(K^{\sep}[\tau])$, (ii) $\oK[t]$-bases $\bsm \in (\cM_{\phi,\oK})^r$ of the $t$-motive $\cM_{\phi,\oK}$ of~$\phi$ and $\bsn \in (\cN_{\phi,\oK})^r$ of its dual $t$-motive, both of which are defined over $K^{\sep}$, and (iii) a rigid analytic trivialization $(\iota_{\phi},\Phi_{\phi},\Psi_{\phi})$ as in Definition~\ref{D:rat}.  We refer back to the notation of \S\ref{subS:Prolongations} regarding the $t$-motive $\rP_n \cM_{\phi,\oK}$ and dual $t$-motive $\rP_n \cN_{\phi,\oK}$ associated to the prolongation $\rP_{n} \phi$ via Proposition~\ref{P:promandmofpro}.  In particular we recall
\[
 \tbsD_n \bss = (D_0 \bss^{\tr}, \dots, D_n \bss^{\tr})^{\tr}, \quad D_i\bss^{\tr} = (D_i\bss_1, \dots, D_i\bss_d),
\]
from \eqref{E:basisM2std}, which comprise the standard basis vectors of $\rP_n \cM_{\phi,\oK}= \Mat_{1 \times (n+1)d}(\oK[\tau])$.  We recall from~\eqref{E:basisM2} that $\tbsD_n \bsm$ is a $\oK[t]$-basis of $\rP_n \cM_{\phi,\oK}$ defined over $K^{\sep}$, where
\[
\tbsD_n \bsm = (D_0\bsm^\tr, D_{1}\bsm^\tr, \dots, D_n\bsm^\tr)^\tr \in \bigl(\cM_{\phi, \oK} \bigr)^{(n+1)r}.
\]
Given an $\bA$-basis $\bslambda_1, \dots, \bslambda_r$ of $\Lambda_{\phi}$, recall from \eqref{E:latticebasispro} that
\[
(\bslambda_v)_u= \bigl( 0, \dots, 0, \bslambda_v^{\tr}, 0, \dots, 0\bigr)^{\tr}, \quad 0 \leqslant v \leqslant n,\ 1 \leqslant v \leqslant r,
\]
form an $\bA$-basis of $\Lambda_{\rP_n\phi}$. By Proposition~\ref{P:qpfunctionpro}, the Anderson generating function of $\rP_n \phi$ associated to $(\bslambda_v)_u$ is
\begin{equation}\label{qpfunctionproprd}
\fG_{(\bslambda_v)_u} = \bigl( 0, \dots, 0, \cG_{\bslambda_v}^{\tr}, \pd_t^1(\cG_{\bslambda_v})^{\tr}, \dots, \pd_t^{n-u}(\cG_{\bslambda_v})^{\tr} \bigr)^{\tr} \in \TT^{(n+1)d},
\end{equation}
where $\cG_{\bslambda_v}$ is the Anderson generating function for $\phi$ with respect to $\bslambda_v$, and $\cG_{\bslambda_v}^{\tr}$ occupies entries $du+1$ through $du+d$.

We now calculate the representing matrix $\Upsilon_{\rP_n \phi} \in \GL_r(\TT_{\theta})$ as in \eqref{E:Upsilondef} for the pairing $(\cdot\, ,\cdot ) : \Der(\rP_n \phi) \times \Lambda_{\rP_n \phi} \to \TT_\theta$ as in \eqref{E:predeRhampairing}. By using \eqref{qpfunctionproprd} and \eqref{E:Dumotivenotation}, a short calculation shows that for $0\leqslant i$, $u \leqslant n$ and $1 \leqslant v \leqslant r$, we have the equality in $\TT_{\theta}^r$:
\[
 \left\langle \begin{pmatrix}
  \tau D_i\bsm_1 \\ \vdots \\ \tau D_i\bsm_r
  \end{pmatrix} \Biggm|
  \fG_{(\bslambda_v)_u} \right\rangle = \begin{cases}
      \boldsymbol{0} & \textup{if $i < u$,} \\
      \left\langle \begin{pmatrix}
  \tau \bsm_1 \\ \vdots \\ \tau \bsm_r
  \end{pmatrix} \Biggm|
  \pd_t^{i-u}(\cG_{\bslambda_v}) \right\rangle & \textup{if $i \geqslant u$.}
   \end{cases}
\]
Using the fact that differentiation with respect to~$t$ and Frobenius twisting commute, it follows from~\eqref{E:Upsilondef} that
\[
 \left\langle \begin{pmatrix}
  \tau \bsm_1 \\ \vdots \\ \tau \bsm_r
  \end{pmatrix} \Biggm|
  \pd_t^{i-u}(\cG_{\bslambda_v}) \right\rangle = \begin{pmatrix}
  \pd_t^{i-u}(\twistop{\tau \bsm_1}{\cG_{\bslambda_1}}) & \cdots & \pd_t^{i-u}(\twistop{\tau \bsm_1}{\cG_{\bslambda_r}}) \\
  \vdots & & \vdots \\
  \pd_t^{i-u}(\twistop{\tau \bsm_r}{\cG_{\bslambda_1}}) & \cdots & \pd_t^{i-u}(\twistop{\tau \bsm_r}{\cG_{\bslambda_r}})
  \end{pmatrix} = \pd_t^{i-u}(\Upsilon_{\phi}).
\]
From this we obtain $\Upsilon_{\rP_n \phi}$ in terms of $\Upsilon_{\phi}$ by observing,
\begin{align} \label{E:Upsilonpro}
\Upsilon_{\rP_n \phi} & = \left\langle \tau \tbsD_n \bsm \Biggm|
  \bigl( \fG_{(\bslambda_1)_0}, \dots, \fG_{(\bslambda_r)_0},  \dots, \fG_{(\bslambda_1)_{n}}, \dots, \fG_{(\bslambda_r)_{n}} \bigr) \right\rangle \\
  &= \begin{pmatrix}
\Upsilon_{\phi} & 0 & \cdots & 0 \\
\pd_t^1(\Upsilon_{\phi}) & \Upsilon_{\phi} & \ddots & \vdots \\
\vdots & \ddots & \ddots & 0 \notag \\
\pd_t^n(\Upsilon_{\phi}) & \cdots & \pd_t^1(\Upsilon_{\phi}) & \Upsilon_{\phi}
\end{pmatrix}\\
&= d_{t,n+1}[\Upsilon_{\phi}^{\tr}]^{\tr}. \notag
\end{align}
Furthermore, we return to the topics of \S\ref{subS:RAT} and recall the matrix $V_{\phi} \in \GL_r(\oK[t])$ from \eqref{E:Vdef} giving the isomorphism $\cM_{\phi}^{\wedge} \to \cN_{\phi}$ in Theorem~\ref{T:isodual}.  Since $\Psi_{\rP_n \phi} = d_{t, n+1}[\Psi_{\phi}]$ by Theorem~\ref{T:prolongprops}(c), it follows from Proposition~\ref{P:PsiUpsilonV} and \eqref{E:Upsilonpro} that
\[
\Psi_{\rP_n \phi} = d_{t, n+1}[V]^{-1}d_{t,n+1}[\Upsilon_{\phi}^{\tr}]^{-1} = d_{t, n+1}[V]^{-1}(\Upsilon_{\rP_n \phi}^{\tr})^{-1}.
\]
Combining this calculation with Theorem~\ref{T:isodual} and Proposition~\ref{P:PsiUpsilonV}, we have obtained the following.

\begin{proposition} \label{P:UpsilonVpro}
Let $\phi : \bA \to \Mat_d(K^{\sep}[\tau])$ be a uniformizable, abelian, and $\bA$-finite $t$-module.  Continuing with the notation above, the following hold.
\begin{enumerate}
\item[(a)] $\Upsilon_{\rP_n \phi} = d_{t,n+1}[ \Upsilon_{\phi}^{\tr}]^{\tr}$.
\item[(b)] $V_{\rP_n \phi} = d_{t,n+1}[V_{\phi}]$.
\end{enumerate}
\end{proposition}

We can now investigate how the theory of prolongations fits into the solutions of associated Frobenius difference equations, as in Lemma~\ref{L:gyha}.

\begin{lemma} \label{L:gyhapro}
Let $\phi : \bA \to \Mat_d(K^{\sep}[\tau])$ be a uniformizable, abelian, and $\bA$-finite $t$-module. Let $\bsalpha$, $\bsy \in \KK^d$ be chosen so that $\Exp_{\phi}(\bsy) = \bsalpha$. Let $\bsg_{\bsy}$ and $\bsh_{\bsalpha}$ be chosen as in Lemma~\ref{L:gyha}. For $1 \leqslant \ell \leqslant n$,  we let $\fG_{(\bsy)_{\ell}}$ be the Anderson generating function of $\rP_n \phi$ associated to $(\bsy)_{\ell}\assign (0, \dots, 0, \bsy, 0, \dots, 0)^{\tr} \in \KK^{(n+1)d}$.  Then applying Lemma~\ref{L:gyha} to $\rP_n \phi$ with associated functions $\fg_{(\bsy)_{\ell}}$ and $\fh_{(\bsalpha)_{\ell}}$, we find
\begin{equation} \label{E:gapro}
\fg_{(\bsy)_{\ell}} = \bigl(0, \dots, 0, \bsg_{\bsy}, \pd_t^1(\bsg_{\bsy}), \dots, \pd_t^{n-\ell}(\bsg_{\bsy}) \bigr) \in \Mat_{1 \times (n+1)r}(\TT_{\theta}),
\end{equation}
and
\begin{equation}\label{E:hapro}
    \fh_{(\bsalpha)_{\ell}} = (0, \dots, 0, \bsh_{\bsalpha}, \pd_t^1(\bsh_{\bsalpha}), \dots, \pd_t^{n-\ell}(\bsh_{\bsalpha})) \in \Mat_{1 \times (n+1)r}(\KK[t]).
\end{equation}
\end{lemma}

\begin{proof}
By Proposition~\ref{P:qpfunctionpro}, we have
\[
\fG_{(\bsy)_{\ell}}= \bigl( 0, \dots, 0, \cG_{\bsy}^{\tr}, \pd_t^1(\cG_{\bsy})^{\tr}, \dots, \pd_t^{n-\ell}(\cG_{\bsy})^{\tr}\bigr)^{\tr}.
\]
Since $V_{\rP_n\phi}= d_{t,n+1}[V_{\phi}]$, it follows from this equation and Lemma~\ref{L:gyha} that
\begin{align*}
\fg_{(\bsy)_{\ell}} &= -\twistop{\tau \bsD_n\bsm}{\fG_{(\bsy)_{\ell}}}^{\tr}\cdot V_{\rP_n \phi} \\
&= -\bigl( 0, \dots, 0, \twistop{\tau \bsm}{\cG_{\bsy}}^\tr, \pd_t^1(\twistop{\tau \bsm}{\cG_{\bsy}})^\tr, \dots, \pd_t^{n-\ell}(\twistop{\tau \bsm}{\cG_{\bsy}}^\tr)\bigr) \cdot d_{t,n+1}[V_{\phi}].
\end{align*}
Thus, using the product rule for hyperderivatives, we obtain
\[
\sum_{j=0}^i\pd_t^{i-j} \bigl(\twistop{\tau \bsm}{\cG_{\bsy}} \bigr)^\tr\cdot \pd_t^{j}(V_{\phi}) = \pd_t^{i}(\twistop{\tau \bsm}{\cG_{\bsy}}^\tr\cdot V_{\phi}),
\]
from which \eqref{E:gapro} follows. Lemma~\ref{L:gyha} then implies that $\fg_{(\bsy)_{\ell}}^{(-1)} \Phi_{\rP_n\phi}-\fg_{(\bsy)_{\ell}} = \fh_{(\bsalpha)_{\ell}}$.  Since $\Phi_{\rP_n \phi}=d_{t,n+1}[\Phi_{\phi}]$, we obtain \eqref{E:hapro} by a similar calculation using the product rule.
\end{proof}
Combining Theorem~\ref{T:quasispancomplete} with \eqref{E:thetaandthypereq} and~\eqref{E:hypertbasisAGFdeg}, we obtain the following corollary.

\begin{corollary} \label{C:rathpqgyhapr0}
Continuing with the notations of Theorem~\ref{T:rathyperquasi0} and Lemma~\ref{L:gyhapro}, we have
\[
\Span{\oK} \bigl( \{1\} \cup \fg_{(\bsy)_0}\big|_{t=\theta} \bigr)= \Span{\oK} \biggl( \{1\} \cup \bigcup_{u=0}^n \bigcup_{v=1}^r \Bigl\{ \pd_\theta^u \bigl( \rF_{\bsdelta_v}(\bsy) \bigr) \Bigr\} \biggr).
\]
\end{corollary}

Recalling the identification of $\rP_n \cN_{\phi}$ with $\Mat_{1 \times (n+1)d}(\KK[\sigma])$ from \eqref{E:dualofproandprodual}, the following corollary is a consequence of Proposition~\ref{P:epsiotapro} and Lemma \ref{L:gyhapro}.  Recall also the notation $D_j(\bsh)$ from~\eqref{E:dualofproandprodual} and the embeddings from Remark~\ref{R:sigmaK}.

\begin{corollary}\label{C:prohyperE0gh}
Let $\phi : \bA \to \Mat_d(K^{\sep}[\tau])$ be a uniformizable, abelian, and $\bA$-finite $t$-module. Continuing with the notations of Proposition~\ref{P:epsiotapro} and  Lemma~\ref{L:gyhapro}, we have
\[
\cE_{0, \rP_n\phi} \bigl(\fg_{(\bsy)_{\ell}} + \fh_{(\bsalpha)_{\ell}} \bigr)
= D_{n-(\ell-1)} \bigl( \cE_{0, \phi}(\bsg_{\bsy}+\bsh_{\bsalpha}) \bigr)^{\tr}
= \bigl( 0, \dots, 0, \cE_{0, \phi}(\bsg_{\bsy}+\bsh_{\bsalpha}), 0, \dots, 0 \bigr)^{\tr},
\]
where $\cE_{0, \phi}(\bsg_{\bsy}+\bsh_{\bsalpha})$ occupies entries $d\ell+1$ through $d\ell+d$.
\end{corollary}

\begin{proof}
For any $\bsalpha = (\bsalpha_0, \dots, \bsalpha_n) \in \Mat_{1\times r(n+1)}(\TT_{\theta})$ as in Proposition~\ref{P:epsiotapro}, rearranging the order of summation we have
\begin{multline} \label{E:rearrangeE0pro}
\sum_{u=0}^{n} \sum_{v=0}^{n-u} D_{n-u-v} \bigl(\cE_{0, \phi}((-1)^{v}\pd_t^v(\bsalpha_u))^\tr \bigr)^\tr \\
= \sum_{u=0}^{n} D_{n-u}\biggl(\cE_{0, \phi} \biggl( \sum_{w=0}^{u}(-1)^{w}\pd_t^w(\bsalpha_{u-w}) \biggr)^\tr \biggr)^\tr.
\end{multline}
Note that by Lemma~\ref{L:gyhapro}, we have $\fg_{(\bsy)_u} + \fh_{(\bsalpha)_u} = (\bsalpha_0, \dots, \bsalpha_n)$, where $\bsalpha_0 = \dots = \bsalpha_{\ell-2}=0$ and for $\ell-1 \leqslant j \leqslant n$ we have $\bsalpha_{j}=\pd_t^{j-(\ell-1)}(\bsg_{\bsy} + \bsh_{\bsalpha})$. Thus for $0 \leqslant u \leqslant \ell-2$ we have
\begin{equation}
\label{E:alternatingsum0}
\sum_{w=0}^{u}(-1)^{w}\pd_t^w(\bsalpha_{u-w}) = 0,
\end{equation}
for $u=\ell-1$ we have
\begin{equation}
\label{E:alternatingsum1}
\sum_{w=0}^{\ell-1}(-1)^{w}\pd_t^w(\bsalpha_{u-w})= \bsalpha_{\ell-1}.
\end{equation}
For $\ell \leqslant u \leqslant n$ we have
\begin{align} \label{E:alternatingsum}
\sum_{w=0}^{u}(-1)^{w} \pd_t^w(\bsalpha_{u-w}) &= \sum_{w=0}^{u-(\ell-1)}(-1)^{w}\pd_t^w(\bsalpha_{u-w}) \\
&=\sum_{w=0}^{u-(\ell-1)}(-1)^{w}\pd_t^w \bigl(\pd_t^{u-w-(\ell-1)}(\bsg_{\bsy}+\bsh_{\bsalpha}) \bigr) \notag \\
&=\sum_{w=0}^{u-(\ell-1)}(-1)^{w}\binom{u-(\ell-1)}{w}\pd_t^{u-(\ell-1)}(\bsg_{\bsy}+\bsh_{\bsalpha}), \notag
\end{align}
where the first equality follows from $\bsalpha_0=\dots=\bsalpha_{\ell-2}=0$ and the third from the composition rule~\eqref{E:comprule}. Since alternating sums of binomial coefficients are $0$, we see that \eqref{E:alternatingsum} is equal to $0$. Combining \eqref{E:rearrangeE0pro}--\eqref{E:alternatingsum}, we obtain
\[
\sum_{u=\ell-1}^{n} D_{n-u} \biggl( \cE_0 \biggl( \sum_{w=0}^{u}(-1)^{w} \pd_t^w \bigl( \pd_t^{u-j-w} (\bsg_{\bsy}+\bsh_{\bsalpha}) \bigr) \biggr)^\tr \biggr)^\tr = D_{n-(\ell-1)} \bigl(\cE_{0, \phi}(\bsg_{\bsy}+\bsh_{\bsalpha}) \bigr).
\]
The result then follows from Proposition~\ref{P:epsiotapro} and \eqref{E:rearrangeE0pro}.
\end{proof}

\begin{remark}
By Lemma~\ref{L:gyhapro}, we have $\fh_{(\bsalpha)_{\ell}} = (\bsalpha_0, \dots, \bsalpha_n)$, where $\bsalpha_0 = \dots = \bsalpha_{\ell-2}=0$ and for $j \in \{\ell-1, \dots, n\}$ we have $\bsalpha_{j}=\pd_t^{j-(\ell-1)}(\bsh_{\bsalpha})$. Then, by using Remark~\ref{R:epsiotapro} and the calculations in the proof of Corollary~\ref{C:prohyperE0gh}, we obtain the following similar result:
\[
  \cE_{1,\rP_n\phi}(\fh_{(\bsalpha)_{\ell}}) = D_{n-(\ell-1)} \bigl( \cE_{1,\phi}(\bsh_{\bsalpha}) \bigr)^{\tr} = \bigl(0, \dots, 0, \cE_{1,\phi}(\bsh_{\bsalpha}), 0, \dots, 0 \bigr)^{\tr},
\]
where the map $\cE_{1,\rP_n\phi}$ on the left-hand side corresponds to the one in \S\ref{subS:AndersonThm} for  $\rP_n\cN_{\phi}$ and $\cE_{1, \phi}$ in the first and second equality corresponds to the one for $\cN_{\phi}$, and in the second equality $\cE_{1, \phi}(\bsh_{\bsalpha})$ occupies entries $d\ell+1$ through $d\ell+d$.
\end{remark}

As mentioned in Remark~\ref{R:quasikbarspans}, the following theorem is a companion to Theorem~\ref{T:rathyperquasi0} to account for non-tractable coordinates of periods and logarithms, as well as their hyperderivatives with respect to~$\theta$.  As such, we should note that the right-hand spaces in both parts~(a) and~(b) are the same as those in Theorem~\ref{T:rathyperquasi0}, though we must choose~$n$ appropriately.  When $\rd\phi_t$ is in Jordan normal form, the proof shows that we can be more precise, and this is stated as Corollary~\ref{C:Jordannontract}.

\begin{theorem} \label{T:hypernontract}
Let $\phi$ be a uniformizable, abelian, and $\bA$-finite Anderson $t$-module defined over $K^\sep$ of rank $r$ and dimension $d$. Let $\bsy \in (K_{\infty}^{\sep})^d$ satisfy $\Exp_\phi(\bsy) \in (K^\sep)^d$, and let $\{\bsdelta_1, \dots, \bsdelta_r\}$ represent a $\oK$-basis of $\rH^1_{\DR}(\phi,\oK)$ defined over $K^\sep$.  Let $j \geqslant 0$, and choose $n \geqslant 0$ so that $(\rd\phi_t - \theta\Id_d)^{n-j}=0$.
\begin{enumerate}
\item[(a)] If $\phi$ is almost strictly pure, then
\[
  \Span{\oK} \bigl( \pd_{\theta}^{j} (\bsy) \bigr) \subseteq
  \Span{\oK} \biggl( \{1\} \cup \bigcup_{u=0}^n \bigcup_{v=1}^r \Bigl\{ \pd_\theta^u \bigl( \rF_{\bsdelta_v}(\bsy) \bigr) \Bigr\} \biggr).
\]
\item[(b)] For all $\phi$, i.e., not necessarily almost strictly pure, let $\bslambda_1, \dots, \bslambda_r$ denote an $\bA$-basis of $\Lambda_{\phi}$.  Then for any $\bslambda \in \Lambda_{\phi}$,
\[
  \Span{\oK} \bigl( \pd_{\theta}^{j}(\bslambda) \bigr) \subseteq \Span{\oK}\Bigl( d_{t, n+1}[\Psi_{ \phi}]^{-1}\big|_{t=\theta} \Bigr) = \Span{\oK} \biggl( \bigcup_{u=0}^n \bigcup_{v=1}^{r} \bigcup_{\ell=1}^r \Bigl\{ \pd_\theta^u \bigl(\rF_{\bsdelta_v}(\bslambda_{\ell}) \bigr) \Bigr\} \biggr),
\]
where $\Psi_{\phi}$ is a rigid analytic trivialization for~$\phi$.
\end{enumerate}
\end{theorem}

\begin{remark}
If $\rd\phi_t - \theta\Id_d=0$, then we can allow $j=n=0$.  However, in this case all coordinates of periods and logarithms are tractable, and the theorem does not contain additional information beyond Theorem~\ref{T:rathyperquasi0}.
\end{remark}

\begin{corollary} \label{C:Jordannontract}
Continuing with the notation of Theorem~\ref{T:hypernontract}, we assume that $\phi$ is almost strictly pure and that it satisfies the conditions of Proposition~\ref{P:epsiota}.  For $\ell_1, \dots, \ell_m$ as in the statement of Proposition~\ref{P:epsiota} and for $1 \leqslant i \leqslant m$, let $d_i \assign \ell_1 + \cdots + \ell_i$.  If we write $\bsy = (y_1, \dots, y_d)^{\tr}$, then for $1 \leqslant i \leqslant m$, $j \geqslant 0$, and $0 \leqslant k \leqslant \ell_i-1$,
\[
\pd_{\theta}^j(y_{d_i-k}) \in \Span{\oK} \biggl( \{1\} \cup \bigcup_{u=0}^{k+j} \bigcup_{v=1}^r \Bigl\{ \pd_\theta^u \bigl( \rF_{\bsdelta_v}(\bsy) \bigr) \Bigr\} \biggr).
\]
\end{corollary}

\begin{remark}
In the next section we show that the condition in Theorem~\ref{T:hypernontract} and Corollary~\ref{C:Jordannontract} that $\phi$ be almost strictly pure can be removed, and so the conclusions hold for any uniformizable, abelian, and $\bA$-finite $t$-module.  See Theorem~\ref{T:hypernontractgen} for more details.  In the proof below the almost strictly pure condition centers around the use of Proposition~\ref{P:Expgyha}, and one may expect this proposition to hold in generality also, though as yet we have been unable to verify it.
\end{remark}

\begin{proof}[Proof of Theorem~\ref{T:hypernontract}]
Suppose $\phi$ satisfies the conditions of Proposition~\ref{P:epsiota}, whose notations we adopt here.  Furthermore, assume that $\phi$ is almost strictly pure, and assume without loss of generality that we have chosen our $\oK[t]$-basis $\bsm_1, \dots, \bsm_r$ of $\cM_{\phi,\oK}$ so that $(\bsdelta_v)_t = \tau\bsm_v$ for each $v$.  Letting $\bsg_{\bsy}+\bsh_{\bsalpha} = (\beta_1, \dots, \beta_r) \in \Mat_{1\times r}(\TT_{\theta})$, we combine Proposition~\ref{P:epsiota}, Lemma~\ref{L:gyha}, and Proposition~\ref{P:Expgyha}, to find
\begin{multline} \label{E:ycoords}
  \bsy = \cE_{0,\phi} \bigl( \bsg_{\bsy} + \bsh_{\alpha} \bigr) \\
  = \Bigl( \pd_t^{\ell_1-1}(\beta_1), \dots, \pd_t^1(\beta_1), \beta_1, \dots, \pd_t^{\ell_m-1}(\beta_m), \dots, \pd_t^1(\beta_m), \beta_m \Bigr)^{\tr}\Big|_{t=\theta}.
\end{multline}
By the definitions of $\bsg_{\bsy}$, $\bsh_{\bsalpha}$ in Lemma~\ref{L:gyha}, it follows that for each $i$, $k$,
\begin{equation} \label{E:betaderivatives}
  \pd_t^{k}(\beta_i) \in \Span{\oK[t]} \biggl(\{ 1 \} \cup \bigcup_{u=0}^k \bigcup_{v=1}^r \Bigl\{ \twistop{(\bsdelta_v)_t}{\pd_t^u(\cG_{\bsy})} \Bigr\} \biggr).
\end{equation}
Thus if $\bsy = (y_1, \dots, y_d)^{\tr}$, then as long as we take $n \geqslant \ell_i-1$ for each $i$, it follows that
\begin{align} \label{E:allentrieslog}
\{ y_1, \dots, y_d \} &\subseteq \Span{\oK} \biggl( \{ 1 \} \cup \bigcup_{u=0}^n \bigcup_{v=1}^r \Bigl\{ \twistop{(\bsdelta_v)_t}{\pd_t^u(\cG_{\bsy})} \big|_{t=\theta} \Bigr\} \biggr) \\
&= \Span{\oK} \biggl( \{1\} \cup \bigcup_{u=0}^n \bigcup_{v=1}^r \Bigl\{ \pd_\theta^u \bigl( \rF_{\bsdelta_v}(\bsy) \bigr) \Bigr\} \biggr), \notag
\end{align}
where the equality is of course from Theorem~\ref{T:rathyperquasi0}.  Fixing such an $n$, it then follows from the composition rule~\eqref{E:comprule} that for $j \geqslant 0$,
\[
\{ \pd_{\theta}^j(y_1), \dots, \pd_{\theta}^j(y_d) \} \subseteq
\Span{\oK} \biggl( \{1\} \cup \bigcup_{u=0}^{n+j} \bigcup_{v=1}^r \Bigl\{ \pd_\theta^u \bigl( \rF_{\bsdelta_v}(\bsy) \bigr) \Bigr\} \biggr).
\]
Part~(a) then follows by reassigning $n \leftarrow n+j$.  We note that Corollary~\ref{C:Jordannontract} follows from~\eqref{E:betaderivatives} by paying closer attention to the particular coordinates in~\eqref{E:ycoords}.

We now prove part~(b) in the case that $\phi$ satisfies the conditions of Proposition~\ref{P:epsiota}, and in particular we no longer assume that $\phi$ is almost strictly pure.  Assume that $n \geqslant \ell_i-1$ for each $i$. By Theorem~\ref{T:PeriodLattice}(b) and Proposition~\ref{P:epsiota}, we see that for $\bslambda \in \Lambda_{\phi}$.
\begin{align} \label{E:allentriesper}
  \Span{\oK}(\bslambda) &\subseteq \Span{\oK} \Bigl( d_{t, n+1}[\Psi_{ \phi}]^{-1}\big|_{t=\theta} \Bigr) \\
  &= \Span{\oK} \biggl( \bigcup_{u=0}^n \bigcup_{v=1}^{r} \bigcup_{\ell=1}^r \Bigl\{ \pd_\theta^u \bigl(\rF_{\bsdelta_v}(\bslambda_\ell) \bigr) \Bigr\} \biggr), \notag
\end{align}
where again the second equality is from Theorem~\ref{T:rathyperquasi0}.  The rest of part~(b) follows in a similar fashion to the end of part~(a).

We turn to the general case where $\phi$ may not satisfy the conditions of Proposition~\ref{P:epsiota}. Let $\rJ$ be the Jordan normal form  of $\rd \phi_t$ as in \eqref{E:Jordanpro1}, and suppose that $\rQ \in \GL_d(K^{\sep})$ is such that $\rd \phi_t = \rQ \rJ \rQ^{-1}$.  Then if we set
\[
\rho_a \assign \rQ^{-1}\phi_a \rQ, \quad a \in \bA,
\]
we obtain an isomorphism of $t$-modules $\rQ : \rho \to \phi$  (as in the definition of morphism of $t$-modules satisfying \eqref{E:morp}), and $\rd\rho_t = \rJ$.  We can assume as in Remark~\ref{R:basechangered} that a $\oK[t]$-basis for $\cN_{\rho,\oK}$ defined over $K^{\sep}$ has been chosen so that $\rho$ satisfies the conditions of Proposition~\ref{P:epsiota}.  Through the isomorphism of dual $t$-motives,
\[
  \rQ_{\dagger} : \cN_{\rho,\oK} \to \cN_{\phi,\oK} \quad (\bsn \mapsto \bsn \rQ^{\tr}),
\]
the image of the $\oK[t]$-basis of $\cN_{\rho,\oK}$ is a $\oK[t]$-basis of $\cN_{\phi,\oK}$, for which $(\Phi_{\rho},\Psi_{\rho}) = (\Phi_{\phi}, \Psi_{\phi})$.  Furthermore, there is the identity of exponentials,
\[
  \Exp_{\phi}(\bsz) = \rQ \Exp_{\rho} \bigl( \rQ^{-1} \bsz \bigr).
\]
Now by the isomorphism $\rQ^{\dagger} : \Der(\phi) \to \Der(\rho)$ defined by $(\rQ^{\dagger}\bsdelta)_a = \bsdelta_a \rQ$ for all $a \in \bA$, we let $\bsepsilon_1, \dots, \bsepsilon_r \in \Der(\rho)$ be chosen so that $\bsepsilon_v = \rQ^{\dagger}\bsdelta_v$ for each~$v$.  This induces the identity of quasi-periodic functions,
\begin{equation} \label{E:QPFbasechange}
  \rF_{\phi,\bsdelta_v}(\rQ\bsz) = \rF_{\rho,\bsepsilon_v}(\bsz), \quad 1 \leqslant v \leqslant r.
\end{equation}
Using \eqref{E:allentrieslog} for $\rho$ and $\rQ^{-1}\bsy$ together with~\eqref{E:QPFbasechange}, we see that
\begin{align*}
  \Span{\oK}(\bsy) = \Span{\oK} ( \rQ^{-1}\bsy ) &\subseteq
  \Span{\oK} \biggl( \{1\} \cup \bigcup_{u=0}^n \bigcup_{v=1}^r \Bigl\{ \pd_\theta^u \bigl( \rF_{\rho,\bsepsilon_v}(\rQ^{-1}\bsy) \bigr) \Bigr\} \biggr),\\
  &= \Span{\oK} \biggl( \{1\} \cup \bigcup_{u=0}^n \bigcup_{v=1}^r \Bigl\{ \pd_\theta^u \bigl( \rF_{\rho,\bsdelta_v}(\bsy) \bigr) \Bigr\} \biggr).
\end{align*}
Just as in the initial case, by taking derivatives with respect to $\theta$ we arrive at part~(a) in general.  The general case of~(b) follows in a similar manner, and we omit the details.
\end{proof}

\subsection{Quasi-periodic extensions of prolongations of \texorpdfstring{$t$}{t}-modules} \label{subS:QuasiExt}
In \S\ref{subS:logquasilog} we saw how hyperderivatives with respect to~$\theta$ of logarithms and quasi-logarithms associated to a $t$-module $\phi$ are captured in terms of hyperderivatives with respect to~$t$ of coordinates of Anderson generating functions of~$\phi$ evaluated at $t=\theta$. Moreover, these quantities also appear in terms of specializations at $t=\theta$ of Anderson generating functions of prolongations~$\rP_n\phi$.  To close the circle, in this section we investigate quasi-periodic extensions of $\rP_n \phi$ themselves to demonstrate how they are related to all coordinates of logarithms and quasi-logarithms of~$\phi$, and their hyperderivatives with respect to~$\theta$.

Our first goal is to determine a beneficial $\oK$-basis of $\rH^1_{\DR}(\rP_n \phi, \oK)$ defined over $K^{\sep}$, which we will use to construct a strictly quasi-periodic extension of $\rP_n \phi$ whose exponential function consists of hyperderivatives with respect to~$\theta$ of the exponential series and quasi-periodic functions of $\rP_n \phi$, in a manner inspired by the situation of Brownawell and Denis in~\S\ref{subS:BD}.  After some preliminary calculations we exhibit this basis in Proposition~\ref{P:deRhambasispro1}.

We start by examining hyperderivatives with respect to~$\theta$ of quasi-periodic functions in a general way.  Let $\phi : \bA \to \Mat_d(K^{\sep}[\tau])$ be an abelian $t$-module of dimension~$d$ and rank~$r$ defined over $K^{\sep}$.  Let $z \assign (z_1, \dots, z_d)^{\tr}$ (previously we have been using `$\bsz$' instead of `$z$,' but in what follows we will need `$\bsz$' to denote a larger vector of variables).  We consider all $z$-variables to be independent variables from $\theta$ and $t$ over~$\FF_q$.  As in~\eqref{E:bracketdef}, for $\rG(z) = \sum_{h\geqslant 0} (g_{h,1}, \dots g_{h,d})\cdot z^{(h)} = \sum_{h\geqslant0} ( g_{h,1}z_1^{q^h}+\dots +g_{h,d}z_d^{q^h}) \in \power{K^{\sep}}{z_1, \dots, z_d}$, we set for $j \in \ZZ$,
\[
\rG^{[j]}(z) \assign \pd_{\theta}^j(\rG(z)) = \sum_{h\geqslant0} \Bigl( \pd_\theta^j(g_{h,1})z_1^{q^h}+\dots +\pd_\theta^j(g_{h,d})z_d^{q^h} \Bigr).
\]
Notably for $j < 0$ we have $\rG^{[j]}(z) = 0$.

Let $\delta \in \Der(\phi,K^{\sep})$ be a $\phi$-biderivation with $\deg_{\tau}(\delta_t) = s \geqslant 1$, and write
\[
\delta_t= \sum_{j\geqslant s} \bsb_{j} \tau^j =  \sum_{j\geqslant s} (b_{j,1} , \dots, b_{j,d})\tau^j \in \cM_{\phi, K^{\sep}} \tau.
\]
Let $\rF_\delta(z)= \sum_{h\geqslant1}\bsa_h \cdot z^{(h)} =\sum_{h\geqslant1} (a_{h,1}z_1^{q^h} +\dots +a_{h,d} z_d^{q^h}) \in \power{K^{\sep}}{z_1, \dots, z_d}$ be the quasi-periodic function associated to $\delta$. From the functional equation~\eqref{E:Fdeltafneq} for quasi-periodic functions, it is straightforward to verify that the first non-zero term of $\rF_{\delta}(z)$ has degree at least $q^{s}$ in each of the variables $z_1, \dots, z_d$, and thus $\rF_\delta(z)= \sum_{h\geqslant s}\bsa_h \cdot z^{(h)}$.  Now fix $i$ with $0 \leqslant i < q^s$.  Then Proposition~\ref{P:Hyperprops}(b) implies that
\begin{equation}\label{E:functionaleqnhyper}
\pd_{\theta}^i \bigl( \rF_{\delta}(\rd\phi_t z) \bigr) = \sum_{h\geqslant s}\pd_{\theta}^i \bigl(\bsa_h (\rd\phi_t)^{(h)} \bigr)\cdot z^{(h)} = \sum_{h\geqslant s}\pd_{\theta}^i(\bsa_h) (\rd\phi_t)^{(h)}\cdot z^{(h)} = \rF^{[i]}_{\delta}(\rd\phi_t z).
\end{equation}
Furthermore, if  $\Exp_{\phi}(z) = \sum_{h=0}^{\infty} C_h \bsz^{(h)} \in \power{K^{\sep}}{z_1, \dots, z_d}^d$, then Proposition~\ref{P:Hyperprops}(b) again implies
\begin{multline} \label{E:functionaleqnhyperbid}
\pd_{\theta}^i \bigl( \delta_t (\Exp_\phi(z)) \bigr) = \sum_{j\geqslant s} \sum_{h=0}^{\infty}\pd_{\theta}^i \bigl( \bsb_j(C_h)^{(j)} \bigr) \cdot z^{(h+j)} \\
= \sum_{j\geqslant s} \sum_{h=0}^{\infty}\pd_{\theta}^i(\bsb_j)(C_h)^{(j)}\cdot z^{(h+j)}
= \delta_t^{[i]}(\Exp_{\phi}(z)).
\end{multline}
By applying $\pd_\theta^i$ to both sides of the functional equation~\eqref{E:Fdeltafneq}, it follows from~\eqref{E:functionaleqnhyper} and~\eqref{E:functionaleqnhyperbid} that
\begin{equation}\label{eq:Di}
\rF_\delta^{[i]}(\rd\phi_t z) = \theta\rF_\delta^{[i]}(z)+\rF_\delta^{[i-1]}(z) +\delta_t^{[i]}(\Exp_\phi(z)).
\end{equation}
In the preceding identities we have taken the variables $z_1, \dots, z_d$ to be independent from $\theta$ and to have all derivatives be~$0$.  On the other hand, if we take $y = (y_1, \dots, y_d)^{\tr} \in (K_{\infty}^{\sep})^d$, then the derivatives of each $y_i$ with respect to $\theta$ may not vanish.  However, if we assume that $0 \leqslant i < q^s$, then as above Proposition~\ref{P:Hyperprops}(b) implies
\begin{equation} \label{E:Fhypereval}
  \pd_{\theta}^i \bigl( \rF_{\delta}(y) \bigr) = \sum_{h \geqslant s} \pd_\theta^i (\bsa_h) \cdot y^{(h)}
  = \rF_{\delta}^{[i]}(y),
\end{equation}
much as in the proof of Corollary~\ref{C:BrDequasilogs} due to Brownawell and Denis.  This last identity will be useful toward the end of this section.

We now fix $n \geqslant 1$, we let $\bsdelta \in \Der(\rP_n \phi,K^{\sep})$, and we investigate the functional equation~\eqref{E:Fdeltafneq} of the quasi-periodic function $\bFF_{\bsdelta}$ associated to $\bsdelta$.  Recall from~\eqref{E:Dumotivenotation} that for $\bsm \in \Mat_{1\times d}(\power{\KK}{\tau})$ and $0 \leqslant u \leqslant n$, we set
\[
 (\bsm)_{u} \assign (0, \dots, 0, \bsm, 0, \dots, 0) \in \Mat_{1\times (n+1)d}(\power{\KK}{\tau}),
\]
where $\bsm$ occupies entries $du+1$ through $du +d$. Note that by the identification $\rP_n\cM_{\phi,\oK}\tau = \Mat_{1\times(n+1)d}(\oK[\tau])\tau$, we can write
\[
\bsdelta_t = (p_0)_0 + (p_1)_1+\dots +(p_n)_n =(p_0, p_1, \dots, p_n),
\]
where each $p_j \in \Mat_{1\times d}(K^{\sep}[\tau])\tau$.  For each $0 \leqslant u \leqslant n$, we set variables $\bsz_n \assign (z_{u,1}, \dots, z_{u,d})^{\tr}$ and take
\[
\bsz \assign \bigl(\bsz_0^{\tr}, \dots \bsz_n^{\tr}\bigr)^{\tr}.
\]
Moreover, setting $\power{K^{\sep}}{\bsz_0, \dots, \bsz_n} \assign \power{K^{\sep}}{z_{0,1}, \dots, z_{0,d}, \dots, z_{n,1}, \dots, z_{n,d}}$, then $\bFF_{\bsdelta}(\bsz)\in \power{K^{\sep}}{\bsz_0, \dots, \bsz_n}$ is the quasi-periodic function of $\rP_n \phi$ associated to $\bsdelta$. Note that we can also write
\[
 \bFF_{\bsdelta}(\bsz) = \rF_0(\bsz_0)+\rF_1(\bsz_1)+\dots+ \rF_n(\bsz_n),
\]
where each $\rF_j(\bsz_j) \in \power{K^{\sep}}{z_{j,1}, \dots, z_{j,d}}$. From the functional equation $\bFF_{\bsdelta}(\rd(\rP_n\phi)_t \bsz) = \theta \bFF_{\bsdelta}(\bsz) + \bsdelta_t(\Exp_{n}(\bsz))$ and Proposition~\ref{P:Exppro}, we see that
\begin{multline*}
\rF_0(\rd\phi_t\bsz_0) -\rF_1(\bsz_0) + \rF_1(\rd\phi_t\bsz_1) -\rF_2(\bsz_1) + \dots + \rF_{n-1}(\rd\phi_t\bsz_{n-1}) -\rF_n(\bsz_{n-1}) + \rF_n(\rd\phi_t \bsz_n)\\
= \theta \rF_0(\bsz_0) + \dots + \theta \rF_n(\bsz_n) + p_0(\Exp_{\phi}(\bsz_0)) + \dots+ p_n(\Exp_{\phi}(\bsz_n)).
\end{multline*}
Thus for $0\leqslant j \leqslant n-1$,
\begin{equation}\label{eq:quasipro}
 \rF_j(\rd\phi_t \bsz_j)-\theta \rF_j(\bsz_j) = \rF_{j+1}(\bsz_j)+p_j (\Exp_\phi(\bsz_j)),
\end{equation}
and
\begin{equation}\label{eq:quasipro1}
 \rF_n(\rd\phi_t\bsz_n)-\theta \rF_n(\bsz_n) = p_n (\Exp_\phi(\bsz_n)).
\end{equation}
Since $p_n\in \Mat_{1\times d}(K^{\sep}[\tau])\tau$, it represents a $\phi$-biderivation by Proposition~\ref{P:biderprops}, and so by \eqref{eq:quasipro1}, $\rF_n(\bsz)$ is the quasi-periodic function of $\phi$ associated to this $\phi$-biderivation.

Let $s = \lfloor \log_q(n) \rfloor + 1$ (if $n \geqslant 1$), and let $\delta \in \Der(\phi, K^{\sep})$ be chosen so that $\deg_{\tau}(\delta_t) \geqslant s$.  For $0\leqslant u \leqslant n$, consider $\bsdelta \in \Der(\rP_n \phi,K^{\sep})$ defined by
\begin{align}\label{E:hyperbid}
\bsdelta_t
&= \bigl( \delta_t^{[u]} \bigr)_0 + \bigl( \delta_t^{[u-1]} \bigr)_1 + \dots + (\delta_t)_{u}\\
&= \bigl( \delta_t^{[u]}, \delta_t^{[u-1]}, \dots, \delta_t^{[1]}, \delta_t, 0, \dots , 0 \bigr) \in \Mat_{1\times (n+1)d}(\power{\KK}{\tau}\tau). \notag
\end{align}
From~\eqref{eq:Di}--\eqref{eq:quasipro1} we see that
\begin{align}\label{E:quasihyperf}
\bFF_{\bsdelta}(\bsz) &= \Bigl( \bigl( \rF_{\delta}^{[u]} \bigr)_0 + \bigl( \rF_{\delta}^{[u-1]} \bigr)_1 + \dots + \bigl( \rF_{\delta}^{[1]} \bigr)_{u-1} + \bigl( \rF_{\delta} \bigr)_{u}\Bigr)(\bsz)\\
&= \bigl( \rF_{\delta}^{[u]}, \rF_{\delta}^{[u-1]}, \dots , \rF_{\delta}^{[1]},\rF_{\delta}, 0, \dots , 0 \bigr) (\bsz) \notag \\
&= \rF_{\delta}^{[u]}(\bsz_0)+ \rF_{\delta}^{[u-1]}(\bsz_1) + \dots + \rF_{\delta}^{[1]}(\bsz_{u-1})+\rF_{\delta}(\bsz_{u}) \in \power{K^{\sep}}{\bsz_0, \dots, \bsz_n}, \notag
\end{align}
and thus $\bFF_{\bsdelta}(\bsz)$ is the sum of hyperderivatives of $\rF_{\delta}(z)$ with respect to~$\theta$.

Now we determine a specific $\oK$-basis of $\rH_{\DR}^1(\rP_n \phi, \oK)$ defined over $K^{\sep}$ and find representatives that are of the form in~\eqref{E:hyperbid}. Then the corresponding quasi-periodic functions will be of the form in~\eqref{E:quasihyperf}.  By Proposition~\ref{P:strictqpextbasis} a strictly quasi-periodic extension (recall Definition~\ref{D:strictlyquasi}) of $\rP_n \phi$ using the associated basis of $\rH_{\sr}^1(\rP_n\phi,\oK)$ is isomorphic to the strictly quasi-periodic extension determined by any other $\oK$-basis of $\rH_{\sr}^1(\rP_n \phi, \oK)$.

As in the proof of Theorem~\ref{T:hypernontract}, we let $\rJ$ denote the Jordan normal form of $\rd\phi_t$ from~\eqref{E:Jordanpro1}, and we let $\rQ \in \GL_d(K^{\sep})$ be chosen so that $\rd\phi_t = \rQ\rJ\rQ^{-1}$.  We let $d_i \assign \ell_1 + \dots + \ell_i$ for $1 \leqslant i \leqslant m$, and note that $d_m=d$.  Using the definition of $\rN_\phi^\perp$ from \S\ref{subS:biderivations}, we see that $\{\bss_{d_1}\rQ^{-1}, \dots, \bss_{d_m}\rQ^{-1}\}$ is a $\oK$-basis of $\rN_\phi^\perp$ defined over $K^\sep$. Moreover, according to Theorem~\ref{T:prolmod} we have
\[
\rN_{\rP_n\phi} = \begin{pmatrix}
\rN_{\phi} &&&\\
-\Id_d & \rN_{\phi} &&\\
&\ddots &\ddots&\\
&&-\Id_d & \rN_{\phi}
\end{pmatrix},
\]
and so we check directly that the rank of $\rN_{\rP_n\phi}$ is $(n+1)d-\sum_{i=1}^m \min(n+1,\ell_i)$ and that a $\oK$-basis of $\rN_{{\rP_n\phi}}^\perp$ defined over $K^\sep$ is given by
\begin{multline} \label{E:basisNpro1}
\bigcup_{i=1}^m \bigcup_{w=0}^{\min(n,\ell_i-1)} \Bigl\{ \bigl(\bss_{d_i-w}\rQ^{-1} \bigr)_0 + \dots + \bigl( \bss_{d_i-1}\rQ^{-1} \bigr)_{w-1} + \bigl( \bss_{d_i}\rQ^{-1} \bigr)_{w} \Bigr\} \\
= \bigcup_{i=1}^m \bigcup_{w=0}^{\min(n,\ell_i-1)} \Bigl\{ \bigl( \bss_{d_i-w}\rQ^{-1}, \dots, \bss_{d_i-1}\rQ^{-1}, \bss_{d_i}\rQ^{-1}, 0, \dots, 0 \bigr) \Bigr\}.
\end{multline}
If $n \geqslant \ell_i-1$ for each $1 \leqslant i \leqslant m$, then the rank of $\rN_{\rP_n\phi}$ is $nd$.

By Proposition~\ref{P:biderprops}, we have
\begin{gather*}
\dim_{\oK} \rH^1_{\sr}(\phi, \oK) = r - d + \rank \rN_{\phi} = r-m, \\
\dim_{\oK} \Der_0(\phi,\oK) = m, \quad \Der_0(\phi,\oK) = (t-\theta) \cdot \rN_{\phi}^{\perp}(\oK).
\end{gather*}
By our calculations above, for $1 \leqslant i \leqslant m$, we define the inner $\phi$-biderivation $\gamma_{i}$ over $K^{\sep}$ by setting
\begin{equation} \label{E:der0basis}
 ( \gamma_{i} )_t \assign (t-\theta)\cdot \bss_{d_i}\rQ^{-1}=  \bss_{d_i}\rQ^{-1}\phi_t - \theta\bss_{d_i}\rQ^{-1} \quad
 \in \Mat_{1\times d}(\oK[\tau])\tau.
\end{equation}
Therefore, if $\{\delta_1, \dots, \delta_{r-m}\}$ is a $\oK$-basis of $\rH^1_{\sr}(\phi, \oK)$ defined over $K^\sep$, then by Proposition~\ref{P:biderprops}(d) we see that
\begin{equation} \label{E:basisSRphi1}
\rH_{\DR}^1(\phi,\oK) = \underbrace{\oK[\delta_1] \oplus \cdots \oplus \oK[\delta_{r-m}]}_{\rH_{\sr}^1(\phi,\oK)} \oplus \underbrace{\oK[\gamma_{1}] \oplus \cdots \oplus \oK[\gamma_{m}]}_{\Der_0(\phi,\oK)}.
\end{equation}
Now for $1 \leqslant i \leqslant m$ and $0 \leqslant w \leqslant \ell_i-1$, we define $\bsepsilon_{w,i} \in \Der_0(\rP_n\phi,K^{\sep})$ by
\begin{multline} \label{E:epsilonwidef}
(\bsepsilon_{w,i})_t \assign (t-\theta)\cdot \bigl( \bss_{d_i-w}\rQ^{-1}, \dots, \bss_{d_i-1}\rQ^{-1}, \bss_{d_i}\rQ^{-1},0, \dots, 0 \bigr) \\ {}\in \Mat_{1\times (n+1)d}(\oK[\tau])\tau,
\end{multline}
and we see from \eqref{E:basisNpro1} that
\begin{equation} \label{E:Der0Pnphi}
  \Der_0(\rP_n\phi,\oK) = \bigoplus_{i=1}^m \bigoplus_{w=0}^{\min(n,\ell_i-1)} \oK \bsepsilon_{w,i}.
\end{equation}
We continue with the preceding calculations to prove the following proposition to construct a specific $\oK$-basis of $\rH_{\DR}^1(\rP_n\phi,\oK)$.

\begin{proposition} \label{P:deRhambasispro1}
Let $\phi : \bA \to \Mat_d(K^{\sep}[\tau])$ be an abelian Anderson $t$-module of rank $r$, and let $\{ \delta_1, \dots, \delta_{r-m}, \gamma_{1}, \dots, \gamma_{m}\}$ be a $\oK$-basis of $\rH_{\DR}^1(\phi, \oK)$ defined over $K^{\sep}$, chosen as in \eqref{E:basisSRphi1}.  Fix $n \geqslant 0$.  For $0 \leqslant u \leqslant n$ and $1 \leqslant v \leqslant r-m$ define $\bsdelta_{u,v} \in \Der(\rP_n \phi,\oK)$ by setting
\[
  (\bsdelta_{u,v})_t \assign ((\delta_v)_t)_u.
\]
For $1\leqslant i \leqslant m$ and $0 \leqslant w \leqslant \min(n,\ell_i-1)$, define $\bsepsilon_{w,i} \in \Der(\rP_n \phi,\oK)$ by setting
\[
  (\bsepsilon_{w,i})_t \assign (t-\theta)\cdot \bigl( \bss_{d_i-w}\rQ^{-1}, \dots, \bss_{d_i-1}\rQ^{-1}, \bss_{d_i}\rQ^{-1},0, \dots, 0 \bigr).
\]
If $\ell_i \leqslant n$, then for each $\ell_i \leqslant e \leqslant n$ define $\bsgamma_{e,i} \in \Der(\rP_n \phi,\oK)$ by setting
\[
\bigl( \bsgamma_{e,i} \bigr)_t \assign ((\gamma_{i})_t)_e.
\]
Then in the decomposition $\rH_{\DR}^1(\rP_n\phi,\oK) = \rH_{\sr}^1(\rP_n\phi,\oK) \oplus \Der_0(\rP_n\phi,\oK)$, we have
\begin{align}
  \rH_{\sr}^1(\rP_n\phi,\oK) &= \bigoplus_{u=0}^n\bigoplus_{v=1}^{r-m} \oK[\bsdelta_{u,v}] \oplus
  \bigoplus_{i=1}^m \bigoplus_{\substack{e=\ell_i \\ \ell_i \leqslant n}}^n \oK[\bsgamma_{e,i}] \\
\intertext{and} \label{E:Der0deRhambasispro1}
  \Der_0(\rP_n\phi,\oK) &= \bigoplus_{i=1}^m \bigoplus_{w=0}^{\min(n,\ell_i-1)} \oK[\bsepsilon_{w,i}].
\end{align}
\end{proposition}

\begin{proof}
For $1\leqslant i \leqslant m$ and $0 \leqslant e \leqslant n$, we extend the definition of $\bsgamma_{e,i} \in \Der(\rP_n\phi,\oK)$ by setting $(\bsgamma_{e,i})_t \assign ((\gamma_i)_t)_e$ (for all $e$, not just those in the range $\ell_i \leqslant e \leqslant n$). If we apply Proposition~\ref{P:deRhambasispro} to the de Rham basis in~\eqref{E:basisSRphi1}, then we see that
\begin{equation}\label{E:basisderhamder0}
\biggl\{ \bigcup_{u=0}^{n}\bigcup_{v=1}^{r-m} \{\bsdelta_{u,v}\} \biggr\}
\cup \biggl\{ \bigcup_{i=1}^m\bigcup_{w=0}^{\min(n,\ell_i-1)} \{ \bsgamma_{w,i}\} \biggr\}
\cup \biggl\{ \bigcup_{i=1}^m\bigcup_{\substack{e=\ell_i\\ \ell_i \leqslant n}}^{n} \{\bsgamma_{e,i}\} \biggr\}
\end{equation}
represents a $\oK$-basis of $\rH_{\DR}^1(\rP_n \phi, \oK)$. Now for $1 \leqslant i \leqslant m$ and $0\leqslant w \leqslant \min(n,\ell_i-1)$, we use Theorem~\ref{T:prolmod} to calculate in $\cM_{\rP_n\phi}$ and find
\begin{align} \label{E:basisder0pro1}
(\bsepsilon_{w,i})_t &= (t-\theta)\cdot \bigl( \bss_{d_i-w}\rQ^{-1}, \dots , \bss_{d_i-1}\rQ^{-1}, \bss_{d_i}\rQ^{-1}, 0, \dots, 0 \bigr)\\
&= \begin{aligned}[t]
\bigl( &\bss_{d_i-w}\rQ^{-1}\phi_t - \theta\bss_{d_i-w}\rQ^{-1} - \bss_{d_i-(w-1)}\rQ^{-1},
\ldots, \\
&\,\,\,\bss_{d_i-1}\rQ^{-1}\phi_t - \theta\bss_{d_i-1}\rQ^{-1} - \bss_{d_i}\rQ^{-1},
\bss_{d_i}\rQ^{-1}\phi_t - \theta\bss_{d_i}\rQ^{-1},0, \ldots, 0 \bigr)
\end{aligned}
\notag \\
&= \sum_{j=0}^{w-1} \bigl((t-\theta)\cdot \bss_{d_i-(w-j)}\rQ^{-1} - \bss_{d_i-(w-(j+1))}\rQ^{-1} \bigr)_j + \bigl((t-\theta)\cdot \bss_{d_i}\rQ^{-1} \bigr)_w, \notag
\end{align}
where in the final terms the quantities `$(t-\theta)\cdot \bss_{d_i-(w-j)}\rQ^{-1}$,' etc., are calculated in $\cM_{\phi,\oK}$ and embedded in $\cM_{\rP_n\phi,\oK}$.  Now~\eqref{E:der0basis} and~\eqref{E:epsilonwidef} together imply that when $0 \leqslant j \leqslant w-1$ we have the calculation in $\cM_{\phi,\oK}$,
\[
  (t-\theta)\cdot \bss_{d_i-(w-j)}\rQ^{-1} - \bss_{d_i - (w-(j+1))}\rQ^{-1} \in \Mat_{1\times d}(\oK[\tau]\tau).
\]
By Proposition~\ref{P:biderprops}(a), we can thus choose $\beta_{w,j} \in \Der(\phi,\oK)$ so that
\begin{equation} \label{E:betawjdef}
  (\beta_{w,j})_t = (t-\theta)\cdot \bss_{d_i-(w-j)}\rQ^{-1} - \bss_{d_i - (w-(j+1))}\rQ^{-1},
\end{equation}
and naturally by~\eqref{E:basisSRphi1},
\[
  [\beta_{w,j}] \in \oK[\delta_1] \oplus \cdots \oplus \oK[\delta_{r-m}] \oplus \oK[\gamma_{1}] \oplus \cdots \oplus \oK[\gamma_{m}].
\]
Now by defining $\bsbeta_{w,j} \in \Der(\rP_n\phi,\oK)$ by $(\bsbeta_{w,j})_t \assign ((\beta_{w,j})_t)_j$, we see that $\bsbeta_{w,j}$ lies within the natural image of $\Der(\rP_j\phi,\oK)$ coming from~\eqref{E:motivespro}. Therefore, Proposition~\ref{P:deRhambasispro} implies that in $\rH_{\DR}^1(\rP_n\phi,\oK)$,
\begin{equation} \label{E:betajDR}
  [\bsbeta_{w,j}] \in \bigoplus_{u=0}^j \bigoplus_{v=1}^{r-m} \oK[\bsdelta_{u,v}] \oplus
  \bigoplus_{i=1}^m \bigoplus_{e=0}^j \oK[\bsgamma_{e,i}].
\end{equation}
Pulling all of this information together, we see that for $1 \leqslant i \leqslant m$ and $0\leqslant w \leqslant \min(n,\ell_i-1)$, it follows from \eqref{E:basisder0pro1} together with~\eqref{E:der0basis} that
\[
(\bsepsilon_{w,i})_t = \bigl( (t-\theta)\bss_{d_i} \rQ^{-1}\bigr)_w + \sum_{j=0}^{w-1} (\bsbeta_{w,j})_t
= (\bsgamma_{w,i})_t + \sum_{j=0}^{w-1} (\bsbeta_{w,j})_t.
\]
Therefore \eqref{E:betajDR} implies that in $\rH_{\DR}^1(\rP_n\phi,\oK)$,
\[
  [\bsepsilon_{w,i}] \in [\bsgamma_{w,i}] + \bigoplus_{u=0}^{w-1} \bigoplus_{v=1}^{r-m} \oK[\bsdelta_{u,v}] \oplus \bigoplus_{i=1}^m \bigoplus_{e=0}^{w-1} \oK[\bsgamma_{e,i}].
\]
Taking $w=0$, we see that $\bsepsilon_{0,i} = \bsgamma_{0,i}$ for each $i$.  Letting $w=1$, we have
\[
  [\bsepsilon_{1,i}] \in [\bsgamma_{1,i}] + \bigoplus_{v=1}^{r-m} \oK[\bsdelta_{0,v}] \oplus \bigoplus_{i=1}^m \oK[\bsgamma_{0,i}] = [\bsgamma_{1,i}] + \bigoplus_{v=1}^{r-m} \oK[\bsdelta_{0,v}] \oplus \bigoplus_{i=1}^m \oK[\bsepsilon_{0,i}],
\]
and so
\[
\bigoplus_{v=1}^{r-m} \oK[\bsdelta_{0,v}] \oplus \bigoplus_{i=1}^m (\oK[\bsgamma_{0,i}]\oplus \oK[\bsgamma_{1,i}])
= \bigoplus_{v=1}^{r-m} \oK[\bsdelta_{0,v}] \oplus \bigoplus_{i=1}^m (\oK[\bsepsilon_{0,i}]\oplus \oK[\bsepsilon_{1,i}]).
\]
Proceeding by induction on~$w$ we find that
\begin{multline*}
\bigoplus_{i=1}^m\bigoplus_{w=0}^{\min(n,\ell_i-1)-1} \bigoplus_{v=1}^{r-m} \oK[\bsdelta_{w,v}] \oplus \bigoplus_{i=1}^m\bigoplus_{w=0}^{\min(n,\ell_i-1)} \oK[\bsgamma_{w,i}] \\
= \bigoplus_{i=1}^m\bigoplus_{w=0}^{\min(n,\ell_i-1)-1} \bigoplus_{v=1}^{r-m} \oK[\bsdelta_{w,v}] \oplus \bigoplus_{i=1}^m\bigoplus_{w=0}^{\min(n,\ell_i-1)} \oK[\bsepsilon_{w,i}].
\end{multline*}
Combining this with~\eqref{E:basisderhamder0}, we see that
\[
\biggl\{ \bigcup_{u=0}^{n}\bigcup_{v=1}^{r-m} \{\bsdelta_{u,v}\} \biggr\}
\cup \biggl\{ \bigcup_{i=1}^m\bigcup_{w=0}^{\min(n,\ell_i-1)} \{ \bsepsilon_{w,i}\} \biggr\}
\cup \biggl\{ \bigcup_{i=1}^m\bigcup_{\substack{e=\ell_i\\ \ell_i \leqslant n}}^{n} \{\bsgamma_{e,i}\} \biggr\}
\]
represents a $\oK$-basis of $\rH_{\DR}^1(\rP_n\phi,\oK)$ as desired, and the rest of the proposition follows from~\eqref{E:Der0Pnphi}.
\end{proof}

We recall that in Theorem~\ref{T:hypernontract}(a) we required the $t$-module in question to be almost strictly pure.  It turns out that this condition can be removed by using elements of the proof of Proposition~\ref{P:deRhambasispro1}, and so the conclusion of Theorem~\ref{T:hypernontract}(a) holds for all uniformizable, abelian, and $\bA$-finite $t$-modules.

\begin{theorem} \label{T:hypernontractgen}
Let $\phi$ be a uniformizable, abelian, and $\bA$-finite Anderson $t$-module defined over $K^\sep$ of rank $r$ and dimension $d$. Let $\bsy \in (K_{\infty}^{\sep})^d$ satisfy $\Exp_\phi(\bsy) \in (K^\sep)^d$, and let $\{ \delta_1, \dots, \delta_{r-m}, \gamma_1, \dots, \gamma_m\}$ be chosen to be the $\oK$-basis of $\rH_{\DR}^1(\phi, \oK)$ defined over $K^{\sep}$ as in \eqref{E:basisSRphi1}.  Let $\cG_{\bsy} \in \TT^d$ be the Anderson generating function for $\phi$ associated to~$\bsy$.  Finally, let $j \geqslant 0$, and choose $n \geqslant 0$ so that $(\rd\phi_t - \theta\Id_d)^{n-j}=0$.  Then
\begin{align*}
  \Span{\oK} \bigl( \pd_{\theta}^{j} (\bsy) \bigr) &\subseteq
  \Span{\oK} \biggl( \{1\} \cup \bigcup_{u=0}^n \bigcup_{v=1}^{r-m} \bigcup_{i=1}^m \Bigl\{ \twistop{(\delta_v)_t}{\pd_t^u(\cG_{\bsy})} \big|_{t=\theta}, \twistop{(\gamma_i)_t}{\pd_t^u(\cG_{\bsy})} \big|_{t=\theta} \Bigr\} \biggr) \\
  &=
  \Span{\oK} \biggl( \{1\} \cup \bigcup_{u=0}^n \bigcup_{v=1}^{r-m}\bigcup_{i=1}^{m} \Bigl\{ \pd_\theta^u \bigl( \rF_{\delta_v}(\bsy), \pd_\theta^u \bigl( \rF_{\gamma_i}(\bsy) \bigr) \Bigr\} \biggr).
\end{align*}
\end{theorem}

\begin{remark}
Taken together, Theorem~\ref{T:rathyperquasi0}, Theorem~\ref{T:hypernontract}, Corollary~\ref{C:Jordannontract}, and Theorem~\ref{T:hypernontractgen} can be viewed as answering Question~\ref{Q:three} and the questions raised at the beginning of~\S\ref{S:Hyper}. We see in these cases that hyperderivatives with respect to~$\theta$ of logarithms and quasi-logarithms of a $t$-module $\phi$ arise as specializations of solutions of Frobenius difference equations associated to Maurischat's prolongation $t$-modules of~$\phi$.  Thus the quantities obtained through the Brownawell-Denis hyperderivative $t$-modules from~\S\ref{subS:BD}, especially in Corollaries~\ref{C:BrDelogs} and~\ref{C:BrDequasilogs}, do arise from abelian and $\bA$-finite Anderson $t$-modules.
\end{remark}

\begin{proof}[Proof of Theorem~\ref{T:hypernontractgen}]
We continue with the notation from previous parts of this section.  As before the equality of the second two sets is Theorem~\ref{T:rathyperquasi0}.  Now writing $\bsy = (y_1, \dots, y_d)^{\tr}$, as in the proof of Theorem~\ref{T:hypernontract} we can use the condition on~$n$ and~$j$ to reduce the argument to showing that
\begin{multline} \label{E:bsyreduction}
  \{ y_1, \dots, y_d \} \subseteq \\ \Span{\oK} \biggl( \{ 1 \} \cup \bigcup_{u=0}^n \bigcup_{v=1}^{r-m} \bigcup_{i=1}^{m} \Bigl\{ \twistop{(\delta_v)_t}{\pd_t^u(\cG_{\bsy})} \big|_{t=\theta},  \twistop{(\gamma_{i})_t}{\pd_t^u(\cG_{\bsy})} \big|_{t=\theta}\Bigr\} \biggr).
\end{multline}
For any $\bsz_0, \dots, \bsz_n \in \KK^d$, let $\bsz \assign (\bsz_0^{\tr}, \dots, \bsz_n^{\tr})^{\tr}\in \KK^{(n+1)d}$ and write $\Exp_{\phi}(\bsz_u) = \bsx_u$ for each $0\leqslant u \leqslant n$.  Then $\Exp_{\rP_n \phi}(\bsz) = (\bsx_0^{\tr}, \dots, \bsx_n^{\tr})^{\tr} \rassign \bsx \in \KK^{(n+1)d}$ by Proposition~\ref{P:Exppro}.  Therefore, for $1\leqslant i \leqslant m$ and $0 \leqslant w \leqslant \ell_i-1$, it follows from \eqref{E:quasiinner} and \eqref{E:epsilonwidef} (and recalling Proposition~\ref{P:quasigen}(a)), that
\begin{multline*}
    \bigl\langle (\bsepsilon_{w,i})_t \bigm| \fG_{\bsz} \bigr\rangle \big|_{t=\theta} = \bFF_{\bsepsilon_{w,i}}(\bsz)\\
    = \bss_{d_i-w}\rQ^{-1}(\bsx_0 - \bsz_0)+  \bss_{d_i-(w-1)}\rQ^{-1}(\bsx_1 - \bsz_1)+ \dots + \bss_{d_i}\rQ^{-1}(\bsx_w - \bsz_w).
\end{multline*}
Now consider $(\bsy)_0 = (\bsy^{\tr}, 0, \dots, 0)^{\tr} \in \KK^{(n+1)d}$, and write $\Exp_{\phi}(\bsy) = \bsalpha$. The previous formula then implies that $\bFF_{\bsepsilon_{w,i}}((\bsy)_0) = \bss_{d_i-w}\rQ^{-1}(\bsalpha-\bsy)$, and therefore,
\begin{equation} \label{E:perlogquasipro}
\bsy = \bsalpha - \rQ\cdot \left(\begin{array}{c}
\bFF_{\bsepsilon_{\ell_1-1,1}}((\bsy)_0) \\
\vdots\\
\bFF_{\bsepsilon_{0,1}}((\bsy)_0)\\[5pt] \hline \\[-10pt]
\bFF_{\bsepsilon_{\ell_2-1,2}}((\bsy)_0)\\
\vdots\\
\bFF_{\bsepsilon_{0,2}}((\bsy)_0)\\
\vdots\\[5pt] \hline \\[-10pt]
\bFF_{\bsepsilon_{\ell_m-1,m}}((\bsy)_0)\\
\vdots\\
\bFF_{\bsepsilon_{0,m}}((\bsy)_0)
\end{array}\right)
=\bsalpha - \rQ\cdot \left. \left( \begin{array}{c}
\twistop{(\bsepsilon_{\ell_1-1,1})_t}{\fG_{(\bsy)_0}} \\
\vdots\\
\twistop{(\bsepsilon_{0,1})_t}{\fG_{(\bsy)_0}}\\[5pt] \hline \\[-10pt]
\twistop{(\bsepsilon_{\ell_2-1,2})_t}{\fG_{(\bsy)_0}}\\
\vdots\\
\twistop{(\bsepsilon_{0,2})_t}{\fG_{(\bsy)_0}}\\
\vdots\\[5pt] \hline \\[-10pt]
\twistop{(\bsepsilon_{\ell_m-1,m})_t}{\fG_{(\bsy)_0}}\\
\vdots\\
\twistop{(\bsepsilon_{0,m})_t}{\fG_{(\bsy)_0}}
\end{array} \right) \right|_{t=\theta}.
\end{equation}
Now recall the definitions of the biderivations $\gamma_i$, $\beta_{w,j} \in \Der(\phi,\oK)$ from \eqref{E:der0basis} and \eqref{E:betawjdef}.  By Proposition~\ref{P:qpfunctionpro} and \eqref{E:basisder0pro1}, it follows that for each $1\leqslant i \leqslant m$ and $1 \leqslant w\leqslant \ell_i-1$ we have
\begin{align} \label{E:AGFDerpro}
\bigl\langle (\bsepsilon_{w,i})_t &\bigm| \fG_{(\bsy)_0} \bigr\rangle \big|_{t=\theta} \\
&=\begin{aligned}[t]
    \biggl( \sum_{j=1}^w \bigl\langle (t-\theta) \cdot \bss_{d_i-j}\rQ^{-1} - \bss_{d_i-(j-1)}&\rQ^{-1} \bigm| \pd_t^{w-j}(\cG_{\bsy}) \bigr\rangle \biggr)\bigg|_{t=\theta} \\
    &{}+ \bigl\langle (t-\theta) \cdot \bss_{d_i}\rQ^{-1} \bigm| \pd_t^w(\cG_{\bsy}) \bigr\rangle \big|_{t=\theta}
    \end{aligned} \notag
\\
&= \sum_{j=1}^w \bigl( \twistop{(\beta_{w,w-j})_t}{\pd_t^{w-j}(\cG_{\bsy})} \bigr) \big|_{t=\theta}
    + \twistop{(\gamma_i)_t}{\pd_t^{w} (\cG_{\bsy})}|_{t=\theta}, \notag
\end{align}
and for $w=0$,
\begin{equation} \label{E:AGFDerpro0}
\bigl\langle (\bsepsilon_{0,i})_t \bigm| \fG_{(\bsy)_0} \bigr\rangle \big|_{t=\theta} = \bigl\langle (t-\theta) \cdot \bss_{d_i}\rQ^{-1} \bigm| \cG_{\bsy} \bigr\rangle \big|_{t=\theta} = \twistop{(\gamma_i)_t}{\cG_{\bsy}}|_{t=\theta}.
\end{equation}
Hence for $1\leqslant i \leqslant m$ and $0 \leqslant w \leqslant \ell_i-1$, it follows from~\eqref{E:basisSRphi1} that
\[
\bigl\langle (\bsepsilon_{w,i})_t \bigm| \fG_{(\bsy)_0} \bigr\rangle \big|_{t=\theta} \in \Span{\oK} \biggl( \bigcup_{u=0}^n \bigcup_{v=1}^{r-m} \bigcup_{i=1}^{m} \Bigl\{ \twistop{(\delta_v)_t}{\pd_t^u(\cG_{\bsy})} \big|_{t=\theta},  \twistop{(\gamma_{i})_t}{\pd_t^u(\cG_{\bsy})} \big|_{t=\theta}\Bigr\} \biggr),
\]
and combining this with~\eqref{E:perlogquasipro} we obtain~\eqref{E:bsyreduction}.
\end{proof}

In order to determine representatives of classes in $\rH_{\DR}^1(\rP_n \phi, \oK)$ of the form \eqref{E:hyperbid}, we refine representatives of the $\oK$-basis of $\rH_{\DR}^1(\rP_n \phi, \oK)$ in Proposition~\ref{P:deRhambasispro1} in the following way.  We continue with $\delta_1, \dots, \delta_{r-m}$, $\gamma_1, \dots, \gamma_m \in \Der(\phi,K^{\sep})$ as in~\eqref{E:basisSRphi1}. Let $s = \lfloor \log_q(n) \rfloor + 1$ (if $n \geqslant 1$), and for each $1\leqslant v \leqslant r-m$, $1 \leqslant i \leqslant m$, let $\delta_v^s$, $\gamma_i^s \in \Der(\phi,K^{\sep})$ be given as in Lemma~\ref{L:degdeRham}. That is, $\delta_v^s$ and $\gamma_i^s$ are in the same de Rham classes over~$\oK$ as $\delta_v$ and $\gamma_i$ respectively, and $\deg_{\tau}((\delta_v^s)_t)$, $\deg_{\tau}((\gamma_i^s)_t) \geqslant s$.

For $0 \leqslant u \leqslant n$ and $1 \leqslant v \leqslant r-m$, we define $\bsvartheta_{u,v} \in \Der(\rP_n \phi,K^{\sep})$ by setting
\begin{equation}\label{E:newprosrbasis}
  (\bsvartheta_{u,v})_t \assign  \sum_{j=0}^u \bigl((\delta_v^s)_t^{[u-j]} \bigr)_j  \in \Mat_{1\times (n+1)d}(K^{\sep}[\tau]\tau),
\end{equation}
and for $1\leqslant i \leqslant m$ and $0 \leqslant e \leqslant n$, we define $\bskappa_{e,i} \in \Der(\rP_n \phi,K^{\sep})$ by setting
\begin{equation}\label{E:newprosrbasis1}
(\bskappa_{e,i})_t \assign \sum_{j=0}^e \bigl( (\gamma_i^s)_t^{[e-j]} \bigr)_j \in \Mat_{1\times (n+1)d}(K^{\sep}[\tau]\tau).
\end{equation}
For the short term we consider only $\bskappa_{e,i}$ for $\ell_i \leqslant e \leqslant n$ (if $\ell_i \leqslant n$), but it will be convenient to refer to all $\bskappa_{e,i}$ at the end of the section in Corollary~\ref{C:hyperQPFAGF}. Each $\bsvartheta_{u,v}$ and $\bskappa_{e,i}$ is in the form of~\eqref{E:hyperbid}, and moreover the following holds.

\begin{lemma} \label{L:dersrpro}
Let $n \geqslant 0$.  With notation as above in~\eqref{E:newprosrbasis} and~\eqref{E:newprosrbasis1},
\[
\biggl\{ \bigcup_{u=0}^{n}\bigcup_{v=1}^{r-m}\{\bsvartheta_{u,v}\} \biggr\}
\cup \biggl\{ \bigcup_{i=1}^m \bigcup_{\substack{e=\ell_i\\ \ell_i \leqslant n}}^{n}\{\bskappa_{e,i}\} \biggr\}
\]
represents a $\oK$-basis of $\rH^{1}_{\sr}(\rP_n\phi,\oK)$.
\end{lemma}

\begin{proof}
We proceed by induction on $n$.  If $n=0$, the proof is done once we recall that $\delta_1^s, \dots, \delta_{r-m}^s$ represent a $\oK$-basis of $\rH_{\sr}^1(\phi,\oK)$, as $\delta_v$ and $\delta_v^s$ represent the same de Rham class for~$\phi$.  We then assume $n \geqslant 1$ and follow a similar argument to the one for Proposition~\ref{P:deRhambasispro}.  In particular we appeal to the exact sequence~\eqref{E:DRSESpro} with $h=n-1$, namely
\[
  0 \to \rH_{\DR}^1(\rP_{n-1} \phi,\oK) \xrightarrow{\bsi_*} \rH_{\DR}^1(\rP_n \phi,\oK) \xrightarrow{\bspr_*} \rH_{\DR}^1(\phi, \oK) \to 0.
\]
As in this previous proof, a $\oK$-basis of the middle term is obtained from the union of the image of $\oK$-basis of the left under $\bsi_*$ with the preimage of a $\oK$-basis of the right under $\bspr_*$. We observe that for appropriate choices of indices (recall also $\bsepsilon_{w,i}$ from \eqref{E:basisder0pro1}),
\begin{align*}
\bspr_*([\bsvartheta_{u,v}]) &= \begin{cases}
[\delta_v^s] & \textup{if $u=n$,} \\
0 & \textup{if $u < n$,}
\end{cases}
\\
\bspr_*([\bskappa_{e,i}]) &= \begin{cases}
[\gamma_i^s] & \textup{if $e=n$,} \\
0 & \textup{if $e < n$,}
\end{cases}
\\
\bspr_*([\bsepsilon_{w,i}]) &= \begin{cases}
[\gamma_i] & \textup{if $w=n$,} \\
0 & \textup{if $w < n$.}
\end{cases}
\end{align*}
For each fixed $i$, we have two mutually exclusive cases: either $\ell_i \leqslant n$ and then $[\bskappa_{n,i}] \mapsto [\gamma_i^s]$; or $\ell_i > n$ and then $[\bsepsilon_{n,i}] \mapsto [\gamma_i]$.  Therefore,
\[
  \bigcup_{v=1}^{r-m} \{ \bsvartheta_{n,v} \} \cup \bigcup_{\substack{i=1 \\ \ell_i \leqslant n}}^m \{ \bskappa_{n,i} \} \cup \bigcup_{\substack{i=1 \\ \ell_i > n}}^m \{ \bsepsilon_{n,i} \}
\]
projects via $\bspr_*$ to a $\oK$-basis of $\rH_{\DR}^1(\phi,\oK)$.  By induction we find that
\[
\biggl\{ \bigcup_{u=0}^{n}\bigcup_{v=1}^{r-m}\{\bsvartheta_{u,v}\} \biggr\}
\cup \biggl\{ \bigcup_{w=0}^{\ell_i-1} \bigcup_{i=1}^m\{ \bsepsilon_{w,i}\} \biggr\}
 \cup \biggl\{ \bigcup_{i=1}^m  \bigcup_{\substack{e=\ell_i\\ \ell_i \leqslant  n}}^{n}\{\bskappa_{e,i}\} \biggr\}
\]
represents a $\oK$-basis of $\rH_{\DR}^1(\rP_n\phi,\oK)$, and the result follows from Proposition~\ref{P:deRhambasispro1} and in particular~\eqref{E:Der0deRhambasispro1}.
\end{proof}

We now use this basis to analyze the associated strictly quasi-periodic extension of~$\rP_n\phi$ and to determine simplified expressions for its quasi-periods and quasi-logarithms.  Considering the $\oK$-basis of $\rH_{\sr}^1(\phi, \oK)$ in Lemma~\ref{L:dersrpro}, we obtain quasi-periodic functions $\bFF_{\bsvartheta_{u,v}}(\bsz)$ and $\bFF_{\bskappa_{e,i}}(\bsz)$ as in~\eqref{E:quasihyperf}.  Notably for $0 \leqslant u \leqslant n$ and $0 \leqslant v \leqslant r-m$,
\begin{equation} \label{E:Fvarthetauv}
\bFF_{\bsvartheta_{u,v}}(\bsz) = \rF_{\delta_v^s}^{[u]}(\bsz_0) + \rF_{\delta_v^s}^{[u-1]}(\bsz_1) + \cdots + \rF_{\delta_v^s}(\bsz_u),
\end{equation}
and for $1 \leqslant i \leqslant m$ and $\ell_i \leqslant e \leqslant n$ (if $\ell_i \leqslant n$),
\begin{equation} \label{E:Fkappaei}
\bFF_{\bskappa_{e,i}}(\bsz) = \rF_{\gamma_i^s}^{[e]}(\bsz_0) + \rF_{\gamma_i^s}^{[e-1]}(\bsz_1) + \cdots + \rF_{\gamma_i^s}(\bsz_e).
\end{equation}
Moreover, we let $\psi_n : \bA \to \Mat_{\bsd_n}(K^{\sep}[\tau])$ be the strictly quasi-periodic extension of ${\rP_n\phi}$ with respect to the biderivations
\begin{multline*}
  \bigl\{ \bsvartheta_{n,1}, \dots, \bsvartheta_{0,1}; \ldots; \bsvartheta_{n,r-m}, \ldots, \bsvartheta_{0,r-m}; \\ \bskappa_{n,1}, \ldots, \bskappa_{\ell_1,1}; \ldots; \bskappa_{n,m}, \ldots, \bskappa_{\ell_m,m} \bigr\} \subseteq \Der_0(\rP_n\phi,K^{\sep}),
\end{multline*}
which as in~\eqref{E:QPext} is a $t$-module of dimension
\[
  \bsd_n \assign (n+1)(r+d-m) + \sum_{i=1}^{m} \max(n-\ell_i+1,0).
\]
We note that if $n \geqslant \ell_i-1$ for each $i$, then this simplifies to
\[
  \bsd_n = nd + (n+1)r.
\]
We see from~\eqref{E:QPExp} that for 
\[
\bsu = (u_{0,1}, \dots, u_{n,1}, \ldots, u_{0,r-m}, \ldots, u_{n,r-m}, v_{\ell_1,1}, \ldots, v_{n,1}, \ldots, v_{\ell_n,m}, \ldots, v_{n,m})^{\tr},
\]
its exponential function has the form
\begin{align} \label{E:quasihyperpro}
\Exp_{\psi_n} \begin{pmatrix} \bsu \\ \bsz \end{pmatrix}
&= \begin{aligned}[t]
\Bigl( u_{0,1} &{}+\bFF_{\bsvartheta_{n,1}}(\bsz), \dots, u_{n,r-m}+\bFF_{\bsvartheta_{0,r-m}}(\bsz), \\
& v_{\ell_1,1}+\bFF_{\bskappa_{n,1}}(\bsz), \dots, v_{n,m} + \bFF_{\bskappa_{\ell_m,m}}(\bsz), \Exp_{\rP_n \phi}(\bsz)^{\tr} \Bigr)^{\tr}.
\end{aligned} \notag
\end{align}
Pulling this all together we obtain the following proposition that identifies quasi-log\-a\-rithms on the prolongation $\rP_n\phi$ with hyperderivatives with respect to~$\theta$ of quasi-log\-a\-rithms on the original $t$-module~$\phi$ itself.  As such we may view this proposition as an extension of Corollary~\ref{C:BrDequasilogs} of Brownawell and Denis.

\begin{proposition} \label{P:period}
Let $\phi : \bA \to \Mat_d(K^{\sep}[\tau])$ be a uniformizable abelian $t$-module of rank~$r$. We let $n \geqslant 0$, and continuing with the notations of Lemma~\ref{L:dersrpro}, we let $\psi_n : \bA \to \Mat_{\bsd_n}(K^{\sep}[\tau])$ be the associated strictly quasi-periodic extension of $\rP_n \phi$. Let $\bsy \assign (\bsy_0^{\tr}, \dots, \bsy_n^{\tr})^{\tr} \in (K_{\infty}^{\sep})^{(n+1)d}$ be chosen so that $\Exp_{\phi}(\bsy_u) = \bsalpha_u \in (K^{\sep})^d$ for each $0 \leqslant u\leqslant n$. Then
\[
\Exp_{\psi_n} \left(\begin{array}{c}
  \vdots\\[5pt] \hline \\[-10pt]
  -\bFF_{\bsvartheta_{n,v}}(\bsy)\\
  \vdots\\
  -\bFF_{\bsvartheta_{1,v}}(\bsy)\\
  -\bFF_{\bsvartheta_{0,v}}(\bsy)\\[5pt] \hline \\[-10pt]
  \vdots\\[5pt] \hline \\[-10pt]
  -\bFF_{\bskappa_{n,i}}(\bsy)\\
  \vdots\\
  -\bFF_{\bskappa_{\ell_i,i}}(\bsy)\\[5pt] \hline \\[-10pt]
  \vdots\\[5pt] \hline \\[-10pt]
  \bsy_0 \\
  \vdots \\
  \bsy_n
\end{array}\right)
=
\Exp_{\psi_n} \left(\begin{array}{c}
\vdots\\[5pt] \hline \\[-10pt]
  -\sum_{w=0}^n \pd_{\theta}^w (\rF_{\delta_v^s}(\bsy_{n-w})) \\
  \vdots\\
  -\rF_{\delta_v^s}(\bsy_1) - \pd_{\theta}^1(\rF_{\delta_v^s}(\bsy_0))\\
  -\rF_{\delta_v^s}(\bsy_0)\\[5pt] \hline \\[-10pt]
  \vdots\\[5pt] \hline \\[-10pt]
  -\sum_{w=0}^n \pd_{\theta}^w(\rF_{\gamma_i^s}(\bsy_{n-w})) \\
  \vdots\\
  -\sum_{w=0}^{\ell_i} \pd_{\theta}^w(\rF_{\gamma_i^s}(\bsy_{\ell_i-w})) \\[5pt] \hline \\[-10pt]
  \vdots\\[5pt] \hline \\[-10pt]
  \bsy_0 \\
  \vdots \\
  \bsy_n
\end{array}\right)
=  \begin{pmatrix}
  0\\
  \vdots\\
  \vdots\\
  0\\
  \bsalpha_0 \\
  \vdots\\
  \bsalpha_n
\end{pmatrix}.
\]
\end{proposition}

\begin{proof}
For $0 \leqslant u \leqslant n$ and $1 \leqslant v \leqslant r-m$, we need to verify that
\begin{equation} \label{E:hyperbFF}
  \bFF_{\bsvartheta_{u,v}}(\bsy) = \sum_{w=0}^u \pd_{\theta}^w \bigl( \rF_{\delta_v^s}(\bsy_{u-w}) \bigr).
\end{equation}
By~\eqref{E:Fvarthetauv}, it suffices to check for $0 \leqslant w \leqslant u$ that
\[
  \pd_{\theta}^w \bigl( \rF_{\delta_v^s}(\bsy_{u-w}) \bigr) = \rF_{\delta_v^s}^{[w]}(\bsy_{u-w}),
\]
but since $\deg_{\tau} (\delta_v^s)_t \geqslant s$, this has been shown in~\eqref{E:Fhypereval}.  For $1 \leqslant i \leqslant m$ and $\ell_i \leqslant e \leqslant n$, the desired identity for $\bFF_{\bskappa_{e,i}}(\bsy)$ follows similarly.
\end{proof}

\begin{remark}
What we notice from Proposition~\ref{P:period} is that quasi-logarithms associated to strictly reduced biderivations on $\rP_n\phi$ can be identified with hyperderivatives with respect to~$\theta$ of quasi-logarithms associated to strictly reduced \emph{and some} inner biderivations on $\phi$ itself.  Also these strictly reduced quasi-logarithms on $\rP_n\phi$ are directly constructed from quantities appearing in Theorems~\ref{T:rathyperquasi0} and~\ref{T:hypernontractgen}.

In~\cite{CP11}, \cite{CP12}, Chang and the second author obtained results on algebraic independence of logarithms and quasi-logarithms of a single Drinfeld module by constructing an appropriate $t$-motive that was an extension of the trivial $t$-motive by the $t$-motive of the Drinfeld module and then by calculating the dimension of its motivic Galois group so as to apply results of~\cite{ABP04}, \cite{P08}.  One can imagine, using Proposition~\ref{P:period} together with Theorems~\ref{T:rathyperquasi0} and~\ref{T:hypernontractgen}, that a similar strategy could be used to determine the algebraic independence of hyperderivatives with respect to~$\theta$ of logarithms and quasi-logarithms on a general abelian, $\bA$-finite, and uniformizable $t$-module.  This strategy has been carried out successfully by Maurischat~\cite{Maurischat18}, \cite{Maurischat21b}, for hyperderivatives of the Carlitz period and by the first author~\cite{NPhD}, \cite{N21}, for hyperderivatives of logarithms and quasi-logarithms of a Drinfeld module.
\end{remark}

To these ends we can use Proposition~\ref{P:period} to capture hyperderivatives of logarithms and quasi-logarithms of~$\phi$ in a more simplified way, as we find in the following corollary.

\begin{corollary} \label{C:quasilog}
Let $\phi : \bA \to \Mat_d(K^{\sep}[\tau])$ be a uniformizable, abelian, and $\bA$-finite $t$-module of rank~$r$, and let $\bsy \in (K_{\infty}^{\sep})^d$ be chosen so that $\Exp_{\phi}(\bsy) = \bsalpha \in (K^{\sep})^d$.  For $n \geqslant 0$, let $\psi_n : \bA \to \Mat_{\bsd_n}(K^{\sep}[\tau])$ be the associated strictly quasi-periodic extension of $\rP_n \phi$ arising from Lemma~\ref{L:dersrpro}.  Then
\[
\Exp_{\psi_n} \left(\begin{array}{c}
\vdots\\[5pt] \hline \\[-10pt]
  -\pd_{\theta}^n (\rF_{\delta_v^s}(\bsy)) \\
  \vdots\\
  -\pd_{\theta}^1(\rF_{\delta_v^s}(\bsy))\\
  -\rF_{\delta_v^s}(\bsy)\\[5pt] \hline \\[-10pt]
  \vdots\\[5pt] \hline \\[-10pt]
  -\pd_{\theta}^n(\rF_{\gamma_i^s}(\bsy)) \\
  \vdots\\
  -\pd_{\theta}^{\ell_i}(\rF_{\gamma_i^s}(\bsy)) \\[5pt] \hline \\[-10pt]
  \vdots\\[5pt] \hline \\[-10pt]
  \bsy \\
  0\\
  \vdots \\
  0
\end{array}\right)
=
\left(\begin{array}{c}
  0\\
  \vdots\\
  \vdots\\
  0\\[5pt] \hline \\[-10pt]
  \bsalpha\\
  0\\
  \vdots\\
  0
\end{array}\right).
\]
\end{corollary}

For $\bsy \in (K_{\infty}^{\sep})^d$, this corollary is derived by simply substituting
\[
(\bsy)_0 = (\bsy^{\tr}, 0, \dots, 0)^{\tr} \in (K_{\infty}^{\sep})^{(n+1)d}
\]
into Proposition~\ref{P:period}, but we can also return to the theory of Anderson generating functions for $\rP_n\phi$ and make aspects of Theorem~\ref{T:rathyperquasi0} more explicit.

\begin{corollary} \label{C:hyperQPFAGF}
Let $\phi : \bA \to \Mat_d(K^{\sep}[\tau])$ be a uniformizable, abelian, and $\bA$-finite $t$-module of rank~$r$, and let $\bsy \in (K_{\infty}^{\sep})^d$.  Let $n \geqslant 0$ and set $s = \lfloor \log_q(n) \rfloor + 1$ if $n \geqslant 1$ and $s=0$ otherwise.  Select $\delta_1, \dots, \delta_{r-m}$, $\gamma_1, \dots, \gamma_m \in \Der(\phi,K^{\sep})$ as in \eqref{E:basisSRphi1}, together with corresponding $\delta_v^s$ and $\gamma_i^s$ as in Lemma~\ref{L:degdeRham}. Then the following hold.
\begin{enumerate}
\item[(a)] For $0 \leqslant u \leqslant n$ and $1 \leqslant v \leqslant r-m$,
\[
\pd_{\theta}^u \bigl( \rF_{\delta_v^s}(\bsy) \bigr) =  \Bigl( \bigl\langle (\delta_v^s)_t^{[u]} \bigm| \cG_{\bsy} \bigr\rangle + \bigl\langle (\delta_v^s)_t^{[u-1]} \bigm| \pd_t^1(\cG_{\bsy}) \bigr\rangle + \dots + \bigl\langle (\delta_v^s)_t \bigm| \pd_t^{u}(\cG_{\bsy}) \bigr\rangle \Bigr) \Big|_{t=\theta}.
\]
\item[(b)] For $1 \leqslant i \leqslant m$ and $0 \leqslant e \leqslant n$,
\[
\pd_{\theta}^e \bigl( \rF_{\gamma_i^s}(\bsy) \bigr)=  \Bigl( \bigl\langle (\gamma_i^s)_t^{[e]} \bigm| \cG_{\bsy} \bigr\rangle + \bigl\langle (\gamma_i^s)_t^{[e-1]} \bigm| \pd_t^1(\cG_{\bsy}) \bigr\rangle + \dots + \bigl\langle (\gamma_i^s)_t \bigm| \pd_t^{e}(\cG_{\bsy}) \bigr\rangle \Bigr) \Big|_{t=\theta}.
\]
\end{enumerate}
\end{corollary}

\begin{remark}
Combining this corollary with \eqref{E:perlogquasipro}--\eqref{E:AGFDerpro0}, we obtain explicit descriptions of hyperderivatives of quasi-logarithms on~$\phi$ with respect to $\theta$ of all representatives of $\rH_{\DR}^1(\phi,\oK)$ in terms of hyperderivatives with respect to~$t$ of the associated Anderson generating functions of~$\phi$.
\end{remark}

\begin{proof}[Proof of Corollary~\ref{C:hyperQPFAGF}]
Letting $0 \leqslant u \leqslant n$ and $1 \leqslant v \leqslant r-m$, define $\bsvartheta_{u,v}$, $\bskappa_{e,i} \in \Der_0(\rP_n\phi,K^{\sep})$ as in~\eqref{E:newprosrbasis} and~\eqref{E:newprosrbasis1}. By Proposition~\ref{P:quasigen}(a) and \eqref{E:hyperbFF}, we have
\[
\bigl\langle (\bsvartheta_{u,v})_t \bigm| \fG_{(\bsy)_0} \bigr\rangle \big|_{t=\theta} = \bFF_{\bsvartheta_{u,v}} \bigl((\bsy)_0 \bigr) = \pd_{\theta}^u \bigl(\rF_{\bsdelta_v^s}(\bsy) \bigr),
\]
and by Proposition~\ref{P:qpfunctionpro} and \eqref{E:newprosrbasis}, it follows that
\[
\bigl\langle (\bsvartheta_{u,v})_t \bigm| \fG_{(\bsy)_0} \bigr\rangle = \bigl\langle (\bsdelta_v^s)_t^{[u]} \bigm| \cG_{\bsy} \bigr\rangle + \bigl\langle (\bsdelta_v^s)_t^{[u-1]} \bigm| \pd_t^1(\cG_{\bsy}) \bigr\rangle + \dots + \bigl\langle (\bsdelta_v^s)_t \bigm| \pd_t^u(\cG_{\bsy}) \bigr\rangle.
\]
This proves part~(a).  Part~(b) follows from a similar computation using \eqref{E:newprosrbasis1}.
\end{proof}

\subsection{Examples} \label{subS:Examplespro}

\begin{example}\label{Ex:Drinfeldmodule3}
\emph{Drinfeld modules.} Let $\phi:\bA \to K^{\sep}[\tau]$ be a Drinfeld module defined by
\[
\phi_t = \theta+b_1\tau+\dots+b_r\tau^r,
\]
such that $b_r \neq 0$. Since $b_r \neq 0$, it follows that $\rH^1_{\DR}(\phi, \oK)$ can be identified with the $\oK$-vector space with $\{m \in \cM_{\phi,\oK}\tau \mid \deg_\tau m \leqslant r\}$. Moreover, since $\rN_\phi = 0$, a basis of $\rH^1_{\DR}(\phi, \oK)$ is given by $\{\delta_1, \dots, \delta_r\}$, where $(\delta_1)_t = \phi_t-\theta$ and $(\delta_j)_t=\tau^{j-1}$ for $j=2, \dots, r$. Then since $\rF_{\delta_1}(z) = \Exp_\phi(z)-z$ by \eqref{E:quasiinner}, we see from Theorem~\ref{T:rathyperquasi0}(b) that
\[
\Span{\oK}\bigl(d_{t, n+1}[\Psi_{\phi}]^{-1}\big|_{t=\theta}\bigr) = \Span{\oK} \bigl( \pd_\theta^u(\lambda),\, \pd_\theta^u(\rF_{\delta_j}(\lambda)) \mid 0 \leqslant u \leqslant n, \, 2\leqslant j \leqslant r,\, \lambda \in \Lambda_{\phi} \bigr).
\]
Let $y\in \KK$ such that $\Exp_\phi(y)=\alpha\in K^\sep$, and so by Lemma~\ref{L:separable}, $y \in K^\sep_\infty$. Let $\cG_y$ be the Anderson generating function of $\phi$ with respect to $y$. Then, as we saw in Example~\ref{Ex:Drinfeldmodule2}, by Corollary~\ref{C:AGFtwistedmult} we obtain
\[
\twistop{(\delta_1)_t}{\cG_{y}} = (t-\theta)\cG_y + \alpha,
\]
and for $2 \leqslant j \leqslant r$
\[
\twistop{(\delta_j)_t}{\cG_{y}} = \cG_{y}^{(j-1)}.
\]
We let $n \geqslant 0$, and consider the prolongation $\rP_n\phi$, which is a $t$-module of dimension $n+1$.  By Theorem~\ref{T:rathyperquasi0}(a) and Corollary~\ref{C:rathpqgyhapr0} we see that
\begin{align*}
\Span{\oK} \bigl(\{1\} \cup \fg_{(y)_0}\big|_{t=\theta} \big) &= \Span{\oK}\biggl(\{1\} \cup \bigcup_{u=0}^n \bigcup_{j=2}^r \Bigl\{ \pd_t^u\bigl((t-\theta)\cG_y \bigr)\big|_{t=\theta},\, \pd_t^u \bigl( \cG_{y}^{(j-1)} \bigr)\big|_{t=\theta} \Bigr\} \biggr)\\
&= \Span{\oK} \biggl( \{1\} \cup \bigcup_{u=0}^n \bigcup_{j=2}^r \Bigl\{ \pd_\theta^u(y),\, \pd_\theta^u \bigl( \rF_{\delta_j}(y) \bigr) \Bigr\} \biggr).
\end{align*}
Moreover, \eqref{E:perlogquasipro}--\eqref{E:AGFDerpro0} also imply
\[
y = \alpha- \twistop{(\bsepsilon_{0,1})_t}{\fG_{(y)_0}}|_{t=\theta}
= \alpha- \twistop{t-\theta}{\cG_y}|_{t=\theta}
= \alpha- \twistop{(\delta_1)_t}{\cG_{y}}|_{t=\theta}.
\]
Finally, if $\psi_n$ is the strictly quasi-periodic extension of $\rP_n\phi$ defined in~\S\ref{subS:QuasiExt}, then Proposition~\ref{P:period} determines its period lattice $\Lambda_{\psi_n}$.  In particular for $\bslambda = (\lambda_0, \dots, \lambda_n)^{\tr} \in \Lambda_{\rP_n\phi} = (\Lambda_{\phi})^{n}$, we have the general form for the corresponding lattice element of $\bsxi(\bslambda) \in \Lambda_{\psi_n}$, and by Corollary~\ref{C:quasilog} if we take $(\lambda)_0 = (\lambda, 0, \dots, 0)^{\tr} \in \Lambda_{\rP_n \phi}$, we have a simplified expression for $\bsxi((\lambda)_0)$:
\[
\bsxi(\bslambda)= \left(\begin{array}{c}
  -\sum_{w=0}^n \pd_{\theta}^w(\rF_{\tau}(\lambda_{n-w}))\\
  \vdots\\
  -\sum_{w=0}^u \pd_{\theta}^w(\rF_{\tau}(\lambda_{u-w}))\\
  \vdots\\
  -\rF_{\tau}(\lambda_0)\\[5pt]\hline \\[-10pt]
  \vdots\\[5pt]\hline\\[-10pt]
  -\sum_{w=0}^n \pd_{\theta}^w(\rF_{\tau^{r-1}}(\lambda_{n-w}))\\
  \vdots\\
  -\sum_{w=0}^u \pd_{\theta}^w(\rF_{\tau^{r-1}}(\lambda_{u-w}))\\
  \vdots \\
  -\rF_{\tau^{r-1}}(\lambda_0)\\[5pt]\hline \\[-10pt]
  \sum_{w=1}^n \pd_{\theta}^w(\lambda_{n-w})\\
  \vdots\\
  \lambda_2 + \pd_{\theta}^1(\lambda_1) + \pd_{\theta}^2(\lambda_0)  \\
  \lambda_1 + \pd_{\theta}^1(\lambda_0)  \\[5pt]\hline\\[-10pt]
  \bslambda
\end{array} \right),
\quad \bsxi((\lambda)_0) = \left( \begin{array}{c}
  -\pd_{\theta}^n(\rF_{\tau}(\lambda))\\
  \vdots\\
  -\pd_{\theta}^{1}(\rF_{\tau}(\lambda))\\
  -\rF_{\tau}(\lambda)\\[5pt]\hline\\[-10pt]
  \vdots\\[5pt]\hline\\[-10pt]
  -\pd_{\theta}^n(\rF_{\tau^{r-1}}(\lambda))\\
  \vdots\\
  -\pd_{\theta}^1(\rF_{\tau^{r-1}}(\lambda))\\
  -\rF_{\tau^{r-1}}(\lambda)\\[5pt]\hline\\[-10pt]
  \pd_{\theta}^n(\lambda)  \\
  \vdots\\
  \pd_{\theta}^1(\lambda)\\[5pt]\hline\\[-10pt]
  \lambda\\
  0\\
  \vdots\\
  0
\end{array} \right).
\]
\end{example}

\begin{example}\label{Ex:Carlitztensor3}
\emph{Carlitz tensor powers.}
We continue with the considerations of Example~\ref{Ex:Carlitztensor2}, and let $\phi : \bA \to \Mat_{d}(K^\sep[\tau])$ be the $d$-th tensor power of the Carlitz module defined by
\[
  \phi_t = \begin{pmatrix}
   \theta & 1 & \cdots & 0 \\
    & \ddots & \ddots & \vdots \\
    & & \ddots & 1 \\
    & & & \theta
  \end{pmatrix}
  + \begin{pmatrix}
  0 & \cdots & \cdots & 0 \\
  \vdots & & & \vdots \\
  \vdots & & & \vdots \\
  1 & \cdots & \cdots & 0
  \end{pmatrix}
  \tau.
\]
The $t$-module $\phi$ has rank $1$ and $\rH^1_{\DR}(\phi,\oK) \cong \Der_{0}(\phi, \oK)$. Using the definition of $\rN_\phi^\perp$ from \S\ref{subS:biderivations}, we see that $\bss_d$ generates $\rN_\phi^\perp$ and is defined over $K^\sep$.  It follows that, if we take $\bsdelta \in \Der_0(\phi)$ to be defined by $\bsdelta_t \assign  \bss_{d}\phi_{t}-\theta \bss_{d}=(t-\theta)\bss_d$, then $\bsdelta$ represents a $\oK$-basis of $\rH^1_{\DR}(\phi, \oK)\cong \Der_{0}(\phi, \oK) \cong (t-\theta)\rN_{\phi}^\perp$.

We recall from Example~\ref{Ex:Carlitztensor2} that $\Psi = (-1)^d\Omega^d$ is a rigid analytic trivialization of $\phi$, and we let $\Pi = \cE_0((-1)^d\Omega^{-d})$ be the generator of its period lattice as in~\eqref{E:Pidef}. Since $\rF_{\bsdelta}(\bsz) = \bss_d(\Exp_\phi(\bsz)-\bsz)$ by \eqref{E:quasiinner}, we have $\rF_{\bsdelta}(\Pi)=-\tpi^d$. It thus follows from Theorem~\ref{T:rathyperquasi0}(b) that
\begin{align*}
\Span{\oK}\bigl( d_{t, n+1} \bigl[ (-1)^d\Omega^{-d} \bigr]^{-1} \big|_{t=\theta} \bigr) &= \Span{\oK} \bigl( \pd_\theta^u(\bss_d(\Pi)) : 0 \leqslant u \leqslant n \bigr) \\
&= \Span{\oK} \bigl( \pd_\theta^u( \tpi^d) : 0 \leqslant u \leqslant n \bigr).
\end{align*}
These identities for tensor powers of the Carlitz module were observed by Maurischat~\cite[\S 8]{Maurischat18}, but what we note further is that the quantities in Theorem~\ref{T:rathyperquasi0}(b) consist of the $\oK$-span of hyperderivatives of the $d$-th power of the Carlitz period with respect to~$\theta$, as shown in~\cite{Maurischat21b}.

Now let $\bsy\in \KK^d$ be chosen so that $\Exp_\phi(\bsy)=\bsalpha \in (K^\sep)^d$, and by Lemma~\ref{L:separable} we have $\bsy \in (K^\sep_\infty)^d$. Let $\cG_{\bsy}$ be the Anderson generating function of $\phi$ with respect to $\bsy$. By Corollary~\ref{C:AGFtwistedmult}, as in Example~\ref{Ex:Carlitztensor2} we have
\[
\twistop{\bsdelta_t}{\cG_{\bsy}}= (t-\theta)\bss_{d}\cG_{\bsy} + \bss_{d}(\bsalpha),
\]
and
\[
\rF_{\delta}(\bsy) = \bss_{d}(\bsalpha-\bsy).
\]
Thus Theorem~\ref{T:rathyperquasi0}(a) and Corollary~\ref{C:rathpqgyhapr0} imply
\begin{align*}
\Span{\oK} \bigl( \{1\} \cup \fg_{(\bsy)_0}\big|_{t=\theta} \bigr) &= \Span{\oK} \biggl(\{1\} \cup\bigcup_{u=0}^n \Bigl\{ \pd_t^u \bigl( (t-\theta)\bss_{d}\cG_{\bsy} \bigr) \big|_{t=\theta} \Bigr\} \biggr)\\
&= \Span{\oK} \bigl(\{1\} \cup \pd_\theta^u(y_d): 0 \leqslant u \leqslant n \bigr),
\end{align*}
where $\bsy = (y_1, \dots, y_d)^{\tr}$.  Assume that $n \geqslant d-1$. Then it follows from \eqref{E:perlogquasipro}--\eqref{E:AGFDerpro0} that
\[
\bsy = \bsalpha- \left.\begin{pmatrix}
\twistop{(\bsepsilon_{d-1,d})_t}{\fG_{(\bsy)_0}} \\
\vdots\\
\twistop{(\bsepsilon_{1,d})_t}{\fG_{(\bsy)_0}}\\
\twistop{(\bsepsilon_{0,d})_t}{\fG_{(\bsy)_0}}
\end{pmatrix}\right|_{t=\theta} = \bsalpha - \left.\begin{pmatrix}
\pd_t^{d-1}(\twistop{\tau\bss_1}{\cG_{\bsy}}) \\
\vdots\\
\pd_t^{1}(\twistop{\tau\bss_1}{\cG_{\bsy}})\\
\twistop{\tau\bss_1}{\cG_{\bsy}}
\end{pmatrix}\right|_{t=\theta}.
\]
Moreover, if $\psi_n$ is the strictly quasi-periodic extension of $\rP_n\phi$, then by Proposition~\ref{P:period} its period lattice is
\[
\Lambda_{\psi_n} =  \left\{\begin{pmatrix}
  \sum_{w=0}^n \bss_d(\pd_{\theta}^w(\bslambda_{n-w}))\\
  \vdots\\
  \sum_{w=0}^{d+1} \bss_d(\pd_{\theta}^w(\bslambda_{d+1-w})) \\
  \sum_{w=0}^d \bss_d(\pd_{\theta}^w(\bslambda_{d-w}))\\
  \bslambda
\end{pmatrix} \Biggm|
\bslambda = \bigl( \bslambda_0^{\tr}, \dots, \bslambda_n^{\tr} \bigr)^\tr \in \Lambda_{\rP_n \phi} = (\Lambda_\phi)^d
\right\}.
\]
In particular, if $\bslambda_0 = (\bslambda^{\tr}, 0, \dots, 0)^{\tr} \in \Lambda_{\rP_n \phi}$ where $\bslambda = (\lambda_1, \dots, \lambda_d)^{\tr} \in \Lambda_{\phi}$ or if $\Pi_0 = (\Pi^{\tr}, 0, \dots, 0)^{\tr}$, then Corollary~\ref{C:quasilog} implies
\[
\begin{pmatrix}
  \pd_{\theta}^n(\lambda_d)  \\
  \vdots\\
  \pd_{\theta}^d(\lambda_d)\\
  \bslambda_0
\end{pmatrix},
\quad
\begin{pmatrix}
  \pd_{\theta}^n(\tpi^d) \\
  \vdots \\
  \pd_{\theta}^d(\tpi^d) \\
  \Pi_0
\end{pmatrix}
\]
are periods in $\Lambda_{\psi_n}$.
\end{example}

\begin{example}\label{Ex:Strictlypure3}
\emph{Strictly pure $t$-modules.}
Recall notation and results on strictly pure $t$-modules from Examples~\ref{Ex:strictlypure1} and~\ref{Ex:strictlypure2}. Let $\phi : \bA \to \Mat_d(K^\sep[\tau])$ be a $t$-module defined by
\[
  \phi_t = \rd \phi_t + B_1 \tau + \cdots + B_{\ell} \tau^{\ell}, \quad B_i \in \Mat_d(K^{\sep}),\ \det B_{\ell} \neq 0,
\]
where $\rd\phi_t =  \theta I_d +\rN_\phi$ is in Jordan normal form, i.e.,
\begin{equation}\label{E:Jordanpro}
    \rd \phi_t = \begin{pmatrix}
d_{\theta,\ell_1}[\theta] & &  \\
 & \ddots & \\
& &  d_{\theta,\ell_m}[\theta] \end{pmatrix}
\end{equation}
with $\ell_1+ \dots+ \ell_m = d$. Since $\det B_\ell \neq 0$, much like for Drinfeld modules it follows that $\rH^1_{\DR}(\phi, \oK)$ can be identified canonically with
\begin{equation}
\rH_{\DR}^1(\phi,\oK) \cong \{(m_1, \dots, m_d) \in \Mat_{1 \times d}(\oK[\tau]\tau) : \deg_\tau m_i \leqslant \ell,\, 1 \leqslant i \leqslant d\}.
\end{equation}
Recalling $d_k \assign \ell_1+\dots+\ell_k$ with $d_m = d$, the definition of $\rN_\phi^\perp$ from \S\ref{subS:biderivations} implies that $\{\bss_{d_1}, \dots, \bss_{d_m} \}$ is a basis of $\rN_\phi^\perp$ defined over $K^\sep$. For $1\leqslant i \leqslant d$ and $1 \leqslant j \leqslant \ell$, we define $\bsdelta_{i,j} \in \Der(\phi,K^{\sep})$ by
\[
(\bsdelta_{i,j})_t \assign \bss_i \tau^j, \quad 1 \leqslant j \leqslant \ell-1,
\]
and
\[
(\bsdelta_{i,\ell})_t \assign \begin{cases}
  \bss_i \tau^\ell & \textup{if $i \neq d_1, \dots, d_m$,} \\
  \bss_i \phi_t - \theta \bss_i = (t-\theta)\cdot \bss_i & \textup{otherwise.}
\end{cases}
\]
It follows that $\{\bsdelta_{i,j}: 1 \leqslant i \leqslant d,\ 1 \leqslant j \leqslant \ell\}$ represents a $\oK$-basis of $\rH^1_{\DR}(\phi, \oK)$. Note that \eqref{E:quasiinner} implies
\[
\rF_{\bsdelta_{d_k,\ell}}(\bsz) = \bss_{d_k}(\Exp_\phi(\bsz) - \bsz).
\]
For $n \geqslant 0$, Theorem~\ref{T:rathyperquasi0}(b) implies
\[
\Span{\oK} \bigl( d_{t, n+1}[\Psi_{ \phi}]^{-1} \big|_{t=\theta} \bigr) = \Span{\oK} \biggl( \bigcup_{u=0}^n \bigcup_{i=1}^{d}\bigcup_{j=1}^{\ell} \bigl\{ \pd_\theta^u(\rF_{\delta_{i,j}}(\bslambda)) : \bslambda \in \Lambda_\phi \bigr\} \biggr),
\]
and we note that $\rF_{\bsdelta_{d_k,\ell}}(\bslambda) = -\bss_{d_k} \cdot \bslambda$ for $1 \leqslant k \leqslant m$ thus producing the tractable coordinates of~$\bslambda$.

Fix $\bsy\in \KK^d$ such that $\Exp_\phi(\bsy) =\bsalpha \in (K^\sep)^d$, and so by Lemma~\ref{L:separable} we have $\bsy \in (K^\sep_\infty)^d$. Let $\cG_{\bsy}$ be the Anderson generating function of $\phi$ with respect to $\bsy$. By Theorem~\ref{T:rathyperquasi0}(a) and Corollary~\ref{C:rathpqgyhapr0} we obtain
\begin{align} \label{E:strictlypurespan}
\Span{\oK} \bigl(\{1\} \cup \fg_{(\bsy)_0}\big|_{t=\theta} \bigr) &= \Span{\oK} \biggl( \{ 1 \} \cup \bigcup_{u=0}^n \bigcup_{i=1}^d \bigcup_{j=1}^\ell \Bigl\{  \twistop{(\bsdelta_{i,j})_t}{\pd_t^u(\cG_{\bsy})}\big|_{t=\theta} \Bigr\} \biggr) \\
&= \Span{\oK} \biggl( \{1 \} \cup  \bigcup_{u=0}^n\bigcup_{i=1}^d \bigcup_{j=1}^\ell \bigl\{ \pd_\theta^u(\rF_{\bsdelta_{i,j}}(\bsy)) \bigr\} \biggr). \notag
\end{align}
However, we can be more precise.  For $1 \leqslant k \leqslant m$ we note that
\[
\twistop{(\bsdelta_{d_k,\ell})_t}{\cG_{\bsy}} = (t-\theta)\cdot \bss_{d_k}\cdot \cG_{\bsy} + \bss_{d_k} \cdot \bsalpha,
\]
and so if $\cG_{\bsy} = (g_1, \dots, g_d)^{\tr} \in \TT^d$ and $\bsy = (y_1, \dots, y_d)^{\tr}$, $\bsalpha = (\alpha_1, \dots, \alpha_d) \in \KK^d$, it follows from Proposition~\ref{P:AGFres} that
\[
\rF_{\bsdelta_{d_k,\ell}}(\bsy) = \twistop{(\bsdelta_{d_k,\ell})_t}{\cG_{\bsy}}|_{t=\theta} = -y_{d_k} + \alpha_{d_k}.
\]
For the other cases, which are attached to strictly reduced biderivations, we have
\[
\rF_{\bsdelta_{i,j}}(\bsy) = \twistop{(\bsdelta_{i,j})_t}{\cG_{\bsy}}|_{t=\theta} = \bss_i \cG_{\bsy}^{(j)} \big|_{t=\theta} = g_i^{(j)}(\theta).
\]

Finally, choose $n\geqslant 0$ so that $n \geqslant \ell_k-1$ for all $1 \leqslant k \leqslant m$.   For $1 \leqslant k \leqslant m$, we recall $\gamma_k \in \Der(\phi,\oK)$ from~\eqref{E:der0basis} defined by
\[
  (\gamma_k)_t = (t-\theta) \cdot \bss_{d_k} \rQ^{-1} = (t-\theta) \cdot \bss_{d_k} = (\bsdelta_{d_k,\ell})_t,
\]
where the second equality follows from $\rd\phi_t$ being in Jordan normal form ($\rQ = \Id_d$) and the third from the definition of $\bsdelta_{d_k,\ell}$.  For $0 \leqslant w\leqslant \ell_k-1$, we calculate $\twistop{(\bsepsilon_{w,k})_t}{\fG_{(\bsy)_0}}|_{t=\theta}$ as in \eqref{E:perlogquasipro}.  By \eqref{E:AGFDerpro0}, we see as above that
\begin{equation} \label{E:eps0ktest}
  \bigl\langle (\bsepsilon_{0,k})_t \bigm| \fG_{(\bsy)_0} \bigr\rangle\big|_{t=\theta} = \twistop{(\gamma_k)_t}{\cG_{\bsy}}|_{t=\theta}= y_{d_k} - \alpha_{d_k}.
\end{equation}
For $1 \leqslant j \leqslant w \leqslant \ell_k-1$, the calculation in \eqref{E:AGFDerpro} leads us to investigate
\begin{align*}
\bigl\langle (t &{}-\theta) \cdot \bss_{d_k-j} - \bss_{d_k-(j-1)} \bigm| \pd_t^{w-j}(\cG_{\bsy}) \bigr\rangle \\
&= \pd_t^{w-j} \bigl( \bigl\langle (t-\theta) \cdot \bss_{d_k-j} - \bss_{d_k-(j-1)}\bigm| \cG_{\bsy} \bigr\rangle \bigr) \\
&= \pd_t^{w-j} \bigl( (t-\theta) \cdot \bss_{d_k-j}\cdot  \cG_{\bsy} + \bss_{d_k-j}\cdot \bsalpha - \bss_{d_k-(j-1)} \cdot \cG_{\bsy} \bigr), \\
\intertext{where this last equality follows from Corollary~\ref{C:AGFtwistedmult}, and furthermore,}
&= \begin{cases}
(t-\theta) \bss_{d_k-w} \cdot \cG_{\bsy} + \bss_{d_k-w}\cdot \bsalpha - \bss_{d_k-w+1} \cdot \cG_{\bsy}
& \textup{if $j=w$,} \\
(t-\theta) \bss_{d_k-j}\cdot \pd_t^{w-j}(\cG_{\bsy}) + \bss_{d_k-j}\cdot \pd_t^{w-j-1}(\cG_{\bsy}) - \bss_{d_k-j+1} \cdot \pd_t^{w-j}(\cG_{\bsy}) & \textup{if $j < w$.}
\end{cases}
\end{align*}
Likewise,
\[
\bigl\langle (t-\theta) \cdot \bss_{d_k} \bigm| \pd_t^w(\cG_{\bsy}) \bigr\rangle = (t-\theta) \bss_{d_k}\cdot \pd_t^{w}(\cG_{\bsy}) + \bss_{d_k} \cdot \pd_t^{w-1}(\cG_{\bsy}).
\]
Furthermore, for $0 \leqslant j, j' \leqslant w$, the proof of Proposition~\ref{P:AGFres} and \eqref{E:giprinparts} yield
\begin{equation} \label{E:Oofone1}
  \bss_{d_k-j} \cdot \cG_{\bsy} = g_{d_k-j} = -\sum_{h=0}^{j} y_{d_k-h} (t-\theta)^{h-j-1} + O(1),
\end{equation}
where $O(1)$ represents $O((t-\theta)^0)$ in $\laurent{\KK}{t-\theta}$.  Therefore,
\begin{equation} \label{E:Oofone2}
  \bss_{d_k-j} \cdot \pd_t^{w-j'}(\cG_{\bsy}) = -\sum_{h=0}^{j} y_{d_k-h} \binom{h-j-1}{w-j'} (t-\theta)^{h-j-1-w+j'} + O(1),
\end{equation}
and so for $1 \leqslant j \leqslant w-1$,
\begin{align} \label{E:Oofone3}
(t-\theta)\bss_{d_k-j} \cdot {}&{} \pd_t^{w-j}(\cG_{\bsy}) + \bss_{d_k-j} \cdot \pd_t^{w-j-1}(\cG_{\bsy}) - \bss_{d_k-j+1}\cdot \pd_t^{w-j}(\cG_{\bsy}) \\
&= \begin{aligned}[t]
-\sum_{h=0}^j y_{d_k-h} &\biggl( \binom{h-j-1}{w-j} + \binom{h-j-1}{w-j-1} \biggr)(t-\theta)^{h-w} \\
&{} + \sum_{h=0}^{j-1} y_{d_k-h} \binom{h-j}{w-j}(t-\theta)^{h-w} + O(1)
\end{aligned} \notag
\\
&= O(1). \notag
\end{align}
Similarly,
\begin{multline} \label{E:Oofone4}
(t-\theta) \bss_{d_k}\cdot \pd_t^{w}(\cG_{\bsy}) + \bss_{d_k} \cdot \pd_t^{w-1}(\cG_{\bsy}) \\
= -y_{d_k} \biggl( \binom{-1}{w} + \binom{-1}{w-1} \biggr) (t-\theta)^{-w} + O(1)
= O(1),
\end{multline}
and finally,
\begin{equation} \label{E:Oofone5}
(t-\theta) \bss_{d_k-w} \cdot \cG_{\bsy} + \bss_{d_k-w}\cdot \bsalpha - \bss_{d_k-w+1} \cdot \cG_{\bsy}
= \alpha_{d_k-w} - y_{d_k-w} + O(1).
\end{equation}
Now this $O(1)$ constant seems to render this last identity meaningless, but it turns out that each of the $O(1)$ quantities in \eqref{E:Oofone1}--\eqref{E:Oofone5} is actually $O((t-\theta)^1)$. Indeed for $1 \leqslant w\leqslant \ell_k-1$, combining the calculations above with \eqref{E:perlogquasipro} and~\eqref{E:AGFDerpro}, we have
\begin{align*}
\alpha_{d_k-w} &{}- y_{d_k-w} \\
&= \bigl\langle (\bsepsilon_{w,k})_t \bigm| \fG_{(\bsy)_0} \bigr\rangle\big|_{t=\theta} \\
&= \begin{aligned}[t]
\sum_{j=1}^{w-1} &\Bigl( (t-\theta) \bss_{d_k-j} \cdot \pd_t^{w-j}(\cG_{\bsy}) + \bss_{d_k-j}\cdot \pd_t^{w-j-1}(\cG_{\bsy}) - \bss_{d_k-j+1} \cdot \pd_t^{w-j}(\cG_{\bsy}) \Bigr)\Big|_{t=\theta} \\
&{}+ \bigl( (t-\theta) \bss_{d_k}\cdot \pd_t^{w}(\cG_{\bsy}) + \bss_{d_k} \cdot \pd_t^{w-1}(\cG_{\bsy}) \bigr)\big|_{t=\theta} \\
&{}+ \bigl( (t-\theta) \bss_{d_k-w} \cdot \cG_{\bsy} + \bss_{d_k-w}\cdot \bsalpha - \bss_{d_k-w+1} \cdot \cG_{\bsy} \bigr)\big|_{t=\theta}.
\end{aligned}
\end{align*}
Thus this identity together with \eqref{E:eps0ktest} provides a complete description of the coordinates of $\bsy$ and $\bsalpha$ in terms of derivatives of Anderson generating functions with respect to~$t$.
\end{example}

\begin{remark}
It would be interesting to have similar descriptions for these findings for almost strictly pure $t$-modules, but as indicated in Example~\ref{Ex:ASP} their behavior can be more subtle and complicated.
\end{remark}


\newpage
\begin{thebibliography}{99}

\bibitem{And86} %
G. W. Anderson, \textit{$t$-motives}, Duke Math. J. \textbf{53} (1986), no. 2, 457--502.

\bibitem{ABP04} %
G. W. Anderson, W. D. Brownawell, and M. A. Papanikolas, \textit{Determination of the algebraic relations among special $\Gamma$-values in positive characteristic}, Ann. of Math. (2) \textbf{160} (2004), no.~1, 237--313.

\bibitem{AndThak90} %
G. W. Anderson and D. S. Thakur, \textit{Tensor powers of the Carlitz module and zeta values}, Ann. of Math. (2) \textbf{132} (1990), no.~1, 159--191.

\bibitem{AnglesNgoDacTavares20} %
B. Angl\`{e}s, T. Ngo Dac, and F. Tavares Ribeiro, \textit{On special $L$-values of $t$-modules}, Adv. Math. \textbf{372} (2020), Art.~107313, 33~pp.

\bibitem{AnglesPellarin15} %
B. Angl\`{e}s and F. Pellarin, \textit{Universal Gauss-Thakur sums and $L$-series}, Invent. Math. \textbf{200}
(2015), no.~2, 653--669.

\bibitem{BosserPellarin08} %
V. Bosser and F. Pellarin, \textit{Hyperdifferential properties of Drinfeld quasi-modular forms}, Int. Math. Res. Not. IMRN (2008), Art. ID rnn032, 56 pp.

\bibitem{BosserPellarin09} %
V. Bosser and F. Pellarin, \textit{On certain families of Drinfeld quasi-modular forms}, J. Number Theory \textbf{129} (2009), no.~12, 2952--2990.

\bibitem{Brownawell93} %
W.~D. Brownawell, \textit{Algebraic independence of Drinfeld exponential and quasi-periodic functions}, in: Advances in Number Theory (Kingston, ON, 1991), Oxford Univ. Press, New York, 1993, pp.~341--365.

\bibitem{Brownawell96} %
W.~D. Brownawell, \textit{Submodules of products of quasi-periodic modules}, Rocky Mountain J. Math. \textbf{26} (1996), no.~3, 847--873.

\bibitem{Brownawell98} %
W.~D. Brownawell, \textit{Transcendence in positive characteristic}, in: Number Theory (Tiruchirapalli, 1996), Amer. Math. Soc., Providence, 1998, pp.~317--332.

\bibitem{Brownawell99} %
W. D. Brownawell, \textit{Linear independence and divided derivatives of a Drinfeld module. I.} in: Number Theory in Progress, Vol.~1 (Zakopane-Ko\'{s}cielisko, 1997), de Gruyter, Berlin, 1999, pp.~47--61.

\bibitem{Brownawell01} %
W. D. Brownawell, \textit{Minimal extensions of algebraic groups and linear independence}, J. Number Theory \textbf{90} (2001), no.~2, 239--254.

\bibitem{BrownawellDenis00} %
W. D. Brownawell and L. Denis, \textit{Linear independence and divided derivatives of a Drinfeld module. II.} Proc. Amer. Math. Soc. \textbf{128} (2000), no.~6, 1581--1593.

\bibitem{BP02} %
W.~D. Brownawell and M.~A. Papanikolas, \textit{Linear independence of Gamma-values in positive characteristic}, J. Reine Angew. Math. \textbf{549} (2002), 91--148.

\bibitem{BPrapid} %
W. D. Brownawell and M. A. Papanikolas, \textit{A rapid introduction to Drinfeld modules, $t$-modules, and $t$-motives}, in: $t$-Motives: Hodge Structures, Transcendence and Other Motivic Aspects, Eur. Math. Soc., Z\"urich, 2020, pp.~3--30.

\bibitem{Chang14} %
C.-Y. Chang, \textit{Linear independence of monomials of multizeta values in positive characteristic}, Compos. Math. \textbf{150} (2014), no.~11, 1789--1808.

\bibitem{ChangGreenMishiba21} %
C.-Y. Chang, N. Green, and Y. Mishiba, \textit{Taylor coefficients of Anderson-Thakur series and explicit formulae}, Math. Ann. \textbf{379} (2021), no.~3--4, 1425--1474.

\bibitem{ChangMishiba19} %
C.-Y. Chang and Y. Mishiba, \textit{On multiple polylogarithms in characteristic $p$: $v$-adic vanishing versus $\infty$-adic Eulerianness}, Int. Math. Res. Not. IMRN (2019), no.~3, 923--947.

\bibitem{ChangMishiba21} %
C.-Y. Chang and Y. Mishiba, \textit{On a conjecture of Furusho over function fields}, Invent. Math. \textbf{223} (2021), no.~1, 49--102.

\bibitem{CP11} %
C.-Y. Chang and M. A. Papanikolas, \textit{Algebraic relations among periods and logarithms of rank~$2$ Drinfeld modules}, Amer. J. Math. \textbf{133} (2011), no.~2, 359--391.

\bibitem{CP12} %
C.-Y. Chang and M. A. Papanikolas, \textit{Algebraic independence of periods and logarithms of Drinfeld modules. With an appendix by B.~Conrad}, J. Amer. Math. Soc. \textbf{25} (2012), no.~1, 123--150.

\bibitem{CPTY10} %
C.-Y. Chang, M. A. Papanikolas, D. S. Thakur, and J. Yu, \textit{Algebraic independence of arithmetic gamma values and Carlitz zeta values}, Adv. Math. \textbf{223} (2010), no.~4, 1137--1154.

\bibitem{CPY10} %
C.-Y. Chang, M. A. Papanikolas, and J. Yu, \textit{Geometric gamma values and zeta values in positive characteristic}, Int. Math. Res. Not. IMRN (2010), no.~8, 1432--1455.

\bibitem{CPY11} %
C.-Y. Chang, M. A. Papanikolas, and J. Yu, \textit{Frobenius difference equations and algebraic independence of zeta values in positive equal characteristic}, Algebra Number Theory \textbf{5} (2011), no.~1, 111--129.

\bibitem{CPY19} %
C.-Y. Chang, M. A. Papanikolas, and J. Yu, \textit{An effective criterion for Eulerian multizeta values in positive characteristic}, J. Eur. Math. Soc. (JEMS) \textbf{21} (2019), no.~2, 405--440.

\bibitem{ChangYu07} %
C.-Y. Chang and J. Yu, \textit{Determination of algebraic relations among special zeta values in positive characteristic}, Adv. Math. \textbf{216} (2007), no.~1, 321--345.

\bibitem{Conrad00} %
K. Conrad, \textit{The digit principle}, J. Number Theory \textbf{84} (2000), no.~2, 230--257.

\bibitem{DavidDenis02} %
S. David and L. Denis, \textit{P\'{e}riodes de modules de Drinfeld ``L'Ind\'{e}pendance quadratique en rang II,''} J. Ramanujan Math. Soc. (2002) \textbf{17}, no.~1, 65--83.

\bibitem{Denis92} %
L. Denis, \textit{Hauteurs canoniques et modules de Drinfeld}, Math. Ann. \textbf{294} (1992), no.~2, 213--223.

\bibitem{Denis93} %
L. Denis, \textit{Transcendence et d\'{e}riv\'{e}es de l'exponentielle de Carlitz}, in: S\'{e}minaire de Th\'{e}orie des Nombres, Paris, 1991--92, Birkh\"{a}user, Boston, MA, 1993, pp.~1--21.

\bibitem{Denis95} %
L. Denis, \textit{D\'{e}riv\'{e}es d'un module de Drinfeld et transcendance}, Duke Math. J. \textbf{80} (1995), no.~1, 1--13.

\bibitem{Denis97} %
L. Denis, \textit{Ind\'{e}pendance alg\'{e}brique en caract\'{e}ristique deux}, J. Number Theory \textbf{66} (1997), no.~1, 183--200.

\bibitem{Denis00} %
L. Denis, \textit{Ind\'{e}pendance alg\'{e}brique des d\'{e}riv\'{e}es d'une p\'{e}riode du module de Carlitz}, J. Austral. Math. Soc. Ser. A \textbf{69} (2000), no.~1, 8--18.

\bibitem{Denis06} %
L. Denis, \textit{Ind\'{e}pendance alg\'{e}brique de logarithmes en caract\'{e}ristique $p$}, Bull. Austral. Math. Soc. \textbf{74} (2006), no.~3, 461--470.

\bibitem{EP14} %
A. El-Guindy and M. A. Papanikolas, \textit{Identities for Anderson generating functions for Drinfeld modules}, Monatsh. Math. \textbf{173} (2014), no. 3--4, 471--493.

\bibitem{FresnelvdPut} %
J. Fresnel and M. van der Put, \textit{Rigid Analytic Geometry and its Applications}, Birkh\"{a}user, Boston, 2004.

\bibitem{Gekeler89a} %
E.-U. Gekeler, \textit{On the de Rham isomorphism for Drinfeld modules}, J. Reine Angew. Math. \textbf{401} (1989), 188--208.

\bibitem{Gekeler89b} %
E.-U. Gekeler, \textit{Quasi-periodic functions and Drinfeld modular forms}, Compositio Math. \textbf{69} (1989), no.~3, 277--293.

\bibitem{Gekeler90} %
E.-U. Gekeler, \textit{de Rham cohomology and the Gauss-Manin connection for Drinfeld modules}, in: $p$-adic Analysis (Trento, 1989), Lect. Notes Math., vol.~1454, Springer, Berlin, 1990, pp.~223--255.

\bibitem{GezmisP19} %
O. Gezmi\c{s} and M. A. Papanikolas, \textit{The de Rham isomorphism for Drinfeld modules over Tate algebras}, J. Algebra \textbf{525} (2019), 454--496.

\bibitem{Goss94} %
D. Goss, \textit{Drinfeld modules: cohomomology and special functions}, in: Motives (Seattle, WA, 1991), Proc. Sympos. Pure Math., vol.~55, Part 2, Amer. Math. Soc., Providence, RI, 1994, pp.~309--362.

\bibitem{Goss} %
D. Goss, \textit{Basic Structures of Function Field Arithmetic}, Springer-Verlag, Berlin, 1996.

\bibitem{Green19a} %
N. Green, \textit{Special zeta values using tensor powers of Drinfeld modules}, Math. Res. Lett. \textbf{26} (2019), no.~6, 1629--1676.

\bibitem{Green22} %
N. Green, \textit{Tensor powers of rank~$1$ Drinfeld modules and periods}, J. Number Theory \textbf{232} (2022), 204--241.

\bibitem{GreenNgoDac20a} %
N. Green and T. Ngo Dac, \textit{Algebraic relations among Goss's zeta values on elliptic curves}, arXiv:2004:08461, 2020.

\bibitem{GreenNgoDac20b} %
N. Green and T. Ngo Dac, \textit{On log-algebraic identities for Anderson $t$-modules and characteristic $p$ multiple zeta values}, arXiv:2007.11060, 2020.

\bibitem{GreenP18} %
N. Green and M. A. Papanikolas, \textit{Special $L$-values and shtuka functions for Drinfeld modules on elliptic curves}, Res. Math. Sci. \textbf{5} (2018), 5:4, 47 pp.

\bibitem{Hamahata93} %
Y. Hamahata, \textit{Tensor products of Drinfeld modules and $v$-adic representations}, Manuscripta Math. \textbf{79} (1993), no.~3--4, 307--327.

\bibitem{HartlJuschka20} %
U. Hartl and A.-K. Juschka, \textit{Pink's theory of Hodge structures and the Hodge conjecture over function fields}, in: $t$-Motives: Hodge Structures, Transcendence and Other Motivic Aspects, Eur. Math. Soc., Z\"urich, 2020, pp.~31--182.

\bibitem{JacobsonI} %
N. Jacobson, \textit{Basic Algebra. I.}, 2nd ed., Dover, Mineola, NY, 2009.

\bibitem{Jeong00} %
S. Jeong, \textit{A comparison of the Carlitz and digit derivative bases in function field arithmetic}, J. Number Theory \textbf{84} (2000), no.~2, 258--275.

\bibitem{Jeong11} %
S. Jeong, \textit{Calculus in positive characteristic $p$}, J. Number Theory \textbf{131} (2011), no.~6, 1089--1104.

\bibitem{Juschka10} %
A.-K. Juschka, \textit{The Hodge conjecture for function fields}, Diploma Thesis, Universit\"{a}t M\"{u}nster, 2010.

\bibitem{LP13} %
B. A. Lutes and M. A. Papanikolas, \textit{Algebraic independence of values of Goss $L$-functions at $s=1$}, J. Number Theory \textbf{133} (2013), no.~3, 1000--1011.

\bibitem{Maurischat18} %
A. Maurischat, \textit{Prolongations of $t$-motives and algebraic independence of periods}, Doc. Math. \textbf{23} (2018), 815--838.

\bibitem{Maurischat19a} %
A. Maurischat, \textit{Anderson $t$-modules with thin $t$-adic Galois representations}, Proc. Amer. Math. Soc. (to appear), https://doi.org/10.1090/proc/15815, arXiv:1907.05144, 2019.

\bibitem{Maurischat21a} %
A. Maurischat, \textit{Abelian equals $\mathbf{A}$-finite for Anderson $\mathbf{A}$-modules}, arXiv:2110.11114, 2021.

\bibitem{Maurischat21b} %
A. Maurischat, \textit{Algebraic independence of the Carlitz period and its hyperderivatives}, arXiv:2104.02630, 2021.

\bibitem{Maurischat22} %
A. Maurischat, \textit{Periods of $t$-modules as special values}, J. Number Theory \textbf{232} (2022), 177--203.

\bibitem{NPhD} %
C. Namoijam, \textit{Hyperderivatives of periods and logarithms of Anderson $t$-modules and algebraic independence}, Ph.D.~thesis, Texas A\&M University, 2020.

\bibitem{N21} %
C. Namoijam, \textit{Algebraic relations among hyperderivatives of periods and logarithms of Drinfeld modules}, arXiv:2103.09485, 2021.

\bibitem{Okugawa} %
K. Okugawa, \textit{Differential algebra of nonzero characteristic}, Kinokuniya Co. Ltd., Tokyo, 1987.

\bibitem{Ore33} %
O. Ore, \textit{On a special class of polynomials}, Trans. Amer. Math. Soc. \textbf{35} (1933), no.~3, 559--584.

\bibitem{P08} %
M. A. Papanikolas, \textit{Tannakian duality for Anderson-Drinfeld motives and algebraic independence of Carlitz logarithms}, Invent. Math. \textbf{171} (2008), no. 1, 123--174.

\bibitem{P22} %
M. A. Papanikolas, \textit{Hyperderivative power sums, Vandermonde matrices, and Carlitz multiplication coefficients}, J. Number Theory \textbf{232} (2022), 317--354.

\bibitem{PLogAlg} %
M. A. Papanikolas, \textit{Log-algebraicity on tensor powers of the Carlitz module and special values of Goss $L$-functions}, in preparation.

\bibitem{PR03} %
M. A. Papanikolas and N. Ramachandran, \textit{A Weil-Barsotti formula for Drinfeld modules}, J. Number Theory \textbf{98} (2003), no.~2, 407--431.

\bibitem{PZ17} %
M. A. Papanikolas and G. Zeng, \textit{Theta operators, Goss polynomials, and $v$-adic modular forms}, J. Th\'{e}or. Nombres Bordeaux \textbf{29} (2017), no.~3, 729--753.

\bibitem{Pellarin08} %
F. Pellarin, \textit{Aspects de l'ind\'{e}pendance alg\'{e}brique en caract\'{e}ristique non nulle}, S\'{e}m. Bourbaki, vol. 2006/2007. Ast\'{e}risque \textbf{317} (2008), no.~973, viii, 205--242.

\bibitem{Pellarin12} %
F. Pellarin, \textit{Values of certain $L$-series in positive characteristic}, Ann. of Math. (2) \textbf{176} (2012), no.~3, 2055--2093.

\bibitem{Pellarin20} %
F. Pellarin, \textit{An introduction to Mahler's method for transcendence and algebraic independence}, in: $t$-Motives: Hodge Structures, Transcendence and Other Motivic Aspects, Eur. Math. Soc., Z\"urich, 2020, pp.~297--349.

\bibitem{PellarinPerkins16} %
F. Pellarin and R. B. Perkins, \textit{On certain generating functions in positive characteristic}, Monatsh. Math. \textbf{180} (2016), no.~1, 123--144.

\bibitem{PellarinPerkins18} %
F. Pellarin and R. B. Perkins, \textit{Vectorial Drinfeld modular forms over Tate algebras}, Int. J. Number Theory \textbf{14} (2018), no.~6, 1729--1783.

\bibitem{Petrov15} %
A. Petrov, \textit{On hyperderivatives of single-cuspidal Drinfeld modular forms with $\mathcal{A}$-expansions}, J. Number Theory \textbf{149} (2015), 153--165.

\bibitem{Poonen96} %
B. Poonen, \textit{Fractional power series and pairings on Drinfeld modules}, J. Amer. Math. Soc. \textbf{9} (1996), no.~3, 783--812.

\bibitem{Robert} %
A. Robert, \textit{A Course in $p$-adic Analysis}, Springer, New York, 2000.

\bibitem{Sinha97} %
S. K. Sinha, \textit{Periods of $t$-motives and transcendence}, Duke Math. J. \textbf{88} (1997), no.~3, 465--535.

\bibitem{Taelman09} %
L. Taelman, \textit{Artin $t$-motifs}, J. Number Theory \textbf{129} (2009), no.~1, 142--157.

\bibitem{Thakur91}
D. S. Thakur, \textit{Gamma functions for function fields and Drinfeld modules}, Ann. of Math. (2) \textbf{134} (1991), no.~1, 25--64.

\bibitem{Thakur96}
D. S. Thakur, \textit{Transcendence of gamma values for $\mathbf{F}_q[T]$}, Ann. of Math. (2) \textbf{144} (1996), no.~1, 181--188.

\bibitem{Thakur} %
D. S. Thakur, \textit{Function Field Arithmetic}, World Scientific Publishing, River Edge, NJ, 2004.

\bibitem{US98} %
Y. Uchino and T. Satoh, \textit{Function field modular forms and higher-derivations}, Math. Ann. \textbf{311}
(1998), no. 3, 439--466.

\bibitem{Vincent15} %
C. Vincent, \textit{Weierstrass points on the Drinfeld modular curve $X_0(\mathfrak{p})$}, Res. Math. Sci. \textbf{2} (2015), 2:10, 40 pp.

\bibitem{Wade41} %
L. I. Wade, \textit{Certain quantities transcendental over $GF(p^n,x)$}, Duke Math. J. \textbf{8} (1941), no.~4, 701--720.

\bibitem{Yu85} %
J. Yu, \textit{Transcendence theory over function fields}, Duke Math. J. \textbf{52} (1985), no.~2, 517--527.

\bibitem{Yu86} %
J. Yu, \textit{Transcendence and Drinfeld modules}, Invent. Math. \textbf{83} (1986), no.~3, 507--517.

\bibitem{Yu89} %
J. Yu, \textit{Transcendence and Drinfeld modules: several variables}, Duke Math. J. \textbf{58} (1989), no.~3, 559--575.

\bibitem{Yu90} %
J.~Yu, \textit{On periods and quasi-periods of Drinfeld modules}, Compositio Math. \textbf{74} (1990), no.~3, 235--245.

\bibitem{Yu91} %
J. Yu, \textit{Transcendence and special zeta values in characteristic $p$}, Ann. of Math. (2) \textbf{134} (1991), no.~1, 1--23.

\bibitem{Yu97} %
J. Yu, \textit{Analytic homomorphisms into Drinfeld modules}, Ann. of Math. (2) \textbf{145} (1997), no.~2, 215--233.

\end{thebibliography}
\end{document}